\documentclass[pmlr,twocolumn,a4paper,10pt, sort&compress,nomaths]{jmlr}
\usepackage{isipta2019}
\usepackage[british]{babel}
\usepackage{tikz} 
\usepackage{pgfplotstable}
\usetikzlibrary{intersections}

\usepackage{mathrsfs}
\usepackage{enumitem}

\newcommand{\nats}{\mathbb{N}}
\newcommand{\nnegints}{\mathbb{Z}_{\geq0}}
\newcommand{\reals}{\mathbb{R}}
\newcommand{\nnegreals}{\mathbb{R}_{\geq0}}
\newcommand{\posreals}{\mathbb{R}_{>0}}

\DeclarePairedDelimiter{\br}{[}{]}
\let\set\relax
\DeclarePairedDelimiter{\set}{\{}{\}}
\DeclarePairedDelimiter{\abs}{\lvert}{\rvert}
\DeclarePairedDelimiter{\card}{\lvert}{\rvert}
\DeclarePairedDelimiter{\norm}{\lVert}{\rVert}

\newcommand\condbase[1][]{\:#1\lvert\:}
\let\cond\condbase

\DeclarePairedDelimiterX\pr[1](){\let\cond\scond #1}
\newcommand{\given}{\colon}

\newcommand{\ccpevents}{\mathcal{E}}
\newcommand{\ccpneevents}{\mathcal{E}_{\emptyset}}
\newcommand{\ccpdomain}{\mathcal{D}}

\newcommand{\prob}{P}
\newcommand{\prev}{E}
\newcommand{\prevdomain}{\mathcal{G}}
\newcommand{\lprev}{\underline{E}}
\newcommand{\uprev}{\overline{E}}
\newcommand{\lpoisprev}{\underline{P}}
\newcommand{\upoisprev}{\overline{P}}

\newcommand{\indic}{\mathbb{I}}
\newcommand{\indica}[1]{\mathbb{I}_{#1}}

\newcommand{\setofallcproc}{\mathbb{P}}
\newcommand{\setofcproc}{\mathscr{P}}
\newcommand{\setofconscproc}[1]{\mathbb{P}_{#1}}

\newcommand{\pth}{\omega}
\newcommand{\setofpths}{\Omega}

\newcommand{\cpfield}{\mathcal{F}}
\newcommand{\cpdomain}{\mathcal{D}_{\mathrm{CP}}}

\newcommand{\setoftseq}{\mathscr{U}}
\newcommand{\setofnetseq}{\mathscr{U}_{\emptyset}}

\newcommand{\stsp}{\mathscr{X}}

\newcommand{\pois}{\psi}

\newcommand{\llambda}{\underline{\lambda}}
\newcommand{\ulambda}{\overline{\lambda}}
\newcommand{\consistent}{\sim}

\newcommand{\setoffn}{\mathcal{L}}
\newcommand{\setofcafn}{\mathcal{L}^{\mathrm{c}}}
\newcommand{\setofbbfn}{\mathcal{K}_{\mathrm{b}}}
\newcommand{\genset}{\mathcal{Y}}

\newcommand{\lowx}{\underline{x}}
\newcommand{\upx}{\overline{x}}
\newcommand{\ltro}{\underline{Q}}

\newcommand{\genltro}{\underline{R}}
\newcommand{\lto}{\underline{T}}

\newcommand{\setoftrm}{\mathcal{Q}}
\newcommand{\trm}{Q}
\newcommand{\setoflambdaseq}{\mathcal{S}}

\jmlrworkshop{ISIPTA 2019}

\title{
  First Steps Towards an Imprecise Poisson Process
}

\author{
  \Name{Alexander Erreygers}\Email{Alexander.Erreygers@UGent.be}\\
  \Name{Jasper {De Bock}}\Email{Jasper.DeBock@UGent.be}\\
  \addr ELIS -- FL\hspace{.055555em}ip, Ghent University, Belgium
}

\begin{document}
\maketitle

\begin{abstract}
  The Poisson process is the most elementary continuous-time stochastic process that models a stream of repeating events.
  It is uniquely characterised by a single parameter called the rate.
  Instead of a single value for this rate, we here consider a rate interval and let it characterise two nested sets of stochastic processes.
  We call these two sets of stochastic process imprecise Poisson processes, explain why this is justified, and study the corresponding lower and upper (conditional) expectations.
  Besides a general theoretical framework, we also provide practical methods to compute lower and upper (conditional) expectations of functions that depend on the number of events at a single point in time.
\end{abstract}

\section{Introduction}
\label{sec:Introduction}
The \emph{Poisson process} is arguably one of the most basic stochastic processes.
At the core of this model is our subject, who is interested in something specific that occurs repeatedly over time, where time is assumed to be continuous.
For instance, our subject could be interested in the arrival of a customer to a queue, to give an example from queueing theory.
For the sake of brevity, we will call such a specific occurrence a \emph{Poisson-event},\footnote{
  We use the term ``Poisson-event'' rather than just ``event'' to avoid confusion with the standard usage of event in probability theory, where event refers to a subset of the sample space; we are indebted to an anonymous reviewer for pointing out this potential confusion, and to Gert de Cooman for suggesting the adopted terminology.
} whence our subject is interested in a stream of Poisson-events.
The time instants at which subsequent Poisson-events occur are uncertain to our subject, hence the need for a probabilistic model.
This set-up is not exclusive to queueing theory; it is also used in renewal theory and reliability theory, to name but a few applications.

There is a plethora of alternative but essentially equivalent characterisations of this Poisson process.
Some of the more well-known and basic characterisations are as the limit of the Bernoulli process \cite[Chapter~VI, Sections~5 and 6]{1968Feller} or as a sequence of mutually independent and exponentially distributed inter-event times \cite[Chapter~5, Section~3.A]{1975KarlinTaylor}.
An alternative way to look at the Poisson process is as a random dispersion of points in some general space---that need not be the real number line---see for instance \cite[Sections~2.1 and 2.2]{2013DaleyVereJones} or \cite[Chapter~2]{1993Kingman}.
More theoretically involved characterisations that follow our set-up are (i) as a counting process---a type of continuous-time stochastic process---that has independent and stationary increments that are Poisson distributed, see for example \cite[Definition~2.1.1]{1996Ross} or \cite[Section~3]{1999Sato}; (ii) as a counting process with a condition on the ``rate''---the rate of change of the probability of having a Poisson-event in a vanishingly small time period---see for instance \cite[Chapter~XVII, Section~2]{1968Feller} or \cite[Definition~2.1.2]{1996Ross}; (iii) as a stationary counting process that has no after-effects---see for instance \cite[Chapter~XVII, Section~2]{1968Feller} or \cite[Section~1]{1960Khintchine}; (iv) as a martingale through the Watanabe characterisation \cite[Theorem~2.3]{1964Watanabe}; or (v) as a pure-birth chain---a type of continuous-time Markov chain---with one birth rate, see for instance \cite[Section~2.4]{1997Norris}.
Many of these characterisations are actually equivalent, see for instance \cite[Theorem~2.4.3]{1997Norris} or \cite[Theorem~2.1.1]{1996Ross}.

Broadly speaking, these characterisations all make the same three assumptions: (i) \emph{orderliness}, in the sense that the probability that two or more Poisson-events occur at the same time is zero; (ii) independence, more specifically the absence of after-effects or \emph{Markovianity}; and (iii) \emph{homogeneity}.
It is essentially well-known that these three assumptions imply the existence of a parameter called the rate, and that this rate uniquely characterises the Poisson process.
We here weaken the three aforementioned assumptions.
First and foremost, we get rid of the implicit assumption that our subject's beliefs can be accurately modelled by a single stochastic process; instead, we assume that her beliefs only allow us to consider a \emph{set} of stochastic processes.
Specifically, we consider a rate interval instead of a precise value for the rate, and examine two distinct sets: (i) the set of all Poisson processes whose rate belongs to this rate interval; and (ii) the set of all processes that are orderly and ``consistent'' with the rate interval.
We then define lower and upper conditional expectations as the infimum and supremum of the conditional expectations with respect to the stochastic processes in these respective sets.
Aside from this general theoretical framework, we focus on computing the lower and upper expectation of functions that depend on the number of occurred Poisson-events at a single future time point.
For the set of Poisson processes, we show that this requires the solution of a one-parameter optimisation problem; for the second set, we show that this can be computed using backwards recursion.
Furthermore, we argue that both sets can be justifiably called imprecise Poisson processes: imprecise because their lower and upper expectations are not equal, and Poisson because their lower and upper expectations satisfy imprecise versions of the defining properties of the (precise) Poisson process.
The interested reader can find proofs for all our results in the Appendix.

Our approach is heavily inspired by the theory of imprecise continuous-time Markov chains \cite{2017Krak}.
For instance, we define the imprecise Poisson process via consistency with a rate interval, whereas \citet{2017Krak} use consistency with a set of transition rate matrices.
In the bigger picture, our contribution can therefore be seen as the first steps towards generalising the theory of imprecise continuous-time Markov chains from finite to countably infinite state spaces.

\section{Counting Processes in General}
\label{sec:Counting processes in general}
Recall from the Introduction that our subject is interested in the occurrences of a Poisson-event.
In this setting, it makes sense to consider the number of Poisson-events that have occurred from the initial time point \(t_{\mathrm{ini}}=0\) up to a time point~\(t\), where \(t\) is a non-negative real number.

\subsection{Counting Paths and the Sample Space}
The temporal evolution of the number of occurred Poisson-events is given by a counting path~\(\pth\colon\nnegreals\to\nnegints\); at any time point~\(t\) in \(\nnegreals\), \(\pth\pr{t}\) is the number of Poisson-events that have occurred from \(t_{\mathrm{ini}}=0\) up to \(t\).\footnote{We use \(\nnegints\) and \(\nats\) to denote the non-negative integers and natural numbers (or positive integers), respectively.
Furthermore, the real numbers, non-negative real numbers and positive real numbers are denoted by \(\reals\), \(\nnegreals{}\) and \(\posreals\), respectively.}
Since the actual temporal evolution of the number of occurred Poisson-events is unknown to the subject, we need a probabilistic model, more specifically a continuous-time stochastic process.
The sample space---the space of all possible outcomes---of this process is a set of counting paths, denoted by~\(\setofpths\).
One popular choice for \(\setofpths\) is the set of all c\`adl\`ag---right-continuous with left limits---counting paths, in this set-up usually also assumed to be non-decreasing.
However, our results do not require such a strong assumption.
Before we state our assumptions on~\(\setofpths\), we first introduce some notation.

In the remainder, we frequently use increasing sequences \(t_1, \dots, t_n\) of time points, that is, sequences \(t_1, \dots, t_n\) in \(\nnegreals\) of arbitrary length---that is, with \(n\) in \(\nats\)---such that \(t_i<t_{i+1}\) for all \(i\) in \(\set{1, \dots, n-1}\).
For the sake of brevity, we follow \cite[Section~2.1]{2017Krak} in denoting such a sequence by \(u\).
We collect all increasing---but possibly empty---sequences of time points in \(\setoftseq\), and let \(\setofnetseq\coloneqq\setoftseq\setminus\set{\emptyset}\).
Observe that as a sequence of time points \(u\) in \(\setoftseq\) is just a finite and ordered set of non-negative real numbers, we can perform common set-theoretic operations on them like unions.
In order to lighten our notation, we identify the single time point~\(t\) with a sequence; as such, we can use \(u\cup t\) as a notational shorthand for \(u \cup\set{t}\).
Also, a statement of the form \(\max u < t\) is taken to be true if \(u=\emptyset\); see for instance Lemma~\ref{lem:Element of cpfield_u}.
With this convention, for any \(t\) in \(\nnegreals\), we let \(\setoftseq_{<t}\coloneqq\set{u\in\setoftseq\colon \max u<t}\) be the set of all sequences of time points of which the last time point precedes \(t\).
Note that if \(t=0\), then there is no such non-empty sequence and so \(\setoftseq_{<t}=\set{\emptyset}\).

In order to better distinguish between general non-negative integers and counts, we let \(\stsp\coloneqq\nnegints\).
For any \(u=t_1, \dots, t_n\) in \(\setofnetseq\), we let \(\stsp_u\) be the set of all \(n\)-tuples \(x_u=\pr{x_{t_1}, \dots, x_{t_n}}\) of non-negative integers that are non-decreasing:
\begin{equation}
\label{eqn:stsp_u}
  \stsp_u
  \coloneqq \set{\pr{x_{t_1}, \dots, x_{t_n}} \in \stsp^n \colon x_{t_1} \leq \cdots \leq x_{t_n}}.
\end{equation}
If \(u\) is the empty sequence~\(\emptyset\), then we let \(\stsp_u=\stsp_\emptyset\) denote the singleton containing the empty tuple, denoted by \(x_\emptyset\).

With all this notation in place, we can now formally state our requirements on \(\setofpths\):
\begin{enumerate}[label=\upshape{}A\arabic*., ref=\upshape(A\arabic*), leftmargin=*]
  \item \label{Paths:Non-decreasing}
  \(\pr{\forall \pth\in\setofpths}\pr{\forall t,\Delta\in\nnegreals}~\pth\pr{t}\leq\pth\pr{t+\Delta}\);
  \item \label{Paths:Existence for any instantiation}
  \(\pr{\forall u\in\setofnetseq}\pr{\forall x_u\in\stsp_u}\pr{\exists\pth\in\setofpths}\pr{\forall t\in u}~\pth\pr{t}=x_t\).
\end{enumerate}
Assumption~\ref{Paths:Non-decreasing} ensures that all paths are non-decreasing, which is essential if we interpret \(\pth\pr{t}\) as the number of Poisson-events that have occurred up to time~\(t\).
Assumption~\ref{Paths:Existence for any instantiation} ensures that the set~\(\setofpths\) is sufficiently large, essentially ensuring that the finitary events of Equation~\eqref{eqn:finitary event} further on are non-empty.

\subsection{Coherent Conditional Probabilities}
We follow \citet{2017Krak} in using the framework of coherent conditional probabilities to model the beliefs of our subject.
What follows is a brief introduction to coherent conditional probabilities; we refer to \cite{1985Regazzini} and \cite[Section~4.1]{2017Krak} for a more detailed exposition.
For any sample space---that is, a non-empty set---\(S\), we let \(\ccpevents\pr{S}\) denote the set all events---that is, subsets of \(S\)---and let \(\ccpneevents\pr{S}\coloneqq\ccpevents\pr{S}\setminus\set{\emptyset}\) denote the set of all non-empty events.
Before we introduce coherent conditional probabilities, we first look at full conditional probabilities.
\begin{definition}
\label{def:Full conditional probability}
  Let \(S\) be a sample space.
  A \emph{full conditional probability}~\(\prob\) is a real-valued map on \(\ccpevents\pr{S} \times \ccpevents_{\emptyset}\pr{S}\) such that, for all \(A\), \(B\) in \(\ccpevents\pr{S}\) and \(C\), \(D\) in \(\ccpneevents\pr{S}\),
  \begin{enumerate}[label=\upshape{P\arabic*.}, leftmargin=*, ref=\upshape{(P\arabic*)}, leftmargin=*]
    \item \label{LOP:geq 0}
    \(\prob\pr{A \cond C} \geq 0\);
    \item \label{LOP:1 if C in A}
    \(\prob\pr{A \cond C} = 1\) if \(C \subseteq A\);
    \item \label{LOP:Additive if disjoint}
    \(\prob\pr{A \cup B \cond C} = \prob\pr{A \cond C} + \prob\pr{B \cond C}\) if \(A \cap B = \emptyset\);
    \item \label{LOP:Bayes rule}
    \(\prob\pr{A \cap D \cond C} = \prob\pr{A \cond D \cap C} \prob\pr{D \cond C}\) if \(D\cap C\neq\emptyset\).
  \end{enumerate}
\end{definition}
Note that \ref{LOP:geq 0}--\ref{LOP:Additive if disjoint} just state that \(\prob\pr{\cdot\cond C}\) is a finitely-additive probability measure, and that \ref{LOP:Bayes rule} is a multiplicative version of Bayes' rule.
We use the adjective \emph{full} because the domain of \(\prob\) is \(\ccpevents\pr{S} \times \ccpevents_{\emptyset}\pr{S}\).
Next, we move to domains that are a subset of \(\ccpevents\pr{S} \times \ccpevents_{\emptyset}\pr{S}\).
\begin{definition}
\label{def:Coherent conditional probability in text}
  Let \(S\) be a sample space.
  A coherent conditional probability~\(\prob\) is a real-valued map on \(\ccpdomain{}\subseteq \ccpevents\pr{S} \times \ccpevents_{\emptyset}\pr{S}\) that can be extended to a full conditional probability.
\end{definition}
Important to emphasise here is that simply demanding that \ref{LOP:geq 0}--\ref{LOP:Bayes rule} hold on the domain~\(\ccpdomain{}\) is in general \emph{not} sufficient to guarantee that \(\prob\) can be extended to a full conditional probability.
A necessary and sufficient condition for the existence of such an extension can be found in \cite[Theorem~3]{1985Regazzini} or \cite[Corollary~4.3]{2017Krak}, but we refrain from stating it here because of its technical nature.
We here only mention that this so-called \emph{coherence} condition---hence explaining the use of the adjective coherent---has an intuitive betting interpretation, and that checking this condition is usually feasible while explicitly constructing the full conditional extension is typically not; this is extremely useful when constructing proofs.
Another strong argument for using coherent conditional probabilities is that they can always be extended to a coherent conditional probability on a larger domain \cite[Theorem~4]{1985Regazzini}.
This too is an essential tool in the proof of many of our main results, including Theorems~\ref{the:Pois transition probabilities are poisson distributed and the converse}, \ref{the:Lower conditional expectation of bounded function with respect to all consistent processes} and \ref{the:Lower and upper expectation of functions that do not increase to fast}.

\subsection{Events and Fields}
For any \(v=t_1, \dots, t_n\) in \(\setofnetseq\) and \(B\subseteq\stsp_v\), we define the \emph{finitary event}
\begin{equation}
\label{eqn:finitary event}
  \pr{X_v\in B}
  \coloneqq \set{\pth\in\setofpths\colon\pr{\pth\pr{t_1}, \dots, \pth\pr{t_n}}\in B}.
\end{equation}
Furthermore, we also let \(\pr{X_\emptyset=x_\emptyset}\coloneqq\setofpths{}\eqqcolon\pr{X_\emptyset\in\stsp_\emptyset}\).
Then for any \(u\) in \(\setoftseq\), we let \(\cpfield_u\) be the field of events---or algebra of sets---generated by the finitary events for all sequences with time points in or succeeding \(u\):
\begin{multline}
\label{eqn:cpfield_u}
  \cpfield_u
  \coloneqq \langle \set{\pr{X_v\in B} \colon v\in\setoftseq, B\subseteq\stsp_v, \\ \pr{\forall t\in v}~t\in u\cup [\max u, +\infty)} \rangle.
\end{multline}
\begin{lemma}
\label{lem:Element of cpfield_u}
  Consider some \(u\) in \(\setoftseq\) and \(A\) in \(\cpfield_u\).
  Then there is some \(v\) in \(\setoftseq\) with \(\min v>\max u\) and some \(B\subseteq\stsp_w\) with \(w\coloneqq u\cup v\) such that \(A=\pr{X_w \in B}\).
\end{lemma}

\subsection{Counting Processes as Coherent Conditional Probabilities}
From here on, we focus on coherent conditional probabilities with the domain
\begin{equation*}
  \cpdomain
  \coloneqq \set{\pr{A, X_u=x_u} \colon u\in\setoftseq, A\in\cpfield_u, x_u\in\stsp_u},
\end{equation*}
which essentially consists of future events conditional on the number of occurred Poisson-events at specified past time-points.
The rationale behind this domain is twofold.
First and foremost, it is sufficiently large to make most inferences that one is usually interested in.
For example, this domain allows us to compute---tight lower and upper bounds on---the expectation of a real-valued function on the number of occurred Poisson-events at a single future time point, as we will see in Section~\ref{sec:Computing lower expectations}.
Second, it guarantees that every rate corresponds to a unique Poisson process, as we will see in Section~\ref{sec:Poisson processes in particular}.
\begin{definition}
\label{def:Counting process}
  A \emph{counting process}~\(\prob\) is a coherent conditional probability on \(\cpdomain\) such that
  \begin{enumerate}[label=\upshape{}CP\arabic*., ref=\upshape(CP\arabic*), leftmargin=*]
    \item \label{CP:Start at 0}
    \(\prob\pr{X_0=0}=1\);
    \item \label{CP:Orederliness}
    for all \(t\) in \(\nnegreals\), \(u\) in \(\setoftseq_{<t}\) and \(\pr{x_u, x}\) in \(\stsp_{u\cup t}\),
    \begin{equation*}
      \lim_{\Delta\to0^+} \frac{\prob\pr{X_{t+\Delta}\geq x+2\cond X_u=x_u, X_t=x}}{\Delta}
      = 0
    \end{equation*}
    and, if \(t>0\),
    \begin{equation*}
      \lim_{\Delta\to0^+} \frac{\prob\pr{X_t\geq x+2\cond X_u=x_u, X_{t-\Delta}=x}}{\Delta}
      = 0.
    \end{equation*}
  \end{enumerate}
\end{definition}
The second requirement~\ref{CP:Orederliness} is---our version of---the \emph{orderliness} property that we previously mentioned in the Introduction.
In essence, it ensures that the probability that two or more Poisson-events occur at the same time is zero.
We collect all counting processes in the set~\(\setofallcproc\).

\subsection{Conditional Expectation with Respect to a Counting Process}
\label{ssec:Conditional expectation with respect to a counting process}
For any counting process~\(\prob\), we let \(\prev_\prob\) denote the associated (conditional) expectation, defined in the usual sense as an integral with respect to the measure~\(\prob\)---see for instance \cite[Theorem~6]{1985Regazzini} or \cite[Section~15.10.1]{2014LowerPrevisions}.

Let \(\setofbbfn\pr{\setofpths}\) denote the set of all real-valued functions on \(\setofpths\) that are bounded below.\footnote{
  Note that we could just as well consider arbitrary real-valued functions instead of restricting ourselves to bounded-below functions.
  Our main reason for doing so is that this facilitates a more elegant treatment.
  Furthermore, many functions of practical interest are bounded-below.
}
Fix some \(u\) in \(\setoftseq\).
Then \(f\) in \(\setofbbfn\pr{\setofpths}\) is \emph{\(\cpfield_u\)-measurable} if for all \(\alpha\) in \([\inf f, +\infty)\), the level set \(\set{f>\alpha}\coloneqq\set{\pth\in\setofpths\colon f\pr{\pth}>\alpha}\) is an element of \(\cpfield_u\).
We collect all such \(\cpfield_u\)-measurable functions in \(\prevdomain_u\).

The (conditional) expectation~\(\prev_\prob\) has domain
\begin{multline*}
  \prevdomain
  \coloneqq\{\pr{f, X_u=x_u}\in \setofbbfn\pr{\setofpths}\times\ccpneevents\pr{\setofpths} \colon \\ u\in\setoftseq, x_u\in\stsp_u, f\in\prevdomain_u\}.
\end{multline*}
For any \(\pr{f, X_u=x_u}\) in \(\prevdomain\), we have
\begin{multline*}
  \prev_\prob\pr{f\cond X_u=x_u} \\
  \coloneqq \inf f + \int_{\inf f}^{\sup f} \prob\pr{\set{f>\alpha}\cond X_u=x_u} \,\mathrm{d}\alpha,
\end{multline*}
where the integral is a---possibly improper---Riemann integral.
Note that this integral always exists because \(\prob\pr{\set{f>\alpha}\cond X_u=x_u}\) is a non-increasing function of \(\alpha\).
This expression simplifies if \(f\) is an \emph{\(\cpfield_u\)-simple function}.
To define these, we let \(\indica{A}\colon\setofpths\to\reals\) denote the indicator of an event~\(A\subseteq\setofpths\), defined for all \(\pth\) in \(\setofpths\) as \(\indica{A}\pr{\pth}\coloneqq1\) if \(\pth\in A\) and \(0\) otherwise.
We then say that \(f\) is \(\cpfield_u\)-simple if it can be written as \(f=\sum_{i=1}^n a_i \indica{A_i}\), with \(n\) in \(\nats\) and, for all \(i\) in \(\set{1, \dots, n}\), \(a_i\) in \(\reals\) and \(A_i\) in \(\cpfield_u\).
In this case, the integral expression reduces to
\begin{equation}
\label{eqn:Expectation of a simple function}
  \prev_\prob\pr{f\cond X_u=x_u}
  = \sum_{i=1}^n a_i \prob\pr{A_i\cond X_u=x_u}.
\end{equation}
For unconditional expectations, we have that
\begin{equation*}
  \prev\pr{\cdot}
  \coloneqq \prev_\prob\pr{\cdot\cond \setofpths}
  = \prev_\prob\pr{\cdot\cond X_\emptyset=x_\emptyset}
  = \prev_\prob\pr{\cdot\cond X_0=0},
\end{equation*}
where the final equality holds due to \ref{CP:Start at 0}.
Therefore, in the remainder, we can restrict ourselves to expectations of the form \(\prev_\prob\pr{\cdot\cond X_u=x_u, X_t=x}\), as \(\prev\pr{\cdot}\) corresponds to the case \(u=\emptyset\), \(t=0\) and \(x=0\).

\section{The Poisson Process in Particular}
\label{sec:Poisson processes in particular}
We now turn to the most well-known counting process, namely the Poisson process.
As explained in the Introduction, there are plenty of alternative characterisations of the Poisson process.
The following definition turns out to capture all its essential properties in our framework.
\begin{definition}
\label{def:Poisson process}
  A \emph{Poisson process}~\(\prob\) is a counting process such that, for all \(t, \Delta\) in \(\nnegreals\), \(u\) in \(\setoftseq_{<t}\), \(\pr{x_u, x}\) in \(\stsp_{u\cup t}\) and \(y\) in \(\stsp\) with \(y\geq x\),
  \begin{enumerate}[label=\upshape{}PP\arabic*., ref=\upshape(PP\arabic*), leftmargin=*]
    \item \label{Pois:Markovian}
    \(\prob\pr{X_{t+\Delta}=y\cond X_u=x_u, X_t=x} = \prob\pr{X_{t+\Delta}=y\cond X_t=x}\);
    \item \label{Pois:State-homogeneous}
    \(\prob\pr{X_{t+\Delta}=y\cond X_t=x} = \prob\pr{X_{t+\Delta}=y-x\cond X_t=0}\);
    \item \label{Pois:Time-homogeneous}
    \(\prob\pr{X_{t+\Delta}=y\cond X_t=x} = \prob\pr{X_{\Delta}=y\cond X_0=x}\).
  \end{enumerate}
\end{definition}
The first condition~\ref{Pois:Markovian} states that the Poisson process is Markovian, while conditions \ref{Pois:State-homogeneous} and \ref{Pois:Time-homogeneous} state that the Poisson process is homogeneous.
Note that---unlike many of the characterisations mentioned in the Introduction---we do \emph{not} impose that the transition probabilities are Poisson distributed, nor do we impose some value for the ``rate''.
It was already observed by \citet[Chapter~XVII, Section~2, Footnote~4]{1968Feller} and \citet[Sections~1 and 2]{1960Khintchine} that assuming---their version of---\ref{Pois:Markovian}--\ref{Pois:Time-homogeneous} is sufficient to obtain the Poisson process.
Our results basically extend these characterisations to our framework for counting processes using coherent conditional probabilities.

First and foremost, we obtain that the transition probabilities are Poisson distributed, hence explaining the name of the process.
\begin{theorem}
\label{the:Pois transition probabilities are poisson distributed and the converse}
  Consider a Poisson process~\(\prob\).
  Then there is a rate~\(\lambda\) in \(\nnegreals\) such that, for all \(t, \Delta\) in \(\nnegreals\), \(u\) in \(\setoftseq_{<t}\), \(\pr{x_u, x}\) in \(\stsp_{u\cup t}\) and \(y\) in \(\stsp\),
  \begin{multline}
  \label{eqn:Poisson transition probabilities}
    \prob\pr{X_{t+\Delta}=y \cond X_u=x_u, X_t=x} \\
    = \begin{cases}
      \pois_{\lambda \Delta}\pr{y-x} &\text{if } y\geq x, \\
      0 &\text{otherwise,}
    \end{cases}
  \end{multline}
  where \(\pois_{\lambda \Delta}\) is the Poisson distribution with parameter \(\lambda \Delta\), defined for all \(k\) in \(\nnegints\) as
  \begin{equation*}
    \pois_{\lambda \Delta}\pr{k}
    \coloneqq e^{-\lambda\Delta}\frac{\pr{\lambda\Delta}^k}{k!}.
  \end{equation*}
  Conversely, for every~\(\lambda\) in \(\nnegreals\), there is a unique coherent conditional probability~\(\prob\) on \(\cpdomain{}\) that satisfies~\ref{CP:Start at 0} and Equation~\eqref{eqn:Poisson transition probabilities}, and this \(\prob\) is a Poisson process.
\end{theorem}
Theorem~\ref{the:Pois transition probabilities are poisson distributed and the converse} might seem somewhat trivial, but its proof is surprisingly lengthy.
Note that it also establishes that any rate~\(\lambda\) gives rise to a unique Poisson process, so in the remainder we can talk of \emph{the} Poisson process with rate~\(\lambda\).
Finally, it has the following obvious corollary.
\begin{corollary}
\label{cor:Poisson has rate lambda}
  Consider a Poisson process~\(\prob\).
  Then there is a rate~\(\lambda\) in \(\nnegreals\) such that, for all \(t\) in \(\nnegreals\), \(u\) in \(\setoftseq_{<t}\) and \(\pr{x_u, x}\) in \(\stsp_{u\cup t}\),
  \begin{equation}
  \label{eqn:Pois:Right rate}
    \lim_{\Delta\to0^+}\frac{\prob\pr{X_{t+\Delta}=x+1\cond X_u=x_u, X_t=x}}{\Delta}
    = \lambda
  \end{equation}
  and, if \(t>0\),
  \begin{equation}
  \label{eqn:Pois:Left rate}
    \lim_{\Delta\to0^+}\frac{\prob\pr{X_t=x+1\cond X_u=x_u, X_{t-\Delta}=x}}{\Delta}
    = \lambda.
  \end{equation}
\end{corollary}

We end our discussion of Poisson processes with the following result, which actually is a---not entirely immediate---consequence of Theorem~\ref{the:Lower conditional expectation of bounded function with respect to all consistent processes} further on.
\begin{theorem}
\label{the:Counting process with rate lambda is Poisson}
  Consider a counting process~\(\prob\).
  If there is a rate~\(\lambda\) in \(\nnegreals\) such that \(\prob\) satisfies Equations~\eqref{eqn:Pois:Right rate} and \eqref{eqn:Pois:Left rate}, then \(\prob\) is the Poisson process with rate~\(\lambda\).
\end{theorem}

\section{Sets of Counting Processes}
\label{sec:Sets of counting processes}
Instead of considering a single counting process, we now study \emph{sets} of counting processes.
With any subset~\(\setofcproc\) of \(\setofallcproc\), we associate a \emph{lower expectation}
\begin{equation}
\label{eqn:Lower expectation with respect to set of counting processes}
  \lprev_{\setofcproc}\pr{\cdot\cond\cdot}
  \coloneqq \inf\set{\prev_\prob\pr{\cdot\cond\cdot} \colon \prob\in\setofcproc}
\end{equation}
and, similarly, an \emph{upper expectation}
\begin{equation}
\label{eqn:Upper expectation with respect to set of counting processes}
  \uprev_{\setofcproc}\pr{\cdot\cond\cdot}
  \coloneqq \sup\set{\prev_\prob\pr{\cdot\cond\cdot} \colon \prob\in\setofcproc}.
\end{equation}
Since the expectation~\(\prev_\prob\) associated with any counting process~\(\prob\) in \(\setofcproc\) has domain~\(\prevdomain{}\), \(\lprev_{\setofcproc}\) and \(\uprev_{\setofcproc}\) are well-defined on the same domain~\(\prevdomain{}\).
Observe that for any \(\pr{f, X_u=x_u}\) in \(\prevdomain\) such that \(f\) is bounded, the lower and upper expectations are conjugate in the sense that \(\uprev_{\setofcproc}\pr{f\cond X_u=x_u}=-\lprev_{\setofcproc}\pr{-f\cond X_u=x_u}\).
Therefore, it suffices to study one of the two if only considering bounded functions; we will focus on lower expectations in the remainder.

\subsection{The Obvious Imprecise Poisson Process}
\label{ssec:Set of Poisson processes}
From here on, we consider a closed interval~\(\Lambda\coloneqq\br{\llambda, \ulambda}\subset\nnegreals\) of rates instead of a single value for the rate~\(\lambda\).
In order not to unnecessarily repeat ourselves, we fix one rate interval~\(\Lambda\) that we use throughout the remainder.
Due to Theorem~\ref{the:Pois transition probabilities are poisson distributed and the converse}, there is one obvious set of counting processes that is entirely characterised by this rate interval~\(\Lambda\): the set
\begin{equation*}
  \setofconscproc{\Lambda}^\star
  \coloneqq \set{\prob_\lambda\colon \lambda\in\Lambda}
\end{equation*}
that consists of all Poisson processes with rate in this interval, where \(\prob_\lambda\) denotes the Poisson process with rate \(\lambda\).

The lower and upper expectation associated with this set~\(\setofconscproc{\Lambda}^\star\) according to Equations~\eqref{eqn:Lower expectation with respect to set of counting processes} and \eqref{eqn:Upper expectation with respect to set of counting processes} are denoted by \(\lprev_\Lambda^\star\) and \(\uprev_\Lambda^\star\), respectively.
It is clear that by construction, determining \(\lprev_\Lambda^\star\pr{f\cond X_u=x_u}\) and/or \(\uprev_\Lambda^\star\pr{f\cond X_u=x_u}\) boils down to solving a one-parameter optimisation problem: one has to minimise and/or maximise \(\prev_{\prob_\lambda}\pr{f\cond X_u=x_u}\)---the conditional expectation of \(f\) with respect to the Poisson process with rate \(\lambda\)---with respect to all values of \(\lambda\) in the rate interval~\(\Lambda\).
For some specific functions~\(f\), see for example Proposition~\ref{prop:Monotone bounded functions:equality between conditional expectations} further on, this one-parameter optimisation problem can be solved analytically.
For more involved functions, the optimisation problem has to be solved numerically, for instance by evaluating \(\prev_{\prob_\lambda}\pr{f\cond X_u=x_u}\) over a (sufficiently fine) grid of values of \(\lambda\) in the rate interval~\(\Lambda\), where \(\prev_{\prob_\lambda}\pr{f\cond X_u=x_u}\) might also have to be numerically approximated.

\subsection{A More Involved Imprecise Poisson Process}
\label{ssec:Set of consistent processes}
A second set of counting processes characterised by the rate interval~\(\Lambda\) is inspired by Theorem~\ref{the:Counting process with rate lambda is Poisson}.
This theorem suggests that the dynamics of a counting process are captured by the rate---that is, the limit expressions in Equations~\eqref{eqn:Pois:Right rate} and \eqref{eqn:Pois:Left rate} of Corollary~\ref{cor:Poisson has rate lambda}.
Essential to our second characterisation is the notion of consistency.
\begin{definition}
\label{def:Consistency of counting processes}
  A counting process~\(\prob\) is \emph{consistent} with the rate interval~\(\Lambda\), denoted by \(\prob\consistent\Lambda\), if for all \(t\) in \(\nnegreals\), \(u\) in \(\setoftseq_{<t}\) and \(\pr{x_u,x }\) in \(\stsp_{u\cup t}\),
  \begin{multline}
  \label{eqn:Consistent rate from right}
    \llambda
    \leq \liminf_{\Delta\to0^+}\frac{\prob\pr{X_{t+\Delta}=x+1\cond X_u=x_u, X_t=x}}{\Delta} \\
    \leq \limsup_{\Delta\to0^+}\frac{\prob\pr{X_{t+\Delta}=x+1\cond X_u=x_u, X_t=x}}{\Delta}
    \leq \ulambda
  \end{multline}
  and, if \(t>0\),
  \begin{multline}
  \label{eqn:Consistent rate from left}
    \llambda
    \leq  \liminf_{\Delta\to0^+}\frac{\prob\pr{X_t=x+1\cond X_u=x_u, X_{t-\Delta}=x}}{\Delta} \\
    \leq \limsup_{\Delta\to0^+}\frac{\prob\pr{X_t=x+1\cond X_u=x_u, X_{t-\Delta}=x}}{\Delta}
    \leq \ulambda.
  \end{multline}
\end{definition}
We let
\begin{equation*}
  \setofconscproc{\Lambda}
  \coloneqq\set{\prob\in\setofallcproc{}\colon \prob\consistent{}\Lambda}
\end{equation*}
denote the set of all \emph{counting} processes that are consistent with the rate interval~\(\Lambda\).
Observe that, as every Poisson process is a counting process,
\begin{equation}
\label{eqn:Relation between sets of consistent processes}
  \setofconscproc{\Lambda}^\star
  \subseteq \setofconscproc{\Lambda}.
\end{equation}
It is essential to realise that \(\setofconscproc{\Lambda}^\star\) is \emph{not} equal to \(\setofconscproc{\Lambda}\), at least not in general.
Indeed, the set \(\setofconscproc{\Lambda}\) will contain counting processes that have much more exotic dynamics than Poisson processes, in the sense that they need not be Markovian nor homogeneous.
However, if \(\Lambda\) is equal to the degenerate interval~\(\br{\lambda, \lambda}\), then it follows from Theorem~\ref{the:Counting process with rate lambda is Poisson} that
\begin{equation}
\label{eqn:Singleton of Poisson with rate lambda}
  \setofconscproc{\Lambda}^\star
  = \setofconscproc{\Lambda}
  = \set{\prob_\lambda},
\end{equation}
where \(\prob_\lambda\) is the Poisson process with rate~\(\lambda\), as before.
Therefore, both \(\setofconscproc{\Lambda}\) and \(\setofconscproc{\Lambda}^\star\) are proper generalisations of the Poisson process.

We let \(\lprev_\Lambda\) and \(\uprev_\Lambda\) denote the lower and upper expectations associated with the set~\(\setofconscproc{\Lambda}\) according to Equations~\eqref{eqn:Lower expectation with respect to set of counting processes} and \eqref{eqn:Upper expectation with respect to set of counting processes}.
It is an immediate consequence of Equations~\eqref{eqn:Lower expectation with respect to set of counting processes}, \eqref{eqn:Upper expectation with respect to set of counting processes} and \eqref{eqn:Relation between sets of consistent processes} that
\begin{equation}
\label{eqn:Relation between lower expectations for consistent processes}
  \lprev_{\Lambda}\pr{\cdot\cond\cdot}
  \leq \lprev_{\Lambda}^\star\pr{\cdot\cond\cdot}
  \leq \uprev_{\Lambda}^\star\pr{\cdot\cond\cdot}
  \leq \uprev_{\Lambda}\pr{\cdot\cond\cdot}.
\end{equation}
The remainder of this contribution is concerned with computing these lower and upper expectations for a specific type of functions, with a particular focus on the outer ones.

We end this section by mentioning that \(\setofconscproc{\Lambda}^\star\) and \(\setofconscproc{\Lambda}\) are not the only two sets of counting processes that are of potential interest, but they are---to some extent---the two most extreme sets.
One set of counting process that lies in between the two is that of the time-inhomogeneous Poisson processes---see for instance \cite[Section~5]{1960Khintchine} or \cite[Section~2.4]{1996Ross}---that are consistent with the rate interval \(\Lambda\).
In order not to unnecessarily complicate our exposition, we have chosen to limit ourselves to the two extreme cases.

\section{The Poisson Generator and Its Corresponding Semi-Group}
\label{sec:The Poisson semi-group}
Our method for computing lower expectations is based on the method used in the theory of imprecise continuous-time Markov chains~\cite{2017Krak}.
Essential to this method of \citet{2017Krak} is a semi-group of ``lower transition operators'' that is generated by a ``lower transition rate operator''.
In Section~\ref{ssec:The Poisson generator}, we extend their method for generating this semi-group to a countably infinite state space, be it only for one specific type of lower transition rate operator.
First, however, we introduce some necessary concepts and terminology.

\subsection{Functions, Operators and Norms}
\label{ssec:Functions, operators and norms}
Consider some non-empty ordered set~\(\genset\) that is at most countably infinite, and let \(\setoffn\pr{\genset}\) be the set of all bounded real-valued functions on \(\genset\).
Observe that \(\setoffn\pr{\genset}\) is clearly a vector space.
Even more, it is well-known that this vector space is complete under the supremum norm
\begin{equation*}
  \norm{f}
  \coloneqq \sup\set{\abs{f\pr{x}}\colon x\in \genset}
  \quad\text{for all } f\in\setoffn\pr{\genset}.
\end{equation*}

A \emph{transformation} is any operator \(A\colon\setoffn\pr{\genset}\to\setoffn\pr{\genset}\).
Such a transformation~\(A\) is \emph{non-negatively homogeneous} if, for all \(f\) in \(\setoffn\pr{\genset}\) and \(\gamma\) in \(\nnegreals\), \(A\pr{\gamma f} = \gamma Af\).
The supremum norm induces an operator norm for non-negatively homogeneous transformations~\(A\):
\begin{equation*}
  \norm{A}
  \coloneqq\sup\set{\norm{A f} \colon f\in\setoffn\pr{\genset}, \norm{f}=1};
\end{equation*}
see \cite{2016DeBock} for a proof that this is indeed a norm.
An important non-negatively homogeneous transformation is the \emph{identity map}~\(I\) that maps any \(f\) in \(\setoffn\pr{\genset}\) to itself.

\subsection{The Poisson Generator}
\label{ssec:The Poisson generator}
A non-negatively homogeneous transformation that will be essential in the remainder is the \emph{Poisson generator}~\(\ltro\colon\setoffn\pr{\stsp}\to\setoffn\pr{\stsp}\) associated with the rate interval~\(\Lambda\), defined for all \(f\) in \(\setoffn\pr{\stsp}\) and \(x\) in \(\stsp\) as
\begin{equation*}
  \br{\ltro f}\pr{x}
  \coloneqq \min\set{\lambda f\pr{x+1}-\lambda f\pr{x}\colon \lambda\in\br{\llambda, \ulambda}}.
\end{equation*}

Fix any \(t, s\) in \(\nnegreals\) with \(t\leq s\).
If \(t<s\), then we let \(\setoftseq_{\br{t, s}}\) denote the set of all non-empty and increasing sequences of time points~\(t_0, \dots, t_n\) that start with \(t_0=t\) and end with \(t_n=s\).
For any sequence \(u\) in this set~\(\setoftseq_{\br{t, s}}\), we let
\begin{equation}
\label{eqn:PoisGen:Phi_u}
  \Phi_u
  \coloneqq \prod_{i=1}^n \pr{I+\Delta_i \ltro},
\end{equation}
where for any \(i\) in \(\set{1, \dots, n}\), we denote the difference between the consecutive time points \(t_i\) and \(t_{i-1}\) by \(\Delta_i\coloneqq t_{i}-t_{i-1}\).
In the remainder, we let \(\sigma\pr{u}\coloneqq\max\set{\Delta_i\colon i\in\set{1, \dots, n}}\) be the largest of these time differences.
If \(t=s\), then we let \(\setoftseq_{\br{t,s}}\coloneqq\set{t}\), \(\sigma\pr{t}\coloneqq0\) and \(\Phi_t\coloneqq I\).

The Poisson generator~\(\ltro\) generates a family of transformations, as is evident from the following result.
This result is very similar to \cite[Corollary~7.11]{2017Krak}, which establishes an analoguous result for imprecise Markov chains with finite state spaces; it should therefore not come as a surprise that their proofs are largely similar as well.
\begin{theorem}
\label{the:Transformation induced by PoisGen in text}
  Fix any \(t, s\) in \(\nnegreals\) with \(t\leq s\).
  For any sequence \(\set{u_i}_{i\in\nats}\) in \(\setoftseq_{\br{t,s}}\) such that \(\lim_{i\to+\infty}\sigma\pr{u_i}=0\), the corresponding sequence \(\set{\Phi_{u_i}}_{i\in\nats}\) converges to a unique non-negatively homogeneous transformation that does not depend on the chosen sequence \(\set{u_i}_{i\in\nats}\).
\end{theorem}
For any \(t, s\) in \(\nnegreals\) with \(t\leq s\), Theorem~\ref{the:Transformation induced by PoisGen in text} allows us to define the non-negatively homogeneous transformation
\begin{equation}
\label{eqn:LTO gen by PoisGen}
  \lto_t^s
  \coloneqq \lim_{\sigma\pr{u}\to0} \set{\Phi_u \colon u\in\setoftseq_{\br{t,s}}},
\end{equation}
where this unconventional notation for the limit denotes the unique limit mentioned in Theorem~\ref{the:Transformation induced by PoisGen in text}.
The family of transformations thus defined has some very interesting properties: in the Appendix, we prove that for any \(t, s\) in \(\nnegreals\) with \(t\leq s\), \(f, g\) in \(\setoffn\pr{\stsp}\) and \(\gamma\) in \(\nnegreals\),
\begin{enumerate}[label=\upshape{}SG\arabic*., ref=\upshape(SG\arabic*), leftmargin=*]
  \item \label{LTO in text:Non-negateive homogeneity} \(\lto_t^s\pr{\gamma f}=\gamma\lto_t^s f\);
  \item \(\lto_t^s\pr{f+g}\geq\lto_t^s f+\lto_t^s g\);
  \item \label{LTO in text:Bound} \(\lto_t^s f\geq\inf f\).
\end{enumerate}
We furthermore prove that this family forms a time-homogeneous semi-group, in the sense that
\begin{enumerate}[resume*]
  \item \label{LTO in text:Identity} \(\lto_t^t = I\);
  \item \(\lto_t^s = \lto_t^r \lto_r^s\) for all \(r\) in \(\nnegreals\) with \(t\leq r\leq s\);
  \item \label{LTO in text:Time-invariance} \(\lto_t^s = \lto_0^{s-t}\).
\end{enumerate}

While the induced transformation~\(\lto_t^s\) is interesting in its own right, we will be mainly interested in (a single component of) the image \(\lto_t^s f\) of some bounded function~\(f\).
Therefore, for any \(x\) in \(\stsp\) and \(t, s\) in \(\nnegreals\) with \(t\leq s\), we define the operator~\(\lpoisprev_t^s\pr{\cdot\cond x}\colon\setoffn\pr{\stsp}\to\reals\) as
\begin{equation*}
  \lpoisprev_t^s\pr{f\cond x}
  \coloneqq \br{\lto_t^s f}\pr{x}
  \qquad\text{for all } f\in\setoffn\pr{\stsp}.
\end{equation*}
The following follows immediately from \ref{LTO in text:Non-negateive homogeneity}--\ref{LTO in text:Bound}.
\begin{corollary}
\label{cor:lpoisprev is a coherent lower prevision}
  For any \(x\) in \(\stsp\) and \(t, s\) in \(\nnegreals\) with \(t\leq s\), \(\lpoisprev_t^s\pr{\cdot\cond x}\) is a coherent lower prevision in the sense of \cite[Definition~4.10]{2014LowerPrevisions}.
\end{corollary}
In the remainder, we let \(\upoisprev_t^s\pr{\cdot\cond x}\coloneqq-\lpoisprev_t^s\pr{-\cdot\cond x}\) be the conjugate coherent upper prevision of the coherent lower prevision~\(\lpoisprev_t^s\pr{\cdot\cond x}\).

\subsection{The Reduced Poisson Generator}
\label{ssec:The reduced Poisson generator}
Fix any \(\lowx, \upx\) in \(\stsp\) such that \(\lowx \leq \upx\), and let
\begin{equation*}
  \chi
  \coloneqq \set{x\in\stsp\colon\lowx\leq x\leq\upx}.
\end{equation*}
We define the \emph{reduced Poisson generator}~\(\ltro^\chi\colon\setoffn\pr{\chi}\to\setoffn\pr{\chi}\) for all \(g\) in \(\setoffn\pr{\chi}\) and \(x\) in \(\chi\) as
\begin{equation*}
\label{eqn:Lower transition rate operator}
  \br{\ltro^\chi g}\pr{x}
  \coloneqq \begin{dcases}
    \min_{\lambda\in\br{\llambda,\ulambda}} \pr{\lambda g\pr{x+1}-\lambda g\pr{x}}&\text{if } \lowx\leq x < \upx, \\
    0 &\text{if } x=\upx.
  \end{dcases}
\end{equation*}
In the Appendix, we verify that this reduced Poisson generator~\(\ltro^\chi\) is a \emph{lower transition rate operator} in the sense of \cite[Definition~7.2]{2017Krak}.
As outlined in \cite[Section~7]{2017Krak}, this lower transition rate operator generates a family of transformations as well.
For any \(t,s\) in \(\nnegreals\) with \(t\leq s\) and any \(u\) in \(\setoftseq_{\br{t,s}}\), we let
\begin{equation*}
  \Phi^\chi_u
  \coloneqq \prod_{i=1}^n \pr{I+\Delta_i\ltro^\chi}.
\end{equation*}
Note the similarity between the equation above and Equation~\eqref{eqn:PoisGen:Phi_u}.
Because \(\ltro^\chi\) is a lower transition rate operator, it follows from \cite[Corollary~7.11]{2017Krak}---a result similar to Theorem~\ref{the:Transformation induced by PoisGen in text}---that the transformation
\begin{equation}
\label{eqn:PoisGenChi:Lto}
  \lto_{t,s}^\chi
  \coloneqq \lim_{\sigma\pr{u}\to0} \set{\Phi_u^\chi\colon u\in\setoftseq_{\br{t,s}}}
\end{equation}
is non-negatively homogeneous.
The limit in this definition is to be interpreted as the limit in Equation~\eqref{eqn:LTO gen by PoisGen}: it does not depend on the actual sequence \(\set{u_i}_{i\in\nats}\) as long as \(\lim_{i\to+\infty}\sigma\pr{u_i}=0\).
Unsurprisingly, \citet{2017Krak} show that this family of transformations~\(\lto_{t,s}^\chi\) also satisfies \ref{LTO in text:Non-negateive homogeneity}--\ref{LTO in text:Time-invariance}.
Observe that Equation~\eqref{eqn:PoisGenChi:Lto} suggests a method to evaluate \(\lto^\chi_{t,s}\) for some \(g\) in \(\setoffn\pr{\chi}\): choose a sufficiently fine grid~\(u\), and compute \(\Phi^\chi_u g\) via backwards recursion.
There is much more to this approximation method than we can cover here; the interested reader is referred to \cite[Section~8.2]{2017Krak} and \cite{2017Erreygers}.

\subsection{The Essential Case of Eventually Constant Functions}
\label{ssec:The case of eventually constant functions}
Our reason for introducing the restricted Poisson generator~\(\ltro^\chi\) and its induced transformation~\(\lto_{t,s}^\chi\) is because the latter can be used to compute \(\lpoisprev_t^s\pr{f\cond x}\).
Essential to our method are those functions~\(f\) in \(\setoffn\pr{\stsp}\) that are \emph{eventually constant}, in the sense that
\begin{equation*}
  \pr{\exists \upx\in\stsp}
  \pr{\forall x \in \stsp, x\geq \upx}~
  f\pr{x}=f\pr{\upx}.
\end{equation*}
In this case, we say that \(f\) is \emph{constant starting from \(\upx\)}.
We collect all real-valued bounded functions~\(f\) on \(\stsp\) that are eventually constant in \(\setofcafn\pr{\stsp}\).

Our next result establishes a link between \(\lpoisprev_t^s\pr{\cdot\cond x}\) and \(\lto_{t,s}^\chi\) for eventually constant functions.
\begin{proposition}
\label{prop:eventually bounded functions:lto f equals lto^chi f}
  Fix some \(t, s\) in \(\nnegreals\) with \(t\leq s\) and some \(f\) in \(\setofcafn\pr{\stsp}\) that is constant starting from \(\upx\).
  Choose some \(\lowx\) in \(\stsp\) with \(\lowx\leq\upx\), and let \(\chi\coloneqq\set{x\in\stsp\colon\lowx\leq x\leq\upx}\).
  Then for any \(x\) in \(\stsp\) with \(x\geq\lowx\),
  \begin{equation*}
    \lpoisprev_t^s\pr{f\cond x}
    = \br{\lto_t^s f}\pr{x}
    = \begin{cases}
      \br{\lto^\chi_{t, s} f^\chi}\pr{x} &\text{if } x \leq \upx, \\
      f\pr{\upx} &\text{if } x\geq\upx,
    \end{cases}
  \end{equation*}
  where \(f^\chi\) is the restriction of \(f\) to \(\chi\).
\end{proposition}
Note that we are free to choose \(\lowx\).
If we are interested in \(\lpoisprev_t^s\pr{f\cond x}\) for a specific value of \(x\), then choosing \(\lowx=\min\set{\upx, x}\) is the optimal choice.
However, if we are interested in \(\lpoisprev_t^s\pr{f\cond x}\) for a finite range~\(R\subset\stsp\) of different \(x\) values, the obvious choice is \(\lowx=\min\pr{R\cup\set{\upx}}\) because we then only have to determine \(\lto^\chi_{t, s} f^\chi\) once!

A method to compute \(\lpoisprev_t^s\pr{\cdot\cond x}\) for all bounded functions~\(f\) follows from combining Proposition~\ref{prop:eventually bounded functions:lto f equals lto^chi f} with the following result.
\begin{proposition}
\label{prop:Approximate lto f using lto f indic leq upx}
  For any \(t, s\) in \(\nnegreals\) with \(t\leq s\), \(f\) in \(\setoffn\pr{\stsp}\) and \(x\) in \(\stsp\),
  \begin{equation*}
    \lpoisprev_t^s\pr{f\cond x}
    = \lim_{\upx\to+\infty} \lpoisprev_t^s\pr{\indica{\leq\upx} f + f\pr{\upx}\indica{>\upx}\cond x},
  \end{equation*}
  where \(\indica{\leq \upx}\) and \(\indica{> \upx}\) are the indicators of \(\set{z\in\stsp\colon z\leq \upx}\) and \(\set{z\in\stsp\colon z>\upx}\), respectively.
\end{proposition}
Observe that \(\indica{\leq\upx} f + f\pr{\upx}\indica{>\upx}\)---with \(\indica{\leq\upx} f\) the point-wise multiplication of \(\indica{\leq\overline{x}}\) and \(f\)---is constant starting from \(\upx\).
Therefore, it follows from Proposition~\ref{prop:eventually bounded functions:lto f equals lto^chi f} that \(\lpoisprev_t^s\pr{\indica{\leq\upx} f + f\pr{\upx}\indica{>\upx}\cond x}=\br{\lto_{t,s}^\chi f^\chi}\pr{x}\), where \(f^\chi\) is the restriction of \(f\) to \(\chi\).
We can combine this observation and Proposition~\ref{prop:Approximate lto f using lto f indic leq upx} to obtain a method to compute \(\lpoisprev_t^s\pr{f\cond x}\) for any bounded function~\(f\): (i) choose some sufficiently large \(\upx\) and let \(\chi\coloneqq\set{y\in\stsp\colon x\leq y\leq\upx}\); (ii) compute \(\lpoisprev_t^s\pr{\indica{\leq\upx} f + f\pr{\upx}\indica{>\upx}\cond x}=\br{\lto_{t,s}^\chi f^\chi}\pr{x}\), using one of the existing approximation methods mentioned at the end of Section~\ref{ssec:The reduced Poisson generator}; (iii) repeat (i)--(ii) for increasingly larger \(\upx\) until convergence is empirically observed.

\section{Computing Lower Expectations of Functions on \texorpdfstring{\(X_s\)}{Xs}}
\label{sec:Computing lower expectations}
Let \(\setofbbfn\pr{\stsp}\) denote the set of all real-valued bounded-below functions on \(\stsp\).
With any \(f\) in \(\setofbbfn\pr{\stsp}\) and \(s\) in \(\nnegreals\), we associate the real-valued bounded-below function
\begin{equation*}
  f\pr{X_s}\colon\setofpths{}\to\reals\colon
  \pth\mapsto\br{f\pr{X_s}}\pr{\pth}
  \coloneqq f\pr{\pth\pr{s}}.
\end{equation*}
In other words, and as suggested by our notation, \(f\pr{X_s}\) is the functional composition of \(f\) with the projector
\begin{equation*}
  X_s
  \colon \setofpths\to\stsp
  \colon \pth\mapsto X_s\pr{\pth}\coloneqq\pth\pr{s}.
\end{equation*}
The (conditional) expectation of \(f\pr{X_s}\) exists for any counting process~\(\prob\), as is established by the following rather obvious result.
\begin{lemma}
\label{lem:f of X_t is a F_u-measurable function}
  Consider some \(s\) in \(\nnegreals\) and \(u\) in \(\setoftseq\) with \(\max u \leq s\).
  Then for any \(f\) in \(\setofbbfn\pr{\stsp}\), \(f\pr{X_s}\) is an \(\cpfield_u\)-measurable function.
\end{lemma}
In the remainder, we provide several methods for computing lower and upper expectations; first for those with respect to the consistent Poisson processes and second for those with respect to all consistent  counting processes.
For the latter, we first limit ourselves to bounded functions and subsequently move on to functions that are bounded-below.

\subsection{With Respect to the Consistent Poisson Processes}
\label{ssec:Lower and upper with respect to consistent Poisson}
Fix some rate~\(\lambda\) in \(\nnegreals\), and let \(\prob\) be the Poisson process with rate \(\lambda\).
It is essentially well-known---and a consequence of Theorem~\ref{the:Pois transition probabilities are poisson distributed and the converse}---that for any \(t, s\) in \(\nnegreals\) with \(t\leq s\), \(u\) in \(\setoftseq_{<t}\), \(\pr{x_u, x}\) in \(\stsp_{u\cup t}\) and \(f\) in \(\setofbbfn\pr{\stsp}\),
\begin{equation}
\label{eqn:prev_lambda f(X_s) in text}
  \prev_{\prob}\pr{f\pr{X_s}\cond X_u=x_u, X_t=x}
  = \sum_{y=x}^{+\infty} f\pr{y}\pois_{\lambda\pr{s-t}}\pr{y-x}.
\end{equation}
Because of this expression, \(\lprev_\Lambda^\star\pr{f\pr{X_s}\cond X_u=x_u, X_t=x}\) can be computed using the straightforward method that we already discussed in Section~\ref{ssec:Set of Poisson processes}: (i) fix a finite grid of \(\lambda\)'s in \(\Lambda=\br{\llambda, \ulambda}\), (ii) (numerically) evaluate the infinite sum in Equation~\eqref{eqn:prev_lambda f(X_s) in text} for each \(\lambda\) in this grid, and (iii) compute the minimum. In some specific cases, it is even possible to know beforehand for which $\lambda$ this minimum will be achieved.
For example, if $f$ is monotone and bounded, or bounded below and non-decreasing, then as we will see in Propositions~\ref{prop:Monotone bounded functions:equality between conditional expectations} and~\ref{prop:Lower and upper expectation of non-decreasing functions}, it suffices to consider $\lambda=\llambda$ or $\lambda=\ulambda$.

\subsection{With Respect to the Consistent Counting Processes}
\label{ssec:Lower and upper with respect to consistent counting}
Computing \(\lprev_\Lambda\pr{f\pr{X_s}\cond X_u=x_u, X_t=x}\) is less straightforward, as in general this does not reduce to a one-parameter optimisation problem.
Nevertheless, as we are about to show, the semi-group of Section~\ref{sec:The Poisson semi-group} allows us to circumvent this issue.
Our first result establishes a method to compute the lower---and hence also the upper---expectation of bounded functions.
\begin{theorem}
\label{the:Lower conditional expectation of bounded function with respect to all consistent processes}
  For any \(t,s\) in \(\nnegreals\) with \(t\leq s\), \(u\) in \(\setoftseq_{<t}\), \(f\) in \(\setoffn\pr{\stsp}\) and \(\pr{x_u, x}\) in \(\stsp_{u\cup t}\),
  \begin{equation*}
    \lprev_\Lambda\pr{f\pr{X_s}\cond X_u=x_u, X_t=x}
    = \lpoisprev_t^s\pr{f\cond x}.
  \end{equation*}
\end{theorem}
Indeed, because of this result, we can use the method that was introduced at the end of Section~\ref{ssec:The case of eventually constant functions} to compute the lower expectation of $f$.

For the special case of monotone bounded functions, we obtain an even stronger result.
\begin{proposition}
\label{prop:Monotone bounded functions:equality between conditional expectations}
  Fix any \(t,s\) in \(\nnegreals\) with \(t\leq s\), \(u\) in \(\setoftseq_{<t}\), \(\pr{x_u, x}\) in \(\stsp_{u\cup t}\) and \(f\) in \(\setoffn\pr{\stsp}\).
  If \(f\) is monotone, then
  \begin{multline*}
    \lprev_\Lambda\pr{f\pr{X_s}\cond X_u=x_u, X_t=x} \\
    = \lprev_\Lambda^\star\pr{f\pr{X_s}\cond X_u=x_u, X_t=x} \\
    = \prev_{\prob_\lambda}\pr{f\pr{X_s}\cond X_u=x_u, X_t=x},
  \end{multline*}
  where \(\prob_\lambda\) is the Poisson process with rate~\(\lambda=\llambda\) if \(f\) is non-decreasing and rate~\(\lambda=\ulambda\) if \(f\) is non-increasing.
\end{proposition}

Almost everything has now been set up to consider a general real-valued bounded below function of \(X_s\).
An essential intermediary step is an extension of Proposition~\ref{prop:Monotone bounded functions:equality between conditional expectations}.
\begin{proposition}
\label{prop:Lower and upper expectation of non-decreasing functions}
  Fix any \(t, s\) in \(\nnegreals\) with \(t\leq s\), \(u\) in \(\setoftseq_{<t}\) and \(\pr{x_u, x}\) in \(\stsp_{u\cup t}\).
  Then for any \(f\) in \(\setofbbfn\pr{\stsp}\) that is non-decreasing,
  \begin{multline*}
    \lprev_\Lambda\pr{f\pr{X_s}\cond X_u=x_u, X_t=x} \\
    = \lprev_\Lambda^\star\pr{f\pr{X_s}\cond X_u=x_u, X_t=x} \\
    = \prev_{\prob_{\llambda}}\pr{f\pr{X_s}\cond X_u=x_u, X_t=x}
  \end{multline*}
  and
  \begin{multline*}
    \uprev_\Lambda\pr{f\pr{X_s}\cond X_u=x_u, X_t=x} \\
    = \uprev_\Lambda^\star\pr{f\pr{X_s}\cond X_u=x_u, X_t=x} \\
    = \prev_{\prob_{\ulambda}}\pr{f\pr{X_s}\cond X_u=x_u, X_t=x},
  \end{multline*}
  where \(\prob_{\llambda}\) and \(\prob_{\ulambda}\) are the Poisson processes with rates~\(\llambda\) and \(\ulambda\), respectively.
\end{proposition}
As an immediate corollary of Proposition~\ref{prop:Lower and upper expectation of non-decreasing functions}, we obtain an interpretation for the rate interval~\(\Lambda\): its bounds provide tight lower and upper bounds on the expected number of Poisson-events in any time period.
\begin{corollary}
\label{cor:Lower and upper expected number of events}
  Fix any \(t, s\) in \(\nnegreals\) with \(t\leq s\), \(u\) in \(\setoftseq_{<t}\) and \(\pr{x_u, x}\) in \(\stsp_{u\cup t}\).
  Then
  \begin{equation*}
    \lprev_\Lambda\pr{X_s\cond X_u=x_u, X_t=x}
    = x + \llambda\pr{s-t}
  \end{equation*}
  and
  \begin{equation*}
    \uprev_\Lambda\pr{X_s\cond X_u=x_u, X_t=x}
    = x + \ulambda\pr{s-t},
  \end{equation*}
  and similarly for \(\lprev_\Lambda^\star\) and \(\uprev_\Lambda^\star\).
\end{corollary}
A more important consequence of Proposition~\ref{prop:Lower and upper expectation of non-decreasing functions} is the following result, which can be regarded as an extension of (the combination of) Proposition~\ref{prop:Approximate lto f using lto f indic leq upx} and Theorem~\ref{the:Lower conditional expectation of bounded function with respect to all consistent processes}.
\begin{theorem}
\label{the:Lower and upper expectation of functions that do not increase to fast}
  Fix any \(t, s\) in \(\nnegreals\) with \(t\leq s\), \(u\) in \(\setoftseq_{<t}\), \(\pr{x_u, x}\) in \(\stsp_{u\cup t}\) and \(f\) in \(\setofbbfn\pr{\stsp}\).
  If
  \begin{equation*}
    \sum_{y=x}^{+\infty} f_{\max}\pr{y} \pois_{\ulambda\pr{s-t}}\pr{y-x}
    < +\infty,
  \end{equation*}
  where \(f_{\max}\) in \(\setofbbfn\pr{\stsp}\) is defined for all \(y\) in \(\stsp\) as
  \begin{equation*}
     f_{\max}\pr{y}
     \coloneqq\max\set{f\pr{z}\colon z\in\stsp, z\leq y},
  \end{equation*}
 then
  \begin{equation*}
    \lprev_\Lambda\pr{f\pr{X_s}\cond X_u=x_u, X_t=x}
    = \lim_{\upx\to+\infty} \lpoisprev_t^s\pr{\indica{\leq\upx}f+f\pr{\upx}\indica{>\upx}\cond x},
  \end{equation*}
  \begin{equation*}
    \uprev_\Lambda\pr{f\pr{X_s}\cond X_u=x_u, X_t=x}
    = \lim_{\upx\to+\infty} \upoisprev_t^s\pr{\indica{\leq\upx}f+f\pr{\upx}\indica{>\upx}\cond x},
  \end{equation*}
  where the two limits are finite.
\end{theorem}
Because of this result, we can compute the lower and upper expectation using the same method as before. Note that it makes no difference that $f$ is no longer bounded; the method still works because \(\indica{\leq\upx} f + f\pr{\upx}\indica{>\upx}\) is bounded.

\subsection{A Numerical Example}
\label{ssec:Numerical example}
\begin{figure}
\floatconts
  {fig:fig-label}
  {\caption{Bounds on transition probabilities as a function of \(t\) for the rate interval \(\Lambda=\br{1, 2}\).}}
  {%
  \begin{tikzpicture}
    \pgfplotstableread[col sep=comma]{no-event.csv}{\data}
    \begin{axis}[
      width=\columnwidth,
      height=.5\columnwidth,
      legend style={font=\footnotesize},
      tick label style={font=\footnotesize},
      label style={font=\footnotesize},
      ylabel={\(\prob\pr{X_t=0\cond X_0=0}\)},
      ymin=0, ymax=1, xmin=0, xmax=4,
      legend entries={
        \(\setofconscproc{\Lambda}^\star\),
        \(\setofconscproc{\Lambda}\),
      },
    ]
      \addplot[dashed, mark=*, mark color=black, mark size=1.3pt, mark repeat=4] table [x=t, y=Plow]{\data};
      \addplot[solid] table [x=t, y=low]{\data};
      \addplot[dashed, mark=*, mark color=black, mark size=1.3pt, mark repeat=4] table [x=t, y=Pup]{\data};
      \addplot[solid] table [x=t, y=up]{\data};
    \end{axis}
  \end{tikzpicture}
  \begin{tikzpicture}
    \pgfplotstableread[col sep=comma]{one-event.csv}{\data}
    \begin{axis}[
      width=\columnwidth,
      height=.5\columnwidth,
      legend style={font=\footnotesize},
      tick label style={font=\footnotesize},
      label style={font=\footnotesize},
      xlabel={Time~\(t\)},
      ylabel={\(\prob\pr{X_t=1\cond X_0=0}\)},
      ymin=0, 
      xmin=0, xmax=4,
      legend entries={
        \(\setofconscproc{\Lambda}^\star\),
        \(\setofconscproc{\Lambda}\),
      },
    ]
      \addplot[name path=Poislow, dashed, mark=*, mark color=black, mark size=1.3pt, mark repeat=4] table [x=t, y=Plow]{\data};
      \addplot[name path=low, solid] table [x=t, y=low]{\data};
      \addplot[name path=Poisup, dashed, mark=*, mark color=black, mark size=1.3pt, mark repeat=4] table [x=t, y=Pup]{\data};
      \addplot[name path=up, solid] table [x=t, y=up]{\data};
    \end{axis}
  \end{tikzpicture}
  }
\end{figure}
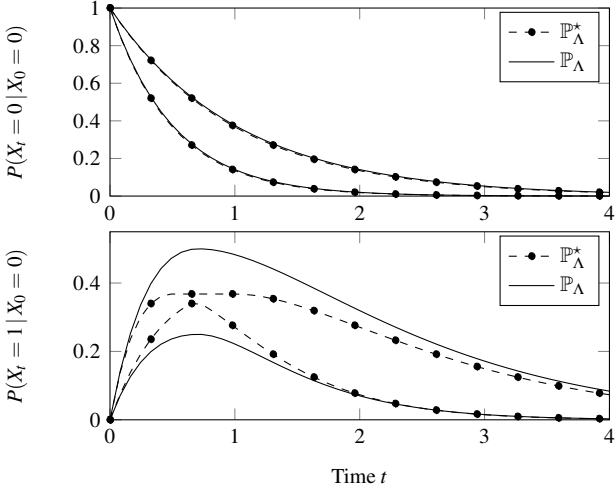
We end this section with a basic numerical example.
We determine tight lower and upper bounds on
\begin{equation*}
  \prob\pr{X_t=x\cond X_0=0}
  = \prev_\prob\pr{\indica{x}\pr{X_t}\cond X_0=0},
\end{equation*}
with \(x\) equal to \(0\) or \(1\).
We use the methods outlined in Sections~\ref{ssec:Lower and upper with respect to consistent Poisson} and \ref{ssec:Lower and upper with respect to consistent counting} to compute lower and upper bounds with respect to the sets~\(\setofconscproc{\Lambda}^\star\) and \(\setofconscproc{\Lambda}\) for \(\Lambda=\br{1,2}\).
The resulting bounds are depicted in Figure~\ref{fig:fig-label}.
Observe that for \(x=0\), the bounds with respect to~\(\setofconscproc{\Lambda}^\star\) and \(\setofconscproc{\Lambda}\) are equal, as is to be expected due to Proposition~\ref{prop:Monotone bounded functions:equality between conditional expectations} because \(\indica{x}\) is monotone for $x=0$.
For \(x=1\), \(\indica{x}\) is not monotone and the bounds with respect to~\(\setofconscproc{\Lambda}^\star\) are clearly \emph{not} equal to those with respect to~\(\setofconscproc{\Lambda}\).

\section{Justification for the Term Imprecise Poisson Process}
\label{sec:Justification for the term imprecise poisson process}
Until now, we have provided little justification for why we call both \(\setofconscproc{\Lambda}^\star\) and \(\setofconscproc{\Lambda}\) imprecise Poisson processes.
In Section~\ref{ssec:Set of consistent processes}, we already briefly mentioned that the two sets are proper generalisations of the Poisson process: if the rate interval~\(\Lambda\) is degenerate, meaning that \(\llambda=\lambda=\ulambda\), then both sets reduce to the singleton containing the Poisson process with rate~\(\lambda\).
Another argument for referring to \(\setofconscproc{\Lambda}^\star\) and \(\setofconscproc{\Lambda}\) as imprecise Poisson processes concerns the (tight lower and upper bounds on the) expected number of Poisson events in a time period of length \(\Delta\).
For a Poisson process, it is well-known that this expectation is equal to \(\Delta\lambda\), and we know from Corollary~\ref{cor:Lower and upper expected number of events} that the corresponding lower and upper expectations are equal to \(\Delta\llambda\) and \(\Delta\ulambda\), respectively.

We end this section with our strongest argument for using the term imprecise Poisson process to refer to both \(\setofconscproc{\Lambda}^\star\) and \(\setofconscproc{\Lambda}^\star\).
The following result establishes that the corresponding lower expectations \(\lprev_{\Lambda}^\star\) and \(\lprev_{\Lambda}\)---and, due to conjugacy, also the corresponding upper expectations \(\uprev_{\Lambda}^\star\) and \(\uprev_{\Lambda}\)---satisfy imprecise generalisations of \ref{CP:Start at 0}, \ref{CP:Orederliness} and \ref{Pois:Markovian}--\ref{Pois:Time-homogeneous}, which are the defining properties of a Poisson process.
\begin{proposition}
\label{prop:Poisson like}
  For all \(t, \Delta\) in \(\nnegreals\), \(u\) in \(\setoftseq_{<t}\), \(\pr{x_u,x}\) in \(\stsp_{u\cup t}\) and \(f\) in \(\setoffn\pr{\stsp}\),
  \begin{enumerate}[label=\upshape(\roman*)]
    \item
    \(\lprev_\Lambda\pr{f\pr{X_0}}=f\pr{0}\);
    \item
    \begin{equation*}
      \lim_{\Delta\to0^+} \frac{\lprev_\Lambda\pr{\indica{\pr{X_{t+\Delta}\geq x+2}}\cond X_u=x_u, X_t=x}}{\Delta}
      =0
    \end{equation*}
    and, if \(t>0\),
    \begin{equation*}
      \lim_{\Delta\to0^+} \frac{\lprev_\Lambda\pr{\indica{\pr{X_t\geq x+2}}\cond X_u=x_u, X_{t-\Delta}=x}}{\Delta}
      =0;
    \end{equation*}
    \item
    \(\lprev_\Lambda\pr{f\pr{X_{t+\Delta}}\cond X_u=x_u, X_t=x} = \lprev_\Lambda\pr{f\pr{X_{t+\Delta}}\cond X_t=x}\);
    \item
    \(\lprev_\Lambda\pr{f\pr{X_{t+\Delta}}\cond X_t=x} = \lprev_\Lambda\pr{f'_x\pr{X_{t+\Delta}}\cond X_t=0}\);
    \item
    \(\lprev_\Lambda\pr{f\pr{X_{t+\Delta}}\cond X_t=x} = \lprev_\Lambda\pr{f\pr{X_{\Delta}}\cond X_0=x}\);
  \end{enumerate}
  with \(f'_x\colon\stsp\to\reals\colon z\mapsto f'_x\pr{z}\coloneqq f\pr{x+z}\).
  The same equalities also hold for \(\lprev_\Lambda^\star\).
\end{proposition}

\section{Conclusion}
In this contribution, we proposed two generalisations of the Poisson process in the form of two sets of counting processes: the set~\(\setofconscproc{\Lambda}^\star\) of all Poisson processes with rate~\(\lambda\) in the rate interval~\(\Lambda\), and the set~\(\setofconscproc{\Lambda}\) of all counting process that are consistent with the rate interval~\(\Lambda\).
We argued why both of these sets can be seen as proper generalisations of the Poisson process.
First and foremost, for a degenerate rate interval they both reduce to the singleton containing the Poisson process with this rate.
Second, the lower and upper expectations with respect to both sets satisfy imprecise generalisations of \ref{CP:Start at 0}, \ref{CP:Orederliness} and \ref{Pois:Markovian}--\ref{Pois:Time-homogeneous}, the defining properties of a Poisson process.
We also presented several methods for computing lower and upper expectations for functions that depend on the number of occurred Poisson-events at a single time point.

We end with two suggestions for future research.
An obvious open question is whether we can efficiently compute lower and upper expectations for functions that depend on the number of occurred Poisson-events at multiple points in time.
Based on similar results of \citet{2017Krak} for imprecise continuous-time Markov chains with a finite state space, we strongly believe that this will be the case for \(\setofconscproc{\Lambda}\) but not for \(\setofconscproc{\Lambda}^\star\), whence providing a practical argument in favour of the former.
A perhaps slightly less obvious open question is whether Theorem~\ref{the:Pois transition probabilities are poisson distributed and the converse} and Corollary~\ref{cor:Poisson has rate lambda} can be generalised to sets of counting processes, in the sense that we can infer the existence of a rate interval rather than specify one, by imposing appropriate conditions on the set of counting processes, including the imprecise generalisations of \ref{CP:Start at 0}, \ref{CP:Orederliness} and \ref{Pois:Markovian}--\ref{Pois:Time-homogeneous}.

\acks{%
  We would like to express our gratitude to the three anonymous reviewers for their valuable time, constructive feedback and sound suggestions for improvements.
  Jasper De Bock's research was partially funded by the European Commission's H2020 programme, through the UTOPIAE Marie Curie Innovative Training Network, H2020-MSCA-ITN-2016, Grant Agreement number 722734.
}

\bibliography{erreygers19}

\begin{thebibliography}{14}
\providecommand{\natexlab}[1]{#1}
\providecommand{\url}[1]{\texttt{#1}}
\expandafter\ifx\csname urlstyle\endcsname\relax
  \providecommand{\doi}[1]{doi: #1}\else
  \providecommand{\doi}{doi: \begingroup \urlstyle{rm}\Url}\fi

\bibitem[Daley and Vere-Jones(2003)]{2013DaleyVereJones}
Daryl~J. Daley and David Vere-Jones.
\newblock \emph{An Introduction to the Theory of Point Processes}, volume I:
  Elementary Theory and Methods.
\newblock Springer, 2003.

\bibitem[De~Bock(2017)]{2016DeBock}
Jasper De~Bock.
\newblock The limit behaviour of imprecise continuous-time {M}arkov chains.
\newblock \emph{Journal of Nonlinear Science}, 27\penalty0 (1):\penalty0
  159--196, 2017.
\newblock \doi{10.1007/s00332-016-9328-3}.

\bibitem[Erreygers and De~Bock(2017)]{2017Erreygers}
Alexander Erreygers and Jasper De~Bock.
\newblock Imprecise continuous-time {M}arkov chains: Efficient computational
  methods with guaranteed error bounds.
\newblock In \emph{Proceedings of ISIPTA'17}, pages 145--156. PMLR, 2017.

\bibitem[Feller(1968)]{1968Feller}
William Feller.
\newblock \emph{An Introduction to Probability Theory and Its Applications},
  volume~1.
\newblock Wiley, 1968.

\bibitem[Karlin and Taylor(1975)]{1975KarlinTaylor}
Samuel Karlin and Howard~M. Taylor.
\newblock \emph{A first course in stochastic processes}.
\newblock Academic Press, 1975.

\bibitem[Khintchine(1960)]{1960Khintchine}
Aleksandr~Yakovlevich Khintchine.
\newblock \emph{Mathematical methods in the theory of queueing}.
\newblock Charles Griffin \& Company, 1960.

\bibitem[Kingman(1993)]{1993Kingman}
John~F.C. Kingman.
\newblock \emph{Poisson Processes}.
\newblock Clarendon Press, 1993.

\bibitem[Krak et~al.(2017)Krak, De~Bock, and Siebes]{2017Krak}
Thomas Krak, Jasper De~Bock, and Arno Siebes.
\newblock Imprecise continuous-time {Ma}rkov chains.
\newblock \emph{International Journal of Approximate Reasoning}, 88:\penalty0
  452--528, 2017.
\newblock \doi{10.1016/j.ijar.2017.06.012}.

\bibitem[Norris(1997)]{1997Norris}
James~R. Norris.
\newblock \emph{{M}arkov chains}.
\newblock Cambridge University Press, 1997.

\bibitem[Regazzini(1985)]{1985Regazzini}
Eugenio Regazzini.
\newblock Finitely additive conditional probabilities.
\newblock \emph{Rendiconti del Seminario Matematico e Fisico di Milano},
  55\penalty0 (1):\penalty0 69--89, 1985.
\newblock \doi{10.1007/BF02924866}.

\bibitem[Ross(1996)]{1996Ross}
Sheldon~M. Ross.
\newblock \emph{Stochastic Processes}.
\newblock Wiley, 1996.

\bibitem[Sato(1999)]{1999Sato}
Ke-iti Sato.
\newblock \emph{L\'evy processes and infinitely divisible distributions}.
\newblock Cambridge University Press, 1999.

\bibitem[Troffaes and de~Cooman(2014)]{2014LowerPrevisions}
Matthias~C.M. Troffaes and Gert de~Cooman.
\newblock \emph{Lower Previsions}.
\newblock Wiley, 2014.

\bibitem[Watanabe(1964)]{1964Watanabe}
Shinzo Watanabe.
\newblock On discontinuous additive functionals and {L}\'evy measures of a
  {M}arkov process.
\newblock \emph{Japanese journal of mathematics: transactions and abstracts},
  34:\penalty0 53--70, 1964.
\newblock \doi{10.4099/jjm1924.34.0_53}.

\end{thebibliography}

\onecolumn
\clearpage
\appendix

\newenvironment{proofof}[1]{\par\noindent{\bfseries\upshape Proof of #1\ }}{\jmlrQED}

\section*{Appendix}
In this appendix, we will not entirely follow the same order as we did in the main text.
Our reason for doing so is that to prove the results in Section~\ref{sec:Poisson processes in particular}, we need some results that are very much related to the transformations that we introduce in Section~\ref{sec:The Poisson semi-group}.
Therefore, we have chosen to start off this appendix with some general results regarding transformations.

\section{Some Preliminary Results Regarding Transformations}
\label{app:Preliminary results regarding transformations}
Throughout this appendix, and as mentioned in Section~\ref{ssec:Functions, operators and norms}, we let \(\genset\) be any non-empty set that is at most countably finite; furthermore, we assume that \(\genset\) is endowed with a total order ``\(\leq\)''.
\subsection{General Non-Negatively Homogeneous Transformations}
\label{sapp:Transformations:General non-negatively homogeneous}
We start of with some essential properties of non-negatively homogeneous transformations.
\begin{lemma}
\label{lem:Properties of non-negatively homogeneous transformations}
  Consider two non-negatively homogeneous transformations~\(A\) and \(B\) on \(\setoffn\pr{\genset}\).
  Then
  \begin{enumerate}[label=\upshape{}NH\arabic*., ref=\upshape(NH\arabic*), leftmargin=*]
    \item \label{BNH: A+B} \(A+B\) is non-negatively homogeneous;
    \item \label{BNH: mu a} \(\mu A\) is non-negatively homogeneous for any \(\mu\) in \(\reals\);
    \item  \label{BNH: A B} \(A B\) is non-negatively homogeneous;
    \item \label{BNH: norm Af <= norm A norm f} \(\norm{A f} \leq \norm{A}\norm{f}\) for any \(f\) in \(\setoffn\pr{\genset}\) ;
    \item  \label{BNH:norm A B <= norm A norm B} \(\norm{A B} \leq \norm{A}\norm{B}\).
  \end{enumerate}
\end{lemma}
\begin{proof}
  The proof of \ref{BNH: A+B}--\ref{BNH: A B} is a matter of straightforward verification.
  We therefore move on to proving \ref{BNH: norm Af <= norm A norm f}.
  Observe first that if \(\norm{f} = 0\), then \(f = 0\) and it follows from the non-negative homogeneity of \(A\) that \(A f = A\pr{0 f} = 0 \pr{A f} = 0\).
  Therefore, \(\norm{Af} = 0\); hence the stated is true.
  Next, we assume that \(\norm{f}>0\).
  Then
  \[
    A f
    = A \pr*{\frac{\norm{f} f}{\norm{f}}}
    = \norm{f} A \pr*{\frac{f}{\norm{f}}}
    = \norm{f} A f',
  \]
  where we let \(f'\coloneqq f / \norm{f}\).
  Note that \(\norm{f'} = \norm{f}/\norm{f} = 1\).
  Consequently,
  \begin{align*}
    \norm{A f}
    &= \sup\set{\abs{\br{Af}\pr{x}} \colon x \in \genset} \\
    &= \sup\set{\abs{\norm{f}\br{Af'}\pr{x}} \colon x \in \genset}
    = \norm{f} \sup\set{\abs{\br{Af'}\pr{x}} \colon x \in \genset}
    = \norm{f} \norm{A f'} \\
    &\leq \norm{A}\norm{f},
  \end{align*}
  where the final inequality holds because \(\norm{f'}=1\) implies that \(\norm{Af'}\leq\norm{A}\).

  Finally, we prove \ref{BNH:norm A B <= norm A norm B}.
  To that end, we observe that
  \begin{align*}
      \norm{AB}
      &= \sup\set{\norm{ABg} \colon g\in\setoffn\pr{\genset}, \norm{g}=1} \\
      &\leq \sup\set{\norm{A} \norm{Bg} \colon g\in\setoffn\pr{\genset}, \norm{g}=1} \\
      &= \norm{A} \sup\set{\norm{Bg} \colon g\in\setoffn\pr{\genset}, \norm{g}=1}
      = \norm{A}\norm{B},
    \end{align*}
    where the inequality follows from \ref{BNH: norm Af <= norm A norm f}.
\end{proof}
Time and time again, we will consider transformations on \(\setoffn\pr{\genset}\) that are constructed using a finite succession of the operations \ref{BNH: A+B}--\ref{BNH: A B}.
For instance, we will often be interested in the (norm of the) ``difference''~\(\norm{A_1\cdots A_k-B_1\cdots B_\ell}\) between the two transformations~\(A\coloneqq A_1 \cdots A_k\) and \(B\coloneqq B_1\cdots B_\ell\), where \(A_1\), \dots \(A_k\) and \(B_1\), \ldots, \(B_\ell\) are non-negatively homogeneous transformations.
That \(A\) is a non-negatively homogeneous transformation follows from repeated application of \ref{BNH: A B}, and similarly for \(B\).
Furthermore, \(A-B\) is a non-negatively homogeneous transformation due to \ref{BNH: mu a} with \(\mu=-1\) and \ref{BNH: A+B}, so the norm of \(A-B\) is indeed well-defined.
In order not to needlessly repeat ourselves, we will usually refrain from explicitly mentioning that the operator constructed by a finite succession of the operations \ref{BNH: A+B}--\ref{BNH: A B} is a non-negatively homogeneous transformation.

Next, we verify that our norm for non-negatively homogeneous transformations on \(\setoffn\pr{\genset}\) satisfies the three conditions of a norm: it should (i) be absolutely homogeneous, (ii) be sub-additive; and (iii) separate points.
We here simply repeat the arguments of \citet{2016DeBock}, who restricted himself to finite sets~\(\genset\).
\begin{enumerate}[label=(\roman*)]
  \item Let \(A\) be some non-negatively homogeneous operator, and fix some real number \(\mu\).
  Observe that \(\mu A\) is a non-negatively homogeneous transformation by \ref{BNH: mu a}.
  Furthermore, some straightforward manipulations yield that
  \begin{align*}
    \norm{\mu A}
    &= \sup\set{\norm{\mu A f}\colon f\in\setoffn\pr{\genset}, \norm{f}=1}
    = \sup\set{\abs{\mu}\norm{A f}\colon f\in\setoffn\pr{\genset}, \norm{f}=1} \\
    &= \abs{\mu}\sup\set{\norm{A f}\colon f\in\setoffn\pr{\genset}, \norm{f}=1}
    =\abs{\mu}\norm{A},
  \end{align*}
  where the second equality holds due to the absolute homogeneity of the supremum norm.
  \item Let \(A\) and \(B\) be two non-negatively homogeneous transformations.
  Recall from \ref{BNH: A+B} that \(A+B\) is also a non-negatively homogeneous transformation.
  It now follows from the sub-additivity of the supremum norm that
  \begin{align*}
    \norm{A+B}
    &= \sup\set{\norm{\pr{A+B} f}\colon f\in\setoffn\pr{\genset}, \norm{f}=1}
    = \sup\set{\norm{Af+Bf}\colon f\in\setoffn\pr{\genset}, \norm{f}=1} \\
    &\leq \sup\set{\norm{A f}+\norm{Bf}\colon f\in\setoffn\pr{\genset}, \norm{f}=1}
    \leq \sup\set{\norm{A}+\norm{B}\colon f\in\setoffn\pr{\genset}, \norm{f}=1}=\norm{A}+\norm{B},
  \end{align*}
  where for the second inequality we have used \ref{BNH: norm Af <= norm A norm f}.
  \item Let \(A\) be a non-negatively homogeneous transformation.
  Recall from \ref{BNH: norm Af <= norm A norm f} that, for any \(f\) in \(\setoffn\pr{\genset}\),
  \begin{equation*}
    0\leq \norm{Af}\leq\norm{A}\norm{f}.
  \end{equation*}
  Hence, if \(\norm{A}=0\), then it follows from these inequalities that \(\norm{Af}=0\) for all \(f\) in \(\setoffn\pr{\genset}\).
  From this, we conclude that \(Af=0\) for all \(f\) in \(\setoffn\pr{\genset}\), and so \(A=0\), because the supremum norm separates points.
\end{enumerate}

\subsection{Lower Counting Transformations}
\label{sapp:Transformations:Lower counting}
The first two types of non-negatively homogeneous transformations that will be essential in the remainder are lower transition transformations and lower counting transformations.
The following definition is a straightforward generalisation (or modification) of the existing concept of a \emph{lower transition operator} on a finite state space, see for instance \cite[Definition~7.1]{2017Krak}.
\begin{definition}
\label{def:Lower transition transformation}
  A \emph{lower transition transformation}~\(\lto\colon\setoffn\pr{\genset}\to\setoffn\pr{\genset}\) is a transformation such that
  \begin{enumerate}[label=\upshape{}LT\arabic*., ref=\upshape(LT\arabic*), leftmargin=*]
    \item \label{def:LTT:Non-negative homgeneity}
    \(\lto\pr{\gamma f} = \gamma \lto f\), for all \(f\) in \(\setoffn\pr{\genset}\) and \(\gamma\) in \(\nnegreals\); \hfill [non-negative homogeneity]
    \item \label{def:LTT:Super-additivity}
    \(\lto\pr{f+g}\geq \lto f + \lto g\), for all \(f, g\) in \(\setoffn\pr{\genset}\); \hfill [super-additivity]
    \item\label{def:LTT:Bound}
    \(\lto  f \geq \inf f\), for all \(f\) in \(\setoffn\pr{\genset}\). \hfill [bound]
  \end{enumerate}
  A \emph{lower counting transformation}~\(\lto\) is a lower transition transformation with
  \begin{enumerate}[resume*]
    \item \label{def:LTT:Counting} \(\br{\lto  f}\pr{x}=\br{\lto \pr{\indica{\geq x}f}}\pr{x}\), for all \(f\) in \(\setoffn\pr{\genset}\) and \(x\) in \(\genset\).
  \end{enumerate}
\end{definition}

Lower transition transformations have many interesting properties.
We start with some basic ones.
\begin{lemma}
\label{lem:LTT:properties from coherence}
  Consider a lower transition transformation~\(\lto\colon\setoffn\pr{\genset}\to\setoffn\pr{\genset}\).
  Then
  \begin{enumerate}[label=\upshape{}LT\arabic*., ref=\upshape(LT\arabic*), start=5, leftmargin=*]
    \item \label{LTT:inf f T f sup f}\(\inf f \leq \lto f \leq -\lto\pr{-f}\leq \sup f\) for all \(f\) in \(\setoffn\pr{\stsp}\);
    \item \label{LTT:constant} \(\lto \mu = \mu\) for all \(\mu\) in \(\reals\);
    \item \(\lto\pr{f+\mu} = \lto f +\mu\) for all \(f\) in \(\setoffn\pr{\genset}\) and all \(\mu\) in \(\reals\);
    \item \label{LTT:monotonicity} \(\lto f \leq \lto g\) for all \(f,g\) in \(\setoffn\pr{\genset}\) such that \(f\leq g\);
    \item \label{LTT:absolute value of difference} \(\abs{\lto f-\lto g} \leq -\lto\pr{-\abs{f-g}}\) for all \(f,g\) in \(\setoffn\pr{\genset}\).
  \end{enumerate}
\end{lemma}
\begin{proof}
  For any \(x\) in \(\genset\), the operator component \(\lto_x\coloneqq\br{\lto \cdot }\pr{x}\) is a coherent lower prevision with domain \(\setoffn\pr{\genset}\), the linear space of all bounded functions on \(\genset\).
  The properties therefore follow from their respective counterparts for lower previsions, see for instance \cite[Theorem~4.13]{2014LowerPrevisions}.
\end{proof}
\citet{2016DeBock} states some additional basic properties that follow from those of Lemma~\ref{lem:LTT:properties from coherence}.
\begin{lemma}
\label{lem:LTT:Further properties}
  Consider a lower transition transformation~\(\lto\colon\setoffn\pr{\genset}\to\setoffn\pr{\genset}\).
  Then for all \(f, g\) in \(\setoffn\pr{\genset}\) and all non-negatively homogeneous transformations \(A\) and \(B\),
\begin{enumerate}[label=\upshape{}LT\arabic*., ref=\upshape(LT\arabic*), start=10, leftmargin=*, widest=12]
    \item \label{LTT:Norm is lower than one} \(\norm{\lto}\leq1\);
    \item \label{LTT:Non-expansiveness} \(\norm{\lto f - \lto g}\leq\norm{f-g}\);
    \item \label{LTT:Composition with nneg hom} \(\norm{\lto A - \lto B}\leq\norm{A-B}\).
  \end{enumerate}
\end{lemma}
\begin{proof}
  \ref{LTT:Norm is lower than one} can be verified by combining the definition of the operator norm and \ref{LTT:inf f T f sup f}.
  Next, \ref{LTT:Non-expansiveness} follows from the definition of the (supremum) norm, \ref{LTT:absolute value of difference} and \ref{LTT:inf f T f sup f}.
  Finally, \ref{LTT:Composition with nneg hom} follows from the definition of the operator norm and \ref{LTT:Non-expansiveness}.
\end{proof}
 The following is an obvious extension/adaptation of \cite[Proposition~7.1]{2017Krak} to our setting.
\begin{lemma}
\label{lem:LTT:Composition is again a LTT}
  For any two lower transition (counting) transformations~\(\lto_1\) and \(\lto_2\), their composition \(\lto_1 \lto_2\) is again a lower transition (counting) transformation.
\end{lemma}
\begin{proof}
  To verify the four conditions of Definition~\ref{def:Lower transition transformation}, we fix some \(f,g\) in \(\setoffn\pr{\genset}\), \(\gamma\) in \(\nnegreals\) and \(x\) in \(\genset\).
  The first condition follows immediately from applying \ref{def:LTT:Non-negative homgeneity} twice:
  \begin{equation*}
    \lto_1 \lto_2 \pr{\gamma f}
    = \lto_1 \pr{\gamma \lto_2 f}
    = \gamma \lto_1\lto_2 f.
  \end{equation*}
  Next, we move on to the second condition.
  As \(\lto_2\) is a lower transition transformation, it follows from \ref{def:LTT:Super-additivity} that \(\lto_2\pr{f+g}\geq\lto_2 f+\lto_2 g\).
  Hence, it follows from \ref{LTT:monotonicity} that
  \begin{equation*}
    \lto_1\lto_2\pr{f+g}
    \geq \lto_1\pr{\lto_2 f+\lto_2 g}
    \geq \lto_1\lto_2 f + \lto_1\lto_2 g,
  \end{equation*}
  where for the final inequality we have again used \ref{def:LTT:Super-additivity}.
  The third condition follows from \ref{def:LTT:Bound}, \ref{LTT:monotonicity} and \ref{LTT:constant}:
  \begin{equation*}
    \lto_1\lto_2 f
    \geq \lto_1 \inf f
    = \inf f.
  \end{equation*}
  Finally, we assume that \(\lto_1\) and \(\lto_2\) are both lower counting transformations, and verify that their composition \(\lto_1\lto_2\) satisfies the fourth condition.
  To that end, we let \(h\coloneqq \lto_2 f\) and \(h'\coloneqq \lto_2 \pr{\indica{\geq x}f}\).
  Observe that, for all \(z\) in \(\stsp\),
  \begin{align*}
    \indica{\geq x}\pr{z} h\pr{z}
    &= \indica{\geq x}\pr{z} \br{\lto_2 f}\pr{z}
    = \indica{\geq x}\pr{z} \br{\lto_2 \pr{\indica{\geq z}f}}\pr{z} \\
    &= \begin{cases}
      \br{\lto_2 \pr{\indica{\geq z}f}}\pr{z} &\text{if } z\geq x \\
      0 &\text{otherwise}
    \end{cases}
    = \begin{cases}
      \br{\lto_2 \pr{\indica{\geq z}\pr{\indica{\geq x}f}}}\pr{z} &\text{if } z\geq x \\
      0 &\text{otherwise}
    \end{cases}
    = \begin{cases}
      \br{\lto_2 \pr{\indica{\geq x}f}}\pr{z} &\text{if } z\geq x \\
      0 &\text{otherwise}
    \end{cases} \\
    &= \indica{\geq x}\pr{z} \br{\lto_2 \pr{\indica{\geq x}f}}\pr{z}
    = \indica{\geq x}\pr{z} h'\pr{z},
  \end{align*}
  where the second and fifth equality hold due to \ref{def:LTT:Counting} because \(\lto_2\) is a lower counting transformation.
  Hence, \(\indica{\geq x}h=\indica{\geq x}h'\).
  As \(\lto_1\) is a lower counting transformation as well, it now follows that
  \begin{align*}
    \br{\lto_1 \lto_2 f}\pr{x}
    &= \br{\lto_1 h}\pr{x}
    = \br{\lto_1 \pr{\indica{\geq x} h}}\pr{x}
    = \br{\lto_1 \pr{\indica{\geq x} h'}}\pr{x}
    = \br{\lto_1 h'}\pr{x}
    = \br{\lto_1 \lto_2 \pr{\indica{\geq x} f}}\pr{x}.
  \end{align*}
\end{proof}
The following is an extension of~\cite[Lemma~E.4]{2017Krak} to our---slightly---more general setting.
\begin{lemma}
\label{lem: bound on norm of A_1 ... A_k - B_1 ... B_k}
  Consider some \(n\) in \(\nats\) and two sequences \(\lto_1, \dots, \lto_n\) and \(\lto'_1, \dots, \lto'_n\) of lower transition transformations on \(\setoffn\pr{\genset}\).
  Then
  \begin{equation*}
    \norm*{\prod_{i=1}^n \lto_i - \prod_{i=1}^n \lto'_i}
    \leq \sum_{i=1}^n \norm{\lto_i-\lto'_i}.
  \end{equation*}
\end{lemma}
\begin{proof}
  Our proof is entirely the same as that of \cite[Lemma~E.4]{2017Krak}, and is one using induction.
  Observe that the stated clearly holds for \(n=1\).
  Fix some \(m\) in \(\nats\) and assume that the stated holds for \(n=m\).
  We now show that the stated then also holds for \(m+1\).
  \begin{align*}
    \norm*{\prod_{i=1}^{m+1} \lto_i - \prod_{i=1}^{m+1} \lto'_i}
    &= \norm*{\prod_{i=1}^{m+1} \lto_i - \pr*{\prod_{i=1}^m \lto_i} \lto'_{m+1} + \pr*{\prod_{i=1}^m \lto_i} \lto'_{m+1} - \prod_{i=1}^{m+1} \lto'_i} \\
    &\leq \norm*{\prod_{i=1}^{m+1} \lto_i - \pr*{\prod_{i=1}^m \lto_i} \lto'_{m+1}}+ \norm*{\pr*{\prod_{i=1}^m \lto_i} \lto'_{m+1} - \prod_{i=1}^{m+1} \lto'_i} \\
    &= \norm*{\pr*{\prod_{i=1}^m \lto_i}\lto_{m+1} -\pr*{\prod_{i=1}^m \lto_i} \lto'_{m+1}}+\norm*{\pr*{\prod_{i=1}^m \lto_i - \prod_{i=1}^m \lto'_i} \lto'_{m+1}} \\
    &\leq \norm{\lto_{m+1} - \lto'_{m+1}}+\norm*{\prod_{i=1}^m \lto_i - \prod_{i=1}^m \lto'_i}\norm{\lto'_{m+1}} \\
    &\leq \norm{\lto_{m+1} - \lto'_{m+1}}+\norm*{\prod_{i=1}^m \lto_i - \prod_{i=1}^m \lto'_i} \\
    &\leq \norm{\lto_{m+1} - \lto'_{m+1}}+\sum_{i=1}^m\norm*{\lto_i - \lto'_i}
    = \sum_{i=1}^{m+1}\norm*{\lto_i - \lto'_i},
  \end{align*}
  where the second inequality follows from Lemma~\ref{lem:LTT:Composition is again a LTT}, \ref{LTT:Composition with nneg hom} and \ref{BNH:norm A B <= norm A norm B}, the third inequality follows from \ref{LTT:Norm is lower than one} and the penultimate inequality follows from the induction hypothesis.
\end{proof}

\section{Lower Transition Rate Transformations and the Corresponding Semi-Group of Lower Transition Transformations}
\label{app:ltro to lto}
In this short section, we introduce lower transition rate transformations, a second type of non-negatively homogeneous transformations.
Additionally, we also briefly explain how these lower transition transformations generate a semi-group of lower transition transformations.
Rather than working with a non-empty, ordered and possibly countably infinite set \(\genset\), in this section we will consider a non-empty and \emph{finite} set \(\chi\).
\begin{definition}[Definition~7.2 in \cite{2017Krak}]
\label{def:Lower transition rate transformation}
  A \emph{lower transition rate transformation} is a transformation~\(\genltro^\chi\colon\setoffn\pr{\chi}\to\setoffn\pr{\chi}\) such that
  \begin{enumerate}[label=\upshape{}LR\arabic*., ref=\upshape(LR\arabic*), leftmargin=*]
    \item \label{ltrt:non-negatively homogeneous} \(\genltro^\chi\pr{\gamma f}=\gamma \genltro^\chi f\), for all \(f\) in \(\setoffn\pr{\chi}\) and \(\gamma\) in \(\nnegreals\); \hfill[non-negative homogeneity]
    \item \(\genltro^\chi\pr{f+g}\geq \genltro^\chi f+\genltro^\chi g\), for all \(f, g\) in \(\setoffn\pr{\chi}\); \hfill[super-additivity]
    \item \label{ltrt:zero row-sums} \(\genltro^\chi\mu=0\), for all \(\mu\) in \(\reals\); \hfill[zero row-sums]
    \item \label{ltrt:non-negative off-diagonal} \(\br{\genltro^\chi\indica{y}}\pr{x}\geq0\), for all \(x,y\in\chi\) with \(x\neq y\). \hfill[non-negative off-diagonal elements]
  \end{enumerate}
\end{definition}

\subsection{The Corresponding Semi-Group}
We first repeat two intermediate results that are essential to the construction method of the semi-group.
We will see in Appendix~\ref{app:The Generalised Poisson Generator} further one that similar results hold in the setting of (generalised) Poisson generators.
\begin{lemma}[Proposition~4 in \cite{2017Erreygers}]
\label{lem:Gen ltro^chi:Norm}
  If \(\genltro^\chi\colon\setoffn\pr{\chi}\to\setoffn\pr{\chi}\) is a lower transition rate transformation, then
  \begin{equation*}
    \norm{\genltro^\chi}
    = 2\max\set{\abs{\br{\genltro^\chi\indica{x}}\pr{x}}\colon x\in \chi}.
  \end{equation*}
\end{lemma}
\begin{lemma}[Proposition~3 in \cite{2017Erreygers}]
\label{lem:LTRO:I+Delta Q is a LTT}
  Consider any lower transition rate transformation~\(\genltro^\chi\colon\setoffn\pr{\chi}\to\setoffn\pr{\chi}\) and \(\Delta\) in \(\nnegreals\).
  Then \(\pr{I+\Delta\genltro^\chi}\) is a lower transition transformation if and only if \(\Delta\norm{\genltro^\chi}\leq 2\).
\end{lemma}
Next, we repeat the two results that establish how a lower transition rate transformation generates a family of lower transition transformations; they are our direct inspiration for Theorem~\ref{the:Transformation induced by PoisGen in text}, as well as for Theorems~\ref{the:GenPoisGen:Phi_u_i converges to a LCT} and \ref{the:LCT induced by GenPoisGen} further on.
\begin{proposition}[Corollary~7.11 in \cite{2017Krak}]
\label{prop:Gen ltro^chi:Sequence converges to ltt}
  Consider some lower transition rate transformation~\(\genltro^\chi\colon\setoffn\pr{\chi}\to\setoffn\pr{\chi}\), and fix some \(t,s\) in \(\nnegreals\) such that \(t\leq s\).
  For every sequence \(\set{u_i}_{i\in\nats}\) in \(\setoftseq_{\br{t,s}}\) with \(\lim_{i\to+\infty}\sigma\pr{u_i}=0\), the corresponding sequence
  \begin{equation*}
    \set*{\prod_{k=1}^{k_i}\pr{I+\Delta^i_{k} \genltro^\chi}}_{i\in\nats}
  \end{equation*}
  converges to a lower transition transformation, where for every \(i\) in \(\nats\), \(k_i+1\) is the length of the sequence \(u_i=t^i_0, \dots, t^i_{k_i}\) and, for every \(k\) in \(\set{1,\dots,k_i}\),\(\Delta^i_k\) is the difference between the consecutive time points $t^i_k$ and \(t^i_{k-1}\) of this sequence.
\end{proposition}
\begin{proposition}[Theorem~7.12 in \cite{2017Krak}]
\label{prop:Gen ltro^chi:Induced lto^chi}
  Consider a lower transition rate transformation~\(\genltro^\chi\colon\setoffn\pr{\chi}\to\setoffn\pr{\chi}\).
  Then for any \(t,s\) in \(\nnegreals\) with \(t\leq s\), there is a unique lower transition transformation~\(\lto\colon\setoffn\pr{\chi}\to\setoffn\pr{\chi}\) such that
  \begin{equation*}
    \pr{\forall \epsilon\in\posreals}
    \pr{\exists\delta\in\posreals}
    \pr{\forall u\in\setoftseq_{\br{t,s}}, \sigma\pr{u}\leq\delta}~
    \norm*{\lto - \prod_{i=1}^n\pr{I+\Delta_i\genltro^\chi}}
    \leq \epsilon.
  \end{equation*}
\end{proposition}

Consider some lower transition rate transformation~\(\genltro^\chi\colon\setoffn\pr{\chi}\to\setoffn\pr{\chi}\), and fix some \(t,s\) in \(\nnegreals\) such that \(t\leq s\).
As explained by \citet[Section~7.3]{2017Krak}, the two results above allow us to define the corresponding lower transition transformation
\begin{equation*}
  \lto^\chi_{t,s}
  \coloneqq \lim_{\sigma\pr{u}\to0} \set*{\prod_{i=1}^n \pr{I+\Delta_i\genltro^\chi} \colon u\in\setoftseq_{\br{t,s}}}.
\end{equation*}
In this definition, the unconventional notation for the limit is used to indicate that the limit does not depend on the choice of sequence \(\set{u_i}_{i\in\nats}\) in \(\setoftseq_{\br{t,s}}\), all that is required is that \(\lim_{i\to+\infty}\sigma\pr{u_i}=0\).
We conclude this brief section by repeating the result that establishes that the family of corresponding lower transition transformations forms a semi-group.
\begin{proposition}[Propositions~7.13--14 in \cite{2017Krak}]
\label{prop:Gen ltro^chi:semi-group properties}
  Consider any lower transition rate operator~\(\genltro^\chi\colon\setoffn\pr{\chi}\to\setoffn\pr{\chi}\).
  Then for any \(t,s\) in \(\nnegreals\) with \(t\leq s\),
  \begin{enumerate}[label=\upshape(\roman*)]
    \item \(\lto^\chi_{t,t}=I\);
    \item \(\lto^\chi_{t,s}=\lto^\chi_{t,r} \lto^\chi_{r,s}\) for any \(r\) in \(\nnegreals\) such that \(t\leq r\leq s\);
    \item \(\lto^\chi_{t,s}=\lto^\chi_{0, s-t}\).
  \end{enumerate}
\end{proposition}

\section{The Generalised Poisson Generator}
\label{app:The Generalised Poisson Generator}
In this section, we essentially generalise the results of Appendix~\ref{app:ltro to lto} to the setting of a countably infinite state space; however, we limit ourselves to one specific type of (a generalisation of) lower transition rate transformations.
Essential to our exposition are sequences \(S\coloneqq\set{\pr{\llambda_x, \ulambda_x}}_{x\in\stsp}\) in \(\nnegreals^2\) such that \(\llambda_x\leq\ulambda_x\) for all \(x\in\stsp\).
Even more, we will usually demand that \(\llambda_x\) and \(\ulambda_x\) are both contained in \(\Lambda=\br{\llambda, \ulambda}\).
We collect all such sequences in the set
\begin{equation*}
  \setoflambdaseq_\Lambda
  \coloneqq\set*{\set{\pr{\llambda_x, \ulambda_x}}_{x\in\stsp} \text{ in } \br{\llambda, \ulambda}^2 \colon\pr{\forall x\in\stsp}~\llambda_x\leq\ulambda_x}.
\end{equation*}
With any \(S=\set{\pr{\llambda_x, \ulambda_x}}_{x\in\stsp}\) in \(\setoflambdaseq_\Lambda\), we associate the generalised Poisson generator~\(\ltro_S\), defined for all \(f\) in \(\setoffn\pr{\stsp}\) and \(x\) in \(\stsp\) as
\begin{equation}
\label{eqn:Generalized Poisson Generator}
  \br{\ltro_S f}\pr{x}
  \coloneqq \min\set{\lambda f\pr{x+1}-\lambda f\pr{x}\colon\lambda\in\br{\llambda_x, \ulambda_x}}.
\end{equation}
Observe that the generalised Poisson generator is a generalisation of the Poisson generator, because clearly
\begin{equation}
\label{eqn:PoisGen as special case of GenPoisGen}
  \ltro
  = \ltro_{S}
  \quad\text{with } S=\set{\pr{\llambda, \ulambda}}_{x\in\stsp}.
\end{equation}

\subsection{From Generalised Poisson Generators \textellipsis}
We first establish that generalised Poisson generators can be seen as lower transition rate transformations with a countably infinite state space; more precisely, we establish that they are transformations that furthermore satisfy properties that are similar to conditions \ref{ltrt:non-negatively homogeneous}--\ref{ltrt:non-negative off-diagonal} of Definition~\ref{def:Lower transition rate transformation}.
\begin{proposition}
\label{prop:GenPoisGen properties}
  Consider a sequence \(S=\set{\pr{\llambda_x,\ulambda_x}}_{x\in\stsp}\) in \(\setoflambdaseq_{\Lambda}\).
  Then \(\ltro_S\) is a transformation on \(\setoffn\pr{\stsp}\).
  Furthermore,
  \begin{enumerate}[label=\upshape{}GP\arabic*., ref=\upshape(GP\arabic*), leftmargin=*]
    \item \label{prop:GenPoisGen:Non-negative homogeneity} \(\ltro_S\pr{\gamma f}=\gamma\ltro_S f\) for all \(\gamma\) in \(\nnegreals\) and \(f\) in \(\setoffn\pr{\stsp}\); \hfill[non-negative homogeneity]
    \item \label{prop:GenPoisGen:Super-additivity}\(\ltro_S\pr{f+g}\geq\ltro_S f+\ltro_S g\) for all \(f, g\) in \(\setoffn\pr{\stsp}\); \hfill[super-additivity]
    \item \label{prop:GenPoisGen:Zero row sums}\(\ltro_s \mu=0\) for all \(\mu\) in \(\reals\); \hfill[zero row-sums]
    \item \label{prop:GenPoisGen:Non-negative off-diagonal}\(\br{\ltro_S \indica{y}}\pr{x}\geq 0\) for all \(x,y\) in \(\stsp\) with \(x\neq y\); \hfill[non-negative off-diagonal elements]
    \item \label{prop:GenPoisGen:Counting} \(\br{\ltro_S f}\pr{x}=\br{\ltro_S\pr{\indica{\geq x}f}}\pr{x}\) for all \(x\) in \(\stsp\) and \(f\) in \(\setoffn\pr{\stsp}\);
    \item \label{prop:GenPoisGen:Values above do not matter} \(\br{\ltro_S f}\pr{x}=\br{\ltro_S\pr{\indica{\leq x+1}f}}\pr{x}\) for all \(x\) in \(\stsp\) and \(f\) in \(\setoffn\pr{\stsp}\).
  \end{enumerate}
\end{proposition}
\begin{proof}
  We first verify that \(\ltro_S\) is a transformation on \(\setoffn\pr{\stsp}\).
  To that end, we fix any \(f\) in \(\setoffn\pr{\stsp}\).
  Observe that, for any \(x\) in \(\stsp\),
  \begin{equation*}
    \abs{\br{\ltro_S f}\pr{x}}
    = \abs*{\min\set{\lambda f\pr{x+1}-\lambda f\pr{x}\colon \lambda\in\pr{\llambda_x, \ulambda_x}}}
    \leq \ulambda_x \abs{f\pr{x+1}-f\pr{x}}
    \leq 2 \ulambda_x \norm{f}
    \leq 2 \ulambda \norm{f}.
  \end{equation*}
  Hence, \(\ltro_S f\) is clearly bounded.
  Since \(f\) was arbitrary, this proves that \(\ltro_S\) is a transformation, as required.

  For the second part of the stated, we observe that properties \ref{prop:GenPoisGen:Non-negative homogeneity}--\ref{prop:GenPoisGen:Values above do not matter} follow immediately from the definition of \(\ltro_S\).
\end{proof}

Next, we consider the norm of a generalised Poisson generator.
Note that the following result is similar to---or an extension of---Lemma~\ref{lem:Gen ltro^chi:Norm} because
\begin{equation*}
  2 \sup\set{\abs{\br{\ltro_S \indica{x}}\pr{x}}\colon x\in\stsp}
  =2 \sup\set{\abs{-\ulambda_x}\colon x\in\stsp}
  =2 \sup\set{\ulambda_x\colon x\in\stsp}.
\end{equation*}
\begin{lemma}
\label{lem:GenPoisGen:Norm}
  For any sequence \(S=\set{\pr{\llambda_x,\ulambda_x}}_{x\in\stsp}\) in \(\setoflambdaseq_{\Lambda}\),
  \begin{equation*}
    \norm{\ltro_S}
    = 2\sup\set{\ulambda_x\colon x\in\stsp}.
  \end{equation*}
\end{lemma}
\begin{proof}
  We first show that \(\norm{\ltro_S}\geq2\sup\set{\ulambda_x\colon x\in\stsp}\).
  For any \(y\) in \(\stsp\), we let \(f_y\coloneqq\indica{y}-\indica{y+1}\).
  Fix any \(y\) in \(\stsp\).
  Then for any \(z\) in \(\stsp\),
  \begin{align*}
    \br{\ltro_S f_y}\pr{z}
    &= \min\set{\lambda f_y\pr{z+1} - \lambda f_y\pr{z}\colon \lambda\in\br{\llambda_z, \ulambda_z}} \\
    &= \min\set{\lambda \indica{y}\pr{z+1}-\lambda \indica{y+1}\pr{z+1}-\lambda\indica{y}\pr{z}+\lambda\indica{y+1}\pr{z}\colon \lambda\in\br{\llambda_z, \ulambda_z}} \\
    &= \begin{cases}
      \llambda_z &\text{if } z=y-1 \text{ or } z=y+1,\\
      - 2\ulambda_z &\text{if } z=y, \\
      0 &\text{otherwise}.
    \end{cases}
  \end{align*}
  Observe that for any \(y\) in \(\stsp\), \(f_y\) is a bounded real-valued function on \(\stsp\) and that
  \begin{equation*}
    \norm{\ltro_S f_y}
    = \sup\set{\abs{\br{\ltro_S f_y}\pr{z}}\colon z\in\stsp}
    = \max\set{\llambda_{y-1}, 2\ulambda_y, \llambda_{y+1}},
  \end{equation*}
  where \(\llambda_{y-1}\) is not included in the set if \(y=0\).
  Therefore
  \begin{equation*}
    \sup\set{\norm{\ltro_S f_x}\colon x\in\stsp}
    = \sup\set{\max\set{\llambda_{x-1}, 2\ulambda_x, \llambda_{x+1}}\colon x\in\stsp}
    = 2\sup\set{\ulambda_x\colon x\in\stsp}.
  \end{equation*}
  Since \(\norm{f_x}=1\) for all \(x\) in \(\stsp\), we observe that
  \begin{equation*}
    \norm{\ltro_S}
    = \sup\set{\norm{\ltro_S f}\colon f\in\setoffn\pr{\stsp}, \norm{f}=1}
    \geq \sup\set{\norm{\ltro_S f_x}\colon x\in\stsp}
    = 2\sup\set{\ulambda_x\colon x\in\stsp},
  \end{equation*}
  as we set out to prove.

  Next, we prove that \(\norm{\ltro_S}\leq2\sup\set{\ulambda_x\colon x\in\stsp}\).
  Fix any \(\epsilon\) in \(\posreals\).
  Then by the definition of \(\norm{\ltro_S}\) and \(\norm{\ltro_S f}\), there is an \(f\) in \(\setoffn\pr{\stsp}\) with \(\norm{f}=1\) and an \(y\) in \(\stsp\) such that
  \begin{equation*}
    \norm{\ltro_S}-\epsilon
    < \norm{\ltro_S f}-\frac{\epsilon}2
    < \abs{\br{\ltro_S f}\pr{y}}.
  \end{equation*}
  From the definition of \(\ltro_S\), it now follows that
  \begin{align*}
    \abs{\br{\ltro_S f}\pr{y}}
    &= \abs{\min\set{\lambda f\pr{y+1}-\lambda f\pr{y}\colon\lambda\in\br{\llambda_y,\ulambda_y}}} \\
    &\leq \max \set{\lambda \abs{f\pr{y+1}-f\pr{y}}\colon\lambda\in\br{\llambda_y,\ulambda_y}} \\
    &= \ulambda_y\abs{f\pr{y+1}-f\pr{y}}
    \leq 2\ulambda_y
    \leq 2\sup\set{\ulambda_x\colon x\in\stsp},
  \end{align*}
  where the second inequality follows from the fact that \(\norm{f}=1\).
  Hence,
  \begin{equation*}
    \norm{\ltro_S}-\epsilon
    < \abs{\br{\ltro_S f}\pr{y}}
    \leq 2\sup\set{\ulambda_x\colon x\in\stsp}.
  \end{equation*}
  As this holds for any arbitrary positive real number \(\epsilon\), we conclude that \(\norm{\ltro_S}\leq 2\sup\set{\ulambda_x\colon x\in\stsp}\), as required.
  The stated now follows because \(\norm{\ltro_S}\geq2\sup\set{\ulambda_x\colon x\in\stsp}\) and \(\norm{\ltro_S}\leq2\sup\set{\ulambda_x\colon x\in\stsp}\).
\end{proof}

\subsection{\textellipsis to Lower Counting Transformations}
\label{sapp:GenPoisGen:To Lower Counting Transformations}
The generalised Poisson generator naturally defines a family of lower transition (or, more precisely, counting) transformations.
Crucial to our exposition are Theorems~\ref{the:GenPoisGen:Phi_u_i converges to a LCT} and \ref{the:LCT induced by GenPoisGen} further on.
In essence, these two results extend Propositions~\ref{prop:Gen ltro^chi:Sequence converges to ltt} and \ref{prop:Gen ltro^chi:Induced lto^chi} to the setting of generalised Poisson generators.
Even more, our reasoning that uniquely defines this family of transformations is largely analoguous to the line of reasoning followed in \cite[Appendix~E]{2017Krak}.
Our first step is the following observation.
\begin{lemma}
\label{lem:I+Delta GenPoisGen is LTT}
  Consider some \(S\) in \(\setoflambdaseq_\Lambda\) and some \(\Delta\) in \(\nnegreals\).
  Then \(\pr{I+\Delta\ltro_S}\) is a lower counting transformation if and only if \(\Delta\norm{\ltro_S}\leq2\).
\end{lemma}
\begin{proof}
  We first check the sufficiency of the condition \(\Delta\norm{\ltro_S}\leq2\).
  To that end, we fix any \(\Delta\) in \(\nnegreals\) that satisfies this condition, and let \(\lto\coloneqq I+\Delta\ltro_S\).
  That \(\lto\) satisfies \ref{def:LTT:Non-negative homgeneity} follows immediately from the non-negative homogeneity of \(I\) and that of \(\ltro\)---that is, \ref{prop:GenPoisGen:Non-negative homogeneity}; similarly, \ref{def:LTT:Super-additivity} follows immediately from the super-additivity of \(I\) and that of \(\ltro\)---that is, \ref{prop:GenPoisGen:Super-additivity}.
  Next, we verify that \ref{def:LTT:Bound} holds.
  To that, we fix any \(f\) in \(\setoffn\pr{\stsp}\) and \(x\) in \(\stsp\), and observe that
  \begin{align}
    \br{\lto f}\pr{x}
    &= \br{If}\pr{x}+\br{\Delta\ltro_Sf}\pr{x}
    = f\pr{x} + \Delta \br{\ltro_S f}\pr{x}
    = f\pr{x} + \Delta\min\set{\lambda_x f\pr{x+1}-\lambda_x f\pr{x}\colon \lambda_x\in\br{\llambda_x, \ulambda_x}} \notag\\
    &= \min\set{f\pr{x}+\Delta\lambda_x f\pr{x+1}-\Delta\lambda_x f\pr{x}\colon \lambda_x\in\br{\llambda_x, \ulambda_x}} \notag\\
    &= \min\set{\pr{1-\Delta\lambda_x} f\pr{x}+\Delta\lambda_x f\pr{x+1}\colon \lambda_x\in\br{\llambda_x, \ulambda_x}},
    \label{eqn:Proof of I+Delta GenPoisGen is LTT:minimum}
  \end{align}
  where the fourth equality holds because \(\Delta\geq 0\).
  Observe now that \(0\leq\Delta\lambda_x\) because \(\Delta\) and \(\lambda_x\) are non-negative by assumption, and that furthermore \(\Delta\lambda_x\leq 1\) because \(\Delta\norm{\ltro_S}\leq 2\) by assumption and \(\lambda_x\leq\sup\set{\lambda_y\colon y\in\stsp}=\norm{\ltro_S}/2\) due to Lemma~\ref{lem:GenPoisGen:Norm}.
  Consequently, the sum in the minimum in Equation~\eqref{eqn:Proof of I+Delta GenPoisGen is LTT:minimum} is a convex combination of \(f\pr{x}\) and \(f\pr{x+1}\).
  Because a convex combination of two real numbers is always greater than or equal to the minimum of these two numbers, it now follows that
  \begin{equation*}
    \br{\lto f}\pr{x}
    = \min\set{\pr{1-\Delta\lambda_x} f\pr{x}+\Delta\lambda_x f\pr{x+1}\colon \lambda_x\in\br{\llambda_x, \ulambda_x}}
    \geq \min\set{f\pr{x}, f\pr{x+1}}
    \geq \inf f.
  \end{equation*}
  Since this holds for all \(f\) in \(\setoffn\pr{\stsp}\) and \(x\) in \(\stsp\), this implies \ref{def:LTT:Bound}.
  Finally, \ref{def:LTT:Counting} follows immediately from \ref{prop:GenPoisGen:Counting} and the observation that \(\br{If}\pr{x}=\br{I\pr{\indica{\geq x}f}}\pr{x}\) for all \(f\) in \(\setoffn\pr{\stsp}\) and \(x\) in \(\stsp\).

  That the condition \(\Delta\norm{\ltro_S}\leq2\) is necessary follows from a counterexample.
  Fix some \(\Delta\) in \(\posreals\) such that \(\Delta\norm{\ltro_S}>2\), and assume ex-absurdo that \(\pr{I+\Delta\ltro_S}\) is a lower counting transformation.
  Fix any \(\epsilon\) in \(\posreals\) such that \(2\Delta\epsilon<\Delta\norm{\ltro_S}-2\).
  Then there is an \(x\) in \(\stsp\) such that
  \begin{equation}
    \br{\ltro_S\indica{x}}\pr{x}
    = -\ulambda_x
    \leq -\sup\set{\ulambda_y\colon y\in\stsp}+\epsilon
    = -\frac{\norm{\ltro_S}}2 + \epsilon,
  \end{equation}
  where the second equality follows from Lemma~\ref{lem:GenPoisGen:Norm}.
  Therefore
  \begin{equation*}
    \br{\pr{I+\Delta\ltro_S}\indica{x}}\pr{x}
    = \indica{x}\pr{x} + \Delta \br{\ltro_S\indica{x}}\pr{x}
    = 1 + \Delta \br{\ltro_S\indica{x}}\pr{x}
    \leq 1-\Delta \frac{\norm{\ltro_S}}2 + \Delta\epsilon
    = \frac12 \pr*{2-\Delta\norm{\ltro_S}+2\Delta\epsilon}
    < 0,
  \end{equation*}
  where the final inequality follows from our condition on \(\epsilon\).
  Since \(\inf \indica{x}=0\), this is a clear contradiction with \ref{def:LTT:Bound}.
  Hence, the condition \(\Delta\norm{\ltro_S}\leq2\) is indeed necessary.
\end{proof}

For our second step, we construct a lower counting transformation as the composition of lower counting transformations of the form of Lemma~\ref{lem:I+Delta GenPoisGen is LTT}.
More formally, we construct this transformation as follows.
For any sequence~\(S=\set{\pr{\llambda_x,\ulambda_x}}_{x\in\stsp}\) in \(\setoflambdaseq_{\Lambda}\), any \(t,s\) in \(\nnegreals{}\) such that \(t\leq s\) and any \(u\) in \(\setoftseq_{\br{t,s}}\), we let
\begin{equation}
\label{eqn:GenPoisGen:Phi_u}
  \Phi_u
  \coloneqq \begin{dcases}
    \prod_{i=1}^n \pr{I+\Delta_i\ltro_S} &\text{if } t<s, \\
    I &\text{otherwise.}
  \end{dcases}
\end{equation}
This notation is clearly reminiscent of the notation that was previously introduced in Section~\ref{ssec:The Poisson generator}; in fact, the latter is a special case of the former because---as was previously observed in Equation~\eqref{eqn:PoisGen as special case of GenPoisGen}---the Poisson generator~\(\ltro\) is a special case of the generalised Poisson generator.
The following result establishes that the operator~\(\Phi_u\) thus defined is a lower counting transformation.
\begin{corollary}
\label{cor:Phi_u is LTT}
  Fix a sequence \(S=\set{\pr{\llambda_x,\ulambda_x}}_{x\in\stsp}\) in \(\setoflambdaseq_{\Lambda}\), some \(t,s\) in \(\nnegreals\) with \(t\leq s\) and a sequence \(u=t_0, \dots, t_n\) in \(\setoftseq_{\br{t,s}}\).
  If \(\sigma\pr{u}\norm{\ltro_S}\leq2\), then \(\Phi_u\), as defined in Equation~\eqref{eqn:GenPoisGen:Phi_u}, is a lower counting transformation.
\end{corollary}
\begin{proof}
  Follows immediately from Lemmas~\ref{lem:LTT:Composition is again a LTT} and \ref{lem:I+Delta GenPoisGen is LTT}.
\end{proof}

Next, we establish some results that will allow us to determine the difference between \(\Phi_u\) and \(\Phi_{u'}\), where \(u\) and \(u'\) are two sequences of time points in \(\setoftseq_{\br{t,s}}\).
\begin{lemma}
\label{lem:GenPoisGen:Delta_i vs Delta}
  Consider a sequence \(S=\set{\pr{\llambda_x,\ulambda_x}}_{x\in\stsp}\) in \(\setoflambdaseq_{\Lambda}\) and a sequence \(\Delta_1, \dots, \Delta_n\) in \(\nnegreals\) with \(n\) in \(\nats\).
  If \(\Delta_i\norm{\ltro_S}\leq2\) for all \(i\) in \(\set{1, \dots, n}\), then
  \[
    \norm*{\prod_{i=1}^n \pr{I+\Delta_i\ltro_S} - \pr{I+\Delta\ltro_S}}
    \leq \norm{\ltro_S}^2 \sum_{i=1}^n \Delta_i \sum_{j=i+1}^n \Delta_j,
  \]
  where \(\Delta\coloneqq\sum_{i=1}^n\Delta_i\).
\end{lemma}
In our proof of Lemma~\ref{lem:GenPoisGen:Delta_i vs Delta}, we will make use of the following corollary.
\begin{corollary}
\label{cor:GenPoisGen:A-B}
  Consider a sequence \(S=\set{\pr{\llambda_x,\ulambda_x}}_{x\in\stsp}\) in \(\setoflambdaseq_{\Lambda}\).
  Then for any two non-negatively homogeneous transformations \(A\) and \(B\) on \(\setoffn\pr{\genset}\),
  \begin{equation*}
    \norm{\ltro_S A - \ltro_S B}
    \leq\norm{\ltro_S}\norm{A-B}.
  \end{equation*}
\end{corollary}
\begin{proof}
  Our proof is---almost---equal to that of \cite[R12]{2016DeBock}.
  Since the stated is clearly true for \(\norm{\ltro_S}=0\), we may assume that \(\norm{\ltro_S}>0\) without loss of generality.
  If we let \(\Delta\coloneqq2/\norm{\ltro_S}\), then it follows from Lemma~\ref{lem:I+Delta GenPoisGen is LTT} that \(\lto\coloneqq I+\Delta\ltro\) is a lower counting transformation.
  Observe that
  \begin{align*}
    \norm{\ltro_S A-\ltro_S B}
    &= \norm*{\frac{\norm{\ltro_S}}2\pr{\lto A-A}-\frac{\norm{\ltro_S}}2\pr{\lto B-B}}
    = \frac{\norm{\ltro_S}}2\norm{\pr{\lto A-\lto B}-\pr{A-B}} \\
    &\leq \frac{\norm{\ltro_S}}2\norm{\lto A-\lto B}+\frac{\norm{\ltro_S}}2\norm{A-B}
    \leq \frac{\norm{\ltro_S}}2\norm{A-B}+\frac{\norm{\ltro_S}}2\norm{A-B}
    = \norm{\ltro_S}\norm{A-B},
  \end{align*}
  where the final inequality follows from Lemma~\ref{lem:LTT:Further properties}~\ref{LTT:Composition with nneg hom} because \(\lto\) is a lower counting transformation by construction.
\end{proof}
\begin{proofof}{Lemma~\ref{lem:GenPoisGen:Delta_i vs Delta}}
  Our proof is one by induction, and is almost equivalent to the one that \citeauthor{2017Krak} provide for \cite[Lemma~E.5]{2017Krak}, although ours yields a (marginally) smaller upper bound.
  First, we observe that for \(n=1\), the stated is trivially true.
  Next, we fix some \(n\geq2\) and assume that the stated holds for \(1\leq n'<n\).
  We now show that this then implies that the stated also holds for \(n\).
  Some straightforward manipulations yield
  \begin{align*}
    \norm*{\prod_{i=1}^n \pr{I+\Delta_i\ltro_S} - \pr{I+\Delta\ltro_S}}
    &= \norm*{\prod_{i=2}^n \pr{I+\Delta_i\ltro_S} + \Delta_1 \ltro_S \prod_{i=2}^n \pr{I+\Delta_i\ltro_S} - I-\pr*{\sum_{i=2}^n\Delta_i}\ltro_S - \Delta_1 \ltro_S} \\
    &\leq \norm*{\prod_{i=2}^n \pr{I+\Delta_i\ltro_S} - I-\pr*{\sum_{i=2}^n\Delta_i}\ltro_S} + \norm*{\Delta_1 \ltro_S \prod_{i=2}^n \pr{I+\Delta_i\ltro_S}  - \Delta_1 \ltro_S} \\
    &\leq \norm{\ltro_S}^2 \sum_{i=2}^n \Delta_i \sum_{j=i+1}^n \Delta_j + \norm*{\Delta_1 \ltro_S \prod_{i=2}^n \pr{I+\Delta_i\ltro_S}  - \Delta_1 \ltro_S} \\
    &\leq \norm{\ltro_S}^2 \sum_{i=2}^n \Delta_i \sum_{j=i+1}^n \Delta_j + \Delta_1 \norm{\ltro_S} \norm*{\prod_{i=2}^n \pr{I+\Delta_i\ltro_S}  - I} \\
    &\leq \norm{\ltro_S}^2 \sum_{i=2}^n \Delta_i \sum_{j=i+1}^n \Delta_j + \Delta_1 \norm{\ltro_S} \sum_{i=2}^n\norm{\pr{I+\Delta_i\ltro_S}  - I} \\
    &= \norm{\ltro_S}^2 \sum_{i=2}^n \Delta_i \sum_{j=i+1}^n \Delta_j + \Delta_1 \norm{\ltro_S}^2 \sum_{i=2}^n \Delta_i
    = \norm{\ltro_S}^2 \sum_{i=1}^n \Delta_i \sum_{j=i+1}^k \Delta_j,
  \end{align*}
  where the first inequality is a consequence of the triangle inequality, the second inequality follows from the induction hypothesis, the third inequality follows from Corollaries~\ref{cor:Phi_u is LTT} and \ref{cor:GenPoisGen:A-B}, and the fourth inequality follows from Lemma~\ref{lem: bound on norm of A_1 ... A_k - B_1 ... B_k}.
\end{proofof}
\begin{lemma}
\label{lem:GenPoisGen:Delta_i,j vs Delta_i}
  Fix a sequence \(S\) in \(\setoflambdaseq_{\Lambda}\) and some \(n\) in \(\nats\).
  Furthermore, for all \(i\) in \(\set{1, \dots, n}\), we fix some sequence \(\Delta_{i,1}, \dots \Delta_{i,k_i}\) in \(\nnegreals\) and let \(\Delta_i\coloneqq\sum_{j=1}^{k_i} \Delta_{i, j}\).
  Let \(\Delta\coloneqq\sum_{i=1}^n\Delta_i\) and \(\Delta^\star\coloneqq \max\set{\Delta_i\given i\in\set{1, \dots, n}}\).
  If \(\Delta^\star\norm{\ltro_S}\leq2\), then
  \begin{equation*}
    \norm*{\prod_{i=1}^n \pr*{\prod_{j=1}^{k_i} \pr{I+\Delta_{i,j}\ltro_S} } - \prod_{i=1}^n \pr{I+\Delta_i\ltro_S}}
    \leq \norm{\ltro_S}^2 \sum_{i=1}^n \Delta_i^2
    \leq \norm{\ltro_S}^2 \Delta \Delta^\star.
  \end{equation*}
\end{lemma}
\begin{proof}
  Our proof is almost the same as that of \cite[Lemma~E.6]{2017Krak}.
  Observe that
  \begin{align*}
    \norm*{\prod_{i=1}^n \pr*{\prod_{j=1}^{k_i} \pr{I+\Delta_{i,j}\ltro_S} } - \prod_{i=1}^n \pr{I+\Delta_i\ltro_S}}
    &\leq \sum_{i=1}^n \norm*{\prod_{j=1}^{k_i} \pr{I+\Delta_{i,j}\ltro_S}  - \pr{I+\Delta_i\ltro_S}}
    \leq \sum_{i=1}^n \norm{\ltro_S}^2 \sum_{j=1}^{k_i} \Delta_{i, j} \sum_{\ell=j+1}^{k_i} \Delta_{i, \ell} \\
    &\leq \norm{\ltro_S}^2 \sum_{i=1}^n \sum_{j=1}^{k_i} \Delta_{i, j} \Delta_i
    \leq \norm{\ltro_S}^2 \sum_{i=1}^n \Delta_i \Delta_i \\
    &\leq \norm{\ltro_S}^2 \sum_{i=1}^n \Delta_{i} \Delta^\star
    = \norm{\ltro_S}^2 \Delta \Delta^\star,
  \end{align*}
  where the first inequality follows from Corollary~\ref{cor:Phi_u is LTT} and Lemma~\ref{lem: bound on norm of A_1 ... A_k - B_1 ... B_k}, and the second inequality follows from Lemma~\ref{lem:GenPoisGen:Delta_i vs Delta}.
\end{proof}
Everything is now set up to establish the following two results regarding the difference between \(\Phi_u\) and \(\Phi_{u'}\).
\begin{corollary}
\label{cor:GenPoisGen:Phi_u vs Phi_u'}
  Consider a sequence \(S\) in \(\setoflambdaseq_{\Lambda}\), some \(t,s\) in \(\nnegreals\) with \(t\leq s\) and some \(u\) in \(\setoftseq_{\br{t,s}}\) such that \(\sigma\pr{u}\norm{\ltro_S}\leq2\).
  Then for any \(u'\) in \(\setoftseq_{\br{t,s}}\) such that \(u\subseteq u'\),
  \begin{equation*}
    \norm{\Phi_u-\Phi_{u'}}
    \leq \sigma\pr{u}\pr{s-t}\norm{\ltro_S}^2.
  \end{equation*}
\end{corollary}
\begin{proof}
   Follows almost immediately from Lemma~\ref{lem:GenPoisGen:Delta_i,j vs Delta_i}
\end{proof}
\begin{lemma}
\label{lem:GenPoisGen:Phi_u-Phi_u'}
  Fix a sequence \(S\) in \(\setoflambdaseq_{\Lambda}\), \(t,s\) in \(\nnegreals\) with \(t\leq s\), \(\delta\) in \(\nnegreals\) with \(\delta\norm{\ltro_S}\leq2\) and \(u, u'\) in \(\setoftseq_{\br{t,s}}\).
  If \(\sigma\pr{u}\leq\delta\) and \(\sigma\pr{u'}\leq\delta\), then
  \begin{equation*}
    \norm{\Phi_u-\Phi_{u'}}
    \leq 2 \delta \pr{s-t} \norm{\ltro_S}^2.
  \end{equation*}
\end{lemma}
\begin{proof}
  Our proof is entirely similar to that of \citet[Proposition~7.9]{2017Krak}.
  Let \(u^\star\) be the sequence of time points in \(\setoftseq_{\br{t,s}}\) that contains all time points in \(u\) and \(u'\).
  It then follows immediately from Corollary~\ref{cor:GenPoisGen:Phi_u vs Phi_u'} that \(\norm{\Phi_u-\Phi_{u^\star}}\leq\delta\pr{s-t}\norm{\ltro_S}^2\) and \(\norm{\Phi_{u'}-\Phi_{u^\star}}\leq\delta\pr{s-t}\norm{\ltro_S}^2\), whence
  \begin{equation*}
    \norm{\Phi_u-\Phi_{u'}}
    \leq \norm{\Phi_u-\Phi_{u^\star}} + \norm{\Phi_{u^\star}-\Phi_{u'}}
    \leq 2\delta\pr{s-t}\norm{\ltro_S}^2.
  \end{equation*}
\end{proof}

Now that we have an upper bound on the measure of the distance between \(\Phi_u\) and \(\Phi_{u'}\), we can fix some sequence \(\set{u_i}_{i\in\nats}\) in \(\setoftseq_{\br{t,s}}\) and study the behaviour of the corresponding sequence \(\set{\Phi_{u_i}}_{i\in\nats}\) in the limit for \(i\to+\infty\).
\begin{lemma}
\label{lem:GenPoisGen:Phu_u_i is a Cauchy sequence}
  Fix a sequence \(S\) in \(\setoflambdaseq_{\Lambda}\) and some \(t,s\) in \(\nnegreals\) with \(t\leq s\).
  Then for every sequence \(\set{u_i}_{i\in\nats}\) in \(\setoftseq_{\br{t,s}}\) such that \(\lim_{i\to+\infty}\sigma\pr{u_i}=0\), the corresponding sequence \(\set*{\Phi_{u_i}}_{i\in\nats}\) is Cauchy.
\end{lemma}
\begin{proof}
  In order to prove the stated, we need to show that for every \(\epsilon\) in \(\posreals\), there exists an \(i^\star\) in \(\nats\) such that \(\norm{\Phi_{u_i}-\Phi_{u_j}}\leq\epsilon\) for all \(i,j\) in \(\nats\) with \(i\geq i^\star\) and \(j\geq i^\star\).
  Fix now any \(\epsilon\) in \(\posreals\).
  Because \(\lim_{i\to+\infty}\sigma\pr{u_i}=0\), there is an \(i^\star\) in \(\nats\) such that (i) \(\sigma\pr{u_i}\norm{\ltro_S}\leq 2\) for all \(i\geq i^\star\), and (ii) \(2\sigma\pr{u_i}\pr{s-t}\norm{\ltro_S}^2\leq\epsilon\) for all \(i\geq i^\star\).
  From this and Lemma~\ref{lem:GenPoisGen:Phi_u-Phi_u'}, it now follows that, for all \(i,j\) in \(\nats\) with \(i\geq i^\star\) and \(j\geq i^\star\),
  \begin{equation*}
    \norm{\Phi_{u_i}-\Phi_{u_j}}
    \leq 2 \max\set{\sigma\pr{u_i}, \sigma\pr{u_j}} \pr{s-t} \norm{\ltro_S}^2
    \leq \epsilon.
  \end{equation*}
  Because \(\epsilon\) was an arbitrary positive real number, this proves the stated.
\end{proof}
\begin{lemma}
\label{lem:GenPoisGen:Phi_u_i f converges to a limit}
  Fix a sequence \(S\) in \(\setoflambdaseq_{\Lambda}\), some \(t,s\) in \(\nnegreals\) with \(t\leq s\) and some \(f\) in \(\setoffn\pr{\stsp}\).
  For every sequence \(\set{u_i}_{i\in\nats}\) in \(\setoftseq_{\br{t,s}}\) such that \(\lim_{i\to+\infty}\sigma\pr{u_i}=0\), the corresponding sequence \(\set{\Phi_{u_i} f}_{i\in\nats}\) converges to a limit~\(f_{\mathrm{lim}}\) in \(\setoffn\pr{\stsp}\) that does not depend on the chosen sequence \(\set{u_i}_{i\in\nats}\).
\end{lemma}
\begin{proof}
  Our proof consists of two parts.
  In the first part, we will prove that \(\set{\Phi_{u_i}f}_{i\in\nats}\) converges to a limit; in the second part, we will prove that this limit does not depend on the chosen sequence \(\set{u_i}_{i\in\nats}\).

  Fix some sequence \(\set{u_i}_{i\in\nats}\) in \(\setoftseq_{\br{t,s}}\) such that \(\lim_{i\to+\infty}\sigma\pr{u_i}=0\).
  The corresponding sequence \(\set{\Phi_{u_i}f}_{i\in\nats}\) converges to a limit because (i) \(\setoffn\pr{\stsp}\) is a complete normed vector space, and (ii) \(\set{\Phi_{u_i}f}_{i\in\nats}\) is a Cauchy sequence in \(\setoffn\pr{\stsp}\).
  We now prove that \(\set{\Phi_{u_i}f}_{i\in\nats}\) is a Cauchy sequence.
  To that end, we fix some \(\epsilon\) in \(\posreals\).
  If \(\norm{f}=0\), then \(f=0\).
  Hence, it follows almost immediately from Equations~\eqref{eqn:GenPoisGen:Phi_u} and \eqref{eqn:Generalized Poisson Generator} and \ref{LTT:constant} that \(\Phi_{u_i} f=\Phi_{u_i} 0=0\) for all \(i\) in \(\nats\).
  Consequently, \(\norm{\Phi_{u_i}f-\Phi_{u_j}}=0\leq\epsilon\) for all \(i,j\) in \(\nats\), and so the veracity of the claim is immediate.

  Next, we consider the alternative case that \(\norm{f}\neq0\).
  By Lemma~\ref{lem:GenPoisGen:Phu_u_i is a Cauchy sequence}, there is a \(i^\star\) in \(\nats\) such that
  \begin{equation*}
    \pr{\forall i,j\in\nats, i\geq i^\star, j\geq i^\star}~
    \norm{\Phi_{u_i} - \Phi_{u_j}}
    \leq \frac{\epsilon}{\norm{f}}.
  \end{equation*}
  Observe now that
  \begin{equation*}
    \pr{\forall i,j\in\nats, i\geq i^\star, j\geq i^\star}~
    \norm{\Phi_{u_i}f - \Phi_{u_j}f}
    \leq \norm{\Phi_{u_i}-\Phi_{u_j}}\norm{f}
    \leq \frac{\epsilon}{\norm{f}}\norm{f}
    = \epsilon,
  \end{equation*}
  where the first inequality holds due to \ref{BNH: norm Af <= norm A norm f} because---for reasons explained right after Lemma~\ref{lem:Properties of non-negatively homogeneous transformations}---\(A\coloneqq\Phi_{u_i}-\Phi_{u_j}\) is a non-negatively homogeneous transformation.
  Since \(\epsilon\) was an arbitrary positive real number, we conclude from this that \(\set{\Phi_{u_i} f}_{i\in\nats}\) is Cauchy.

  Next, we prove that the limit does not depend on the chosen sequence.
  To that end, we fix two sequences \(\set{u_i}_{i\in\nats}\) and \(\set{u'_i}_{i\in\nats}\) such that \(\lim_{i\to+\infty}\sigma\pr{u_i}=0\) and \(\lim_{i\to+\infty}\sigma\pr{u'_i}=0\).
  Furthermore, we let \(f_{\lim}\) and \(f'_{\lim}\) denote the limits of \(\set{\Phi_{u_i}f}_{i\in\nats}\) and \(\set{\Phi_{u'_i}f}_{i\in\nats}\), respectively.
  In order to prove the stated, we need to verify that \(f_{\lim}=f'_{\lim}\).
  To that end, we observe that, for all \(i\) in \(\nats\),
  \begin{equation}
  \label{eqn:Proof of GenPoisGen:Phi_u_i f converges to a limit:Triangle inequality}
    \norm{f_{\lim}-f'_{\lim}}
    = \norm{f_{\lim}-\Phi_{u_i}f+\Phi_{u_i}f-\Phi_{u'_i}f+\Phi_{u'_i}f-f'_{\lim}}
    \leq \norm{f_{\lim}-\Phi_{u_i}f}+\norm{\Phi_{u_i}f-\Phi_{u'_i}}f+\norm{f'_{\lim}-\Phi_{u'_i}f},
  \end{equation}
  where the inequality follows from the triangle inequality.
  Fix now any \(\epsilon\) in \(\posreals\), and choose any \(\epsilon'\) in \(\posreals\) such that \(3\epsilon'\leq\epsilon\), and additionally choose any \(\delta\) in \(\posreals\) such that \(2\delta\pr{s-t}\norm{\ltro_S}^2\norm{f}\leq \epsilon'\).
  Due to the first part of the statement, and because \(\lim_{i\to+\infty}\sigma\pr{u_i}=0=\lim_{i\to+\infty}\sigma\pr{u'_i}\), there is some \(j\) in \(\nats\) such that
  \begin{equation}
  \label{eqn:Proof of GenPoisGen:Phi_u_i f converges to a limit:bounds on diff with limit}
    \norm{\Phi_{u_j}-f_{\lim}}
    \leq \epsilon'
    \qquad\text{and}\qquad
    \norm{\Phi_{u'_j}-f'_{\lim}}
    \leq \epsilon'
  \end{equation}
  and
  \begin{equation}
  \label{eqn:Proof of GenPoisGen:Phi_u_i f converges to a limit:bounds on max granularity}
    \sigma\pr{u_j}
    \leq \delta
    \qquad\text{and}\qquad
    \sigma\pr{u'_j}
    \leq \delta.
  \end{equation}
  Fix any such \(j\).
  Observe furthermore that
  \begin{equation}
  \label{eqn:Proof of GenPoisGen:Phi_u_i f converges to a limit:bounds on diff between approx}
    \norm{\Phi_{u_j}f-\Phi_{u'_j}f}
    \leq \norm{\Phi_{u_j}-\Phi_{u'_j}}\norm{f}
    \leq 2\delta\pr{s-t}\norm{\ltro_S}^2 \norm{f}
    \leq \epsilon',
  \end{equation}
  where the first inequality holds due to \ref{BNH: norm Af <= norm A norm f} because---for reasons mentioned right after Lemma~\ref{lem:Properties of non-negatively homogeneous transformations}---\(\Phi_{u_j}-\Phi_{u'_j}\) is a non-negatively homogeneous transformation, the second inequality follows from Equation~\eqref{eqn:Proof of GenPoisGen:Phi_u_i f converges to a limit:bounds on max granularity} and Corollary~\ref{cor:GenPoisGen:Phi_u vs Phi_u'} and the final inequality is precisely our condition on \(\delta\).
  We now use Equations~\eqref{eqn:Proof of GenPoisGen:Phi_u_i f converges to a limit:bounds on diff with limit} and \eqref{eqn:Proof of GenPoisGen:Phi_u_i f converges to a limit:bounds on diff between approx} to bound the terms in Equation~\eqref{eqn:Proof of GenPoisGen:Phi_u_i f converges to a limit:Triangle inequality} for \(i=j\), to yield
  \begin{equation*}
    \norm{f_{\lim}-f'_{\lim}}
    \leq \norm{f_{\lim}-\Phi_{u_j}}+\norm{\Phi_{u_j}-\Phi_{u'_j}}+\norm{f'_{\lim}-\Phi_{u'_j}}
    \leq 3\epsilon'
    \leq \epsilon,
  \end{equation*}
  where the final inequality is precisely our condition on \(\epsilon'\).
  Because \(\epsilon\) was an arbitrary positive real number, we infer from this inequality that \(\norm{f_{\lim}-f'_{\lim}}=0\), which in turn implies that \(f_{\lim}=f'_{\lim}\).
\end{proof}
We now have all the necessary intermediary results to establish the two main results regarding the limit behaviour of the sequences \(\set{\Phi_{u_i}}_{i\in\nats}\).
Our first result establishes that the sequence always converges to a lower counting transformation.
In this sense, it is similar to Proposition~\ref{prop:Gen ltro^chi:Sequence converges to ltt}---that is, \cite[Corollary~7.11]{2017Krak}.
\begin{theorem}
\label{the:GenPoisGen:Phi_u_i converges to a LCT}
  Consider a sequence \(S\) in \(\setoflambdaseq_{\Lambda}\), and fix some \(t,s\) in \(\nnegreals\) with \(t\leq s\).
  For any sequence \(\set{u_i}_{i\in\nats}\) in \(\setoftseq_{\br{t,s}}\) such that \(\lim_{i\to+\infty}\sigma\pr{u_i}=0\), the corresponding sequence \(\set{\Phi_{u_i}}_{i\in\nats}\) converges to a lower counting transformation.
\end{theorem}
\begin{proof}
  Recall from Lemma~\ref{lem:GenPoisGen:Phi_u_i f converges to a limit} that, for all \(f\) in \(\setoffn\pr{\stsp}\), the sequence \(\set{\Phi_{u_i} f}_{i\in\nats}\) converges to the bounded function \(f_{\lim}\).
  Let \(\lto\) be the transformation that maps any \(f\) in \(\setoffn\pr{\stsp}\) to the corresponding limit \(f_{\lim}\):
  \begin{equation}
  \label{eqn:Proof of GenPoisGen:Phi_u_i converges to a LCT:Definition of LTO}
    \lto
    \colon\setoffn\pr{\stsp}\to\setoffn\pr{\stsp}
    \colon f\mapsto \lto f
    \coloneqq \lim_{i\to+\infty} \Phi_{u_i} f
    = f_{\lim}.
  \end{equation}

  We now first verify that \(\lto\) is a lower counting transformation---and therefore also a non-negatively homogeneous transformation.
  Because \(\lim_{i\to+\infty}\sigma\pr{u_i}=0\), there is an \(i^\star\) in \(\nats\) such that \(\sigma\pr{u_i}\norm{\ltro_S}\leq 2\) for all \(i\geq i^\star\).
  From this and Corollary~\ref{cor:Phi_u is LTT}, it follows that \(\Phi_{u_i}\) is a lower counting transformation for all \(i\geq i^\star\).
  This implies that \(\lto\), as defined in Equation~\eqref{eqn:Proof of GenPoisGen:Phi_u_i converges to a LCT:Definition of LTO}, is a lower counting transformation as well because the (in)equalities in the conditions \ref{def:LTT:Non-negative homgeneity}--\ref{def:LTT:Counting} are preserved under taking limits.

  Next, we verify that \(\set{\Phi_{u_i}}_{i\in\nats}\) converges to \(\lto\).
  To that end, we fix any \(\epsilon\) in \(\posreals\), and choose some \(\epsilon'\) in \(\posreals\) such that \(3\epsilon'\leq\epsilon\).
  Recall from Lemma~\ref{lem:GenPoisGen:Phu_u_i is a Cauchy sequence} that \(\set{\Phi_{u_i}}_{i\in\nats}\) is a Cauchy sequence.
  Hence, there is an \(i_{\epsilon}\) in \(\nats\) such that, for all \(i,j\) in \(\nats\) with \(i\geq i_{\epsilon}\) and \(j\geq i_{\epsilon}\),
  \begin{equation}
  \label{eqn:Proof of GenPoisGen:Phi_u_i converges to a LCT:Cauchy condition}
    \norm{\Phi_{u_i}-\Phi_{u_j}}
    \leq \epsilon'.
  \end{equation}
  Fix now any \(i\) in \(\nats\) such that \(i\geq i_{\epsilon}\).
  From the definition of the norm for non-negatively homogeneous transformations, it follows that there is some \(f_1\) in \(\setoffn\pr{\stsp}\) with \(\norm{f_1}=1\) such that
  \begin{equation}
  \label{eqn:Proof of GenPoisGen:Phi_u_i converges to a LCT:Bound 1}
    \norm{\lto-\Phi_{u_i}}
    \leq \norm{\lto f_1-\Phi_{u_i}f_1}+\epsilon'.
  \end{equation}
  Furthermore, due to Equation~\eqref{eqn:Proof of GenPoisGen:Phi_u_i converges to a LCT:Definition of LTO}, there is a \(j\) in \(\nats\) such that \(j\geq i_{\epsilon}\) and \(\norm{\lto f_1 - \Phi_{u_j} f_1}\leq \epsilon'\).
  We now use this and Equation~\eqref{eqn:Proof of GenPoisGen:Phi_u_i converges to a LCT:Bound 1}, to yield
  \begin{align*}
    \norm{\lto-\Phi_{u_i}}
    &\leq \norm{\lto f_1-\Phi_{u_i}f_1}+\epsilon'
    = \norm{\lto f_1-\Phi_{u_j}f_1+\Phi_{u_j}f_1-\Phi_{u_i}f_1}+\epsilon' \\
    &\leq \norm{\lto f_1-\Phi_{u_j}f_1}+\norm{\Phi_{u_j}f_1-\Phi_{u_i}f_1}+\epsilon' \\
    &\leq \norm{\Phi_{u_j} f_1-\Phi_{u_i} f_1}+2\epsilon'.
  \end{align*}
  Finally, we use \ref{BNH: norm Af <= norm A norm f} and the fact that \(\norm{f_1}=1\), to yield
  \begin{equation*}
    \norm{\lto-\Phi_{u_i}}
    \leq \norm{\Phi_{u_j}-\Phi_{u_i}} \norm{f_1}+2\epsilon'
    = \norm{\Phi_{u_j}-\Phi_{u_i}}+2\epsilon'
    \leq 3\epsilon'
    \leq \epsilon,
  \end{equation*}
  where the penultimate inequality follows from Equation~\eqref{eqn:Proof of GenPoisGen:Phi_u_i converges to a LCT:Cauchy condition} because \(i\geq i_{\epsilon}\) and \(j\geq i_{\epsilon}\), and where the final inequality is precisely our condition on \(\epsilon'\).
  Because this inequality holds for any \(i\geq i_\epsilon\), and because \(\epsilon\) was an arbitrary positive real number, we infer from this that \(\lim_{i\to+\infty} \Phi_{u_i}=\lto\), as required.
\end{proof}
Our second result establishes that the limit of \(\set{\Phi_{u_i}}_{i\in\nats}\) is unique, in the sense that it does not depend on the choice of \(\set{u_i}_{i\in\nats}\).
Note the similarity with Proposition~\ref{prop:Gen ltro^chi:Induced lto^chi}---that is, \cite[Theorem~7.12]{2017Krak}.
\begin{theorem}
\label{the:LCT induced by GenPoisGen}
  Consider a sequence \(S\) in \(\setoflambdaseq_{\Lambda}\).
  For any \(t,s\) in \(\nnegreals\) with \(t\leq s\), there is a unique lower counting transformation~\(\lto\) such that
  \begin{equation*}
    \pr{\forall\epsilon\in\posreals}
    \pr{\exists\delta\in\posreals}
    \pr{\forall u\in\setoftseq_{\br{t,s}},\sigma\pr{u}\leq\delta}~
    \norm{\lto-\Phi_u}
    \leq \epsilon.
  \end{equation*}
\end{theorem}
\begin{proof}
  Consider any sequence \(\set{u_i}_{i\in\nats}\) such that \(\lim_{i\to+\infty}\sigma\pr{u_i}=0\).
  Let \(\lto\coloneqq\lim_{i\to+\infty}\Phi_{u_i}\), where this limit exists and is a lower counting transformation due to Theorem~\ref{the:GenPoisGen:Phi_u_i converges to a LCT}.
  We now verify that this lower counting transformation~\(\lto\) satisfies the condition of the statement.
  To that end, we fix any \(\epsilon\) in \(\posreals\), and choose any \(\epsilon'\) in \(\posreals\) such that \(3\epsilon'\leq\epsilon\).
  Additionally, we choose any \(\delta\) in \(\posreals\) such that \(2\delta\pr{s-t}\norm{\ltro_S}^2\leq\epsilon'\).
  We now proceed in a similar fashion as in the second part of the proof of Theorem~\ref{the:GenPoisGen:Phi_u_i converges to a LCT}.
  Fix any \(u\) in \(\setoftseq_{\br{t,s}}\) such that \(\sigma\pr{u}\leq\delta\).
  By definition of the norm for non-negatively homogeneous transformations, there is some \(f_1\) in \(\setoffn\pr{\stsp}\) with \(\norm{f_1}=1\) such that
  \begin{equation*}
    \norm{\lto-\Phi_u}
    \leq \norm{\lto f_1-\Phi_u f_1} + \epsilon'.
  \end{equation*}
  Because \(\lim_{i\to+\infty}\Phi_{u_i}=\lto\) and \(\lim_{i\to+\infty}\sigma\pr{u_i}=0\), there is an \(i\) in \(\nats\) such that \(\sigma\pr{u_i}\leq\delta\) and \(\norm{\lto-\Phi_{u_i}}\leq\epsilon'\).
  Observe now that
  \begin{align*}
    \norm{\lto-\Phi_u}
    &\leq \norm{\lto f_1-\Phi_u f_1} + \epsilon_1
    \leq \norm{\lto f_1-\Phi_{u_i} f_1}+\norm{\Phi_{u_i}f_1-\Phi_u f_1} + \epsilon' \\
    &\leq \norm{\lto-\Phi_{u_i}}+\norm{\Phi_{u_i}-\Phi_u} + \epsilon'
    \leq \norm{\Phi_{u_i}-\Phi_u} + 2\epsilon' \\
    &\leq 2\delta\pr{s-t}\norm{\ltro_S}^2 + 2\epsilon'
    \leq 3\epsilon'
    \leq\epsilon,
  \end{align*}
  where the second inequality follows from the triangle inequality, the third inequality follows from \ref{BNH: norm Af <= norm A norm f} and\(\norm{f_1}=1\), the fourth inequality holds because \(i\) was fixed in such a way that \(\norm{\lto-\Phi_{u_i}}\leq\epsilon'\), the fifth inequality follows from Lemma~\ref{lem:GenPoisGen:Phi_u-Phi_u'} because \(\sigma\pr{u}\leq\delta\) and \(\sigma\pr{u_i}\leq\delta\) and the penultimate inequality follows from our condition on \(\delta\).
  Because this inequality holds for any \(u\) in \(\setoftseq_{\br{t,s}}\) such that \(\sigma\pr{u}\leq\delta\), and because \(\epsilon\) was an arbitrary positive real number, this verifies the condition of the stated.

  Finally, we verify that \(\lto\) is unique.
  To that end, we let \(\lto'\) be any lower counting transformation that (also) satisfies the condition of the stated.
  For any \(\epsilon\) in \(\posreals\), we then clearly have that there is a \(u\) in \(\setoftseq_{\br{t,s}}\) such that \(\norm{\lto-\Phi_u}\leq\epsilon/2\) and \(\norm{\lto'-\Phi_u}\leq\epsilon/2\).
  Hence, \(\norm{\lto-\lto'}\leq\norm{\lto-\Phi_u}+\norm{\lto'-\Phi_u}\leq\epsilon\).
  Since \(\epsilon\) is an arbitrary positive real number, we conclude from this that \(\norm{\lto-\lto'}=0\), which in turn implies that \(\lto=\lto'\), as required.
\end{proof}
We end with two useful properties of the approximation~\(\Phi_u\).
\begin{lemma}
\label{lem:I+Delta GenPoisGen is state homogeneous}
  Consider a sequence \(S=\set{\pr{\llambda_x,\ulambda_x}}_{x\in\stsp}\) in \(\setoflambdaseq_{\Lambda}\).
  Fix an \(n\) in \(\nats\) and, for all \(i\) in \(\set{1, \dots, n}\), a \(\Delta_i\) in \(\nnegreals\) with \(\Delta_i\norm{\ltro_S}\leq2\).
  Then for any \(f\) in \(\setoffn\pr{\stsp}\) and \(x\) in \(\stsp\),
  \begin{equation*}
    \br*{\prod_{i=1}^n\pr{I+\Delta_i\ltro_S} f}\pr{y}
    = \br*{\prod_{i=1}^n\pr{I+\Delta_i\ltro_S} f'_x}\pr{y-x}
    \qquad\text{for all } y\in\stsp \text{ with } y\geq x,
  \end{equation*}
  where \(f'_x\colon\stsp\to\reals\colon z\mapsto f'_x\pr{z}\coloneqq f\pr{x+z}\).
\end{lemma}
\begin{proof}
  Our proof is one by induction.
  First, we consider the case \(n=1\).
  Then clearly
  \begin{align*}
    \br{\pr{I+\Delta_1\ltro_S} f}\pr{y}
    &= f\pr{y} + \Delta_1\br{\ltro_S f}\pr{y}
    = f\pr{y} + \Delta_1\min_{\lambda\in\br{\llambda_y, \ulambda_y}}\lambda\pr{f\pr{y+1}-f\pr{y}} \\
    &= f'_x\pr{y-x} + \Delta_1\min_{\lambda\in\br{\llambda_y, \ulambda_y}}\lambda\pr{f'_x\pr{y+1-x}-f'_x\pr{y-x}}
    = f'_x\pr{y-x} + \Delta_1\br{\ltro_S f'_x}\pr{y-x} \\
    &= \br{\pr{I+\Delta_1\ltro_S}f'_x}\pr{y-x}.
  \end{align*}

  Next, we fix some \(n\) in \(\nats\) with \(n\geq 2\) and assume that the stated holds for all \(n'\) in \(\nats\) with \(1\leq n'<n\).
  We now show that this implies that the stated holds for \(n\) as well.
  Let \(g\coloneqq\prod_{i=2}^n\pr{I+\Delta_i\ltro_S} f\).
  Then
  \begin{equation}
  \label{eqn:Proof of state homogeneity of I+delta trm_lambda:equality}
    \br*{\prod_{i=1}^n\pr{I+\Delta_i\ltro_S} f}\pr{y}
    = \br{\pr{I+\Delta_1\ltro_S} g}\pr{y}
    = \br{\pr{I+\Delta_1\ltro_S} g'_x}\pr{y-x},
  \end{equation}
  where we let \(g'_x\colon\stsp\to\reals\colon z\mapsto g'_x\pr{z}\coloneqq g\pr{z+x}\) and where the second equality follows from the induction hypothesis for \(n'=1\).
  Observe now that, for any \(z\) in \(\stsp\),
  \begin{equation*}
    g'_x\pr{z}
    = g\pr{z+x}
    = \br*{\prod_{i=2}^n\pr{I+\Delta_i\ltro_S} f}\pr{z+x}
    = \br*{\prod_{i=2}^n\pr{I+\Delta_i\ltro_S} f'_x}\pr{z},
  \end{equation*}
  where the third equality follows from the induction hypothesis for \(n'=n-1\).
  Since this holds for all \(z\), this implies that \(g'_x=\prod_{i=2}^n\pr{I+\Delta_i\ltro_S} f'_x\).
  We now substitute this equality in Equation~\eqref{eqn:Proof of state homogeneity of I+delta trm_lambda:equality} to obtain the stated:
  \begin{equation*}
    \br*{\prod_{i=1}^n\pr{I+\Delta_i\ltro_S} f}\pr{y}
    = \br*{\pr{I+\Delta_1\ltro_S} \prod_{i=2}^n\pr{I+\Delta_i\ltro_S} f'_x}\pr{y-x}
    = \br*{\prod_{i=1}^n\pr{I+\Delta_i\ltro_S} f'_x}\pr{y-x}.
  \end{equation*}
\end{proof}
\begin{lemma}
\label{lem:Phi_u:Value of f does not matter above treshold}
  Consider a sequence \(S=\set{\pr{\llambda_x,\ulambda_x}}_{x\in\stsp}\) in \(\setoflambdaseq_{\Lambda}\).
  Fix some \(t,s\) in \(\nnegreals\) with \(t\leq s\), a sequence \(u=t_0, \dots, t_n\) in \(\setoftseq_{\br{t,s}}\) with \(\sigma\pr{u}\norm{\ltro_S}\leq2\) and an \(f\) in \(\setoffn\pr{\stsp}\).
  Then for any \(x,y\) in \(\stsp\) with \(y\geq x+n\),
  \begin{equation*}
    \br{\Phi_u f}\pr{x}
    = \br{\Phi_u \pr{\indica{\leq y}f + f\pr{y}\indica{>y}}}\pr{x}.
  \end{equation*}
\end{lemma}
\begin{proof}
  In order to simplify our notation, for any \(y\) in \(\stsp\) we let \(f_y\coloneqq\indica{\leq y}f+f\pr{y}\indica{>y}\).
  Observe that \(f_y\) is eventually constant starting from \(y\) by construction, and that \(f_y\pr{x}=f\pr{x}\) for all \(x\) in \(\stsp\) such that \(x\leq y\).

  We first consider the case \(t=s\).
  In this case, we have \(n=0\) and \(\Phi_u=I\).
  Fix any \(y\geq x+n=x\).
  We immediately see that
  \begin{equation*}
    \br{\Phi_u f}\pr{x}
    = \br{I f}\pr{x}
    = f\pr{x}
    = \indica{\leq y}\pr{x}f\pr{x} + \indica{>y}\pr{x}f\pr{y}
    = f_y\pr{x}
    = \br{I f_y}\pr{x}
    = \br{\Phi_u f_y}\pr{x}
    = \br{\Phi_u \pr{\indica{\leq y}f+f\pr{y}\indica{>y}}}\pr{x},
  \end{equation*}
  as required.

  Next, we consider the case \(t<s\).
  We will prove the stated by induction.
  Assume first that \(n=1\).
  Fix any \(y\geq x+n=x+1\).
  Observe that
  \begin{equation*}
    \br{\Phi_u f}\pr{x}
    = \br{\pr{I+\Delta_1 \ltro_S}f}\pr{x}
    = f\pr{x} + \Delta_1 \br{\ltro_S f}\pr{x}.
  \end{equation*}
  Recall from the beginning of this proof that \(f\pr{x}=f_y\pr{x}\) by construction.
  Furthermore, because \(\indica{\leq x+1}f=\indica{\leq x+1}f_y\), it follows from Proposition~\ref{prop:GenPoisGen properties}~\ref{prop:GenPoisGen:Values above do not matter} that \(\br{\ltro_S f}\pr{x}=\br{\ltro_S \pr{\indica{\leq x+1}f}}\pr{x}=\br{\ltro_S \pr{\indica{\leq x+1}f_y}}\pr{x}=\br{\ltro_S f_y}\pr{x}\).
  Hence,
  \begin{equation*}
    \br{\Phi_u f}\pr{x}
    = f\pr{x} + \Delta_1 \br{\ltro_S f}\pr{x}
    = f_y\pr{x} + \Delta_1 \br{\ltro_S f_y}\pr{x}
    = \br{\pr{I+\Delta_1\ltro_S} f_y}\pr{x}
    = \br{\Phi_u f_y}\pr{x},
  \end{equation*}
  as required.

  Fix now any \(n\) in \(\nats\) with \(n\geq2\), and assume that the stated holds for all \(1\leq n'<n\).
  We now show that this implies the stated for \(n\).
  Fix any \(y\) in \(\stsp\) with \(y\geq x+n\).
  Let \(v\coloneqq t_1, \dots, t_n\).
  Then
  \begin{equation}
  \label{eqn:Proof of Phi_u:Value of f does not matter above treshold:intermed}
    \br{\Phi_u f}\pr{x}
    = \br{\pr{I+\Delta_1\ltro_S} \Phi_v f}\pr{x}
    = \br{\pr{I+\Delta_1\ltro_S} g}\pr{x}
    = \br{\pr{I+\Delta_1\ltro_S}\pr{\indica{\leq x+1} g + \indica{>x+1}g\pr{x+1}}}\pr{x},
  \end{equation}
  where we let \(g\coloneqq\Phi_v f\) and where the final equality follows from the induction hypothesis with \(n'=1\).

  It now follows from the induction hypothesis with \(n'=n-1\) that, for any \(z\) in \(\stsp\) such that \(y\geq z+n-1\) (or equivalently, \(z\leq y-n+1\)), \(\br{\Phi_v f}\pr{z}=\br{\Phi_v f_y}\pr{z}\).
  As furthermore \(x\leq y-n\) by assumption, we conclude from this that, for any \(z\) in \(\stsp\) such that \(z\leq x+1\leq y-n+1\), \(\br{\Phi_v f}\pr{z}=\br{\Phi_v f_y}\pr{z}\).
  Consequently,
  \begin{equation*}
    \indica{\leq x+1} g + \indica{>x+1}g\pr{x+1}
    = \indica{\leq x+1} \pr{\Phi_v f} + \indica{>x+1}\br{\Phi_v f}\pr{x+1}
    = \indica{\leq x+1} \pr{\Phi_v f_y} + \indica{>x+1}\br{\Phi_v f_y}\pr{x+1}.
  \end{equation*}
  We now substitute this equality in Equation~\eqref{eqn:Proof of Phi_u:Value of f does not matter above treshold:intermed}, to yield
  \begin{equation*}
    \br{\Phi_u f}\pr{x}
    = \br{\pr{I+\Delta_1\ltro_S}\pr{\indica{\leq x+1} \Phi_v f_y + \indica{>x+1}\br{\Phi_v f_y}\pr{x+1}}}\pr{x}.
  \end{equation*}
  We now invoke the induction hypothesis with \(n'=1\) for the second time, to yield
  \begin{equation*}
    \br{\Phi_u f}\pr{x}
    = \br{\pr{I+\Delta_1\ltro_S}\pr{\indica{\leq x+1} \Phi_v f_y + \indica{>x+1}\br{\Phi_v f_y}\pr{x+1}}}\pr{x}
    = \br{\pr{I+\Delta_1\ltro_S} \Phi_v f_y}\pr{x}
    = \br{\Phi_u f_y}\pr{x}.
  \end{equation*}
\end{proof}

\subsection{The Corresponding Semi-Group of Lower Counting Transformations}
Let \(S=\set{\pr{\llambda_x,\ulambda_x}}_{x\in\stsp}\) be a sequence in \(\setoflambdaseq_{\Lambda}\).
Due to Theorems~\ref{the:GenPoisGen:Phi_u_i converges to a LCT} and \ref{the:LCT induced by GenPoisGen}, for any \(t,s\) in \(\nnegreals\) with \(t\leq s\), we can uniquely define the corresponding lower counting transformation
\begin{equation*}
\label{eqn:LCT by GenPoisGen}
  \lto_t^s
  \coloneqq \lim_{\sigma\pr{u}\to0} \set{\Phi_u\colon u\in\setoftseq_{\br{t,s}}}.
\end{equation*}
As explained right after Equation~\eqref{eqn:LTO gen by PoisGen}, this unconventional notation for the limit is used to emphasise that the limit does \emph{not} depend on the chosen sequence \(\set{u_i}_{i\in\nats}\) in \(\setoftseq_{\br{t,s}}\) so long as \(\lim_{i\to+\infty}\sigma\pr{u_i}=0\).

This way, we have defined an entire family \(\set{\lto_t^s\colon t,s\in\nnegreals, t\leq s}\) of lower counting transformations.
The following result establishes that this family is a time-homogeneous semi-group.
\begin{proposition}
\label{prop:Properties of lto_S induced by GenPoisGen}
  Consider a sequence \(S\) in \(\setoflambdaseq_{\Lambda}\).
  Then
  \begin{enumerate}[label=\upshape(\roman*)]
    \item \label{prop:Properties of lto_S by GenPoisGen:IDentity} \(\lto_t^{t} = I\) for all \(t\) in \(\nnegreals\);
    \item \label{prop:Properties of lto_S by GenPoisGen:Semi-group} \(\lto_t^s=\lto_t^{r}\lto_r^s\) for all \(t,r,s\) in \(\nnegreals\) with \(t\leq r\leq s\);
    \item \label{prop:Properties of lto_S by GenPoisGen:Time-homogeneity}
    \(\lto_t^s=\lto_0^{s-t}\) for all \(t,s\) in \(\nnegreals\) with \(t\leq s\).
  \end{enumerate}
\end{proposition}
\begin{proof}
  The proofs of these properties are almost entirely the same as their counterparts in \cite{2017Krak}.
  \begin{enumerate}[label=\upshape(\roman*)]
    \item This is an immediate consequence of Theorem~\ref{the:LCT induced by GenPoisGen} because \(\setoftseq_{\br{t,t}}=\set{t}\), \(\sigma\pr{t}=0\) and \(\Phi_t=I\).
    \item Fix any arbitrary \(\epsilon\) in \(\posreals\).
    By Theorem~\ref{the:LCT induced by GenPoisGen}, there are sequences \(u_1\) in \(\setoftseq_{\br{t,r}}\) and \(u_2\) in \(\setoftseq_{\br{r,s}}\) such that (i) \(\norm{\lto_t^r-\Phi_{u_1}}\leq \epsilon/3\) and \(\sigma\pr{u_1}\norm{\ltro_S}\leq2\); (ii) \(\norm{\lto_r^s-\Phi_{u_2}}\leq \epsilon/3\) and \(\sigma\pr{u_2}\norm{\ltro_S}\leq2\); and (iii) \(\norm{\lto_t^s-\Phi_u}\leq \epsilon/3\) and \(\sigma\pr{u}\norm{\ltro_S}\leq2\), where \(u\coloneqq u_1\cup u_2\) is an element of \(\setoftseq_{\br{t,s}}\).
    Observe that
    \begin{align*}
      \norm{\lto_t^s-\lto_t^r\lto_r^s}
      &= \norm{\lto_t^s-\Phi_u+\Phi_{u_1}\Phi_{u_2}-\lto_t^r\lto_r^s}
      \leq \norm{\lto_t^s-\Phi_u} +\norm{\Phi_{u_1}\Phi_{u_2}-\lto_t^r\lto_r^s} \\
      &\leq \norm{\lto_t^s-\Phi_u} +\norm{\Phi_{u_1}-\lto_t^r}+\norm{\Phi_{u_2}-\lto_r^s}
      \leq 3\frac{\epsilon}3
      = \epsilon,
    \end{align*}
    where the first inequality follows from the triangle inequality and the second inequality follows from Lemma~\ref{lem: bound on norm of A_1 ... A_k - B_1 ... B_k} and Corollary~\ref{cor:Phi_u is LTT}.
    Because \(\epsilon\) was any arbitrary positive real number, we conclude from this inequality that the stated holds.
    \item Fix any sequence \(\set{u_i}_{i\in\nats}\) in \(\setoftseq_{\br{t,s}}\) such that \(\lim_{i\to+\infty}\sigma\pr{u_i}=0\).
    For any \(i\) in \(\nats\), we define \(u'_i\coloneqq t_0-t, t_1-t, \dots, t_n-t\), with \(u_i=t_0, \dots, t_n\).
    Observe that, by construction, \(\Phi_{u_i}-\Phi_{u'_i}\) for all \(i\) in \(\nats\).
    Because furthermore \(\lim_{i\to+\infty} \Phi_{u_i}=\lto_t^s\) and \(\lim_{i\to+\infty}\Phi_{u'_i}=\lto_0^{s-t}\) due to Theorems~\ref{the:GenPoisGen:Phi_u_i converges to a LCT} and \ref{the:LCT induced by GenPoisGen}, the stated now follows immediately.
  \end{enumerate}
\end{proof}

We conclude this section of the Appendix with some technical results regarding the family \(\set{\lto_t^s\colon t,s\in\nnegreals, t\leq s}\) of lower counting transformations.
\begin{lemma}
\label{lem:Additional properties of lto_S induced by GenPoisGen}
  Consider a sequence \(S\) in \(\setoflambdaseq_{\Lambda}\).
  Then for all \(t,s\) in \(\nnegreals\) with \(t\leq s\),
  \begin{enumerate}[label=\upshape(\roman*)]
    \item\label{lem:Additional properties of lto_S by GenPoisGen:Diff with I+Delta ltro}
    \(\norm{\lto_t^s - \pr{I+\pr{s-t}\ltro_S}}\leq\pr{s-t}^2\norm{\ltro_S}^2\);
    \item\label{lem:Additional properties of lto_S by GenPoisGen:Diff with I}
    \(\norm{\lto_t^s - I}\leq\pr{s-t}\norm{\ltro_S}\).
  \end{enumerate}
\end{lemma}
\begin{proof}
  \begin{enumerate}[label=\upshape(\roman*)]
    \item Fix any arbitrary \(\epsilon\) in \(\posreals\).
    By Theorem~\ref{the:LCT induced by GenPoisGen}, there is a sequence \(u=t_0, \dots, t_n\) in \(\setoftseq_{\br{t,s}}\) such that \(\norm{\lto_t^s-\Phi_u}\leq \epsilon\) and \(\sigma\pr{u}\norm{\ltro_S}\leq2\).
    Observe that
    \begin{equation*}
      \norm{\lto_t^s-\pr{I+\pr{s-t}\ltro_S}}
      \leq \norm{\lto_t^s-\Phi_u}+\norm{\Phi_u-\pr{I+\pr{s-t}\ltro_S}}
      \leq \epsilon+\norm{\ltro_S}^2\sum_{i=1}^n \Delta_i \sum_{j=i+1}^n \Delta_j
      \leq \epsilon+\pr{s-t}^2\norm{\ltro_S}^2,
    \end{equation*}
    where the first inequality follows from the triangle inequality and the second inequality follows from Lemma~\ref{lem:GenPoisGen:Delta_i vs Delta}.
    The stated now holds because this inequality holds for any arbitrary positive real number \(\epsilon\).
    \item Fix any \(n\) in \(\nats\) such that \(\pr{s-t}\norm{\ltro_S}\leq 2n\), and let \(\Delta\coloneqq\pr{s-t}/n\).
    Observe that
    \begin{equation*}
      \norm{\lto_0^\Delta - I}
      = \norm{\lto_0^\Delta-\pr{I+\Delta\ltro_S}+\Delta\ltro_S}
      \leq \norm{\lto_0^\Delta-\pr{I+\Delta\ltro_S}}+\Delta\norm{\ltro_S}
      \leq \Delta^2\norm{\ltro_S}^2+\Delta\norm{\ltro_S},
    \end{equation*}
    where the first inequality follows from the triangle inequality and the second inequality follows from \ref{lem:Additional properties of lto_S by GenPoisGen:Diff with I+Delta ltro}.
    By repeatedly applying Proposition~\ref{prop:Properties of lto_S induced by GenPoisGen}~\ref{prop:Properties of lto_S by GenPoisGen:Semi-group} and Proposition~\ref{prop:Properties of lto_S induced by GenPoisGen}~\ref{prop:Properties of lto_S by GenPoisGen:Time-homogeneity}, we obtain that \(\lto_t^s=\pr{\lto_0^\Delta}^n\).
    We combine our findings, to yield
    \begin{equation*}
      \norm{\lto_t^s-I}
      = \norm*{\pr{\lto_0^\Delta}^n-I^n}
      \leq n\norm{\lto_0^\Delta-I}
      \leq n\Delta^2\norm{\ltro_S}^2+n\Delta\norm{\ltro_S}
      = \frac{\pr{s-t}^2}{n}\norm{\ltro_S}^2+\pr{s-t}\norm{\ltro_S},
    \end{equation*}
    where the first inequality follows from Lemma~\ref{lem: bound on norm of A_1 ... A_k - B_1 ... B_k}.
    The stated now follows if we take the limit for \(n\) going to \(+\infty\).
  \end{enumerate}
\end{proof}
The second of these technical results establishes an upper bound on the error made by approximating \(\lto_t^s\) by \(\Phi_u\).
\begin{lemma}
\label{lem:diff between lto_S and Phi_u}
  Fix a sequence \(S\) in \(\setoflambdaseq_{\Lambda}\), some \(t,s\) in \(\nnegreals\) with \(t\leq s\) and a sequence \(u\) in \(\setoftseq_{\br{t,s}}\).
  If \(\sigma\pr{u}\norm{\ltro_S}\leq2\), then
  \begin{equation*}
    \norm{\lto_t^s-\Phi_u}
    \leq\sigma\pr{u}\pr{s-t}\norm{\ltro_S}^2.
  \end{equation*}
\end{lemma}
\begin{proof}
  Our proof is entirely similar to that of \cite[Lemma~E.8]{2017Krak}.
  Fix any \(\epsilon\) in \(\posreals\).
  By Theorem~\ref{the:LCT induced by GenPoisGen}, there is a \(u_\epsilon\) in \(\setoftseq_{\br{t,s}}\) such that \(u\subseteq u_{\epsilon}\) and \(\norm{\lto_t^s-\Phi_{u_\epsilon}}\leq\epsilon\).
  From this, the triangle inequality and Corollary~\ref{cor:GenPoisGen:Phi_u vs Phi_u'}, it follows that
  \begin{equation*}
    \norm{\lto_t^s-\Phi_u}
    = \norm{\lto_t^s-\Phi_{u_\epsilon}+\Phi_{u_\epsilon}-\Phi_u}
    \leq \norm{\lto_t^s-\Phi_{u_\epsilon}}+\norm{\Phi_{u_\epsilon}-\Phi_u}
    \leq \epsilon + \sigma\pr{u}\pr{t-s}\norm{\ltro_S}^2.
  \end{equation*}
  The stated now follows because \(\epsilon\) is an arbitrary positive real number.
\end{proof}
Our next technical result can be interpreted as dealing with the ``time-derivative'' of \(\lto_t^s\), as is explained by \citet[Right after Proposition~7.15]{2017Krak}.
\begin{lemma}
\label{lem:ltro satisfies differential equation}
  Consider a sequence \(S\) in \(\setoflambdaseq_{\Lambda}\).
  Then for any \(t,s\) in \(\nnegreals\) with \(t\leq s\),
  \begin{equation*}
    \pr{\forall\epsilon\in\posreals}
    \pr{\exists\delta\in\posreals}
    \pr{\forall\Delta\in\reals, 0<\abs{\Delta}<\delta, 0\leq t+\Delta\leq s}~
    \norm*{\frac{\lto_{t+\Delta}^s-\lto_t^s}{\Delta}+\ltro_S \lto_t^s}
    \leq\epsilon.
  \end{equation*}
  and
  \begin{equation*}
    \pr{\forall\epsilon\in\posreals}
    \pr{\exists\delta\in\posreals}
    \pr{\forall\Delta\in\reals, 0<\abs{\Delta}<\delta, t\leq s+\Delta}~
    \norm*{\frac{\lto_t^{s+\Delta}-\lto_t^s}{\Delta}-\ltro_S \lto_t^s}
    \leq\epsilon.
  \end{equation*}
\end{lemma}
\begin{proof}
  The proof is just the proof of \citeauthor{2017Krak} for \cite[Proposition~7.15]{2017Krak} with some obvious modifications.
  Fix any \(\epsilon\) in \(\posreals\), and fix any \(\delta\) in \(\posreals\) such that \(2\delta\norm{\ltro_S}^2\leq\epsilon\).
  Consider now any \(\Delta\) in \(\reals\) such that \(0<\abs{\Delta}<\delta\) and \(0\leq t+\Delta\leq s\).
  If we let \(t^\star\coloneqq\max\set{t,t+\Delta}\), then
  \begin{equation*}
    \norm{\lto_{t+\Delta}^s-\lto_t^s + \Delta\ltro_S\lto_t^s}
    = \norm{\lto_{t^\star}^s - \lto_{t^\star-\abs{\Delta}}^s + \abs{\Delta}\ltro_S\lto_t^s}
    = \norm{\lto_{t^\star}^s - \lto_{t^\star-\abs{\Delta}}^{t^\star}\lto_{t^\star}^s + \abs{\Delta}\ltro_S\lto_t^{t^\star}\lto_{t^\star}^s},
  \end{equation*}
  where the last equality follows Proposition~\ref{prop:Properties of lto_S induced by GenPoisGen}~\ref{prop:Properties of lto_S by GenPoisGen:Semi-group} because \(t\leq t^\star\leq s\).
  We now use \ref{BNH:norm A B <= norm A norm B} and \ref{LTT:Norm is lower than one}, to yield
  \begin{equation*}
    \norm{\lto_{t+\Delta}^s-\lto_t^s + \Delta\ltro_S\lto_t^s}
    \leq \norm{I - \lto_{t^\star-\abs{\Delta}}^{t^\star} + \abs{\Delta}\ltro_S\lto_t^{t^\star}}\norm{\lto_{t^\star}^s}
    \leq \norm{I - \lto_{t^\star-\abs{\Delta}}^{t^\star} + \abs{\Delta}\ltro_S\lto_t^{t^\star}}.
  \end{equation*}
  Further manipulations now yield
  \begin{align*}
    \norm{\lto_{t+\Delta}^s-\lto_t^s + \Delta\ltro_S\lto_t^s}
    &\leq \norm{I - \lto_{t^\star-\abs{\Delta}}^{t^\star} +\abs{\Delta} \ltro_S -\abs{\Delta}\ltro_S + \abs{\Delta}\ltro_S\lto_t^{t^\star}} \\
    &\leq \norm{I+\abs{\Delta}\ltro_S - \lto_{t^\star-\abs{\Delta}}^{t^\star}} + \norm{\abs{\Delta}\ltro_S - \abs{\Delta}\ltro_S\lto_t^{t^\star}} \\
    &\leq \norm{I+\abs{\Delta}\ltro_S - \lto_{t^\star-\abs{\Delta}}^{t^\star}} + \abs{\Delta}\norm{\ltro_S}\norm{\lto_t^{t^\star}-I} \\
    &\leq \Delta^2\norm{\ltro_S}^2 + \abs{\Delta}\norm{\ltro_S}\norm{\lto_t^{t^\star}-I}
    \leq \Delta^2\norm{\ltro_S}^2 + \abs{\Delta}\pr{t^\star-t}\norm{\ltro_S} \\
    &\leq 2\Delta^2\norm{\ltro_S}^2,
  \end{align*}
  where the third inequality follows from Corollary~\ref{cor:GenPoisGen:A-B}, the fourth inequality follows from Lemma~\ref{lem:Additional properties of lto_S induced by GenPoisGen}~\ref{lem:Additional properties of lto_S by GenPoisGen:Diff with I+Delta ltro}, the fifth inequality follows from Lemma~\ref{lem:Additional properties of lto_S induced by GenPoisGen}~\ref{lem:Additional properties of lto_S by GenPoisGen:Diff with I} and the final inequality follows from the fact that \(0\leq t^\star-t\leq\abs{\Delta}\).
  From this inequality, it now follows that
  \begin{equation*}
    \norm*{\frac{\lto_{t+\Delta}^s-\lto_t^s}{\Delta}+\ltro_S\lto_t^s}
    = \frac{1}{\abs{\Delta}} \norm{\lto_{t+\Delta}^s-\lto_t^s + \Delta\ltro_S\lto_t^s}
    \leq 2\abs{\Delta}\norm{\ltro_S}^2
    \leq 2\delta\norm{\ltro_S}^2
    \leq \epsilon,
  \end{equation*}
  which proves the first part of the stated.

  The second part of the stated follows from the first part.
  To see this, we fix any \(\epsilon, \tau\) in \(\posreals\), and let \(t'\coloneqq t+\tau\) and \(s'\coloneqq s+\tau\).
  It follows from the first part of the stated that there is some \(\delta'\) in \(\posreals\) such that
  \begin{equation}
  \label{eqn:Proof of ltro satisfies differential equation:First part prime}
    \pr{\forall \Delta'\in\reals, 0<\abs{\Delta'}<\delta', 0\leq t'+\Delta'\leq s'}~
    \norm*{\frac{\lto_{t'+\Delta'}^{s'}-\lto_{t'}^{s'}}{\Delta'}+\ltro_S\lto_{t'}^{s'}}
    \leq \epsilon.
  \end{equation}
  We now let \(\delta\coloneqq\min\set{\delta', \tau}\), and fix any \(\Delta\) in \(\reals\) such that \(0<\abs{\Delta}<\delta\).
  Observe that
  \begin{equation*}
    \norm*{\frac{\lto_{t}^{s+\Delta}-\lto_{t}^{s}}{\Delta}-\ltro_S\lto_{t}^{s}}
    = \norm*{\frac{\lto_{t'}^{s'+\Delta}-\lto_{t'}^{s'}}{\Delta}-\ltro_S\lto_{t'}^{s'}}
    = \norm*{\frac{\lto_{t'-\Delta}^{s'}-\lto_{t'}^{s'}}{\Delta}-\ltro_S\lto_{t'}^{s'}}
    = \norm*{\frac{\lto_{t'-\Delta}^{s'}-\lto_{t'}^{s'}}{-\Delta}+\ltro_S\lto_{t'}^{s'}},
  \end{equation*}
  where the first and second equality follow from Proposition~\ref{prop:Properties of lto_S induced by GenPoisGen}~\ref{prop:Properties of lto_S by GenPoisGen:Time-homogeneity}.
  Since furthermore \(t'-\Delta=t+\tau-\Delta\geq t\geq 0\) and \(t'-\Delta=t+\tau-\Delta\leq s+\tau=s'\), it follows from Equation~\eqref{eqn:Proof of ltro satisfies differential equation:First part prime} with \(\Delta'=-\Delta\) that
  \begin{equation*}
    \norm*{\frac{\lto_{t}^{s+\Delta}-\lto_{t}^{s}}{\Delta}-\ltro_S\lto_{t}^{s}}
    = \norm*{\frac{\lto_{t'-\Delta}^{s'}-\lto_{t'}^{s'}}{-\Delta}+\ltro_S\lto_{t'}^{s'}}
    \leq \epsilon,
  \end{equation*}
  as required.
\end{proof}
We now use Proposition~\ref{lem:ltro satisfies differential equation} to establish the limit behaviour of \(\br{\lto_t^{t+\Delta} \indica{\geq x+2}}\pr{x}\) and \(\br{\lto_{t-\Delta}^t \indica{\geq x+2}}\pr{x}\) for \(\Delta\to0^+\).
\begin{lemma}
\label{lem:lto_S satisfies orederliness property}
  Consider a sequence \(S\) in \(\setoflambdaseq_{\Lambda}\).
  Then for any \(t\) in \(\nnegreals\) and \(x\) in \(\stsp\),
  \begin{equation*}
    \lim_{\Delta\to0^+}\frac{\br{\lto_t^{t+\Delta}\indica{\geq x+2}}\pr{x}}{\Delta}
    = 0
  \end{equation*}
  and, if \(t>0\),
  \begin{equation*}
    \lim_{\Delta\to0^+}\frac{\br{\lto_{t-\Delta}^t\indica{\geq x+2}}\pr{x}}{\Delta}
    = 0.
  \end{equation*}
\end{lemma}
\begin{proof}
  Fix any \(\epsilon\) in \(\posreals\).
  From Lemma~\ref{lem:ltro satisfies differential equation} (with \(s=t\)), we know that there is a \(\delta\) in \(\posreals\) such that
  \begin{equation}
  \label{eqn:Proof of lto_S satisfies orederliness property:Derivative}
    \pr{\forall \Delta\in\posreals, \Delta<\delta}~
    \norm*{\frac{\lto_t^{t+\Delta}-\lto_t^t}{\Delta}-\ltro_S \lto_t^t}
    \leq \epsilon.
  \end{equation}
  Fix any \(\Delta\) in \(\posreals\) such that \(\Delta<\delta\), and observe that
  \begin{align*}
    \abs*{\frac{\br{\lto_t^{t+\Delta}\indica{\geq x+2}}\pr{x}-\br{\lto_t^t\indica{\geq x+2}}\pr{x}}{\Delta}-\br{\ltro_S\lto_t^t\indica{\geq x+2}}\pr{x}}
    &\leq\norm*{\frac{\lto_t^{t+\Delta}\indica{\geq x+2}-\lto_t^t\indica{\geq x+2}}{\Delta}-\ltro_S\lto_t^t\indica{\geq x+2}} \\
    &\leq \norm*{\frac{\lto_t^{t+\Delta}-\lto_t^t}{\Delta}-\ltro_S\lto_t^t}\norm{\indica{\geq x+2}}
    = \norm*{\frac{\lto_t^{t+\Delta}-\lto_t^t}{\Delta}-\ltro_S\lto_t^t} \\
    &\leq \epsilon,
  \end{align*}
  where the first inequality follows from the definition of the supremum norm, the second inequality follows from \ref{BNH: norm Af <= norm A norm f}, the equality holds because \(\norm{\indica{\geq x+2}}=1\) and the final inequality follows from Equation~\eqref{eqn:Proof of lto_S satisfies orederliness property:Derivative} because \(0<\Delta<\delta\).
  Next, we take a closer look at the terms in the absolute value on the left hand side of the above inequality: (a) it follows from \ref{def:LTT:Bound} that \(\br{\lto_t^{t+\Delta}\indica{\geq x+2}}\pr{x}\geq \inf \indica{\geq x+2}=0\); (b) it follows from Proposition~\ref{prop:Properties of lto_S induced by GenPoisGen}~\ref{prop:Properties of lto_S by GenPoisGen:IDentity} that \(\br{\lto_t^t\indica{\geq x+2}}\pr{x}=\br{I\indica{\geq x+2}}\pr{x}=\indica{\geq x+2}\pr{x}=0\); and (c) it follows from Proposition~\ref{prop:Properties of lto_S induced by GenPoisGen}~\ref{prop:Properties of lto_S by GenPoisGen:IDentity} and Equation~\eqref{eqn:Generalized Poisson Generator} that \(\br{\ltro_S\lto_t^t\indica{\geq x+2}}\pr{x} = \br{\ltro_S \indica{\geq x+2}}\pr{x}=0\).
  Consequently,
  \begin{equation*}
    0 \leq
    \frac{\br{\lto_t^{t+\Delta}\indica{\geq x+2}}\pr{x}}{\Delta}
    = \abs*{\frac{\br{\lto_t^{t+\Delta}\indica{\geq x+2}}\pr{x}-\br{T_t^t\indica{\geq x+2}}\pr{x}}{\Delta}-\br{\ltro_S\lto_t^t\indica{\geq x+2}}\pr{x}}
    \leq \epsilon.
  \end{equation*}
  Since this holds for all~\(\Delta\) in \(\posreals\) such that \(\Delta<\delta\), and because \(\epsilon\) was an arbitrary positive real number, we have shown that
  \begin{equation*}
    \lim_{\Delta\to0^+} \frac{\br{\lto_t^{t+\Delta}\indica{\geq x+2}}\pr{x}}\Delta
    =0,
  \end{equation*}
  which is the first limit of the stated.
  The second limit of the stated follows from the the first limit and Proposition~\ref{prop:Properties of lto_S induced by GenPoisGen}~\ref{prop:Properties of lto_S by GenPoisGen:Time-homogeneity}.
\end{proof}
Finally, we conclude this section with a useful ``translation-invariance'' property.
\begin{lemma}
\label{lem:State homogeneity of lto induced by genpoisgen}
  Consider a sequence~\(S\) in \(\setoflambdaseq_{\Lambda}\).
  For any \(t,s\) in \(\nnegreals\) with \(t\leq s\), \(f\) in \(\setoffn\pr{\stsp}\) and \(x\) in \(\stsp\),
  \begin{equation*}
    \br{\lto_t^s f}\pr{y}
    = \br{\lto_t^s f'_x}\pr{y-x}
    \quad\text{for all } y\in\stsp \text{ with } y\geq x,
  \end{equation*}
  where \(f'_x\colon\stsp\to\reals\colon z\mapsto f'_x\pr{z}\coloneqq f\pr{x+z}\).
\end{lemma}
\begin{proof}
  Fix any \(\epsilon\) in \(\posreals\), and choose some \(\epsilon'\) in \(\posreals\) such that \(2\epsilon'\norm{f}\leq\epsilon\).
  By Theorem~\ref{the:LCT induced by GenPoisGen}, there is a \(u\) in \(\setoftseq_{\br{t,s}}\) such that
  \begin{equation*}
    \norm{\lto_t^s-\Phi_u}
    \leq \epsilon'.
  \end{equation*}
  Observe now that
  \begin{equation*}
    \abs*{\br{\lto_t^s f}\pr{y} - \br{\Phi_u f}\pr{y}}
    \leq \norm{\lto_t^s f - \Phi_u f}
    \leq \norm{\lto_t^s - \Phi_u}\norm{f}
    \leq \epsilon'\norm{f},
  \end{equation*}
  where the second inequality follows from \ref{BNH: norm Af <= norm A norm f}.
  Similarly, we find that
  \begin{equation*}
    \abs*{\br{\lto_t^s f'_x}\pr{y-x} - \br{\Phi_u f'_x}\pr{y-x}}
    \leq \norm{\lto_t^s - \Phi_u}\norm{f'_x}
    \leq \epsilon'\norm{f'_x}
    \leq \epsilon'\norm{f},
  \end{equation*}
  where for the penultimate inequality we have used the obvious inequality \(\norm{f'_x}\leq\norm{f}\).
  Using Lemma~\ref{lem:I+Delta GenPoisGen is state homogeneous}, we furthermore find that \(\br{\Phi_u f}\pr{y}=\br{\Phi_u f'_x}\pr{y-x}\).
  We now combine our two previous findings, to yield
  \begin{align*}
    \abs*{\br{\lto_t^s f}\pr{y} - \br{\lto_t^s f'_x}\pr{y-x}}
    &= \abs*{\br{\lto_t^s f}\pr{y} - \br{\Phi_u f}\pr{y} + \br{\Phi_u f'_x}\pr{y-x} - \br{\lto_t^s f'_x}\pr{y-x}} \\
    &\leq \abs*{\br{\lto_t^s f}\pr{y} - \br{\Phi_u f}\pr{y}} + \abs*{\br{\Phi_u f'_x}\pr{y-x} - \br{\lto_t^s f'_x}\pr{y-x}} \\
    &\leq \epsilon'\norm{f} + \epsilon'\norm{f}
    \leq \epsilon,
  \end{align*}
  where the final inequality is precisely our condition on \(\epsilon'\).
  Since \(\epsilon\) was an arbitrary positive real number, this proves the stated.
\end{proof}

\section{Linear Transformations}
\label{app:Linear Transformations}
Another important type of transformations on \(\setoffn\pr{\genset}\) are the linear ones.
A transformation~\(A\) on \(\setoffn\pr{\genset}\) is \emph{linear} if it is (i) homogeneous, in the sense that \(A\pr{\mu f}=\mu Af\) for all \(f\) in \(\setoffn\pr{\genset}\) and \(\mu\) in \(\reals\); and (ii) additive, in the sense that \(A\pr{f+g}=Af+Ag\) for all \(f,g\) in \(\setoffn\pr{\genset}\).
Observe that a linear transformation is always non-negatively homogeneous, and conversely, that a non-negatively homogeneous transformation is linear if and only if it is additive.

The special case that \(\genset\) is finite deserves some additional attention.
It is well-known that in this case, the linear transformation~\(A\) can be identified with a \(\card{\genset}\times\card{\genset}\) matrix, the \(\pr{x,y}\)-component of which is equal to \(A\pr{x,y}\coloneqq\br{A\indica{y}}\pr{x}\).
If convenient, we will sometimes prefer the matrix interpretation over the transformation interpretation.
It is also well-known that if \(\genset\) is finite,
\begin{equation}
\label{eqn:Norm of linear transformation of finite state space}
  \norm{A}
  = \max\set*{\sum_{y\in\genset}\abs{A\pr{x,y}}\colon x\in\genset}.
\end{equation}

\subsection{Linear Transition Transformations}
\label{ssec:Linear Transition Transformations}
If a lower transition (counting) transformation~\(\lto\) is linear and not just super-additive, in the sense that the inequality in \ref{def:LTT:Super-additivity} is actually an equality, then we will call it a linear transition (counting) transformation.
More formally, we have the following definition.
\begin{definition}
\label{def:Linear transition transformation}
  A \emph{linear transition transformation} is any transformation~\(T\colon\setoffn\pr{\genset}\to\setoffn\pr{\genset}\) such that
  \begin{enumerate}[label=\upshape{}T\arabic*., ref=\upshape(T\arabic*), leftmargin=*]
    \item\label{def:LinTT:Homogeneity} \(T\pr{\mu f} = \mu Tf\) for all \(f\) in \(\setoffn\pr{\genset}\) and \(\mu\) in \(\reals\); \hfill [homogeneity]
    \item\label{def:LinTT:Additivity} \(T\pr{f+g}=Tf+Tg\) for all \(f, g\) in \(\setoffn\pr{\genset}\); \hfill [additivity]
    \item\label{def:LinTT:Bounded by inf} \(T f \geq \inf f\) for all \(f\) in \(\setoffn\pr{\genset}\). \hfill [bound]
  \end{enumerate}
  A \emph{linear counting transformation} is a linear transition transformation~\(T\) with
  \begin{enumerate}[resume*]
    \item \label{def:LinTT:Counting} \(\br{T f}\pr{x}=\br{T\pr{\indica{\geq x}f}}\pr{x}\) for all \(f\) in \(\setoffn\pr{\genset}\) and \(x\) in \(\genset\).
  \end{enumerate}
\end{definition}

We now state some useful properties of linear counting transformations.
The first result establishes some basic properties of transition/counting transformations that follow almost immediately from \ref{def:LinTT:Homogeneity}--\ref{def:LinTT:Counting}.
\begin{lemma}
\label{lem:LinTT:Some frequently used properties}
  Consider a linear transition transformation~\(T\).
  Then
  \begin{enumerate}[label=\upshape{}T\arabic*., ref=\upshape(T\arabic*), start=5, leftmargin=*]
    \item \label{LinTT:inf f T f sup f}\(\inf f \leq T f \leq \sup f\) for all \(f\) in \(\setoffn\pr{\stsp}\);
    \item \label{LinTT:constant} \(T \mu = \mu\) for all \(\mu\) in \(\reals\);
    \item \label{LinTT:monotonicity} \(T f \leq T g\) for all \(f,g\) in \(\setoffn\pr{\genset}\) such that \(f\leq g\).
  \end{enumerate}
  If \(T\) is a linear counting transformation, then
  \begin{enumerate}[resume*]
    \item \label{LinTT:No transition to below} \(\br{T\indica{y}}\pr{x}=0\) for all \(x,y\) in \(\stsp\) such that \(y<x\).
  \end{enumerate}
\end{lemma}
\begin{proof}
  Properties \ref{LinTT:inf f T f sup f}--\ref{LinTT:monotonicity} follow immediately from Lemma~\ref{lem:LTT:Further properties}~\ref{LTT:inf f T f sup f}--\ref{LTT:monotonicity}.
  Property~\ref{LinTT:No transition to below} follows from \ref{def:LinTT:Counting} with \(f=\indic{y}\) and \ref{LinTT:constant}, as \(\indica{\geq x} f=\indica{\geq x}\indica{y}=0\).
\end{proof}
The second result is a specialisation of Lemma~\ref{lem:LTT:Composition is again a LTT}
\begin{lemma}
\label{lem:LinTT:Composition is again a LinTT}
  For any two linear transition (counting) transformations~\(T_1\) and \(T_2\), their composition \(T_1 T_2\) is again a linear transition (counting) transformation.
\end{lemma}
\begin{proof}
  The linearity---that is, \ref{def:LinTT:Homogeneity} and \ref{def:LinTT:Additivity}---of \(T_1\) and \(T_2\) implies the linearity of their composition~\(T_1T_2\), as one can easily verify.
  Because \(T_1\) and \(T_2\) are lower transition transformations (as they are linear ones by assumption), it follows from Lemma~\ref{lem:LTT:Composition is again a LTT} that their composition \(T_1 T_2\) satisfies \ref{def:LTT:Bound}, which is equivalent to \ref{def:LinTT:Bounded by inf}.
  Similarly, the composition~\(T_1 T_2\) of the two linear counting transformations satisfies \ref{def:LTT:Counting}, which is equivalent to \ref{def:LinTT:Counting}.
\end{proof}
The final general result will play an important role in the proof of Lemma~\ref{lem:Construction lemma for trans probs with full conditional} further on.
\begin{lemma}
\label{lem:LinPT:Sum decomposition with indicator}
  Consider two linear counting transformations~\(T_1\) and \(T_2\).
  Then for all \(x, y\) in \(\stsp\),
  \begin{equation*}
    \br{T_1 T_2 \indica{y}}\pr{x}
    = \begin{cases}
      \sum_{z=x}^y \br{T_1\indica{z}}\pr{x}\br{T_2\indica{y}}\pr{z} &\text{if } x\leq y, \\
      0 &\text{otherwise}.
    \end{cases}
  \end{equation*}
\end{lemma}
\begin{proof}
  We first consider the case \(y<x\).
  By Lemma~\ref{lem:LinTT:Composition is again a LinTT} and \ref{def:LinTT:Counting},
  \begin{equation*}
    \br{T_1 T_2 \indica{y}}\pr{x}
    = \br{T_1 T_2 \pr{\indica{\geq x}\indica{y}}}\pr{x}
    = \br{T_1 T_2 0}\pr{x}
    = 0,
  \end{equation*}
  where the last equality follows from \ref{def:LinTT:Homogeneity}.

  Next, we consider the case \(y\geq x\).
  By \ref{def:LinTT:Counting},
  \begin{equation}
  \label{eqn:Proof of sum decomposition with indicator}
    \br{T_1 T_2 \indica{y}}\pr{x}
    = \br{T_1\pr{\indica{\geq x}\pr{T_2 \indica{y}}}}\pr{x}.
  \end{equation}
  Fix any \(z\) in \(\stsp\), and consider \(\indica{\geq x}\pr{z} \br{T_2 \indica{y}}\pr{z}\).
  If \(z<x\), then \(\indica{\geq x}\pr{z} \br{T_2 \indica{y}}\pr{z}=0\).
  If \(z>y\), then
  \begin{equation*}
    \indica{\geq x}\pr{z} \br{T_2 \indica{y}}\pr{z}
    = \br{T_2 \indica{y}}\pr{z}
    = \br{T_2 \pr{\indica{\geq z}\indica{y}}}\pr{z}
    = \br{T_2 0}\pr{z}
    = 0,
  \end{equation*}
  where the second equality follows from \ref{def:LinTT:Counting} and the final equality follows from \ref{def:LinTT:Homogeneity}.
  From this, we conclude that
  \begin{equation*}
    \indica{\geq x}\pr{T_2 \indica{y}}
    = \sum_{z=x}^y \br{T_2 \indica{y}}\pr{z}\indica{z}.
  \end{equation*}
  We now substitute this equality in Equation~\eqref{eqn:Proof of sum decomposition with indicator}, to yield
  \begin{equation*}
    \br{T_1 T_2 \indica{y}}\pr{x}
    = \br{T_1\pr{\indica{\geq x}\pr{T_2 \indica{y}}}}\pr{x}
    = \br*{T_1 \pr*{\sum_{z=x}^y \br{T_2 \indica{y}}\pr{z}\indica{z}}}\pr{x}
    = \sum_{z=x}^y \br{T_1\pr{\br{T_2 \indica{y}}\pr{z}\indica{z}}}\pr{x}
    = \sum_{z=x}^y \br{T_2 \indica{y}}\pr{z} \br{T_1\indica{z}}\pr{x},
  \end{equation*}
  using \ref{def:LinTT:Additivity} and \ref{def:LinTT:Homogeneity} for the final two equalities.
\end{proof}

\subsection{Linear Generalised Poisson Generators And The Semi-Groups They Induce}
We now follow the same pattern as we did in Appendix~\ref{app:The Generalised Poisson Generator}, but the with \emph{linear} transformations: we introduce linear generalised Poisson generators and subsequently show that these transformations generate a family of linear counting transformations.
With any sequence \(S\coloneqq\set{\lambda_x}_{x\in\stsp}\) in \(\Lambda=\br{\llambda,\ulambda}\), we associate the operator~\(\trm_S\) defined by
\begin{equation*}
  \br{\trm_S f}\pr{x}
  \coloneqq \lambda_x f\pr{x+1}-\lambda_x f\pr{x}
  \quad\text{for all } x\in\stsp, f\in\setoffn\pr{\stsp}.
\end{equation*}
Observe that \(\trm_S\) is indeed a linear generalised Poisson generator.
That it is linear follows immediately from its definition.
That is is a generalised Poisson generator follows from the fact that
\begin{equation}
\label{eqn:LinGenPoisGen and GenPoisGen}
  \trm_S
  = \ltro_{S'}
  \qquad\text{with } S=\set{\lambda_x}_{x\in\stsp} \text{ and } S'=\set{\pr{\lambda_x, \lambda_x}}_{x\in\stsp}.
\end{equation}
This relation allows us to immediately establish the following result.
\begin{corollary}
\label{cor:Gen trm:norm}
  Consider a sequence \(S\coloneqq\set{\lambda_x}_{x\in\stsp}\) in \(\Lambda=\br{\llambda, \ulambda}\).
  Then
  \begin{equation*}
    \norm{\trm_S}
    = 2\sup\set{\lambda_x\colon x\in\stsp}.
  \end{equation*}
\end{corollary}
\begin{proof}
  This is a simple corollary of Equation~\eqref{eqn:LinGenPoisGen and GenPoisGen} and Lemma~\ref{lem:GenPoisGen:Norm}.
\end{proof}

We now establish that the linear generalised Poisson generator~\(\trm_S\) generates a family of linear counting transformations.
In essence, we simply combine the results of Appendix~\ref{sapp:GenPoisGen:To Lower Counting Transformations} with Equation~\eqref{eqn:LinGenPoisGen and GenPoisGen}.
\begin{corollary}
\label{cor:I + delta Gen trm is a linear counting transformation}
  Consider a sequence \(S\coloneqq\set{\lambda_x}_{x\in\stsp}\) in \(\Lambda=\br{\llambda, \ulambda}\).
  Then for any \(\Delta\) in \(\nnegreals\) with \(\Delta\norm{\trm_S}\leq2\), \(\pr{I+\Delta\trm_S}\) is a linear counting transformation.
\end{corollary}
\begin{proof}
  It follows from Equation~\eqref{eqn:LinGenPoisGen and GenPoisGen} and Lemma~\ref{lem:I+Delta GenPoisGen is LTT} that \(\pr{I+\Delta\trm_S}\) is a lower counting transformation.
  That it is furthermore linear follows directly from the linearity of \(I\) and \(\trm_S\).
\end{proof}
Because the linear generalised Poisson generator~\(\trm_S\) is a generalised Poisson generator, we can use here use the notation~\(\Phi_u\) as introduced in Equation~\eqref{eqn:GenPoisGen:Phi_u} as well; we here simply replace \(\ltro_S\) by \(\trm_S=\ltro_{S'}\) in the definition.
\begin{corollary}
\label{cor:Phi_u is LinTT}
  Fix a sequence \(S=\set{\lambda_x}_{x\in\stsp}\) in \(\Lambda\), some \(t,s\) in \(\nnegreals\) with \(t\leq s\) and a sequence \(u=t_0, \dots, t_n\) in \(\setoftseq_{\br{t,s}}\).
  If \(\sigma\pr{u}\norm{\trm_S}\leq2\), then \(\Phi_u\), as defined in Equation~\eqref{eqn:GenPoisGen:Phi_u}, is a linear counting transformation.
\end{corollary}
\begin{proof}
  Follows immediately from Lemma~\ref{lem:LinTT:Composition is again a LinTT} and Corollary~\ref{cor:I + delta Gen trm is a linear counting transformation}.
\end{proof}
\begin{corollary}
\label{cor:LinGenPoisGen:Phi_u_i converges to a LinCT}
  Consider a sequence \(S\coloneqq\set{\lambda_x}_{x\in\stsp}\) in \(\Lambda=\br{\llambda, \ulambda}\), and fix some \(t,s\) in \(\nnegreals\) with \(t\leq s\).
  For any sequence \(\set{u_i}_{i\in\nats}\) in \(\setoftseq_{\br{t,s}}\) such that \(\lim_{i\to+\infty}\sigma\pr{u_i}=0\), the corresponding sequence \(\set{\Phi_{u_i}}_{i\in\nats}\) converges to a linear counting transformation.
\end{corollary}
\begin{proof}
  Recall from Equation~\eqref{eqn:LinGenPoisGen and GenPoisGen} that \(\trm_S\) is equal to the generalised Poisson generator~\(\ltro_{S'}\) associated with the sequence~\(S'=\set{\pr{\lambda_x,\lambda_x}}_{x\in\stsp}\).
  It therefore follows from Theorem~\ref{the:GenPoisGen:Phi_u_i converges to a LCT} (with the sequence \(S'=\set{\pr{\lambda_x,\lambda_x}}_{x\in\stsp}\)) that the corresponding sequence~\(\set{\Phi_{u_i}}_{i\in\nats}\) converges to a lower counting transformation.
  Hence, what remains for us to verify is that this limit is a linear counting transformation.
  To that end, we observe that, because \(\lim_{i\to+\infty}\sigma\pr{u_i}=0\), there is a \(i^\star\) in \(\nats\) such that \(\sigma\pr{u_i}\norm{\trm_S}\leq2\) for all \(i\geq i^\star\).
  Hence, it follows from Corollary~\ref{cor:Phi_u is LinTT} that \(\Phi_{u_i}\) is a linear counting transformation for all \(i\geq i^\star\).
  As \ref{def:LinTT:Homogeneity}--\ref{def:LinTT:Counting} are preserved under taking the limit for \(i\to+\infty\), this implies that the limit of the corresponding sequence is a linear counting transformation, as required.
\end{proof}
Finally, we establish that the limit mentioned in Corollary~\ref{cor:LinGenPoisGen:Phi_u_i converges to a LinCT} does not depend on the exact choice of the sequence~\(\set{u_i}_{i\in\nats}\).
\begin{corollary}
\label{cor:Gen trm:Existence of T_S}
  Consider a sequence \(S\coloneqq\set{\lambda_x}_{x\in\stsp}\) in \(\Lambda=\br{\llambda, \ulambda}\).
  Then for any \(t, s\) in \(\nnegreals\) with \(t\leq s\), there is a unique linear counting transformation~\(T\) such that
  \begin{equation*}
    \pr{\forall\epsilon\in\posreals}
    \pr{\exists\delta\in\posreals}
    \pr{\forall u\in\setoftseq_{\br{t,s}}, \sigma\pr{u}\leq\delta}~
    \norm*{T - \Phi_u}
    \leq\epsilon.
  \end{equation*}
\end{corollary}
\begin{proof}
  Recall from Equation~\eqref{eqn:LinGenPoisGen and GenPoisGen} that \(\trm_S\) is equal to the generalised Poisson generator~\(\ltro_{S'}\) associated with the sequence~\(S'=\set{\pr{\lambda_x,\lambda_x}}_{x\in\stsp}\).
  It therefore follows from Theorem~\ref{the:LCT induced by GenPoisGen} (with the sequence \(S'=\set{\pr{\lambda_x,\lambda_x}}_{x\in\stsp}\)) that there is a unique lower counting transformation~\(\lto\) that satisfies the condition of the stated.
  All that remains for us is to verify that this unique lower counting transformation~\(\lto\) is linear.
  To that end, we fix any sequence \(\set{u_i}_{i\in\nats}\) such that \(\lim_{i\to+\infty}\sigma\pr{u_i}=0\).
  Because \(\lim_{i\to+\infty}\sigma\pr{u_i}=0\), it follows from the first part of the proof that
  \begin{equation*}
    \pr{\forall \epsilon\in\posreals}\pr{\exists i^\star\in\nats}\pr{\forall i\in\nats, i\geq i^\star}~
    \norm{\lto-\Phi_{u_i}}
    \leq \epsilon.
  \end{equation*}
  Hence, \(\lim_{i\to+\infty}\Phi_{u_i}=\lto\).
  That \(\lto\) is a linear counting transformation now follows from this equality if we recall from Corollary~\ref{cor:LinGenPoisGen:Phi_u_i converges to a LinCT} that \(\lim_{i\to+\infty}\Phi_{u_i}\) is a linear counting transformation.
\end{proof}

Using Corollaries~\ref{cor:LinGenPoisGen:Phi_u_i converges to a LinCT} and \ref{cor:Gen trm:Existence of T_S}, we now define the unique family of linear counting transformations that is generated by the linear generalised Poisson generator~\(\trm_S\).
Consider any sequence~\(S=\set{\lambda_x}_{x\in\stsp}\).
Then for any \(t,s\) in \(\nnegreals\) with \(t\leq s\), we define the corresponding linear counting transformation
\begin{equation*}
  T_{t,S}^s
  \coloneqq \lim_{\sigma\pr{u}\to+\infty} \set{\Phi_u\colon u\in\setoftseq_{\br{t,s}}}.
\end{equation*}
We collect all these transformations in the family~\(\mathcal{T}_S\coloneqq\set{T_{t,S}^s\colon t,s\in\nnegreals, t\leq s}\).

\subsection{Counting Transformation Systems}
We now provide a method to construct more intricate families of linear counting transformations.
This construction method is essential in the proof of Proposition~\ref{prop:Lower bound for conditional expectation is reached by a consistent process}, where we will construct a counting process with transition probabilities that are derived from these linear counting transformations.
Specifically, we are interested in families of the following type, the definition of which is based on \cite[Definition~3.3]{2017Krak}.
\begin{definition}
\label{def:Counting transformation system}
  A \emph{counting transformation system} is a family~\(\mathcal{T}=\set{T_t^s\colon t,s\in\nnegreals, t\leq s}\) of linear counting transformations such that
  \begin{enumerate}[label=\upshape{}S\arabic*., ref=\upshape(S\arabic*), leftmargin=*]
    \item \label{def:CTS:Identity}
    \(T_t^t=I\) for all \(t\) in \(\nnegreals\);
    \item \label{def:CTS:Semi-group}
    \(T_t^s=T_t^r T_r^s\) for all \(t,r,s\) in \(\nnegreals\) with \(t\leq r\leq s\);
    \item \label{def:CTS:Limit of more than two}
    \(\lim_{\Delta\to0^+}\frac{\br{T_t^{t+\Delta}\indica{\geq x+2}}\pr{x}}\Delta=0\) and, if \(t>0\), \(\lim_{\Delta\to0^+}\frac{\br{T_{t-\Delta}^t\indica{\geq x+2}}\pr{x}}\Delta=0\), for all \(t\) in \(\nnegreals\) and \(x\) in \(\stsp\).
  \end{enumerate}
\end{definition}
One example of a counting transformation system is the family~\(\mathcal{T}_S\) of linear counting transformations generated by the linear generalised Poisson generator~\(\trm_S\), as is established in the next result.
\begin{corollary}
\label{cor:Properties of T_S}
  Consider a sequence \(S\coloneqq\set{\lambda_x}_{x\in\stsp}\) in \(\Lambda=\br{\llambda, \ulambda}\).
  Then \(\mathcal{T}_S=\set{T_{t,S}^s\colon t, s\in\nnegreals, t\leq s}\), the corresponding family of linear counting transformations, is a counting transformation system.
  Furthermore,
  \begin{equation}
  \label{eqn:Prop of T_S:Time-homogeneity}
    T_{t,S}^s
    = T_{0,S}^{s-t}
    \qquad\text{for all } t,s\in\nnegreals, t\leq s.
  \end{equation}
\end{corollary}
\begin{proof}
  This is a corollary of Equation~\eqref{eqn:LinGenPoisGen and GenPoisGen}, Proposition~\ref{prop:Properties of lto_S induced by GenPoisGen}~\ref{prop:Properties of lto_S by GenPoisGen:IDentity}--\ref{prop:Properties of lto_S by GenPoisGen:Time-homogeneity} and Lemma~\ref{lem:lto_S satisfies orederliness property}.
\end{proof}
These simple systems can be used to construct more intricate systems.
First, we restrict these counting transformation systems.
Consider any counting transformation system~\(\mathcal{T}=\set{T_t^s\colon t,s\in\nnegreals, t\leq s}\) and any interval \(\mathcal{I}\) in \(\nnegreals\), here and in the remainder assumed to be of the form \(\br{t,s}\) or \([t,+\infty)\).
With this system~\(\mathcal{T}\) and the interval~\(\mathcal{I}\), we associate the \emph{restricted counting transformation system}
\begin{equation*}
  \mathcal{T}^{\mathcal{I}}
  \coloneqq\set{T_t^s \in \mathcal{T} \colon t,s\in\mathcal{I}, t\leq s}.
\end{equation*}
Next, we concatenate two restricted transformation systems.
Consider two counting transformation systems \(\mathcal{T}_1=\set{T_t^{s,1}\colon t,s\in\nnegreals, t\leq s}\) and \(\mathcal{T}_2=\set{T_t^{s,2}\colon t,s\in\nnegreals, t\leq s}\) and two intervals \(\mathcal{I}_1\) and \(\mathcal{I}_2\) in \(\nnegreals\) such that \(\mathcal{I}_1\) is closed and \(\max\mathcal{I}_1=\min\mathcal{I}_2\).
Then the associated \emph{concatenated transformation system}~\(\mathcal{T}_1^{\mathcal{I}_1}\otimes\mathcal{T}_2^{\mathcal{I}_2}\) is defined as the family of transformations~\(\set{T_t^s \colon t,s\in\mathcal{I}_1\cup\mathcal{I}_2, t\leq s}\) such that for all \(t,s\) in \(\mathcal{I}_1\cup\mathcal{I}_2\) with \(t\leq s\),
\begin{equation}
\label{eqn:Concatenated system of linear counting transformations}
  T_t^s
  \coloneqq \begin{cases}
    T_t^{s,1} &\text{if } s\leq r, \\
    T_t^{r,1} T_r^{s,2} &\text{if }t\leq r\leq s, \\
    T_t^{s,2} &\text{if } r\leq t,
  \end{cases}
\end{equation}
where \(r\coloneqq\max\mathcal{I}_1=\min\mathcal{I}_2\).
The following result establishes that the concatenated counting transformation system is again a (restricted) counting transformation system.
\begin{lemma}
\label{lem:Concatenation of restricted systems is again a restricted system}
  Consider two counting transformation systems \(\mathcal{T}_1=\set{T_{t}^{s,1}\colon t,s\in\nnegreals, t\leq s}\) and \(\mathcal{T}_2=\set{T_{t}^{s,2}\colon t,s\in\nnegreals, t\leq s}\) and fix some \(r\) in \(\posreals\).
  Then the concatenated transformation system~\(\mathcal{T}_1^{\br{0,r}}\otimes\mathcal{T}_2^{[r,+\infty)}\) is a restricted counting transformation system.
\end{lemma}
\begin{proof}
  It follows from Equation~\eqref{eqn:Concatenated system of linear counting transformations} and Lemma~\ref{lem:LinTT:Composition is again a LinTT} that every operator~\(T_t^s\) in the concatenation~\(\mathcal{T}_1^{\br{0,r}}\otimes\mathcal{T}_2^{[r,+\infty)}\) is a linear counting transformation.
  That \(\mathcal{T}_1^{\br{0,r}}\otimes\mathcal{T}_2^{[r,+\infty)}\) furthermore satisfies \ref{def:CTS:Identity}--\ref{def:CTS:Limit of more than two} follows immediately from Equation~\eqref{eqn:Concatenated system of linear counting transformations} and the fact that \(\mathcal{T}_1\) and \(\mathcal{T}_2\) satisfy \ref{def:CTS:Identity}--\ref{def:CTS:Limit of more than two}.
\end{proof}
\begin{corollary}
\label{cor:Constructed counting transformation system}
  Consider some \(u=t_0, \dots, t_n\) in \(\setofnetseq\) with \(t_0=0\) and, for all \(i\) in \(\set{0,\dots,n}\), some sequence \(S_i\coloneqq\set{\lambda_{i,x}}_{x\in\stsp}\) in \(\br{\llambda, \ulambda}\).
  Then
  \begin{equation*}
    \mathcal{T}
    \coloneqq \mathcal{T}^{\br{0, t_1}}_{S_0}\otimes\mathcal{T}^{\br{t_1, t_2}}_{S_1}\otimes\cdots\otimes\mathcal{T}^{\br{t_{n-1}, t_n}}_{S_{n-1}}\otimes\mathcal{T}^{[t_n, +\infty)}_{S_n}
  \end{equation*}
  is a counting transformation system.
\end{corollary}
\begin{proof}
  This essentially follows from Corollary~\ref{cor:Properties of T_S} and Lemma~\ref{lem:Concatenation of restricted systems is again a restricted system}.
  Let \(\mathcal{T}'_n\coloneqq \mathcal{T}_{S_n}\) and \(\mathcal{T}'_{n-1}\coloneqq \mathcal{T}_{S_{n-1}}^{\br{0, t_n}}\otimes\mathcal{T}_{n}^{[t_n,+\infty)}\).
  Recall from Corollary~\ref{cor:Properties of T_S} that \(\mathcal{T}_{S_{n-1}}\) and \(\mathcal{T}'_n=\mathcal{T}_{S_n}\) are counting transformation systems.
  Hence, it follows from Lemma~\ref{lem:Concatenation of restricted systems is again a restricted system} with \(\mathcal{T}_1=\mathcal{T}_{S_{n-1}}\),\(\mathcal{T}_2=\mathcal{T}'_n\) and \(r=t_n\) that \(\mathcal{T}'_{n-1}=\mathcal{T}_{S_{n-1}}^{\br{0, t_n}}\otimes\mathcal{T}_{n}^{[t_n,+\infty)}\) is a counting transformation system.

  Next, we let \(\mathcal{T}'_{n-2}\coloneqq \mathcal{T}_{S_{n-2}}^{\br{0, t_{n-1}}}\otimes\mathcal{T}_{n-1}^{\prime~[t_{n-1},+\infty)}\).
  We have just proven that \(\mathcal{T}'_{n-1}\) is a counting transformation system, and it follows from Corollary~\ref{cor:Properties of T_S} that \(\mathcal{T}_{S_{n-2}}\) is a counting transformation system.
  Hence, it follows from Lemma~\ref{lem:Concatenation of restricted systems is again a restricted system} with \(\mathcal{T}_1=\mathcal{T}_{S_{n-2}}\),\(\mathcal{T}_2=\mathcal{T}'_{n-1}\) and \(r=t_{n-1}\) that \(\mathcal{T}_{n-2}=\mathcal{T}_{S_{n-2}}^{\br{0, t_{n-1}}}\otimes\mathcal{T}_{n-1}^{\prime~[t_{n-1},+\infty)}\) is a counting transformation system.

  It is now clear that if we repeat the same argument an additional \(n-2\) times, we have verified the statement.
\end{proof}

\subsection{From a Linear Counting Transformation System to the Poisson Distribution}
We conclude this section of the Appendix with a study of the special case of constant sequences \(S=\set{\lambda}_{x\in\stsp}\) in \(\nnegreals\).
For any \(\lambda\) in \(\nnegreals\), we let
\begin{equation}
\label{eqn:trm_lambda}
  \trm_\lambda
  \coloneqq\trm_S
  =\ltro_{S'}
  \qquad\text{with }
  S\coloneqq\set{\lambda}_{x\in\stsp}
  \text{ and }
  S'\coloneqq\set{\pr{\lambda,\lambda}}_{x\in\stsp}.
\end{equation}
Similarly, we let \(T_{t,\lambda}^s\coloneqq T_{t,S}^s\) for all \(t,s\) in \(\nnegreals\) with \(t\leq s\), and \(\mathcal{T}_\lambda\coloneqq\mathcal{T}_S\) to simplify our notation.
\begin{corollary}
\label{cor:T_lambda is state homogeneous}
  Consider any \(\lambda\) in \(\nnegreals\).
  Then for any \(t,s\) in \(\nnegreals\) with \(t\leq s\), \(f\) in \(\setoffn\pr{\stsp}\) and \(x\) in \(\stsp\),
  \begin{equation*}
  \br*{T_{t,\lambda}^s f}\pr{y}
    = \br*{T_{t,\lambda}^s f'_x}\pr{y-x}
    \qquad\text{for all } y\in\stsp \text{ with } y\geq x,
  \end{equation*}
  where \(f'_x\colon\stsp\to\reals\colon z\mapsto f'_x\pr{z}\coloneqq f\pr{x+z}\).
\end{corollary}
\begin{proof}
  From Equation~\eqref{eqn:trm_lambda}, we know that \(T_{t,\lambda}^s=\lto_t^s\), where \(\lto_t^s\) is the lower counting transformation generated by the (linear) generalised Poisson generator~\(\ltro_{S'}=\trm_\lambda\).
  Therefore, it follows from Lemma~\ref{lem:State homogeneity of lto induced by genpoisgen} that
  \begin{equation*}
    \br{T_{t,\lambda}^sf}\pr{y}
    =\br{\lto_t^sf}\pr{y}
    =\br{\lto_t^sf'_x}\pr{y-x}
    =\br{T_{t,\lambda}^sf'_x}\pr{y-x},
  \end{equation*}
  as required.
\end{proof}
\begin{corollary}
\label{cor:T_lambda satisfies differential equation}
  Consider any \(\lambda\) in \(\nnegreals\).
  Then for any \(t,s\) in \(\nnegreals\) with \(t\leq s\),
  \begin{equation*}
    \pr{\forall\epsilon\in\posreals}
    \pr{\exists\delta\in\posreals}
    \pr{\forall\Delta\in\reals, 0<\abs{\Delta}<\delta, 0\leq t+\Delta\leq s}~
    \norm*{\frac{T_{t+\Delta,\lambda}^s-T_{t,\lambda}^s}{\Delta}+\trm_\lambda T_{t,\lambda}^s}
    \leq\epsilon.
  \end{equation*}
  and
  \begin{equation*}
    \pr{\forall\epsilon\in\posreals}
    \pr{\exists\delta\in\posreals}
    \pr{\forall\Delta\in\reals, 0<\abs{\Delta}<\delta, t\leq s+\Delta}~
    \norm*{\frac{T_{t,\lambda}^{s+\Delta}-T_{t,\lambda}^s}{\Delta}-\trm_\lambda T_{t,\lambda}^s}
    \leq\epsilon.
  \end{equation*}
\end{corollary}
\begin{proof}
  This is a specialisation of Lemma~\ref{lem:ltro satisfies differential equation}, as by Equation~\eqref{eqn:trm_lambda}, \(T_{t,\lambda}^s=\lto_t^s\) where \(\lto_t^s\) is the lower counting transformation generated by the (linear) generalised Poisson generator~\(\ltro_{S'}=\trm_\lambda\) associated with \(S'=\set{\pr{\lambda, \lambda}}_{x\in\stsp}\).
\end{proof}
Everything is now set up to state and prove the main result of this section, namely how the Poisson distribution is obtained from a counting transformation system.
\begin{proposition}
\label{prop:T_lambda indic gives Poisson distribution}
  Consider any \(\lambda\) in \(\nnegreals\).
  Then for all \(t, \Delta\) in \(\nnegreals\) and \(x, y\) in \(\stsp\),
  \begin{equation*}
    \br{T_{t, \lambda}^{t+\Delta}\indica{y}}\pr{x}
    = \begin{cases}
      \pois_{\lambda\Delta}\pr{y-x} &\text{if } y\geq x, \\
      0 &\text{otherwise}.
    \end{cases}
  \end{equation*}
\end{proposition}
\begin{proof}
  Let \(f\coloneqq \indica{y}\), and consider the function \(f'_x\colon\stsp\to\reals\colon z\mapsto f'_x\pr{z}\coloneqq f\pr{x+z}=\indica{y}\pr{x+z}\).
  We first consider the case that \(x>y\).
  Observe that \(f'_x=0\) because \(x>y\), whence it follows from Corollary~\ref{cor:T_lambda is state homogeneous} with \(f=\indica{y}\) that
  \begin{equation*}
    \br{T_{t,\lambda}^{t+\Delta} \indica{y}}\pr{x}
    = \br{T_{t,\lambda}^{t+\Delta} f}\pr{x}
    = \br{T_{t,\lambda}^{t+\Delta} f'_x}\pr{x-x}
    = \br{T_{t,\lambda}^{t+\Delta} 0}\pr{0}
    = 0,
  \end{equation*}
  where for the final equality we have used Corollary~\ref{cor:Properties of T_S} and \ref{LinTT:constant}.
  This equality clearly agrees with the stated.

  Second, we consider the case that \(x\leq y\).
  Observe that \(f'_x=\indica{y-x}\).
  Hence, it follows from Corollary~\ref{cor:T_lambda is state homogeneous} with \(f=\indica{y}\) that
  \begin{equation*}
    \br{T_{t,\lambda}^s \indica{y}}\pr{x}
    = \br{T_{t,\lambda}^{t+\Delta} f}\pr{x}
    = \br{T_{t,\lambda}^{t+\Delta} f'_x}\pr{x-x}
    = \br{T_{t,\lambda}^{t+\Delta} \indica{y-x}}\pr{0}
    = \br{T_{0,\lambda}^{\Delta} \indica{y-x}}\pr{0}.
  \end{equation*}
  where for the final equality we have used Equation~\eqref{eqn:trm_lambda} and Equation~\eqref{eqn:Prop of T_S:Time-homogeneity} of Corollary~\ref{cor:Properties of T_S}.
  Hence, to verify the stated we need to show that
  \begin{equation}
    \phi_z\pr{\Delta}
    \coloneqq \br{T_{0,\lambda}^\Delta \indica{z}}\pr{0}
    = \pois_{\lambda\Delta}\pr{z}
    \qquad\text{for all } \Delta\in\nnegreals \text{ and } z\in\stsp.
  \end{equation}
  Due to Equation~\eqref{eqn:trm_lambda}, Corollary~\ref{cor:Properties of T_S} and \ref{def:CTS:Identity}, we already know that
  \begin{equation*}
    \phi_{z}\pr{0}
    = \br{T_{0,\lambda}^0\indica{z}}\pr{0}
    = \br{I\indica{z}}\pr{0}
    = \indica{z}\pr{0}
    = \begin{cases}
      1 &\text{if } z=0, \\
      0 &\text{otherwise}.
    \end{cases}
  \end{equation*}
  To determine the other values, we start by fixing any \(\Delta\) in \(\nnegreals\) and \(z\) in \(\stsp\).
  Fix an \(\epsilon\) in \(\posreals\).
  By Corollary~\ref{cor:T_lambda satisfies differential equation}, there is a \(\delta^\star\) in \(\posreals\) such that
  \begin{equation*}
    \pr{\forall \delta\in\reals, 0<\abs{\delta}<\delta^\star, 0\leq\Delta+\delta}~
    \norm*{\frac{T_{0,\lambda}^{\Delta+\delta}-T_{0,\lambda}^\Delta}{\delta}-\trm_\lambda T_{0,\lambda}^\Delta}
    \leq \epsilon.
  \end{equation*}
  Fix any real number \(\delta\) such that \(0<\abs{\delta}<\delta^\star\) and \(0\leq\Delta+\delta\), and observe that
  \begin{align*}
    \abs*{\frac{\br{T_{0,\lambda}^{\Delta+\delta}\indica{z}}\pr{0}-\br{T_{0,\lambda}^\Delta\indica{z}}\pr{0}}{\delta}-\br{\trm_\lambda T_{0,\lambda}^\Delta\indica{z}}\pr{0}}
    &\leq\norm*{\frac{T_{0,\lambda}^{\Delta+\delta}\indica{z}-T_{0,\lambda}^\Delta\indica{z}}{\delta}-\trm_\lambda T_{0,\lambda}^\Delta\indica{z}} \\
    &\leq \norm*{\frac{T_{0,\lambda}^{\Delta+\delta}-T_{0,\lambda}^\Delta}{\delta}-\trm_\lambda T_{0,\lambda}^\Delta}\norm{\indica{z}}
    = \norm*{\frac{T_{0,\lambda}^{\Delta+\delta}-T_{0,\lambda}^\Delta}{\delta}-\trm_\lambda T_{0,\lambda}^\Delta}
    \leq \epsilon,
  \end{align*}
  where for the second inequality we have used \ref{BNH: norm Af <= norm A norm f} and the equality holds because \(\norm{\indica{z}}=1\).
  Note that
  \begin{equation*}
    \br{\trm_\lambda T_{0,\lambda}^\Delta\indica{z}}\pr{0}
    = \lambda \br{T_{0,\lambda}^\Delta\indica{z}}\pr{1} - \lambda \br{T_{0,\lambda}^\Delta\indica{z}}\pr{0}
    = \lambda \br{T_{0,\lambda}^\Delta\indica{z-1}}\pr{0} - \lambda \br{T_{0,\lambda}^\Delta\indica{z}}\pr{0}
    = \lambda\phi_{z-1}\pr{\Delta}-\lambda\phi_z\pr{\Delta},
  \end{equation*}
  where for the second equality we have used Corollary~\ref{cor:T_lambda is state homogeneous} and where for ease of notation we let \(\phi_{-1}\coloneqq0\) because if \(z=0\), then \(\indica{z-1}=0\) and hence it follows from Equation~\eqref{eqn:trm_lambda}, Corollary~\ref{cor:Properties of T_S} and \ref{def:LinTT:Homogeneity} that \(\br{T_{0,\lambda}^\Delta\indica{z-1}}\pr{0}=\br{T_{0,\lambda}^\Delta0}\pr{0}=0\).
  We substitute this equality in our previous inequality, to obtain
  \begin{equation*}
    \abs*{\frac{\phi_z\pr{\Delta+\delta}-\phi_z\pr{\Delta}}{\delta}-\lambda\phi_{z-1}\pr{\Delta}+\lambda\phi_z\pr{\Delta}}
    \leq \epsilon.
  \end{equation*}
  Since this holds for any \(\delta\) in \(\reals\) such that \(0<\abs{\delta}<\delta^\star\), and because \(\epsilon\) was an arbitrary positive real number, it follows from this inequality and the definition of the derivative that
  \begin{equation*}
    \mathrm{D}\phi_z\pr{\Delta}
    =\lambda\phi_{z-1}\pr{\Delta}-\lambda\phi_z\pr{\Delta}
    \qquad\text{for all } z\in\stsp \text{ and } \Delta\in\nnegreals,
  \end{equation*}
  where \(\mathrm{D}\phi_z\pr{\Delta}\) denotes the derivative of \(\phi_z\) evaluated in \(\Delta\).
  It is well-known---see for instance \cite[Section~3]{1960Khintchine}---that together with the initial condition~\(\phi_z\pr{0}=\indica{z}\pr{0}\), the resulting family of recursively defined initial value problems has a unique solution, namely \(\phi_z\pr{\Delta}=\pois_{\lambda\Delta}\pr{z}\) for all \(\Delta\) in \(\nnegreals\) and \(z\) in \(\stsp\).
\end{proof}

\section{Supplementary Material for Section~\ref{sec:Counting processes in general}}
\subsection{Coherent Conditional Probabilities}
We start this section with establishing some well-known properties of coherent conditional probabilities.
Our first result establishes a necessary and sufficient condition for the real-valued map to be a coherent conditional probability.
It is actually the condition that \citet{1985Regazzini} uses to define coherent conditional probabilities, but it follows from \cite[Theorems~3 and 4]{1985Regazzini} that our definition---that is, Definition~\ref{def:Coherent conditional probability in text}---is equivalent; see for instance also \cite[Appendix~B]{2017Krak}.
\begin{proposition}
\label{prop:Coherence of conditional probability}
  Let \(S\) be a sample space.
  The real-valued map \(\prob\) on \(\ccpdomain\subseteq\ccpevents\pr{S}\times\ccpneevents\pr{S}\) is a coherent conditional probability if and only if for all \(n\) in \(\nats\), \(\alpha_1, \dots, \alpha_n\) in \(\reals\) and \(\pr{A_1, C_1}, \dots, \pr{A_n, C_n}\) in \(\ccpdomain{}\),
  \begin{equation}
  \label{eqn:Coherence condition}
    \max\set*{\sum_{i=1}^n \alpha_i\indica{C_i}\pr{s}\pr{\prob\pr{A_i\cond C_i}-\indica{A_i}\pr{s}}\colon s\in\bigcup_{i=1}^n C_i}
    \geq 0.
  \end{equation}
\end{proposition}
The next result allows us to always extend coherent conditional probabilities to a larger domain.
\begin{proposition}[Theorem~4 in \cite{1985Regazzini}]
\label{prop:CCP can be coherently extended}
  Consider a sample space~\(S\) and a coherent conditional probability~\(\prob\) on~\(\ccpdomain\subseteq\ccpevents\pr{S}\times\ccpneevents\pr{S}\).
  Then for any \(\ccpdomain^\star\) with \(\ccpdomain\subseteq\ccpdomain^\star\subseteq\ccpevents\pr{S}\times\ccpneevents\pr{S}\), \(\prob\) can be extended to a coherent conditional probability~\(\prob^\star\) on \(\ccpdomain^\star\).
\end{proposition}
Finally, we establish some well-known properties of coherent conditional probabilities.
First, we here recall from \cite[Section~2]{1985Regazzini} that any coherent conditional probability~\(\prob\) satisfies \ref{LOP:geq 0}--\ref{LOP:Bayes rule} on its domain \(\cpdomain{}\).
Additionally, it satisfies the following well-known properties; we refer to \cite[Appendix~B]{2017Krak} for proofs.
\begin{lemma}
  Consider a sample space \(S\) and a coherent conditional probability~\(\prob\) on \(\ccpdomain{}\subseteq\ccpevents\pr{S}\times\ccpneevents\pr{S}\).
  Then for any \(\pr{A, C}\) in \(\ccpdomain\),
  \begin{enumerate}[label=\upshape{P\arabic*.}, leftmargin=*, ref=\upshape{(P\arabic*)}, leftmargin=*, start=5]
    \item \label{LOP:bounds}
    \(0\leq \prob\pr{A \cond C} \leq 1\);
    \item \label{LOP:A and A C}
    \(\prob\pr{A \cond C} = \prob\pr{A\cap C\cond C}\) if \(\pr{A\cap C, C}\in\ccpdomain{}\);
    \item \label{LOP:empty}
    \(\prob\pr{\emptyset \cond C} = 0\) if \(\pr{\emptyset, C}\in\ccpdomain{}\);
    \item \label{LOP:sample space}
    \(\prob\pr{S\cond C} = 1\) if \(\pr{S, C}\in\ccpdomain{}\).
  \end{enumerate}
\end{lemma}
In the remainder, we will make frequent use of \ref{LOP:geq 0}--\ref{LOP:sample space}.
As these are just the standard laws of probability, we will usually do this without explicitly referring to them.

\subsection{Counting Processes in Particular}
We first establish two obvious properties of coherent conditional probabilities on \(\cpdomain\) that will be useful throughout the remainder; see for example Lemma~\ref{lem:System results in counting process} or Proposition~\ref{prop:Poisson transition probabilities define a Poisson process} further on.
\begin{lemma}
\label{lem:Coherent prob on cpdomain:prob of <= x}
  Let \(\prob\) be a coherent conditional probability on the domain \(\ccpdomain{}\subseteq\ccpevents\pr{\setofpths{}}\times\ccpneevents\pr{\setofpths}\) that contains \(\cpdomain\).
  Fix some \(t,s\) in \(\nnegreals\) with \(t\leq s\), \(u\) in \(\setoftseq_{<t}\) and \(\pr{x_u,x}\) in \(\stsp_{u\cup t}\).
  Then for all \(y\) in \(\stsp\),
  \begin{equation*}
    \prob\pr{X_s\leq y\cond X_u=x_u, X_t=x}
    = \begin{cases}
      \sum_{z=x}^y \prob\pr{X_s=z\cond X_u=x_u, X_t=x} &\text{if } y\geq x, \\
      0 &\text{otherwise.}
    \end{cases}
  \end{equation*}
  Consequently, if \(y<x\), then
  \begin{equation*}
    \prob\pr{X_s= y\cond X_u=x_u, X_t=x}
    = 0.
  \end{equation*}
\end{lemma}
\begin{proof}
  To prove the first part of the statement, we observe that
  \begin{equation*}
    \prob\pr{X_s\leq y\cond X_u=x_u, X_t=x}
    = \prob\pr{X_u=x_u, X_t=x, X_s\leq y\cond X_u=x_u, X_t=x}.
  \end{equation*}
  If \(y<x\), then it follows from \ref{Paths:Non-decreasing} that \(\pr{X_u=x_u, X_t=x, X_s\leq y}=\emptyset\).
  Hence,
  \begin{equation*}
    \prob\pr{X_s\leq x\cond X_u=x_u, X_t=x}
    =\prob\pr{\emptyset\cond X_u=x_u, X_t=x}
    = 0,
  \end{equation*}
  which agrees with the stated.
  Alternatively, if \(y\geq x\), then it follows from \ref{Paths:Non-decreasing} and the finite additivity of \(\prob\) that
  \begin{equation*}
    \prob\pr{X_s\leq y\cond X_u=x_u, X_t=x}
    = \sum_{z=x}^y \prob\pr{X_u=x_u, X_t=x, X_s=z\cond X_u=x_u, X_t=x}
    = \sum_{z=x}^y \prob\pr{X_s=z\cond X_u=x_u, X_t=x},
  \end{equation*}
  as required.

  For the second part of the statement, we observe that \(\pr{X_s=y}\subseteq\pr{X_s\leq y}\).
  Together with the first part of the statement and the monotonicity of \(\prob\), this implies that
  \begin{equation*}
    \prob\pr{X_s= y\cond X_u=x_u, X_t=x}
    \leq \prob\pr{X_s\leq y\cond X_u=x_u, X_t=x}
    = 0.
  \end{equation*}
  As furthermore \(\prob\pr{X_s= y\cond X_u=x_u, X_t=x}\geq 0\) by \ref{LOP:geq 0}, this clearly implies the second part of the statement.
\end{proof}
\begin{lemma}
\label{lem:Coherent prob on cpdomain:from t to s via r}
  Let \(\prob\) be a coherent conditional probability on the domain \(\ccpdomain{}\subseteq\ccpevents\pr{\setofpths{}}\times\ccpneevents\pr{\setofpths}\) that contains \(\cpdomain\).
  Fix some \(t,s\) in \(\nnegreals\) with \(t<s\), \(u\) in \(\setoftseq_{<u}\) and \(\pr{x_u,x,y}\) in \(\stsp_{u\cup \set{t,s}}\).
  Then for all \(r\) in \(\nnegreals\) such that \(t<r<s\),
  \begin{equation*}
    \prob\pr{X_s=y\cond X_u=x_u, X_t=x}
    = \sum_{z=x}^y \prob\pr{X_s=y\cond X_u=x_u, X_t=x, X_r=z}\prob\pr{X_r=z\cond X_u=x_u, X_t=x}
  \end{equation*}
  and
  \begin{multline*}
    \prob\pr{X_s\geq y\cond X_u=x_u, X_t=x}
    = \sum_{z=x}^{y-1} \prob\pr{X_s\geq y\cond X_u=x_u, X_t=x, X_r=z}\prob\pr{X_r=z\cond X_u=x_u, X_t=x} \\ + \prob\pr{X_r\geq y\cond X_u=x_u, X_t=x}.
  \end{multline*}
\end{lemma}
\begin{proof}
  Observe that
  \begin{equation*}
    \prob\pr{X_s=y\cond X_u=x_u, X_t=x}
    = \prob\pr{X_u=x_u, X_t=x, X_s=y\cond X_u=x_u, X_t=x}
  \end{equation*}
  due to \ref{LOP:A and A C}.
  As \(\pr{X_u=x_u, X_t=x, X_s=y}=\cup_{z=x}^{y}\pr{X_u=x_u, X_t=x, X_r=z, X_s=y}\) due to \ref{Paths:Non-decreasing}, it follows from \ref{LOP:Additive if disjoint} and \ref{LOP:A and A C} that
  \begin{align*}
   \prob\pr{X_s=y\cond X_u=x_u, X_t=x}
    &= \sum_{z=x}^y \prob\pr{X_u=x_u, X_t=x, X_r=z, X_s=y\cond X_u=x_u, X_t=x} \\
    &= \sum_{z=x}^y \prob\pr{X_r=z, X_s=y\cond X_u=x_u, X_t=x}.
  \end{align*}
  Finally, we use \ref{LOP:Bayes rule}, to yield
  \begin{equation*}
    \prob\pr{X_s=y\cond X_u=x_u, X_t=x}
    = \sum_{z=x}^y \prob\pr{X_r=z\cond X_u=x_u, X_t=x} \prob\pr{X_s=y\cond X_u=x_u, X_t=x, X_r=z},
  \end{equation*}
  which is the first equality of the statement.

  For the second equality of the stated, we observe that due to \ref{Paths:Non-decreasing}, \(\pr{X_u=x_u, X_t=x, X_s\geq y}\) is the union of the pairwise disjoint events \(\pr{X_u=x_u, X_t=x, X_r\geq y}\) and \(\pr{X_u=x_u, X_t=x, X_r=z, X_s\geq y}\) for all \(z\) in \(\stsp\) such that \(x\leq z< y\).
  We now again use \ref{LOP:Additive if disjoint}, \ref{LOP:A and A C} and \ref{LOP:Bayes rule} to yield the second equality of the statement:
  \begin{align*}
    \prob\pr{X_s\geq y\cond X_u=x_u, X_t=x}
    &= \sum_{z=x}^{y-1} \prob\pr{X_r=z, X_s\geq y\cond X_u=x_u, X_t=x} \\
    &\qquad\qquad + \prob\pr{X_r\geq y\cond X_u=x_u, X_t=x} \\
    &= \sum_{z=x}^{y-1} \prob\pr{X_r=z\cond X_u=x_u, X_t=x} \prob\pr{X_s\geq y\cond X_u=x_u, X_t=x, X_r=z} \notag \\ &\qquad\qquad + \prob\pr{X_r\geq y\cond X_u=x_u, X_t=x}.
  \end{align*}
\end{proof}

Next, we prove the first result that is given in the main text.
In this proof, we will need the following---slightly stronger---lemma.
\begin{lemma}
\label{lem:Collection of finitary events is algabra}
  Consider some \(u\) in \(\setoftseq\).
  Then the corresponding collection of finitary events
  \begin{equation}
  \label{eqn:collection of finitary events}
    \mathcal{C}_u
    \coloneqq \set{\pr{X_v\in B} \colon v\in\setoftseq, B\subseteq\stsp_v, \\ \pr{\forall t\in v}~t\in u\cup [\max u, +\infty)}.
  \end{equation}
  is an algebra.
  Therefore, \(\cpfield_u=\mathcal{C}_u\).
\end{lemma}
\begin{proof}
  The second part of the stated---that is, that \(\cpfield_u=\mathcal{C}_u\)---is an immediate consequence of the first part---that is, that the collection of finitary events \(\mathcal{C}_u\) is an algebra---because \(\cpfield_u\) is defined in Equation~\eqref{eqn:cpfield_u} as the smallest algebra of sets that contains~\(\mathcal{C}_u\).
  Hence, we only need to verify that the collection of finitary events \(\mathcal{C}_u\) is an algebra of sets (sometimes also called a field of events), in the sense that
  \begin{enumerate}[label={\upshape{}F\arabic*.}, ref=\upshape(F\arabic*)]
    \item\label{FOE:Empty set} \(\emptyset\) belongs to \(\mathcal{C}_u\);
    \item\label{FOE:complement} \(\setofpths\setminus A_1\) belongs to \(\mathcal{C}_u\) for all \(A_1\) in \(\mathcal{C}_u\);
    \item\label{FOE:union} \(A_1\cup A_2\) belongs to \(\mathcal{C}_u\) for all \(A_1, A_2\) in \(\mathcal{C}_u\).
  \end{enumerate}

  Observe first that \ref{FOE:Empty set} holds, as the empty set~\(\emptyset\) belongs to \(\mathcal{C}_u\).
  For instance, take any \(t\) in \([\max u, +\infty)\) and let \(B\coloneqq \emptyset\).
  Then clearly
  \begin{equation*}
    \pr{X_t\in B}
    = \pr{X_t\in\emptyset}
    = \set{\pth\in\setofpths\colon \pth\pr{t}\in\emptyset}
    = \emptyset,
  \end{equation*}
  and \(\pr{X_t\in B}=\emptyset\) belongs to \(\mathcal{C}_u\) due to Equation~\eqref{eqn:collection of finitary events}.
  Similarly, \(\setofpths\) belongs to \(\mathcal{C}_u\) because \(\setofpths=\pr{X_t\in\stsp}\) for any \(t\) in \([\max u, +\infty)\).
  Observe that \ref{FOE:complement} and \ref{FOE:union} are trivially satisfied for \(A_1=\emptyset\) and \(A_1=\setofpths\).

  Hence, we now fix any \(A_1=\pr{X_{v_1}\in B_1}\) and \(A_2=\pr{X_{v_2}=B_2}\) in \(\mathcal{C}_u\) such that \(\emptyset\neq A_1\neq\setofpths{}\) and \(\emptyset\neq A_2\neq\setofpths{}\), and verify that \ref{FOE:complement} and \ref{FOE:union} hold.
  To that end, we recall that \(\pr{X_\emptyset\in\stsp_\emptyset}=\pr{X_\emptyset=x_\emptyset}=\setofpths\), so the condition \(\emptyset\neq A_1\neq\stsp\) implies that \(v_1\neq\emptyset\), and similarly \(v_2\neq\emptyset\).
  Hence, we can enumerate the time points in \(v_1\) as \(t_1, \dots, t_n\) and the time points in \(v_2\) as \(s_1, \dots, s_m\).

  In order to verify \ref{FOE:complement}, we observe that
  \begin{align*}
    \setofpths\setminus A_1
    &=\setofpths{}\setminus A_1
    = \setofpths{}\setminus\set{\pth\in\setofpths\colon \pr{\pth\pr{t_1}, \cdots, \pth\pr{t_n}} \in B_1}
    = \set{\pth\in\setofpths\colon \pr{\pth\pr{t_1}, \cdots, \pth\pr{t_n}} \notin B_1} \\
    &= \set{\pth\in\setofpths\colon \pr{\pth\pr{t_1}, \cdots, \pth\pr{t_n}} \in B_1^c}
    = \pr{X_{v_1}\in B_1^c},
  \end{align*}
  with \(B_1^c=\stsp_{v_1}\setminus B_1\) and where the third equality follows from \ref{Paths:Non-decreasing} and Equation~\eqref{eqn:stsp_u}.
  As \(B_1^c\subseteq\stsp_{v_1}\), it follows from Equation~\eqref{eqn:collection of finitary events} that \(\pr{X_{v_1}\in B_1^c}\) belongs to \(\mathcal{C}_u\).
  Hence, we may conclude that \(\setofpths\setminus A_1\) belongs to \(\mathcal{C}_u\), as required.

  To verify \ref{FOE:union}, we observe that
  \begin{align*}
    A_1\cup A_2
    &= \set{\pth\in\setofpths\colon \pr{\pth\pr{t_1},\dots, \pth\pr{t_n}}\in B_1}\cup\set{\pth\in\setofpths\colon \pr{\pth\pr{s_1},\dots, \pth\pr{s_m}}\in B_2} \\
    &=\set{\pth\in\setofpths\colon \pr{\pth\pr{t_1},\dots, \pth\pr{t_n}}\in B_1 \text{ or } \pr{\pth\pr{s_1},\dots, \pth\pr{s_m}}\in B_2}.
  \end{align*}
  Let \(v\) be the ordered union of \(v_1=t_1,\dots,t_n\) and \(v_2=s_1,\dots,s_m\).
  We now furthermore enumerate the time-points in \(v\) as \(r_1, \dots, r_k\), and let
  \begin{equation}
    B
    \coloneqq \set{\pr{x_{r_1}, \dots, x_{r_k}}\in\stsp_v\colon \pr{x_{t_1}, \dots, x_{t_n}}\in B_1 \text{ or }  \pr{x_{s_1}, \dots, x_{s_m}}\in B_2}.
  \end{equation}
  It then follows from \ref{Paths:Non-decreasing} and Equation~\eqref{eqn:stsp_u} that
  \begin{align*}
    A_1\cup A_2
    &= \set{\pth\in\setofpths\colon \pr{\pth\pr{t_1},\dots, \pth\pr{t_n}}\in B_1 \text{ or } \pr{\pth\pr{s_1},\dots, \pth\pr{s_m}}\in B_2} \\
    &= \set{\pth\in\setofpths\colon \pr{\pth\pr{r_1},\dots, \pth\pr{r_k}}\in B}
    = \pr{X_v\in B}.
  \end{align*}
  Because \(v\) and \(B\) trivially satisfy the requirements of Equation~\eqref{eqn:collection of finitary events}, we may conclude from this equality that \(A_1\cup A_2\) belongs to \(\mathcal{C}_u\), as required.
\end{proof}
\begin{proofof}{Lemma~\ref{lem:Element of cpfield_u}}
  It follows from Lemma~\ref{eqn:collection of finitary events} that \(A=\pr{X_v\in B_v}\), where \(v\) is some sequence of time points in \(\setoftseq\) such that \(t\) belongs to \(u\cup [\max u,+\infty)\) for all \(t\) in \(v\), and where \(B_v\) is a subset of \(\stsp_v\).

  Let \(w'\coloneqq v\setminus u\).
  If \(w'=\emptyset\), then we fix any \(t\) in the open interval \((\max u, +\infty)\), and we let \(w\) be the sequence containing \(t\).
  Alternatively, if \(w'\neq\emptyset\), then we let \(w\) be the (ordered) sequence of time points in \(w\).

  In any case, if we enumerate the time points in \(u\cup w\) as \(t_1, \dots, t_n\) and those in \(v\) as \(s_1, \dots, s_m\), then we can define the set
  \begin{equation}
  \label{eqn:Proof of :Def of B}
    B
    \coloneqq \set{\pr{x_{t_1}, \dots, x_{t_n}}\in\stsp_{u\cup w}\colon \pr{x_{s_1},\dots,x_{s_m}}\in B_v}.
  \end{equation}
  To obtain the stated, we observe that
  \begin{equation*}
    A
    = \pr{X_v\in B_v}
    = \set{\pth\in\setofpths\colon \pr{\pth\pr{s_1}, \dots, \pth\pr{s_m}}\in B_v}
    = \set{\pth\in\setofpths\colon \pr{\pth\pr{t_1}, \dots, \pth\pr{t_n}}\in B}
    = \pr{X_{u\cup w}\in B},
  \end{equation*}
  where for the third equality we have used Equation~\eqref{eqn:Proof of :Def of B}, Equation~\eqref{eqn:stsp_u} and \ref{Paths:Non-decreasing}.
\end{proofof}

\subsection{Constructing Counting  Processes}
We end this section with a number of interesting technical results about the construction of counting processes.
In the proof of the first of these technical results, we need the following intermediary result.
\begin{lemma}
\label{lem:Helper lemma with not-quite convex combination}
  Consider some non-empty finite index set \(\mathcal{I}\) and, for all \(i\) in \(\mathcal{I}\), some \(\alpha_i\) and \(p_i\) in \(\reals\).
  Let \(\alpha^\star\coloneqq\min\set{\alpha_i\colon i\in\mathcal{I}}\).
  If \(p_i\geq 0\) for all \(i\) in \(\mathcal{I}\) and \(\sum_{i\in\mathcal{I}}p_i\leq 1\), then
  \begin{equation*}
    \sum_{i\in\mathcal{I}} \alpha_i p_i
    \geq \min\set{0, \alpha^\star}.
  \end{equation*}
\end{lemma}
\begin{proof}
  We distinguish two cases based on the sign of \(\alpha^\star\).
  If \(\alpha^\star\geq 0\), then \(\alpha_i\geq 0\) for all \(i\) in \(\mathcal{I}\).
  Since furthermore \(p_i\geq 0\) for all \(i\) in \(\mathcal{I}\), we observe that \(\sum_{i\in\mathcal{I}} \alpha_i p_i \geq 0 = \min\set{0, \alpha^\star}\) as this is a sum of non-negative terms.

  Next, we consider the case \(\alpha^\star<0\).
  If \(\sum_{i\in\mathcal{I}}p_i=0\), then clearly \(\sum_{i\in\mathcal{I}}\alpha_i p_i = 0 \geq \min\set{0, \alpha^\star}\).
  If \(\sum_{i\in\mathcal{I}}p_i>0\), then we observe that
  \begin{equation*}
    \sum_{i\in\mathcal{I}} \alpha_i p_i
    = \pr*{\sum_{j\in\mathcal{I}}p_j} \sum_{i\in\mathcal{I}} \alpha_i \dfrac{p_i}{\sum_{j\in\mathcal{I}}p_j}
    \geq \pr*{\sum_{j\in\mathcal{I}}p_j} \alpha^\star
    \geq \alpha^\star
    = \min\set{0, \alpha^\star},
  \end{equation*}
  where the first inequality holds because a convex combination of real numbers is greater than or equal to the minimum of these real numbers, and the second inequality holds because \(\alpha^\star<0\) and \(0<\sum_{j\in\mathcal{I}}p_j\leq 1\).
\end{proof}
Everything is now set up to prove the following technical lemma, which is crucial when constructing counting processes in general and Poisson processes in particular.
\begin{lemma}
\label{lem:Construction lemma with sequence of time points and pmf}
  Let \(w=w_0, \dots, w_\ell\) be an element of \(\setofnetseq\).
  Let \(\prob_w\) be a real-valued function on
  \begin{equation}
  \label{eqn:ccpdomain_w}
    \ccpdomain_w
    \coloneqq \set{\pr{X_{w_j}=y, X_{w^j}=x_{w^j}} \in \cpdomain\colon j\in\set{0, \dots, \ell}, w^j\coloneqq\set{w_0, \dots, w_{j-1}}, x_{w^j} \in \stsp_{w^j}, y\in\stsp}
  \end{equation}
   such that, for any \(j\) in \(\set{0, \dots, \ell}\) and \(x_{w^{j}}\) in \(\stsp_{w^{j}}\) with \(w^{j}=\set{w_0, \dots, w_{j-1}}\)---and specifically \(w^0=\emptyset\)---(i) if \(j>0\), then \(\prob_w\pr{X_{w_j}=y\cond X_{w^j}=x_{w^j}}=0\) for all \(y\) in \(\stsp\) with \(y< x_{w_{j-1}}\); (ii) \(0\leq\prob_w\pr{X_{w_j}=y\cond X_{w^j}=x_{w^j}}\leq1\) for all \(y\) in \(\stsp\) with \(y\geq x_{w_{j-1}}\); and (iii) \(0\leq\sum_{y\in B}\prob_w\pr{X_{w_j}=y\cond X_{w^j}=x_{w^j}}\leq1\) for all finite subsets \(B\) of \(\stsp\).
  Then \(\prob_w\) is a coherent conditional probability.
\end{lemma}
\begin{proof}
  Our proof by induction is inspired by that of \citet[Lemma~C.1]{2017Krak}: we verify that \(\prob_w\) satisfies the necessary and sufficient condition for coherence of Proposition~\ref{prop:Coherence of conditional probability}.
  First, we observe that this is the case if \(\ell=0\).
  To verify this, we fix any \(n\) in \(\nats\) and, for any \(i\) in \(\set{1,\dots,n}\), some \(\pr{A_i, C_i}=\pr{X_{w_0}=y_i, \setofpths}\) in \(\ccpdomain_w\) and \(\alpha_i\) in \(\reals\).
  Observe that
  \begin{equation}
  \label{eqn:Proof of Construction lemma with sequence of time points and pmf:First step coherent}
    \max\set*{\sum_{i=1}^n \alpha_i \indica{C_i}\pr{\pth}\pr{\prob_w\pr{A_i\cond C_i} - \indica{A_i}\pr{\pth}} \colon \pth\in \bigcup_{i=1}^n C_i}
    = \max\set*{\sum_{i=1}^n \alpha_i \pr{\prob_w\pr{X_{w_0}=y_i\cond \setofpths} - \indica{A_i}\pr{\pth}} \colon \pth\in \setofpths}.
  \end{equation}
  Let \(B\coloneqq\set{y\in\stsp\colon\pr{\exists i\in\set{1,\dots,n}}~y_i=y}\), and observe that \(B\) is non-empty and has a finite number of elements.
  We now partition \(\set{1, \dots, n}\) according to the events~\(A_i=\pr{X_{w_\ell}=y_i}\): for any \(y\) in \(B\), we let
  \begin{equation*}
    \mathcal{I}_y
    \coloneqq \set{i\in\set{1,\dots,n}\colon y_i=y}.
  \end{equation*}
  Furthermore, we let \(\alpha^\star\coloneqq\min\set{\sum_{i\in\mathcal{I}_y}\alpha_i\colon y\in B}\), and let \(y^\star\) be the element of \(B\) that reaches this minimum.
  Observe that
  \begin{align}
    \sum_{i=1}^n \alpha_i \prob\pr{X_{w_0}=y_i\cond\setofpths}
    &= \sum_{y\in B}\sum_{i\in\mathcal{I}_y} \alpha_i \prob\pr{X_{w_0}=y\cond\setofpths}
    = \sum_{y\in B}\pr*{\sum_{i\in\mathcal{I}_y} \alpha_i} \prob\pr{X_{w_0}=y\cond\setofpths} \notag \\
    &\geq \min\set{0, \alpha^\star},
    \label{eqn:Proof of Construction lemma with sequence of time points and pmf:First step alpha inequality}
  \end{align}
  where the inequality follows from Lemma~\ref{lem:Helper lemma with not-quite convex combination} and the conditions (ii) and (iii) on \(\prob_w\) of the stated.
  If \(\alpha^\star\geq 0\), we let \(\pth^\star\) be any path such that \(\pth^\star\pr{w_0}\notin B\); note that this path exists by \ref{Paths:Existence for any instantiation} because \(B\) is a finite subset of \(\stsp\).
  This way,
  \begin{equation}
  \label{eqn:Proof of Construction lemma with sequence of time points and pmf:First step pth case 1}
    \sum_{y\in B}\pr*{\sum_{i\in\mathcal{I}_y}\alpha_i} \indica{\pr{X_{w_0}=y}}\pr{\pth^\star}
    = 0
    = \min\set{0, \alpha^\star}
  \end{equation}
  because \(\indica{\pr{X_{w_0}=y}}\pr{\pth^\star}=0\) for all \(y\) in \(B\).
  Otherwise, that is if \(\alpha^\star<0\), we let \(\pth^\star\) be any path such that \(\pth^\star\pr{w_0}=y^\star\); again, this path exists due to \ref{Paths:Existence for any instantiation}.
  This way,
  \begin{equation}
  \label{eqn:Proof of Construction lemma with sequence of time points and pmf:First step pth case 2}
    \sum_{y\in B}\pr*{\sum_{i\in\mathcal{I}_y}\alpha_i} \indica{\pr{X_{w_0}=y}}\pr{\pth^\star}
    = \alpha^\star
    = \min\set{0, \alpha^\star}
  \end{equation}
  because, for all \(y\) in \(B\), \(\indica{\pr{X_{w_0}=y}}\pr{\pth^\star}=0\) if \(y\neq y^\star\) and \(\indica{\pr{X_{w_0}=y}}\pr{\pth^\star}=1\) if \(y= y^\star\).
  Our choice of \(\pth^\star\) guarantees that
  \begin{align}
    \sum_{i=1}^n \alpha_i \pr{\prob\pr{X_{w_0}=y_i\cond\setofpths} - \indica{A_i}\pr{\pth^\star}}
    &= \sum_{i=1}^n \alpha_i\prob\pr{X_{w_0}=y_i\cond\setofpths} - \sum_{i=1}^n \alpha_i\indica{A_i}\pr{\pth^\star} \notag \\
    &= \sum_{i=1}^n \alpha_i \prob\pr{X_{w_0}=y_i\cond\setofpths} - \sum_{y\in B}\sum_{i\in\mathcal{I}_y} \alpha_i \indica{A_i}\pr{\pth^\star} \notag \\
    &= \sum_{i=1}^n \alpha_i \prob\pr{X_{w_0}=y\cond\setofpths} - \sum_{y\in B}\pr*{\sum_{i\in\mathcal{I}_y} \alpha_i} \indica{\pr{X_{w_0}=y}}\pr{\pth^\star} \notag \\
    &= \sum_{i=1}^n\alpha_i \prob\pr{X_{w_0}=y\cond\setofpths} - \min\set{0, \alpha^\star} \notag \\
    &\geq \min\set{0, \alpha^\star}-\min\set{0, \alpha^\star}
    = 0, \label{eqn:Proof of Construction lemma with sequence of time points and pmf:First step coherent <0}
  \end{align}
  where the fourth equality follows from Equations~\eqref{eqn:Proof of Construction lemma with sequence of time points and pmf:First step pth case 1} and \eqref{eqn:Proof of Construction lemma with sequence of time points and pmf:First step pth case 2}, and the inequality holds due to Equation~\eqref{eqn:Proof of Construction lemma with sequence of time points and pmf:First step alpha inequality}.
  From this we infer that Equation~\eqref{eqn:Proof of Construction lemma with sequence of time points and pmf:First step coherent} holds, which by Proposition~\ref{prop:Coherence of conditional probability} implies that \(\prob_w\) is a coherent conditional probability.

  Next, we fix any \(\ell\) in \(\nats\) with \(\ell\geq1\) and assume assume that the stated holds for any \(\ell'\) in \(\nats\) with \(0\leq\ell'<\ell\).
  We will now show that the stated then follows for \(\ell\).
  To verify the coherence of \(\prob_w\), we fix any \(n\) in \(\nats\), \(\pr{A_1, C_1}, \dots, \pr{A_n, C_n}\) in \(\ccpdomain_w\) and \(\alpha_1, \dots, \alpha_n\) in \(\reals\).
  We need to show that
  \begin{equation}
  \label{eqn:Proof of construction lemma:Coherenece in induction step}
    \max\set*{\sum_{i=1}^n \alpha_i \indica{C_i}\pr{\pth}\pr{\prob_w\pr{A_i\cond C_i} - \indica{A_i}\pr{\pth}} \colon \pth\in C}
    \geq 0,
  \end{equation}
  where \(C\coloneqq\cup_{i=1}^k C_i\).

  For any \(i\) in \(\set{1, \dots, n}\), there is some \(j_i\) in \(\set{0, \dots, \ell}\), some \(y_i\) in \(\stsp\) and \(x^i_{u^i}\) in \(\stsp_{u^i}\) with \(u^i\coloneqq\set{w_0, \dots, w_{j_i-1}}\), such that
  \begin{equation*}
    A_i
    = \pr{X_{w_{j_i}}=y_i}
  \quad\text{and}\quad
    C_i
    = \pr{X_{u^i}=x^i_{u^i}}.
  \end{equation*}
  Let \(\mathcal{I}_{<\ell}\coloneqq\set{i\in\set{1, \dots, n}\colon j_i<\ell}\).
  If \(\mathcal{I}_{<\ell}\neq\emptyset\), then it follows from the induction hypothesis that
  \begin{equation*}
    \max\set*{\sum_{i\in\mathcal{I}_{<\ell}} \alpha_i \indica{C_i}\pr{\pth}\pr{\prob_w\pr{A_i\cond C_i} - \indica{A_i}\pr{\pth}} \colon \pth\in C_{<\ell}}
    \geq 0,
  \end{equation*}
  with \(C_{<\ell}\coloneqq\cup_{i\in\mathcal{I}_{<\ell}} C_i\).
  From this, it follows that there is some \(\pth^\star\) in \(C_{<\ell}\subseteq C\) such that
  \begin{equation}
  \label{eqn:Proof of ...:first coherence ineqaulity}
    \sum_{i\in\mathcal{I}_{<\ell}} \alpha_i \indica{C_i}\pr{\pth^\star}\pr{\prob_w\pr{A_i\cond C_i} - \indica{A_i}\pr{\pth^\star}}
    \geq 0.
  \end{equation}
  If \(\mathcal{I}_{<\ell}=\emptyset\), then we let \(\pth^\star\) be an arbitrary element of \(C\).
  In any case, we have chosen a \(\pth^\star\) in \(C\) that satisfies Equation~\eqref{eqn:Proof of ...:first coherence ineqaulity}.

  Let \(C^\star\coloneqq\cap_{j=0}^{\ell-1} \pr{X_{w_j}=\pth^\star\pr{w_j}}\) and \(\mathcal{I}_{C^\star}\coloneqq\set{i\in\set{1, \dots, n}\colon C_i=C^\star}\).
  Observe that, by construction, \(j_i=\ell\) for all \(i\) in \(\mathcal{I}_{C^\star}\).
  We now execute the same trick as we did before.
  Since by construction \(\mathcal{I}_{C^\star}\) is clearly finite, the sets \(B_1^\star\coloneqq\set{y\in\stsp\colon y<\pth^\star\pr{w_{\ell-1}}, \pr{\exists i\in\mathcal{I}_{C^\star}}~y_i=y}\) and \(B_2^\star\coloneqq\set{y\in\stsp\colon y\geq \pth^\star\pr{w_{\ell-1}}, \pr{\exists i\in\mathcal{I}_{C^\star}}~y_i=y}\) have (at most) a finite number of elements.
  We now partition \(\mathcal{I}_{C^\star}\) according to the events~\(A_i=\pr{X_{w_\ell}=y_i}\): for any \(y\) in \(B_1^\star\cup B_2^\star\), we let
  \begin{equation*}
    \mathcal{I}_y
    \coloneqq \set{i\in\mathcal{I}_{C^\star}\colon A_i = \pr{X_{w_\ell}=y}}.
  \end{equation*}
  If we furthermore let \(\alpha^\star\coloneqq\min\set{\sum_{i\in\mathcal{I}_y}\alpha_i\colon y\in B_2^\star}\)---where we follow the convention that the minimum of the empty set is zero---then
  \begin{align}
    \sum_{i\in\mathcal{I}_{C^\star}} \alpha_i \prob_w\pr{A_i\cond C_i}
    &= \sum_{i\in\mathcal{I}_{C^\star}} \alpha_i \prob_w\pr{A_i\cond C^\star}
    = \sum_{y\in B_1^\star}\sum_{i\in\mathcal{I}_y} \alpha_i\prob_w\pr{A_i\cond C^\star} + \sum_{y\in B_2^\star}\sum_{i\in\mathcal{I}_y} \alpha_i\prob_w\pr{A_i\cond C^\star} \notag\\
    &= \sum_{y\in B_1^\star}\pr*{\sum_{i\in\mathcal{I}_y} \alpha_i}\prob_w\pr{X_{w_\ell} = y\cond C^\star} + \sum_{y\in B_2^\star}\pr*{\sum_{i\in\mathcal{I}_y} \alpha_i}\prob_w\pr{X_{w_\ell} = y\cond C^\star} \notag\\
    &= \sum_{y\in B_2^\star}\pr*{\sum_{i\in\mathcal{I}_y} \alpha_i}\prob_w\pr{X_{w_\ell} = y\cond C^\star} \notag\\
    &= \sum_{y\in B_2^\star}\pr*{\sum_{i\in\mathcal{I}_y} \alpha_i}\prob_w\pr{X_{w_\ell}=y \cond X_{w_0}=\pth^\star\pr{w_0}, \dots, X_{w_{\ell-1}}=\pth^\star\pr{w_{\ell-1}}} \notag\\
    &\geq \min\set{0, \alpha^\star},
    \label{eqn:Proof of Construction lemma with sequence of time points and pmf:Second step alpha inequality}
  \end{align}
  where for the fourth equality we use that \(\prob_w\pr{X_{w_\ell} = y\cond C^\star}=0\) for all \(y\) in \(B_1^\star\)---which follows from condition (i) on \(\prob_w\) of the statement because, by definition of \(B_1^\star\), \(y<\pth^\star\pr{w_{\ell-1}}\) for all \(y\) in \(B_1^\star\)---and where the inequality follows from Lemma~\ref{lem:Helper lemma with not-quite convex combination} due to the conditions (ii) and (iii) on \(\prob_w\) of the statement.

  If \(B_2^\star\) is non-empty and \(\alpha^\star<0\), then we let \(y^\star\) be any element of \(B_2^\star\) that reaches the minimum in the definition of \(\alpha^\star\); if \(B_2^\star\) is non-empty and \(\alpha^\star\geq0\), then we let \(y^\star\) be any element of \(\stsp\) such that \(y^\star\geq \pth^\star\pr{w_{\ell-1}}\) and \(y^\star\notin B_2^\star\); and finally, if \(B_2^\star\) is empty, then we set \(y^\star\coloneqq\pth^\star\pr{w_{\ell-1}}\).
  Because \(\pth^\star\) belongs to \(C^\star\) by definition and \(y^\star\geq\pth^\star\pr{w_{\ell-1}}\) by construction, it follows from \ref{Paths:Existence for any instantiation} that there is at least one path \(\pth\) in \(C^\star\) with \(\pth\pr{w_\ell}=y^\star\).
  Let \(\pth^{\star\star}\) be any such path.
  We have chosen \(\pth^{\star\star}\) such that
  \begin{equation}
  \label{eqn:Proof of construction lemma:Second step pth case 1}
    \sum_{y\in B^\star_1}\sum_{i\in\mathcal{I}_y} \alpha_i\indica{A_i}\pr{\pth^{\star\star}}
    = 0
  \end{equation}
  because \(\indica{A_i}\pr{\pth^{\star\star}}=\indica{\pr{X_{w_\ell}=y_i}}\pr{\pth^{\star\star}}=0\) for all \(i\) in \(\cup_{y\in B_1^\star}\mathcal{I}_y\) as \(\pth^{\star\star}\pr{w_\ell}=y^\star\geq\pth^\star\pr{w_{\ell-1}}>y_i\).
  Similarly,
  \begin{equation}
  \label{eqn:Proof of construction lemma:Second step pth case 2}
    \sum_{y\in B^\star_2}\sum_{i\in\mathcal{I}_y} \alpha_i\indica{A_i}\pr{\pth^{\star\star}}
    = \begin{cases}
      \sum_{i\in \mathcal{I}_{y^\star}} \alpha_i &\text{if } B^\star_2\neq\emptyset \text{ and } \alpha_\star<0 \\
      0 &\text{otherwise}
    \end{cases}
    = \min\set{0, \alpha^\star}
  \end{equation}
  because, for all \(i\) in \(\cup_{y\in B_2^\star}\mathcal{I}_y\), we have that \(\indica{A_i}\pr{\pth^{\star\star}}=\indica{\pr{X_{w_\ell}=y_i}}\pr{\pth^{\star\star}}=0\) if \(y_i\neq y^\star\) and \(1\) if \(y_i=y^\star\)---where the latter only occurs if \(B_2^\star\) is non-empty and \(\alpha^\star<0\).
  Furthermore, \(\indica{C_i}\pr{\pth^{\star\star}}=1\) for all \(i\) in \(\mathcal{I}_{C^\star}\) because \(\pth^{\star\star}\) is an element of \(C^\star\) by construction.
  Hence,
  \begin{align}
    \sum_{i\in\mathcal{I}_{C^\star}} \alpha_i \indica{C_i}\pr{\pth^{\star\star}} \pr{\prob_w\pr{A_i\cond C_i} - \indica{A_i}\pr{\pth^{\star\star}}}
    &= \sum_{i\in\mathcal{I}_{C^\star}} \alpha_i \pr{\prob_w\pr{A_i\cond C_i} - \indica{A_i}\pr{\pth^{\star\star}}}
    = \sum_{i\in\mathcal{I}_{C^\star}} \alpha_i \prob_w\pr{A_i\cond C_i} - \sum_{i\in\mathcal{I}_{C^\star}} \alpha_i \indica{A_i}\pr{\pth^{\star\star}} \notag \\
    &= \sum_{i\in\mathcal{I}_{C^\star}} \alpha_i \prob_w\pr{A_i\cond C_i} - \sum_{y\in B_1^\star}\sum_{i\in\mathcal{I}_y} \alpha_i\indica{A_i}\pr{\pth^{\star\star}} - \sum_{y\in B_2^\star}\sum_{i\in\mathcal{I}_y} \alpha_i\indica{A_i}\pr{\pth^{\star\star}} \notag \\
    &= \sum_{i\in\mathcal{I}_{C^\star}} \alpha_i\prob_w\pr{A_i\cond C_i} -0-\min\set{0, \alpha^\star} \notag\\
    &\geq \min\set{0, \alpha^\star} -\min\set{0, \alpha^\star}
    = 0, \label{eqn:Proof of construction lemma:Inequality for C^star}
  \end{align}
  where for the fourth equality we have used Equations~\eqref{eqn:Proof of construction lemma:Second step pth case 1} and \eqref{eqn:Proof of construction lemma:Second step pth case 2} and for the inequality we have used Equation~\eqref{eqn:Proof of Construction lemma with sequence of time points and pmf:Second step alpha inequality}.

  We now summarise our findings.
  Recall that, by construction, \(\pth^{\star\star}\) belongs to \(C^\star\) and that \(C^\star\) is a subset of \(C_{<\ell}\) if the latter is non-empty.
  Therefore, it follows from Equation~\eqref{eqn:Proof of ...:first coherence ineqaulity} that
  \begin{equation*}
    \sum_{i\in\mathcal{I}_{<\ell}} \alpha_i \indica{C_i}\pr{\pth^{\star\star}}\pr{\prob_w\pr{A_i\cond C_i}-\indica{A_i}\pr{\pth^{\star\star}}}
    \geq 0.
  \end{equation*}
  Similarly, it follows from Equation~\eqref{eqn:Proof of construction lemma:Inequality for C^star} that
  \begin{equation*}
    \sum_{i\in\mathcal{I}_{C^\star}} \alpha_i \indica{C_i}\pr{\pth^{\star\star}}\pr{\prob_w\pr{A_i\cond C_i}-\indica{A_i}\pr{\pth^{\star\star}}}
    \geq 0.
  \end{equation*}
  Finally, if we let \(\mathcal{I}\coloneqq\set{1, \dots,n}\setminus\pr{\mathcal{I}_{<\ell}\cup\mathcal{I}_{C^\star}}\), then clearly \(\indica{C_i}\pr{\pth^{\star\star}}=0\) for all \(i\) in \(\mathcal{I}\).
  Therefore,
  \begin{equation*}
    \sum_{i\in\mathcal{I}} \alpha_i \indica{C_i}\pr{\pth^{\star\star}}\pr{\prob_w\pr{A_i\cond C_i}-\indica{A_i}\pr{\pth^{\star\star}}}
    = 0.
  \end{equation*}
  From the three previous (in)equalities, we infer that
  \begin{equation*}
    \sum_{i=1}^n \alpha_i \indica{C_i}\pr{\pth^{\star\star}}\pr{\prob_w\pr{A_i\cond C_i}-\indica{A_i}\pr{\pth^{\star\star}}}
    = \sum_{i\in\pr{\mathcal{I}_{<\ell}\cup\mathcal{I}_{C^\star}\cup\mathcal{I}}} \alpha_i \indica{C_i}\pr{\pth^{\star\star}}\pr{\prob_w\pr{A_i\cond C_i}-\indica{A_i}\pr{\pth^{\star\star}}}
    \geq 0.
  \end{equation*}
  If the path \(\pth^{\star\star}\) belongs to \(C\), then we may conclude from this that Equation~\eqref{eqn:Proof of construction lemma:Coherenece in induction step} holds, which is what we set out to prove.
  Recall that this inequality holds for any path \(\pth^{\star\star}\) as long as (i) it belongs to \(C^\star\), in the sense that it coincides with \(\pth^\star\) on the time points \(w_0, \dots, w_{\ell-1}\); and (ii) it satisfies \(\pth^{\star\star}\pr{w_\ell}=y^\star\geq \pth^{\star\star}\pr{w_{\ell-1}}=\pth^{\star}\pr{w_{\ell-1}}\).
  Furthermore, we recall that \(\pth^\star\) belongs to \(C\), and that the paths in \(C\) are only ``specified'' on (a subset of) the time points \(w_0, \dots, w_{\ell-1}\), because
  \begin{equation*}
    C
    = \bigcup_{i=1}^k C_i
    = \bigcup_{i=1}^k \pr{X_{u^i}=x_{u^i}^i}
  \end{equation*}
  with \(u^i=\set{w_0, \dots, w_{j_i}-1}\) and \(j_i\leq \ell\).
  Consequently, it follows from \ref{Paths:Non-decreasing} and \ref{Paths:Existence for any instantiation} that there is a path \(\pth^{\star\star}\) in \(C\) that satisfies the two requirements.
\end{proof}
We continue with a second construction lemma, this time using a counting transformation system.
\begin{lemma}
\label{lem:Construction lemma for trans probs with full conditional}
  Consider a non-empty and ordered sequence of time points \(w=\set{w_0, \dots, w_\ell}\) in \(\setoftseq\) and a counting transformation system~\(\mathcal{T}=\set{T_t^s\colon t,s\in\nnegreals, t\leq s}\).
  Let \(\prob_w^\star\) be any coherent conditional probability on \(\ccpevents\pr{\setofpths{}}\times\ccpneevents\pr{\setofpths}\) such that, for all \(j\) in \(\set{1, \dots, \ell}\), \(x_{w^j}\) in \(\stsp_{w^j}\) with \(w^j\coloneqq\set{w_0, \dots, w_{j-1}}\) and \(x\) in \(\stsp\),
  \begin{equation*}
    \prob^\star_w\pr{X_{w_j}=x \cond X_{w^j}=x_{w^j}}
    = \br{T_{w_{j-1}}^{w_j} \indica{x}}\pr{x_{w_{j-1}}}.
  \end{equation*}
  Then for any \(t\) in \(w\) and \(u\) in \(\setoftseq_{<t}\) with
  \(0\neq u\subseteq w\), \(x_u\) in \(\stsp_u\) and \(x\) in \(\stsp\),
  \begin{equation*}
    \prob^\star_w\pr{X_t=x \cond X_u=x_u}
    = \br{T_{\max u}^t \indica{x}}\pr{x_{\max u}}.
  \end{equation*}
\end{lemma}
\begin{proof}
  Our proof of follows that of \cite[Lemma~C.2]{2017Krak} very closely.
  By assumption, there is some \(j\) in \(\set{1, \dots, \ell}\) such that \(t=w_j\) and \(u\subseteq\set{w_0, \dots, w_{j-1}}\).
  We prove the stated using induction.
  First, we observe that if \(t=w_1\), then \(u=\set{w_0}\) and the stated is trivially satisfied.
  Next, we assume that the stated holds for \(t=w_{j-1}\) with \(1<j\leq \ell\), and prove that the stated then also holds for \(t=w_j\).
  In the remainder, we distinguish between two cases: \(\max u=w_{j-1}\) and \(\max u<w_{j-1}\).

  Let us first consider the case that \(\max u=w_{j-1}\).
  Observe that, due to the laws of probability,
  \begin{align}
    \prob^\star_w\pr{X_t=x \cond X_u=x_u}
    = \prob^\star_w\pr{\pr{X_t=x} \cap \pr{X_u=x_u} \cond X_u=x_u}.
  \end{align}
  Note that, due to \ref{Paths:Non-decreasing} and Equation~\eqref{eqn:stsp_u},
  \begin{equation}
  \label{eqn:Proof of Construction lemma for trans probs with full conditional:X_u=x_u with B}
    \pr{X_u=x_u}
    = \pr{X_{w^j} \in B},
  \end{equation}
  with
  \begin{equation*}
    B
    \coloneqq \set*{y_{w^j} \in \stsp_{w^j} \colon \pr{\forall s\in u}~ y_s=x_s}.
  \end{equation*}
  Note that \(B\) is clearly a finite set, as by  construction the last component \(y_{w_{j-1}}\) of any \(y_{w^j}\) in \(B\) is equal to \(x_{w_{j-1}}\).
  Substituting this equality in the previous, we now obtain that
  \begin{align*}
    \prob^\star_w\pr{X_t=x \cond X_u=x_u}
    &= \prob^\star_w\pr{\pr{X_t=x} \cap \pr{X_{w^j}\in B} \cond X_u=x_u}
    = \prob^\star_w\pr{\pr{X_t=x} \cap \pr{\cup_{y_{w^j}\in B} \pr{X_{w^j}=y_{w^j}}} \cond X_u=x_u} \\
    &= \prob^\star_w\pr{\cup_{y_{w^j}\in B} \pr{X_t=x} \cap \pr{X_{w^j}=y_{w^j}} \cond X_u=x_u}
    = \sum_{y_{w^j}\in B}\prob^\star_w\pr{\pr{X_t=x} \cap \pr{X_{w^j}=y_{w^j}} \cond X_u=x_u} \\
    &= \sum_{y_{w^j}\in B}\prob^\star_w\pr{X_t=x\cond \pr{X_{w^j}=y_{w^j}} \cap\pr{X_u=x_u}} \prob^\star_w\pr{X_{w^j}=y_{w^j} \cond X_u=x_u} \\
    &= \sum_{y_{w^j}\in B}\prob^\star_w\pr{X_t=x\cond X_{w^j}=y_{w^j}} \prob^\star_w\pr{X_{w^j}=y_{w^j} \cond X_u=x_u} \\
    &= \sum_{y_{w^j}\in B} \br{T_{w_{j-1}}^{w_j}\indica{x}}\pr{y_{w_{j-1}}} \prob^\star_w\pr{X_{w^j}=y_{w^j} \cond X_u=x_u} \\
    &= \br{T_{w_{j-1}}^{w_j}\indica{x}}\pr{x_{w_{j-1}}} \sum_{y_{w^j}\in B} \prob^\star_w\pr{X_{w^j}=y_{w^j} \cond X_u=x_u},
  \end{align*}
  where for the penultimate equality we have used the equality \(t=w_j\) and the condition on \(\prob_w^\star\) of the stated, and where the last equality holds because, by construction, \(y_{w_{j-1}}=x_{w_{j-1}}\) for all \(y_{w^j}\) in \(B\).
  If \(x<x_{w_{j-1}}\), then because \(\mathcal{T}\) is a counting transformation system, it follows from \ref{LinTT:No transition to below} that \(\br{T_{w_{j-1}}^{w_j}\indica{x}}\pr{x_{w_{j-1}}}=0\).
  This agrees with the stated, since
  \begin{equation*}
    \br{T_{\max u}^t\indica{x}}\pr{x_{\max u}}
    = \br{T_{w_{j-1}}^t\indica{x}}\pr{x_{w_{j-1}}}
    = 0,
  \end{equation*}
  where the second equality follows from \ref{LinTT:No transition to below} because \(x<x_{w_{j-1}}\).
  In case \(x\geq x_{w_{j-1}}\), we observe that
  \begin{align*}
    \prob^\star_w\pr{X_t=x \cond X_u=x_u}
    &= \br{T_{w_{j-1}}^{w_j}\indica{x}}\pr{x_{w_{j-1}}} \sum_{y_{w^j}\in B} \prob^\star_w\pr{X_{w^j}=y_{w^j} \cond X_u=x_u} \\
    &= \br{T_{w_{j-1}}^{w_j}\indica{x}}\pr{x_{w_{j-1}}} \prob^\star_w\pr{\cup_{y_{w^j}\in B}\pr{X_{w^j}=y_{w^j}} \cond X_u=x_u}
    = \br{T_{w_{j-1}}^{w_j}\indica{x}}\pr{x_{w_{j-1}}} \prob^\star_w\pr{X_{w^j}\in B \cond X_u=x_u} \\
    &= \br{T_{w_{j-1}}^{w_j}\indica{x}}\pr{x_{w_{j-1}}} \prob^\star_w\pr{X_u=x_u \cond X_u=x_u}
    = \br{T_{w_{j-1}}^t\indica{x}}\pr{x_{w_{j-1}}},
  \end{align*}
  where for the penultimate equality we have used Equation~\eqref{eqn:Proof of Construction lemma for trans probs with full conditional:X_u=x_u with B}.
  This proves the induction step in case \(\max u = w_{j-1}\).

  Next, we consider the case \(\max u<w_{j-1}\).
  We execute exactly the same trick, but now we consider the---clearly finite---set
  \begin{equation}
  \label{eqn:Proof of Construction lemma for trans probs with full conditional:B in second case}
    B
    \coloneqq \set{y_{w^j} \in \stsp_{w^j} \colon y_{w_{j-1}}\leq x, \pr{\forall s\in u} y_s=x_s}
  \end{equation}
  and, for any \(y\) in \(\stsp\) such that \(x_{\max u}\leq y \leq x\), the---again clearly finite---set
  \begin{equation}
  \label{eqn:Proof of Construction lemma for trans probs with full conditional:B_y in second case}
    B_y
    \coloneqq \set*{y_{w^{j-1}} \in \stsp_{w^{j-1}} \colon \pr{y_{w^{j-1}}, y} \in B}
    = \set{y_{w^{j-1}}\in\stsp_{w^{j-1}}\colon y_{w_{j-2}}\leq y, \pr{\forall s\in u}~y_s=x_s}.
  \end{equation}
  Note that these two sets are connected, as
  \begin{equation}
  \label{eqn:Proof of Construction lemma for trans probs with full conditional:B as union}
    B
    = \bigcup_{y=x_{\max u}}^{x} \set{\pr{y_{w^{j-1}}, y} \colon y_{w^{j-1}} \in B_y}.
  \end{equation}
  Observe that
  \begin{equation*}
    \pr{X_u=x_u}\cap\pr{X_t=x}
    = \pr{X_{w^j} \in B}\cap\pr{X_t=x}.
  \end{equation*}
  Therefore
  \begin{equation*}
    \prob^\star_w\pr{X_t=x \cond X_u=x_u}
    = \prob^\star_w\pr{\pr{X_t=x} \cap \pr{X_u=x_u} \cond X_u=x_u}
    = \prob^\star_w\pr{\pr{X_t=x} \cap \pr{X_{w^j}\in B} \cond X_u=x_u}.
  \end{equation*}
  If \(x<x_{\max u}\), then it follows from Equations~\eqref{eqn:Proof of Construction lemma for trans probs with full conditional:B in second case} and \eqref{eqn:stsp_u} that \(B=\emptyset\).
  Therefore
  \begin{equation*}
    \prob^\star_w\pr{X_t=x \cond X_u=x_u}
    = \prob^\star_w\pr{\emptyset \cond X_u=x_u}
    = 0
    = \br{T_{\max u}^t \indica{x}}\pr{x_{\max u}},
  \end{equation*}
  where the final equality holds due to \ref{LinTT:No transition to below}.
  This case therefore agrees with the stated.

  Next, we consider the case \(x\geq x_{\max u}\).
  Then
  \begin{align*}
    \prob^\star_w\pr{X_t=x \cond X_u=x_u}
    &= \sum_{y_{w^j}\in B} \prob^\star_w\pr{\pr{X_{w_j}=x} \cap \pr{X_{w^j} = y_{w^j}} \cond X_u=x_u} \\
     &= \sum_{y_{w^j}\in B} \prob^\star_w\pr{X_{w_j}=x \cond \pr{X_{w^j}=y_{w^j}}\cap\pr{X_u=x_u}} \prob^\star_w\pr{X_{w^j} = y_{w^j} \cond X_u=x_u} \\
    &= \sum_{y_{w^j}\in B} \prob^\star_w\pr{X_{w_j}=x \cond X_{w^j}=y_{w^j}} \prob^\star_w\pr{X_{w^j} = y_{w^j} \cond X_u=x_u} \\
    &= \sum_{y=x_{\max u}}^{x} \sum_{y_{w^{j-1}}\in B_y} \prob^\star_w\pr{X_{w_j}=x \cond X_{w^{j-1}}=y_{w^{j-1}}, X_{w_{j-1}}=y} \\ &\hspace{12em} \prob^\star_w\pr{X_{w^{j-1}} = y_{w^{j-1}}, X_{w_{j-1}}=y \cond X_u=x_u}.
  \end{align*}
  where for the third equality we have used Equation~\eqref{eqn:Proof of Construction lemma for trans probs with full conditional:B in second case} and for the final equality we have used Equation~\eqref{eqn:Proof of Construction lemma for trans probs with full conditional:B as union}.
  We now use the condition on \(\prob_w\) of the stated to substitute the terms of the form \(\prob^\star_w\pr{X_{w_j}=x \cond X_{w^{j-1}}=y_{w^{j-1}}, X_{w_{j-1}}=y}\) with \(\br{T_{w_{j-1}}^{w_j}\indica{x}}\pr{y}\), to yield
  \begin{align*}
    \prob^\star_w\pr{X_t=x \cond X_u=x_u}
    &= \sum_{y=x_{\max u}}^{x} \br{T_{w_{j-1}}^{w_j} \indica{x}}\pr{y} \sum_{y_{w^{j-1}}\in B_y} \prob^\star_w\pr{X_{w^{j-1}} = y_{w^{j-1}}, X_{w_{j-1}}=y \cond X_u=x_u} \\
    &= \sum_{y=x_{\max u}}^{x} \br{T_{w_{j-1}}^{w_j} \indica{x}}\pr{y} \prob^\star_w\pr{X_{w^{j-1}} \in B_y, X_{w_{j-1}}=y \cond X_u=x_u} \\
    &= \sum_{y=x_{\max u}}^{x} \br{T_{w_{j-1}}^{w_j} \indica{x}}\pr{y} \prob^\star_w\pr{X_u=x_u, X_{w^{j-1}} \in B_y, X_{w_{j-1}}=y \cond X_u=x_u}.
  \end{align*}
  Observe now that, for all \(y\) in \(\stsp\) with \(x_{\max u}\leq y\leq x\),
  \begin{multline*}
    \pr{X_u=x_u, X_{w_{j-1}}=y}
    = \set{\pth\in\setofpths{}\colon \pth\pr{w_{j-1}}=y, \pr{\forall s\in u}~ \pth\pr{s}=x_s} \\
    \subseteq \set{\pth\in\setofpths{}\colon \pth\pr{w_{j-2}}\leq y, \pr{\forall s\in u}~ \pth\pr{s}=x_s}
    = \pr{X_{w^{j-1}}\in B_y},
  \end{multline*}
  where the inclusion follows from \ref{Paths:Non-decreasing} and the final equality follows from Equation~\eqref{eqn:Proof of Construction lemma for trans probs with full conditional:B_y in second case}.
  We infer from this that \(\pr{X_u=x_u, X_{w^{j-1}}\in B_y, X_{w_{j-1}}=y}=\pr{X_u=x_u, X_{w_{j-1}}=y}\).
  We now use this equality to yield
  \begin{align*}
    \prob^\star_w\pr{X_t=x \cond X_u=x_u}
    &= \sum_{y=x_{\max u}}^{x} \br{T_{w_{j-1}}^{w_j} \indica{x}}\pr{y} \prob^\star_w\pr{X_u=x_u, X_{w_{j-1}}=y \cond X_u=x_u} \\
    &= \sum_{y=x_{\max u}}^{x} \br{T_{w_{j-1}}^{w_j} \indica{x}}\pr{y} \prob^\star_w\pr{X_{w_{j-1}}=y \cond X_u=x_u} \\
    &= \sum_{y=x_{\max u}}^{x} \br{T_{w_{j-1}}^{w_j} \indica{x}}\pr{y} \br{T_{\max u}^{w_{j-1}}\indica{y}}\pr{x_{\max u}}
    = \br{T_{\max u}^{w_j}\indica{y}}\pr{x_{\max u}}
    = \br{T_{\max u}^t\indica{y}}\pr{x_{\max u}},
  \end{align*}
  where the third equality follows from the induction hypothesis and the penultimate equality follows from Lemma~\ref{lem:LinPT:Sum decomposition with indicator}.
\end{proof}

\begin{lemma}
\label{lem:System results in counting process}
  Consider a counting transformation system~\(\mathcal{T}\).
  Let \(\tilde{\prob}\) be the real-valued map with domain
  \begin{equation*}
    \tilde{\ccpdomain}
    \coloneqq \set{\pr{X_{t+\Delta}=y, \pr{X_u=x_u, X_t=x}}\in\cpdomain\colon t,\Delta\in\nnegreals, u\in\setoftseq_{<t}, \pr{x_u,x}\in\stsp_{u\cup t}, y\in\stsp} \cup \set{\pr{X_0=x, \setofpths{}}\in\cpdomain\colon x\in\stsp{}},
  \end{equation*}
  that is defined for all \(t, \Delta\) in \(\nnegreals\), \(u\) in \(\setoftseq_{<t}\), \(\pr{x_u, x}\) in \(\stsp_{u\cup t}\) and \(y\) in \(\stsp\) as
  \begin{equation*}
    \tilde{\prob}\pr{X_{t+\Delta}=y \cond X_u=x_u, X_t=x}
    \coloneqq \br{T_t^{t+\Delta}\indica{y}}\pr{x}
  \end{equation*}
  and for all \(x\) in \(\stsp\) as
  \begin{equation*}
    \tilde{\prob}\pr{X_0=x \cond \setofpths}
    \coloneqq \begin{cases}
      1 &\text{if } x = 0, \\
      0 &\text{otherwise}.
    \end{cases}
  \end{equation*}
  Then \(\tilde{\prob}\) is coherent, and any coherent extension of \(\tilde{\prob}\) to \(\cpdomain{}\) is a counting process.
\end{lemma}
\begin{proof}
  Our proof follows that of \citet[Therorem~5.2]{2017Krak} closely.
  We first verify that \(\tilde{\prob}\) is coherent using~Proposition~\ref{prop:Coherence of conditional probability}.
  To that end, we fix any arbitrary \(n\) in \(\nats\) and, for all \(i\) in \(\set{1, \dots, n}\), some \(\pr{A_i, C_i} = \pr{X_{t_i}=x, X_{u_i} = x_{u_i}}\) in \(\tilde{\ccpdomain}\) and \(\alpha_i\) in \(\reals\).
  We need to show that
  \begin{equation}
  \label{eqn:Coherence of tilde prob}
    \max\set*{\sum_{i=1}^n \alpha_i \indica{C_i}\pr{\pth} \pr*{\tilde{\prob}\pr{A_i\cond C_i}-\indica{A_i}\pr{\pth}} \colon \pth \in \bigcup_{i=1}^k C_i}
    \geq 0.
  \end{equation}
  Clearly, there is some non-empty, finite and increasing sequence \(w=w_0, \dots, w_{\ell}\) of time points with \(w_0=0\) such that \(u_i \subseteq w\) and \(t_i\in w\) for all \(i\) in \(\set{1, \dots, \ell}\).
  Let \(\tilde{\prob}_w\) be the restriction of \(\tilde{\prob}\) to \(\ccpdomain_w\), with \(\mathcal{D}_w\) as defined in Equation~\eqref{eqn:ccpdomain_w} of Lemma~\ref{lem:Construction lemma with sequence of time points and pmf}.
  In order to verify that \(\tilde{\prob}_w\) satisfies the three conditions of Lemma~\ref{lem:Construction lemma with sequence of time points and pmf}, we fix some \(j\) in \(\set{0, \dots, \ell}\) and \(x_{w^j}\) in \(\stsp_{w^j}\), with \(w^j\coloneqq w_0, \dots, w_{j-1}\).
  \begin{enumerate}[label=(\roman*)]
    \item Assume that \(j>0\), and fix some \(y\) in \(\stsp\) such that \(y<x_{w_{j-1}}\).
    Observe that
    \begin{equation*}
      \tilde{\prob}_w\pr{X_{w_j}=y\cond X_{w^j}=x_{w^j}}
      = \tilde{\prob}\pr{X_{w_j}=y\cond X_{w^j}=x_{w^j}}
      = \br{T_{w_{j-1}}^{w_j} \indica{y}}\pr{x_{w_{j-1}}}
      = 0,
    \end{equation*}
    where the last equality holds due to \ref{LinTT:No transition to below} because \(y<x_{w_{j-1}}\).
    \item Fix some \(y\) in \(\stsp\) such that \(y\geq x_{w_{j-1}}\).
    Observe that
    \begin{equation*}
      \tilde{\prob}_w\pr{X_{w_j}=y\cond X_{w^j}=x_{w^j}}
      = \tilde{\prob}\pr{X_{w_j}=y\cond X_{w^j}=x_{w^j}}
      = \br{T_{w_{j-1}}^{w_j} \indica{y}}\pr{x_{w_{j-1}}}.
    \end{equation*}
    The second condition is now satisfied because \(0=\inf\indica{y}\leq\br{T_{w_{j-1}}^{w_j} \indica{y}}\pr{x_{w_{j-1}}}\leq\sup\indica{y}=1\) due to \ref{LinTT:inf f T f sup f}.
    \item Fix some finite subset \(B\) of \(\stsp\), and observe that
    \begin{align*}
      \tilde{\prob}_w\pr{X_{w_j}\in B\cond X_{w^j}=x_{w^j}}
      &= \tilde{\prob}\pr{X_{w_j}\in B\cond X_{w^j}=x_{w^j}}
      = \sum_{y\in B} \tilde{\prob}\pr{X_{w_j}=y\cond X_{w^j}=x_{w^j}} \\
      &= \sum_{y\in B} \br{T_{w_{j-1}}^{w_j} \indica{y}}\pr{x_{w_{j-1}}}
      = \br[\bigg]{T_{w_{j-1}}^{w_j} \pr[\bigg]{\sum_{y\in B}\indica{y}}}\pr{x_{w_{j-1}}}
      = \br{T_{w_{j-1}}^{w_j} \indica{B}}\pr{x_{w_{j-1}}},
    \end{align*}
    where for the fourth equality we have used the additivity~\ref{def:LinTT:Additivity} of the linear counting transformation~\(T_{w_{j-1}}^{w_j}\).
    The third condition is now satisfied because \(0=\inf\indica{B}\leq\br{T_{w_{j-1}}^{w_j} \indica{B}}\pr{x_{w_{j-1}}}\leq\sup\indica{B}=1\) due to \ref{LinTT:inf f T f sup f}.
  \end{enumerate}
  Consequently, it follows from Lemma~\ref{lem:Construction lemma with sequence of time points and pmf} that \(\tilde{\prob}_w\) is a coherent conditional probability.
  By Proposition~\ref{prop:CCP can be coherently extended}, we can therefore extend \(\tilde{\prob}_w\) to a coherent conditional probability on \(\ccpevents\pr{\setofpths}\times\ccpneevents\pr{\setofpths}\).
  Let \(\tilde{\prob}^\star_w\) be any such extension.
  It then follows from Proposition~\ref{prop:Coherence of conditional probability} that
  \begin{equation}
  \label{eqn:Coherence of tilde prob star w}
    \max\set*{\sum_{i=1}^n \alpha_i \indica{C_i}\pr{\pth} \pr*{\tilde{\prob}^\star_w\pr{A_i\cond C_i}-\indica{A_i}\pr{\pth}} \colon \pth \in \bigcup_{i=1}^k C_i}
    \geq 0.
  \end{equation}

  We now claim that \(\tilde{\prob}^\star_w\pr{A_i\cond C_i} = \tilde{\prob}\pr{A_i\cond C_i}\) for all \(i\) in \(\set{1, \dots, n}\).
  To verify this claim, we fix any such \(i\).
  If \(u_i=\emptyset\), then \(t_i=0\) and so \(\pr{A_i, C_i}\) is an element of \(\ccpdomain_w\); therefore, \(\tilde{\prob}^\star_w\pr{A_i\cond C_i} = \tilde{\prob}_w\pr{A_i\cond C_i} = \tilde{\prob}\pr{A_i\cond C_i}\).
  If \(u_i\neq\emptyset\), then \(u_i\subseteq w\) and \(t_i\in w\).
  In this case, since \(\tilde{\prob}^\star_w\) satisfies the conditions of Lemma~\ref{lem:Construction lemma for trans probs with full conditional}, it follows from this lemma that \(\tilde{\prob}^\star_w\pr{A_i\cond C_i} = \tilde{\prob}\pr{A_i\cond C_i}\).
  Since \(\tilde{\prob}^\star_w\pr{A_i\cond C_i} = \tilde{\prob}\pr{A_i\cond C_i}\) for all \(i\) in \(\set{1, \dots, n}\), Equation~\eqref{eqn:Coherence of tilde prob} now follows from Equation~\eqref{eqn:Coherence of tilde prob star w}.

  Now that we have verified that \(\tilde{\prob}\) is coherent, it follows from Proposition~\ref{prop:CCP can be coherently extended} that it can be extended to a coherent conditional probability~\(\tilde{\prob}^\star\) on \(\cpdomain{}\).
  Let \(\tilde{\prob}^\star\) be any such coherent extension.

  We need to verify that \(\tilde{\prob}^\star\) is a counting process.
  That \ref{CP:Start at 0} is satisfied is immediate:
  \begin{equation*}
    \tilde{\prob}^\star\pr{X_0=0}
    = \tilde{\prob}^\star\pr{X_0=0\cond\setofpths}
    = \tilde{\prob}\pr{X_0=0\cond\setofpths}
    = 1.
  \end{equation*}
  To check \ref{CP:Orederliness}, we fix any \(t,\Delta\) in \(\nnegreals\), \(u\) in \(\setoftseq_{<t}\) and \(\pr{x_u, x}\) in \(\stsp_{u\cup t}\).
  Observe that
  \begin{align*}
    \tilde{\prob}^\star\pr{X_{t+\Delta}\geq x+2\cond X_u=x_u,X_t=x}
    &= 1-\tilde{\prob}^\star\pr{X_{t+\Delta}< x+2\cond X_u=x_u,X_t=x} \\
    &= 1-\tilde{\prob}^\star\pr{X_{t+\Delta}\leq x+1\cond X_u=x_u,X_t=x} \\
    &= 1 - \tilde{\prob}^\star\pr{X_{t+\Delta}=x\cond X_u=x_u,X_t=x} - \tilde{\prob}^\star\pr{X_{t+\Delta}=x+1\cond X_u=x_u,X_t=x} \\
    &= 1 - \tilde{\prob}\pr{X_{t+\Delta}=x\cond X_u=x_u,X_t=x} - \tilde{\prob}\pr{X_{t+\Delta}=x+1\cond X_u=x_u,X_t=x} \\
    &= 1 - \br{T_t^{t+\Delta}\indica{x}}\pr{x} - \br{T_t^{t+\Delta}\indica{x+1}}\pr{x},
  \end{align*}
  where we have used Lemma~\ref{lem:Coherent prob on cpdomain:prob of <= x} for the third equality.
  Observe that \(1=\br{T_t^{t+\Delta}\indica{\geq x}}\pr{x}\) because \(T_t^{t+\Delta}1=1\) due to \ref{LinTT:constant} and \(\br{T_t^{t+\Delta}1}\pr{x}=\br{T_t^{t+\Delta}\indica{\geq x}}\pr{x}\) due to \ref{def:LinTT:Counting}.
  Therefore,
  \begin{align*}
    \tilde{\prob}^\star\pr{X_{t+\Delta}\geq x+2\cond X_u=x_u,X_t=x}
    &= \br{T_t^{t+\Delta}\indica{\geq x}}\pr{x}-\br{T_t^{t+\Delta}\indica{x}}\pr{x} - \br{T_t^{t+\Delta}\indica{x+1}}\pr{x} \\
    &= \br{T_t^{t+\Delta}\pr{\indica{\geq x}-\indica{x}-\indica{x+1}}}\pr{x}
    = \br{T_t^{t+\Delta}\indica{\geq x+2}}\pr{x},
  \end{align*}
  where for the second equality we have used the additivity~\ref{def:LinTT:Additivity} of the linear counting transformation~\(T_t^{t+\Delta}\).
  Similarly, in case \(\max u < t-\Delta\),
  \begin{equation*}
    \tilde{\prob}^\star\pr{X_t\geq x+2\cond X_u=x_u,X_{t-\Delta}=x}
    =\br{T_{t-\Delta}^t\indica{\geq x+2}}\pr{x}.
  \end{equation*}
  That \ref{CP:Orederliness} is satisfied now follows if we combine these two equalities with \ref{def:CTS:Limit of more than two}.
\end{proof}

\section{Supplementary Material for Section~\ref{sec:Poisson processes in particular}}
Our proof of Theorem~\ref{the:Pois transition probabilities are poisson distributed and the converse} is rather lengthy, and therefore we split it into two parts.
For the first part, we first establish some convenient properties of a Poisson process.
\begin{lemma}
\label{lem:Properties of Poisson}
  Consider a Poisson process~\(\prob\).
  Then
  \begin{enumerate}[label=\upshape(\roman*), ref=\upshape(\roman*)]
    \item\label{lem:Properties of Poisson:Bounds} \(0\leq \prob\pr{X_t=x\cond X_0=0}\leq 1\) for all \(t\) in \(\nnegreals\) and \(x\) in \(\stsp\);
    \item\label{lem:Properties of Poisson:In 0} \(\prob\pr{X_0=x\cond X_0=0}=1\) if \(x=0\) and \(0\) otherwise for all \(x\) in \(\stsp\);
    \item\label{lem:Properties of Poisson:t_1 t_2} \(\prob\pr{X_{t_1+t_2}=x\cond X_0=0} = \sum_{y=0}^x\prob\pr{X_{t_1}=y\cond X_0=0}\prob\pr{X_{t_2}=x-y\cond X_0=0}\) for all \(t_1,t_2\) in \(\nnegreals\) and \(x\) in \(\stsp\);
    \item\label{lem:Properties of Poisson:Limit to 0} \(\lim_{t\to 0^+}\prob\pr{X_t=x\cond X_0=0}=\prob\pr{X_0=x\cond X_0=0}\) for all \(x\) in \(\stsp\).
  \end{enumerate}
\end{lemma}
\begin{proof}
  \begin{enumerate}[label=\upshape(\roman*)]
    \item Follows immediately from \ref{LOP:geq 0}.
    \item Follows almost immediately from \ref{CP:Start at 0} and \ref{LOP:1 if C in A}.
    \item
    If \(t_1\) or \(t_2\) is zero, then this follows almost immediately from \ref{lem:Properties of Poisson:In 0}.
    Hence, we now consider the case that \(t_1\neq0\neq t_2\).
    It follows from Lemma~\ref{lem:Coherent prob on cpdomain:from t to s via r} with \(t=0\), \(r=t_1\) and \(s=t_1+t_2\) that
    \begin{equation*}
      \prob\pr{X_{t_1+t_2}=x\cond X_0=0}
      = \sum_{y=0}^x \prob\pr{X_{t_1+t_2}=x\cond X_0=0, X_{t_1}=y}\prob\pr{X_{t_1}=y\cond X_0=0}.
    \end{equation*}
    We now use \ref{Pois:Markovian}, \ref{Pois:Time-homogeneous} and \ref{Pois:State-homogeneous}, to yield
    \begin{align*}
      \prob\pr{X_{t_1+t_2}=x\cond X_0=0}
      &= \sum_{y=0}^x \prob\pr{X_{t_1+t_2}=x\cond X_{t_1}=y}\prob\pr{X_{t_1}=y\cond X_0=0} \\
      &= \sum_{y=0}^x \prob\pr{X_{t_2}=x\cond X_0=y}\prob\pr{X_{t_1}=y\cond X_0=0} \\
      &= \sum_{y=0}^x \prob\pr{X_{t_1}=y\cond X_0=0}\prob\pr{X_{t_2}=x-y\cond X_0=}.
    \end{align*}
    \item Observe that if \(x\geq 2\), then it follows from \ref{LOP:bounds} and the monotonicity of \(\prob\) that for all \(t\) in \(\posreals\),
    \begin{equation*}
      0
      \leq \prob\pr{X_t=x\cond X_0=0}
      \leq \prob\pr{X_t\geq 2\cond X_0=0}.
    \end{equation*}
    The stated now follows from this inequality because \(\lim_{t\to 0^+}\prob\pr{X_t\geq 2\cond X_0=0}=0\) due to \ref{CP:Orederliness}.

    Next, we consider the case \(x=0\).
    Recall from \ref{lem:Properties of Poisson:Bounds} that \(\prob\pr{X_t=x\cond X_0=0}\) is bounded.
    Furthermore, as \(\pr{X_{t+\Delta}=0}\subseteq\pr{X_t=0}\) due to \ref{Paths:Non-decreasing}, it follows from the monotonicity of \(\prob\) that \(\prob\pr{X_t=x\cond X_0=0}\geq \prob\pr{X_{t+\Delta}=x\cond X_0=0}\) for all \(t, \Delta\) in \(\nnegreals\).
    In other words, \(\prob\pr{X_t=x\cond X_0=0}\) is a bounded and non-increasing function of \(t\).
    It is a standard result from analysis that the left limit of a bounded and non-increasing function on \(\nnegreals\) exists everywhere.
    Consequently, \(\lim_{t\to 0^+}\prob\pr{X_t=0\cond X_0=0}\) exists, and we denote this limit by \(\ell\).

    We need to show that \(\ell=\prob\pr{X_0=0\cond X_0=0}=1\), where the final equality holds due to \ref{lem:Properties of Poisson:In 0}.
    Observe that \(0\leq\ell\leq 1\) due to \ref{lem:Properties of Poisson:Bounds}.
    Our proof is one by contradiction: we assume ex-absurdo that \(\ell<1\).
    Fix any \(t\) in \(\posreals\) and \(n\) in \(\nats\), and let \(\Delta\coloneqq t/n\).
    It then follows from \ref{lem:Properties of Poisson:t_1 t_2} that
    \begin{equation*}
      \prob\pr{X_t=1\cond X_0=0}
      = \prob\pr{X_{\Delta}=0\cond X_0=0}\prob\pr{X_{t-\Delta}=1\cond X_0=0}+\prob\pr{X_{\Delta}=1\cond X_0=0}\prob\pr{X_{t-\Delta}=0\cond X_0=0}.
    \end{equation*}
    We now apply \ref{lem:Properties of Poisson:t_1 t_2} \(\pr{n-1}\) additional times, to yield
    \begin{equation*}
      \prob\pr{X_t=1\cond X_0=0}
      = n \prob\pr{X_\Delta=1\cond X_0=0} \prob\pr{X_\Delta=0\cond X_0=0}^{n-1}.
    \end{equation*}
    Observe that \(\prob\pr{X_\Delta=1\cond X_0=0}\leq 1\) due to \ref{LOP:bounds}, and that \(\prob\pr{X_\Delta=0\cond X_0=0}\leq\ell\) as \(\prob\pr{X_s=0\cond X_0=0}\) is non-increasing in \(s\) and has \(\ell\) as its right limit in \(s=0\).
    Hence,
    \begin{equation*}
      \prob\pr{X_t=1\cond X_0=0}
      \leq n \ell^{n-1}.
    \end{equation*}
    It is clear that if we take \(n\) sufficiently large, then this upper bound is arbitrarily close to \(0\).
    From this and \ref{lem:Properties of Poisson:Bounds}, it follows that \(\prob\pr{X_t=1\cond X_0=0}=0\) for all \(t\) in \(\posreals\).
    Consequently, \(\lim_{t\to0^+}\prob\pr{X_t=1\cond X_0=0}=0\).
    To obtain our contradiction, we observe that it follows from \ref{LOP:Additive if disjoint} and \ref{LOP:1 if C in A} that, for all \(t\) in \(\posreals\),
    \begin{equation}
    \label{eqn:Proof of Properties of Poisson:equality with 1}
      \prob\pr{X_t=0\cond X_0=0}
      = 1 - \prob\pr{X_t=1\cond X_0=0}-\prob\pr{X_t\geq 2\cond X_0=0}.
    \end{equation}
    Taking the limit for \(t\to0^+\) on both sides of the equality yields our contradiction, as
    \begin{equation*}
      \lim_{t\to+0^+}\prob\pr{X_t=0\cond X_0=0}
      = 1-\lim_{t\to0^+}\prob\pr{X_t=1\cond X_0=0}-\lim_{t\to0^+}\prob\pr{X_t\geq 2\cond X_0=0}
      =1,
    \end{equation*}
    where for the final equality we have also used that \(\lim_{t\to 0^+}\prob\pr{X_t\geq 2\cond X_0=0}=0\), a consequence of \ref{CP:Orederliness}.

    For the remaining case that \(x=1\), we use Equation~\eqref{eqn:Proof of Properties of Poisson:equality with 1} to yield the stated:
    \begin{equation*}
      \lim_{t\to0^+} \prob\pr{X_t=1\cond X_0=0}
      = 1 - \lim_{t\to0^+} \prob\pr{X_t=0\cond X_0=0}-\lim_{t\to0^+} \prob\pr{X_t\geq 2\cond X_0=0}
      = 0
      = \prob\pr{X_0=1\cond X_0=0},
    \end{equation*}
    where the second equality follows from the previous and the final equality follows from \ref{lem:Properties of Poisson:In 0}.
  \end{enumerate}
\end{proof}
Next, we establish the main result of the first part.
\begin{proposition}
\label{Prop:Pois transition probabilities are poisson distributed}
  Consider a Poisson process~\(\prob\).
  Then there is a rate~\(\lambda\) in \(\nnegreals\) such that, for all \(t, \Delta\) in \(\nnegreals\), \(u\) in \(\setoftseq_{<t}\), \(\pr{x_u, x}\) in \(\stsp_{u\cup t}\) and \(y\) in \(\stsp\),
  \begin{equation*}
    \prob\pr{X_{t+\Delta}=y \cond X_u=x_u, X_t=x}
    = \begin{cases}
      \pois_{\lambda \Delta}\pr{y-x} &\text{if } y\geq x, \\
      0 &\text{otherwise,}
    \end{cases}
  \end{equation*}
\end{proposition}
\begin{proof}
  The stated for the case \(y<x\) follows immediately from Lemma~\ref{lem:Coherent prob on cpdomain:prob of <= x}.
  Hence, we immediately move on to the case that \(y\geq x\).
  In this case,
  \begin{equation*}
    \prob\pr{X_{t+\Delta}=y \cond X_u=x_u, X_t=x}
    = \prob\pr{X_{t+\Delta}=y \cond X_t=x}
    = \prob\pr{X_{t+\Delta}=y-x \cond X_t=0}
    = \prob\pr{X_{\Delta}=y-x \cond X_0=0},
  \end{equation*}
  where the first equality follows from \ref{Pois:Markovian}, the second from \ref{Pois:State-homogeneous} and the third from \ref{Pois:Time-homogeneous}.
  To verify the stated, we now need to show that there is a \(\lambda\) in \(\nnegreals\) such that
  \begin{equation}
  \label{eqn:Proof of Prop:Pois transition probabilities are poisson distributed:Trans prob}
    \pr{\forall\Delta\in\nnegreals}\pr{\forall z\in\stsp}~
    \prob\pr{X_{\Delta}=z \cond X_0=0}
    = \pois_{\lambda\Delta}\pr{z}.
  \end{equation}

  What follows is a standard argument; see for instance \cite[Chapter~XVII, Section~6]{1968Feller} or \cite[Section~2]{1960Khintchine}.
  We start with the case \(z=0\).
  For notational simplicity, we let \(\theta\coloneqq\prob\pr{X_1=0\cond X_0=0}\).
  Recall from Lemma~\ref{lem:Properties of Poisson}~\ref{lem:Properties of Poisson:Bounds} that \(0\leq\theta\leq 1\).
  Furthermore, we observe that, for any \(n\) in \(\nats\),
  \begin{equation*}
    \prob\pr{X_1=0\cond X_0=0}
    = \prob\pr{X_{\frac{1}{n}}=0\cond X_0=0}^n,
  \end{equation*}
  where the equality follows from applying Lemma~\ref{lem:Properties of Poisson}~\ref{lem:Properties of Poisson:t_1 t_2} \(n\) times.
  Clearly, this implies that, for any \(n\) in \(\nats\),
  \begin{equation}
  \label{eqn:Proof of Pois transition probabilities are poisson distributed:theta^1/n}
    \prob\pr{X_{\frac{1}{n}}=0\cond X_0=0}
    = \theta^\frac{1}{n}.
  \end{equation}

  Recall from before that \(0\leq\theta\leq 1\).
  We infer from these inequalities and Equation~\eqref{eqn:Proof of Pois transition probabilities are poisson distributed:theta^1/n} that \(0<\theta\leq 1\).
  Indeed, in case \(\theta=0\), then it follows from Equation~\eqref{eqn:Proof of Pois transition probabilities are poisson distributed:theta^1/n} that \(\lim_{n\to+\infty}\prob\pr{X_{\frac{1}{n}}=0\cond X_0=0}=0\), which is not correct because this limit is equal to \(1\) due to Lemma~\ref{lem:Properties of Poisson}~\ref{lem:Properties of Poisson:In 0} and \ref{lem:Properties of Poisson:Limit to 0}.

  Next, we observe that, for any \(n\) and \(k\) in \(\nats\),
  \begin{equation}
  \label{eqn:Proof of Pois transition probabilities are poisson distributed:theta^k/n}
    \prob\pr{X_{\frac{k}{n}}=0\cond X_0=0}
    = \prob\pr{X_{\frac{1}{n}}=0\cond X_0=0}^k
    = \theta^{\frac{k}{n}},
  \end{equation}
  where for the first equality we have applied Lemma~\ref{lem:Properties of Poisson}~\ref{lem:Properties of Poisson:t_1 t_2} \(k\) times and for the second equality we have used Equation~\eqref{eqn:Proof of Pois transition probabilities are poisson distributed:theta^1/n}.

  Next, we fix any \(\Delta\) in \(\posreals\).
  Choose any \(n\) in \(\nats\), and let \(k\) be the non-negative integer such that
  \begin{equation*}
    \frac{k-1}{n}
    \leq \Delta
    \leq \frac{k}{n}.
  \end{equation*}
  Because \(\prob\pr{X_s=0\cond X_0=0}\) is a non-increasing function of \(s\)---as we have argued in the proof of Lemma~\ref{lem:Properties of Poisson}~\ref{lem:Properties of Poisson:Limit to 0}---it follows from these inequalities and Equation~\eqref{eqn:Proof of Pois transition probabilities are poisson distributed:theta^k/n} that
  \begin{equation*}
    \theta^{\frac{k-1}{n}}
    \geq \prob\pr{X_\Delta=0\cond X_0=0}
    \geq \theta^{\frac{k}{n}}.
  \end{equation*}
  It is now clear that in the limit for \(n\to+\infty\), the lower and upper bound both converge to \(\theta^\Delta\).
  We now let \(\lambda\coloneqq-\ln\pr{\theta}\), which yields a non-negative real number as \(0<\theta\leq 1\).
  Observe furthermore that \(\pois_{\lambda0}\pr{0}=1=\prob\pr{X_0=0\cond X_0=0}\), where the final equality is precisely Lemma~\ref{lem:Properties of Poisson}~\ref{lem:Properties of Poisson:In 0}.
  In conclusion,
  \begin{equation}
  \label{eqn:Proof of Prop:Pois transition probabilities are poisson distributed:Trans prob for 0}
    \pr{\forall\Delta\in\nnegreals}~
    \prob\pr{X_\Delta=0\cond X_0=0}
    = \theta^\Delta
    = e^{-\lambda\Delta},
  \end{equation}
  so Equation~\eqref{eqn:Proof of Prop:Pois transition probabilities are poisson distributed:Trans prob} is satisfied for the case \(z=0\).

  Next, we verify Equation~\eqref{eqn:Proof of Prop:Pois transition probabilities are poisson distributed:Trans prob} for \(z>0\).
  First, we introduce some additional notation.
  For any \(z\) in \(\stsp\), we let
  \begin{equation*}
    \phi_z\colon\nnegreals{}\to\reals
    \colon \Delta\mapsto\phi_z\pr{\Delta}
    \coloneqq \prob\pr{X_\Delta=z\cond X=0}.
  \end{equation*}
  Additionally, we also let \(\phi_{-1}\coloneqq0\).
  Observe that \(\pois_z\pr{0}=\indica{z}\pr{0}\) due to Lemma~\ref{lem:Properties of Poisson}~\ref{lem:Properties of Poisson:In 0}.
  We now claim that
  \begin{equation}
  \label{eqn:Proof of Prop:Pois transition probabilities are poisson distributed:System of diff eqns}
    \mathrm{D}\phi_z\pr{\Delta}
    =\lambda\phi_{z-1}\pr{\Delta}-\lambda\phi_z\pr{\Delta}
    \qquad\text{for all } z\in\stsp \text{ and } \Delta\in\nnegreals,
  \end{equation}
  where---as in the proof of Lemma~\ref{prop:T_lambda indic gives Poisson distribution}---\(\mathrm{D}\phi_z\pr{\Delta}\) denotes the derivative of \(\phi_z\) evaluated in \(\Delta\).
  It is well-known---see for instance \cite[Section~3]{1960Khintchine}---that together with the initial condition~\(\phi_z\pr{0}=\indica{z}\pr{0}\), the resulting family of recursively defined initial value problems has a unique solution, namely \(\phi_z\pr{\Delta}=\pois_{\lambda\Delta}\pr{z}\) for all \(\Delta\) in \(\nnegreals\) and \(z\) in \(\stsp\).
  Hence, our claim Equation~\eqref{eqn:Proof of Prop:Pois transition probabilities are poisson distributed:System of diff eqns} implies Equation~\eqref{eqn:Proof of Prop:Pois transition probabilities are poisson distributed:Trans prob}.

  In order to verify our claim, we first study the derivative of \(\phi_z\) in \(0\).
  It follows immediately from Equation~\eqref{eqn:Proof of Prop:Pois transition probabilities are poisson distributed:Trans prob for 0} that
  \begin{equation}
  \label{eqn:Proof of Prop:Pois transition probabilities are poisson distributed:Derivative in zero for zero}
    \lim_{\delta\to0^+} \frac{\phi_0\pr{\delta}-\phi_0\pr{0}}{\delta}
    = \lim_{\delta\to0^+} \frac{e^{-\lambda\delta}-1}{\delta}
    = \lambda.
  \end{equation}
  Next, we use that \(\phi_1\pr{0}=0\), execute some straightforward manipulations, use Equation~\eqref{eqn:Proof of Prop:Pois transition probabilities are poisson distributed:Trans prob for 0} and also \ref{CP:Orederliness}, to yield
  \begin{align}
    \lim_{\delta\to0^+} \frac{\phi_1\pr{\delta}-\phi_1\pr{0}}{\delta}
    &= \lim_{\delta\to0^+} \frac{\phi_1\pr{\delta}}{\delta}
    = \lim_{\delta\to0^+} \frac{\prob\pr{X_\delta=1\cond X_0=0}}{\delta} \notag\\
    &= \lim_{\delta\to0^+} \frac{1-\prob\pr{X_\delta=0\cond X_0=0}-\prob\pr{X_\delta\geq 2\cond X_0=0}}{\delta} \notag\\
    &= \lim_{\delta\to0^+} \frac{1-\prob\pr{X_\delta=0\cond X_0=0}}{\delta} - \lim_{\delta\to0^+} \frac{\prob\pr{X_\delta\geq 2\cond X_0=0}}{\delta} \notag\\
    &= \lim_{\delta\to0^+} \frac{1-e^{-\lambda\delta}}{\delta} - \lim_{\delta\to0^+} \frac{\prob\pr{X_\delta\geq 2\cond X_0=0}}{\delta} \notag\\
    &= -\lambda.
    \label{eqn:Proof of Prop:Pois transition probabilities are poisson distributed:Derivative in zero for one}
  \end{align}
  Finally, we observe that, for any \(z\) in \(\stsp\) such that \(z\geq 2\) and any \(\delta\) in \(\posreals\),
  \begin{equation*}
    0
    \leq \phi_z\pr{\delta}
    = \prob\pr{X_\delta=z\cond X_0=0}
    \leq \prob\pr{X_\delta\geq 2\cond X_0=0},
  \end{equation*}
  where the second inequality follows from the monotonicity of \(\prob\).
  From these inequalities, from the equality \(\phi_z\pr{0}=\indica{z}\pr{0}\) and from \ref{CP:Orederliness}, we infer that, for all \(z\) in \(\stsp\) such that \(z\geq 2\),
  \begin{equation}
  \label{eqn:Proof of Prop:Pois transition probabilities are poisson distributed:Derivative in zero for two or more}
    \lim_{\delta\to0^+} \frac{\phi_z\pr{\delta}-\phi_z\pr{0}}{\delta}
    = 0.
  \end{equation}

  We are now finally ready to study the derivative of \(\phi_z\) in a general time point---that is, to verify our claim Equation~\eqref{eqn:Proof of Prop:Pois transition probabilities are poisson distributed:System of diff eqns}.
  First, we observe that it is an immediate consequence of Equation~\eqref{eqn:Proof of Prop:Pois transition probabilities are poisson distributed:Trans prob for 0} that, for all \(\Delta\) in \(\nnegreals\),
  \begin{equation*}
    \mathrm{D}\phi_0\pr{\Delta}
    = \lim_{\delta\to0} \frac{\phi_0\pr{\Delta+\delta}-\phi_0\pr{\Delta}}{\delta}
    = \lambda\phi_0\pr{\Delta}
    = -\lambda\phi_{-1}\pr{\Delta}+\lambda\phi_0\pr{\Delta},
  \end{equation*}
  where the limit is a right limit if \(\Delta=0\).
  Hence, we move on to the case \(z\geq 1\).
  We will only consider the right limit, the left limit can be verified using a similar---but slightly more involved--- argument.
  To that end, we fix any \(\Delta\) in \(\nnegreals\) and \(z\) in \(\stsp\) with \(z\geq 1\).
  We use Lemma~\ref{lem:Properties of Poisson}~\ref{lem:Properties of Poisson:t_1 t_2} and Equations~\eqref{eqn:Proof of Prop:Pois transition probabilities are poisson distributed:Derivative in zero for zero}--\eqref{eqn:Proof of Prop:Pois transition probabilities are poisson distributed:Derivative in zero for two or more}, to yield
  \begin{align*}
    \lim_{\delta\to0^+} \frac{\phi_z\pr{\Delta+\delta}-\phi_z\pr{\Delta}}{\delta}
    &= \lim_{\delta\to0^+} \frac{\sum_{z'=0}^z \phi_{z-z'}\pr{\Delta}\phi_{z'}\pr{\delta}-\phi_z\pr{\Delta}}{\delta} \\
    &= \sum_{z'=0}^{z-1} \phi_{z'}\pr{\Delta} \lim_{\delta\to0^+} \frac{\phi_{z'}\pr{\delta}}{\delta}+\phi_z\pr{\Delta}\lim_{\delta\to0^+} \frac{\phi_0\pr{\delta}-1}{\delta} \\
    &= -\lambda \phi_{z-1}\pr{\Delta}+\lambda\phi_z\pr{\Delta},
  \end{align*}
  as required.
\end{proof}

For the second part of the proof of Theorem~\ref{the:Pois transition probabilities are poisson distributed and the converse}, we construct a Poisson process using only the transition probabilities.
\begin{proposition}
\label{prop:Poisson transition probabilities define a Poisson process}
  Fix any \(\lambda\) in \(\nnegreals\).
  Consider the real-valued map \(\tilde{\prob}\) with domain
  \begin{equation*}
    \tilde{\ccpdomain}
    \coloneqq \set{\pr{X_{t+\Delta}=y, \pr{X_u=x_u, X_t=x}}\in\cpdomain\colon t,\Delta\in\nnegreals, u\in\setoftseq_{<t}, \pr{x_u,x}\in\stsp_{u\cup t}, y\in\stsp} \cup \set{\pr{X_0=x, \setofpths{}}\in\cpdomain{}\colon x\in\stsp},
  \end{equation*}
  that is defined for all \(t, \Delta\) in \(\nnegreals\), \(u\) in \(\setoftseq_{<t}\), \(\pr{x_u, x}\) in \(\stsp_{u\cup t}\) and \(y\) in \(\stsp\) as
  \begin{equation*}
    \tilde{\prob}\pr{X_{t+\Delta}=y \cond X_u=x_u, X_t=x}
    \coloneqq \begin{cases}
      \pois_{\lambda \Delta}\pr{y-x} &\text{if } y \geq x, \\
      0 &\text{otherwise},
    \end{cases}
  \end{equation*}
  and for all \(x\) in \(\stsp\) as
  \begin{equation*}
    \tilde{\prob}\pr{X_0=0 \cond \setofpths}
    \coloneqq \begin{cases}
      1 &\text{if } x=0, \\
      0 &\text{otherwise.}
    \end{cases}1
  \end{equation*}
  Then \(\tilde{\prob}\) is a coherent conditional probability that has a unique extension~\(\tilde{\prob}^\star\) to \(\cpdomain\).
  Even more, this extension~\(\tilde{\prob}^\star\) is a Poisson process.
\end{proposition}
\begin{proof}
  Observe that for all \(t, \Delta\) in \(\nnegreals\), \(u\) in \(\setoftseq_{<t}\), \(\pr{x_u, x}\) in \(\stsp_{u\cup t}\) and \(y\) in \(\stsp\),
  \begin{equation}
  \label{eqn:Proof of Poisson transition probabilities define a Poisson process:Poisson via system with rate lambda}
    \tilde{\prob}\pr{X_{t+\Delta}=y\cond X_u=x_u, X_t=x}
    = \begin{cases}
      \pois_{\lambda \Delta}\pr{y-x} &\text{if } y \geq x \\
      0 &\text{otherwise}
    \end{cases}
    = \br{T_{t,\lambda}^{t+\Delta}\indica{y}}\pr{x},
  \end{equation}
  where the final equality follows from Proposition \ref{prop:T_lambda indic gives Poisson distribution}.
  Therefore, it follows from Corollary~\ref{cor:Properties of T_S} and Lemma~\ref{lem:System results in counting process} that the map~\(\tilde{\prob}\) is indeed coherent, and any coherent extension~\(\tilde{\prob}^\star\) of this map to \(\cpdomain{}\) is a counting process.

  We now set out to prove that this coherent extension~\(\tilde{\prob}^\star\) is unique.
  To that end, we fix any \(\pr{A, X_u=x_u}\) in \(\cpdomain\).
  We distinguish two cases: \(u=\emptyset\) and \(u\neq\emptyset\).
  For the first case, we let \(\tilde{\prob}'\) be any coherent extension of \(\tilde{\prob}^\star\) to \(\ccpevents\pr{\setofpths}\times\ccpneevents\pr{\setofpths}\).
  Observe that
  \begin{align}
    \tilde{\prob}^\star\pr{A\cond X_\emptyset=x_\emptyset}
    &= \tilde{\prob}'\pr{A\cond X_\emptyset=x_\emptyset}
    = \tilde{\prob}'\pr{A\cond \setofpths}
    = \tilde{\prob}'\pr{A\cap\pr{\pr{X_0=0}\cup\pr{X_0>0}}\cond \setofpths} \notag\\
    &= \tilde{\prob}'\pr{A\cap\pr{X_0=0}\cond \setofpths} + \tilde{\prob}'\pr{A\cap\pr{X_0>0}\cond \setofpths} \notag\\
    &= \tilde{\prob}'\pr{A\cond X_0=0}\tilde{\prob}'\pr{X_0=0\cond \setofpths} + \tilde{\prob}'\pr{A\cond X_0>0}\tilde{\prob}'\pr{X_0>0\cond \setofpths} \notag\\
    &= \tilde{\prob}^\star\pr{A\cond X_0=0}\tilde{\prob}^\star\pr{X_0=0\cond \setofpths} + \tilde{\prob}'\pr{A\cond X_0>0}\tilde{\prob}^\star\pr{X_0>0\cond \setofpths} \notag\\
    &= \tilde{\prob}^\star\pr{A\cond X_0=0}\tilde{\prob}^\star\pr{X_0=0\cond \setofpths} + \tilde{\prob}'\pr{A\cond X_0>0}\pr{1-\tilde{\prob}^\star\pr{X_0=0\cond \setofpths}} \notag\\
    &= \tilde{\prob}^\star\pr{A\cond X_0=0},
    \label{eqn:Proof of Poisson transition probabilities define a Poisson process:Annoying edge case}
  \end{align}
  where the third equality follows from \ref{LOP:A and A C} and the obvious fact that \(\pr{X_0=0}\) and \(\pr{X_0>0}\) partition \(\setofpths\), the fourth equality follows from \ref{LOP:Additive if disjoint}, the fifth equality follows from \ref{LOP:Bayes rule} and the last equality holds because \(\tilde{\prob}^\star\pr{X_0=0\cond \setofpths}=\tilde\prob\pr{X_0=0\cond \setofpths}=1\) by the conditions of the statement.
  Hence, \(\tilde{\prob}^\star\pr{A\cond X_\emptyset=x_\emptyset}\) is uniquely defined if \(\tilde{\prob}^\star\pr{A\cond X_0=0}\) is.

  We therefore immediately move on to the case \(u\neq\emptyset\).
  From Lemma~\ref{lem:Element of cpfield_u}, we know that there is a \(v\) in \(\setoftseq\) with \(\min v>\max u\) and a subset \(B\) of \(\stsp_w\) with \(w\coloneqq u\cup v\) such that \(A = \pr{X_w\in B}\).
  It follows from \ref{LOP:A and A C} that
  \begin{equation*}
    \tilde{\prob}^\star\pr{A\cond X_u=x_u}
    = \tilde{\prob}^\star\pr{A \cap \pr{X_u=x_u}\cond X_u=x_u}
    = \tilde{\prob}^\star\pr{X_u=x_u, X_v \in B' \cond X_u=x_u},
  \end{equation*}
  where we let \(B'\coloneqq\set{y_v\in\stsp_v \colon \pr{x_u, y_v}\in B}\).

  If \(B'=\emptyset\), then it immediately follows from \ref{LOP:empty} that
  \begin{equation*}
    \tilde{\prob}^\star\pr{A\cond X_u=x_u}
    = \tilde{\prob}^\star\pr{X_u=x_u, X_v \in B' \cond X_u=x_u}
    = \tilde{\prob}^\star\pr{\emptyset \cond X_u=x_u}
    = 0.
  \end{equation*}
  We therefore assume that \(B'\neq\emptyset\).
  Fix any arbitrary \(\epsilon\) in \(\posreals\), and let \(\Delta\coloneqq \max w-\max u\).
  As the Poisson distribution~\(\pois_{\lambda\Delta}\) is sigma-additive and normed, it follows that there is a \(z_\epsilon\) in \(\stsp\) such that for all \(z\) in \(\stsp\) with \(z\geq z_\epsilon\geq x_{\max u}\),
  \begin{equation}
  \label{eqn:Proof of Poisson transition probabilities define a Poisson process:Poisson normed condition}
    1-\epsilon
    \leq \sum_{y=x_{\max u}}^z \pois_{\lambda\Delta}\pr{y-x_{\max u}}
    \leq 1.
  \end{equation}
  Fix now any such \(z\).
  Observe that
  \begin{align*}
    \tilde{\prob}^\star\pr{X_{\max w}>z\cond X_u=x_u}
    &= 1-\tilde{\prob}^\star\pr{X_{\max u +\Delta}\leq z\cond X_u=x_u}
    = 1-\sum_{y=x_{\max u}}^z \tilde{\prob}^\star\pr{X_{\max u +\Delta}=y\cond X_u=x_u} \\
    &=1-\sum_{y=x_{\max u}}^z \tilde{\prob}\pr{X_{\max u +\Delta}=y\cond X_u=x_u}
    = 1-\sum_{y=x_{\max u}}^z \pois_{\lambda\Delta}\pr{y-x_{\max u}},
  \end{align*}
  where for the second equality we have used Lemma~\ref{lem:Coherent prob on cpdomain:prob of <= x}.
  We combine this equality with Equation~\eqref{eqn:Proof of Poisson transition probabilities define a Poisson process:Poisson normed condition}, to yield
  \begin{equation}
  \label{eqn:Proof of Poisson transition probabilities define a Poisson process:Epsilon bound}
    0
    \leq \tilde{\prob}^\star\pr{X_{\max w}>z\cond X_u=x_u}
    \leq \epsilon.
  \end{equation}

  Let \(B'_z\coloneqq\set{y_v\in B' \colon y_{\max v}\leq z}\), and observe that
  \begin{equation*}
    \pr{X_u=x_u, X_v \in B'_z}
    \subseteq \pr{X_u=x_u, X_v \in B'}
    \subseteq \pr{X_u=x_u, X_v\in B'_z} \cup \pr{X_{\max w} > z}.
  \end{equation*}
  Note that \(\pr{X_u=x_u, X_v\in B'_z}\) and \(\pr{X_{\max w} > z}\) are clearly disjoint.
  Therefore, it follows from the monotonicity and additivity of \(\tilde\prob^\star\) that
  \begin{equation}
  \label{eqn:Proof of Poisson transition probabilities define a Poisson process:Double inequality}
    \tilde{\prob}^\star\pr{X_u=x_u, X_v \in B'_z \cond X_u=x_u}
    \leq \tilde{\prob}^\star\pr{X_u=x_u, X_v \in B' \cond X_u=x_u}
    \leq \tilde{\prob}^\star\pr{X_u=x_u, X_v\in B'_z\cond X_u=x_u} + \epsilon,
  \end{equation}
  where for the last inequality we have also used the upper bound on \(\tilde{\prob}^\star\pr{X_{\max w} > z\cond X_u=x_u}\) of Equation~\eqref{eqn:Proof of Poisson transition probabilities define a Poisson process:Epsilon bound}.

  We now consider the communal term in Equation~\eqref{eqn:Proof of Poisson transition probabilities define a Poisson process:Double inequality}.
  Since \(B'_z\) is finite by construction, it follows from \ref{LOP:Additive if disjoint} that
  \begin{equation*}
    \tilde{\prob}^\star\pr{X_u=x_u, X_v \in B'_z \cond X_u=x_u}
    = \tilde{\prob}^\star\pr*{\bigcup_{y_v\in B'_z}\pr{X_u=x_u, X_v=y_v} \cond X_u=x_u}
    = \sum_{y_v\in B'_z} \tilde{\prob}^\star\pr{X_u=x_u, X_v=y_v \cond X_u=x_u}.
  \end{equation*}
  If we enumerate the time points in \(v\) as \(t_1, \dots, t_{\ell}\), then for any \(y_v\) in \(B'_z\), we find that
  \begin{multline*}
    \tilde{\prob}^\star\pr{X_u=x_u, X_v=y_v \cond X_u=x_u} \\
    = \tilde{\prob}^\star\pr{X_u=x_u, X_{t_1}=y_{t_1}\cond X_u=x_u} \cdots \tilde{\prob}^\star\pr{X_u=x_u, X_{t_\ell}=y_{t_\ell}\cond X_u=x_u, X_{t_1}=y_{t_1}, \dots, X_{t_{\ell-1}}=y_{t_{\ell-1}}} \\
    = \tilde{\prob}^\star\pr{X_{t_1}=y_{t_1}\cond X_u=x_u} \cdots \tilde{\prob}^\star\pr{X_{t_\ell}=y_{t_\ell}\cond X_u=x_u, X_{t_1}=y_{t_1}, \dots, X_{t_{\ell-1}}=y_{t_{\ell-1}}} \\
    = \tilde{\prob}\pr{X_{t_1}=y_{t_1}\cond X_u=x_u} \cdots \tilde{\prob}\pr{X_{t_\ell}=y_{t_\ell}\cond X_u=x_u, X_{t_1}=y_{t_1}, \dots, X_{t_{\ell-1}}=y_{t_{\ell-1}}},
  \end{multline*}
  where the last equality holds because all arguments of \(\tilde{\prob}^\star\) in the factors of this product are elements of \(\tilde{\ccpdomain{}}\).
  Hence, \(\tilde{\prob}^\star\pr{X_u=x_u, X_v \in B'_z \cond X_u=x_u}\) is uniquely defined by \(\tilde{\prob}\).
  Because of this, and also because Equation~\eqref{eqn:Proof of Poisson transition probabilities define a Poisson process:Double inequality} holds for any positive real number \(\epsilon\), we infer from this that \(\tilde{\prob}^\star\pr{A\cond X_u=x_u}\) is completely defined by the values of \(\tilde{\prob}\) on its domain.
  As \(\pr{A, X_u=x_u}\) was an arbitrary element of \(\cpdomain{}\), we conclude that there is a unique extension of \(\tilde{\prob}\) to \(\cpdomain{}\), and so \(\lambda\) indeed uniquely characterises a coherent conditional probability on \(\cpdomain\).

  Finally, we verify that the unique extension~\(\tilde{\prob}^\star\) is a Poisson process.
  To that end, we first verify that the coherent conditional probability~\(\tilde{\prob}^\star\) is a counting process.
  That \(\tilde{\prob}^\star\) satisfies \ref{CP:Start at 0} follows immediately from the definition of \(\tilde\prob\).
  Furthermore, \ref{CP:Orederliness} is also satisfied due to the definition of \(\tilde\prob\).
  In order to verify this, we fix any \(t\) in \(\nnegreals\), \(u\) in \(\setoftseq_{<t}\) and \(\pr{x_u, x}\) in \(\stsp_{u\cup t}\).
  Observe that, for any \(\Delta\) in \(\posreals\),
  \begin{align*}
    \tilde{\prob}^\star\pr{X_{t+\Delta}\geq x+2\cond X_u=x_u, X_t=x}
    &= 1-\tilde{\prob}^\star\pr{X_{t+\Delta}\leq x+1\cond X_u=x_u, X_t=x} \\
    &= 1-\tilde{\prob}^\star\pr{X_{t+\Delta}=x\cond X_u=x_u, X_t=x}-\tilde{\prob}^\star\pr{X_{t+\Delta}= x+1\cond X_u=x_u, X_t=x} \\
    &= 1-\tilde{\prob}\pr{X_{t+\Delta}=x\cond X_u=x_u, X_t=x}-\tilde{\prob}\pr{X_{t+\Delta}= x+1\cond X_u=x_u, X_t=x} \\
    &= 1-\pois_{\lambda\Delta}\pr{0}-\pois_{\lambda\Delta}\pr{1},
  \end{align*}
  where we have used Lemma~\ref{lem:Coherent prob on cpdomain:prob of <= x} for the second equality.
  Consequently,
  \begin{equation*}
    \lim_{\Delta\to0^+} \frac{\tilde{\prob}^\star\pr{X_{t+\Delta}\geq x+2\cond X_u=x_u, X_t=x}}{\Delta}
    = \lim_{\Delta\to0^+} \frac{1-e^{-\lambda\Delta}-\lambda\Delta e^{-\lambda\Delta}}{\Delta}
    = 0,
  \end{equation*}
  as required.
  If \(t>0\), then similar reasoning can be used to verify this equality for the limit from the left.
  Hence, \(\tilde{\prob}^\star\) is a counting process.
  That \(\tilde{\prob}^\star\) furthermore satisfies \ref{Pois:Markovian}--\ref{Pois:Time-homogeneous} follows immediately from the definition of \(\tilde{\prob}\).
\end{proof}

\begin{proofof}{Theorem~\ref{the:Pois transition probabilities are poisson distributed and the converse}}
  The first part of the stated follows from Proposition~\ref{Prop:Pois transition probabilities are poisson distributed}.
  The second part essentially follows from Proposition~\ref{prop:Poisson transition probabilities define a Poisson process}.
  The requirement of Proposition~\ref{prop:Poisson transition probabilities define a Poisson process} regarding \(\prob\pr{X_0=x\cond \setofpths}\) seems to be more restrictive, but it is not.
  To see this, we observe that any coherent conditional probability on \(\cpdomain{}\) that satisfies \ref{CP:Start at 0}, will also satisfy \(\prob\pr{X_0=x\cond \setofpths}=0\) for all \(x>0\), as
  \begin{equation*}
    0
    \leq \prob\pr{X_0=x\cond \setofpths}
    = 1-\prob\pr{X_0\in\stsp\setminus\set{x}\cond \setofpths}
    \leq 1-\prob\pr{X_0=0\cond \setofpths}
    = 0.
  \end{equation*}
\end{proofof}

\begin{proofof}{Corollary~\ref{cor:Poisson has rate lambda}}
  Let \(\prob\) be a Poisson process, and let \(\lambda\) be the rate mentioned in Theorem~\ref{the:Pois transition probabilities are poisson distributed and the converse}.
  Then
  \begin{equation*}
    \lim_{\Delta\to0^+}\frac{\prob\pr{X_{t+\Delta}=x+1\cond X_u=x_u, X_t=x}}{\Delta}
    = \lim_{\Delta\to0^+}\frac{\pois_{\lambda\Delta}\pr{1}}{\Delta}
    = \lim_{\Delta\to0^+}\frac{\lambda\Delta e^{-\lambda\Delta}}{\Delta}
    = \lambda
  \end{equation*}
  and, if \(t>0\),
  \begin{equation*}
    \lim_{\Delta\to0^+}\frac{\prob\pr{X_t=x+1\cond X_u=x_u, X_{t-\Delta}=x}}{\Delta}
    = \lim_{\Delta\to0^+}\frac{\pois_{\lambda\Delta}\pr{1}}{\Delta}
    = \lim_{\Delta\to0^+}\frac{\lambda\Delta e^{-\lambda\Delta}}{\Delta}
    = \lambda.
  \end{equation*}
\end{proofof}

\begin{proofof}{Theorem~\ref{the:Counting process with rate lambda is Poisson}}
  To prove this result, we make use of the terminology and results of Sections~\ref{sec:Sets of counting processes}--\ref{sec:Computing lower expectations}.
  Observe that, by assumption, \(\prob\) is consistent with the degenerate rate interval~\(\Lambda=\br{\lambda,\lambda}\).
  Fix some \(t, \Delta\) in \(\nnegreals\), \(u\) in \(\setoftseq_{<t}\), \(\pr{x_u, x}\) in \(\stsp_{u\cup t}\) and \(y\) in \(\stsp\).
  Observe that
  \begin{multline*}
    \prob\pr{X_{t+\Delta}=y\cond X_u=x_u, X_t=x}
    = \prev_\prob\pr{\indica{\pr{X_{t+\Delta}=y}}\cond X_u=x_u, X_t=x}
    = \prev_\prob\pr{\indica{y}\pr{X_{t+\Delta}}\cond X_u=x_u, X_t=x} \\
    \geq \lprev_{\br{\lambda,\lambda}}\pr{\indica{y}\pr{X_{t+\Delta}}\cond X_u=x_u, X_t=x}
    = \br{T_{t, \lambda}^{t+\Delta}\indica{y}}\pr{x},
  \end{multline*}
  where for the first equality we have used Equation~\eqref{eqn:Expectation of a simple function}, for the inequality we have used Equation~\eqref{eqn:Lower expectation with respect to set of counting processes} and for the final equality we have used Theorem~\ref{the:Lower conditional expectation of bounded function with respect to all consistent processes}.
  Similarly, we also find that
  \begin{align*}
    \prob\pr{X_{t+\Delta}=y\cond X_u=x_u, X_t=x}
    &\leq \uprev_{\br{\lambda,\lambda}}\pr{\indica{y}\pr{X_{t+\Delta}}\cond X_u=x_u, X_t=x}
    = -\lprev_{\br{\lambda,\lambda}}\pr{-\indica{y}\pr{X_{t+\Delta}}\cond X_u=x_u, X_t=x} \\
    &= -\br{T_{t, \lambda}^{t+\Delta}\pr{-\indica{y}}}\pr{x}
    = \br{T_{t, \lambda}^{t+\Delta}\indica{y}}\pr{x},
  \end{align*}
  where for the first equality we have used conjugacy and for the final equality we have used \ref{def:LinTT:Homogeneity}.
  Therefore,
  \begin{equation*}
    \prob\pr{X_{t+\Delta}=y\cond X_u=x_u, X_t=x}
    = \br{T_{t, \lambda}^{t+\Delta}\indica{y}}\pr{x}
    = \begin{cases}
      \pois_{\lambda\Delta}\pr{y-x} &\text{if } y\geq x, \\
      0 &\text{otherwise,}
    \end{cases}
  \end{equation*}
  where the last equality follows from Proposition~\ref{prop:T_lambda indic gives Poisson distribution}.
  This implies the stated due to Theorem~\ref{the:Pois transition probabilities are poisson distributed and the converse}.
\end{proofof}

\section{Supplementary Material for Section~\ref{sec:The Poisson semi-group}}
Most of the results in this section are specialisations of results in Appendices~\ref{app:Preliminary results regarding transformations}, \ref{app:ltro to lto} and \ref{app:The Generalised Poisson Generator}.

\subsection{The Poisson Generator}
All of the properties mentioned in the main text concerning the Poisson generator and the semi-group it induces, follow from results in Appendix~\ref{app:The Generalised Poisson Generator}.
Indeed, as we have previously mentioned---see Equation~\eqref{eqn:PoisGen as special case of GenPoisGen}---the Poisson generator~\(\ltro\) as defined in Section~\ref{ssec:The Poisson generator} is precisely the generalised Poisson generator~\(\ltro_S\) corresponding to the sequence~\(S=\set{\pr{\llambda,\ulambda}}_{x\in\stsp}\).
Hence, we obtain Theorem~\ref{the:Transformation induced by PoisGen in text} as a corollary of Theorems~\ref{the:GenPoisGen:Phi_u_i converges to a LCT} and \ref{the:LCT induced by GenPoisGen}.
\begin{proofof}{Theorem~\ref{the:Transformation induced by PoisGen in text}}
  Recall from Theorem~\ref{the:GenPoisGen:Phi_u_i converges to a LCT} that the corresponding sequence \(\set{\Phi_{u_i}}_{i\in\nats}\) converges to a lower counting transformation, which is a special type of non-negatively homogeneous transformation.
  Furthermore, it follows from Theorem~\ref{the:LCT induced by GenPoisGen} that this limit does not depend on the chosen sequence \(\set{u_i}_{i\in\nats}\).
\end{proofof}
The properties \ref{LTO in text:Non-negateive homogeneity}--\ref{LTO in text:Bound} of the family of transformations of the form~\(\lto_t^s\) just state that \(\lto_t^s\) is a lower counting transformation, as is established in Theorem~\ref{the:LCT induced by GenPoisGen}.
Furthermore, the properties \ref{LTO in text:Identity}--\ref{LTO in text:Time-invariance} are stated in Proposition~\ref{prop:Properties of lto_S induced by GenPoisGen}.

In the remainder, more specifically in the proof of Proposition~\ref{prop:Lower bound for conditional expectation is reached by a consistent process} further on, we will need the following intermediary result.
\begin{lemma}
\label{lem:lto:Approximate with prod of T_S}
  Fix some \(t, s\) in \(\nnegreals\) with \(t\leq s\) and \(f\) in \(\setoffn\pr{\stsp}\).
  Then for any \(\epsilon\) in \(\posreals\), there is a sequence \(u=t_0,\dots, t_n\) in \(\setoftseq_{\br{t,s}}\) and, for all \(i\) in \(\set{1, \dots, n}\), a sequence \(S_i=\set{\lambda_x^i}_{x\in\stsp}\) in \(\br{\llambda, \ulambda}\) such that
  \begin{equation*}
    \norm*{\lto_t^s f - \prod_{i=1}^n T_{t_{i-1},S_i}^{t_i} f}
    \leq \epsilon.
  \end{equation*}.
\end{lemma}
In our proof of Lemma~\ref{lem:lto:Approximate with prod of T_S}, we need the following obvious observation.
\begin{lemma}
\label{lem:ltro:reached by trm_S}
  For any \(f\) in \(\setoffn\pr{\stsp}\), there is a sequence \(S=\set{\lambda_x}_{x\in\stsp}\) in \(\Lambda=\br{\llambda, \ulambda}\) such that \(\ltro f=\trm_S f\).
\end{lemma}
\begin{proof}
  This is immediate from the definition of \(\ltro\) and \(\trm_S\).
\end{proof}
\begin{proofof}{Lemma~\ref{lem:lto:Approximate with prod of T_S}}
  Fix any \(\epsilon'\) in \(\posreals\) such that \(2\epsilon'\norm{f}\leq\epsilon/2\) and \(\delta\) in \(\posreals\) such that \(\delta\pr{s-t}\norm{\ltro}^2\norm{f}\leq\epsilon/2\).
  By Theorem~\ref{the:LCT induced by GenPoisGen}, there is a sequence \(u=t_0, \dots, t_n\) in \(\setoftseq_{\br{t, s}}\) such that \(\sigma\pr{u}\leq\delta\), \(\sigma\pr{u}\norm{\ltro}\leq2\) and
  \begin{equation}
  \label{eqn:Proof of T_S approx:First bound on norm}
    \norm{\lto_t^s - \Phi_u}
    \leq \epsilon'.
  \end{equation}
  Let \(g_i\coloneqq\prod_{j=i+1}^n\pr{I+\Delta_j\ltro}f\) for all \(i\) in \(\set{1, \dots, n}\), where \(g_n=If=f\).
  It now follows from Lemma~\ref{lem:ltro:reached by trm_S} that for any \(i\) in \(\set{1,\dots, n}\), there is a sequence \(S_i=\set{\lambda^i_x}_{x\in\stsp}\) in \(\br{\llambda,\ulambda}\) such that \(\ltro g_i=\trm_{S_i} g_i\).
  By construction,
  \begin{equation}
  \label{eqn:Proof of T_S approx:Phi is prod}
    \Phi_u f
    = \prod_{i=1}^n \pr{I+\Delta_i\ltro}f
    = \prod_{i=1}^n \pr{I+\Delta_i\trm_{S_i}}f.
  \end{equation}
  Furthermore, we use Lemma~\ref{lem:GenPoisGen:Norm}, to yield
  \begin{equation}
  \label{eqn:Proof of T_S approx:norm Q}
    \pr{\forall i\in\set{1,\dots,n}}~
    \norm{\trm_{S_i}}
    = 2\sup\set{\lambda^i_x\colon x\in\stsp}
    \leq 2\ulambda
    = \norm{\ltro}.
  \end{equation}

  Observe now that
  \begin{align*}
    \norm*{\lto_t^s f - \prod_{i=1}^n T_{t_{i-1}, S_i}^{t_i} f}
    &= \norm*{\lto_t^s f -\prod_{i=1}^n \pr{I+\Delta_i\trm_{S_i}}f +\prod_{i=1}^n \pr{I+\Delta_i\trm_{S_i}}f - \prod_{i=1}^n T_{t_{i-1}, S_i}^{t_i} f} \\
    &\leq \norm*{\lto_t^s f -\prod_{i=1}^n \pr{I+\Delta_i\trm_{S_i}}f} +\norm*{\prod_{i=1}^n \pr{I+\Delta_i\trm_{S_i}}f - \prod_{i=1}^n T_{t_{i-1}, S_i}^{t_i} f} \\
    &=\norm*{\lto_t^s f -\Phi_u f} +\norm*{\prod_{i=1}^n \pr{I+\Delta_i\trm_{S_i}}f - \prod_{i=1}^n T_{t_{i-1}, S_i}^{t_i} f} \\
    &\leq \norm*{\lto_t^s -\Phi_u}\norm{f} +\norm*{\prod_{i=1}^n \pr{I+\Delta_i\trm_{S_i}} - \prod_{i=1}^n T_{t_{i-1}, S_i}^{t_i}}\norm{f},
  \end{align*}
  where for the final equality we have used Equation~\eqref{eqn:Proof of T_S approx:Phi is prod} and for the final inequality we have used \ref{BNH: norm Af <= norm A norm f}.
  We use Equation~\eqref{eqn:Proof of T_S approx:First bound on norm}, to yield
  \begin{equation*}
    \norm*{\lto_t^s f - \prod_{i=1}^n T_{t_{i-1}, S_i}^{t_i} f}
    \leq \epsilon'\norm{f} +\norm*{\prod_{i=1}^n \pr{I+\Delta_i\trm_{S_i}} - \prod_{i=1}^n T_{t_{i-1}, S_i}^{t_i}}\norm{f}.
  \end{equation*}
  Next, we use Lemma~\ref{lem: bound on norm of A_1 ... A_k - B_1 ... B_k} to rewrite the second term on the right hand side of the inequality, to yield
  \begin{equation*}
    \norm*{\lto_t^s f - \prod_{i=1}^n T_{t_{i-1}, S_i}^{t_i} f}
    \leq \epsilon'\norm{f} +\sum_{i=1}^n\norm*{\pr{I+\Delta_i\trm_{S_i}} - T_{t_{i-1}, S_i}^{t_i}}\norm{f}.
  \end{equation*}
  Finally, we use Lemma~\ref{lem:Additional properties of lto_S induced by GenPoisGen}~\ref{lem:Additional properties of lto_S by GenPoisGen:Diff with I+Delta ltro} and Equation~\eqref{eqn:Proof of T_S approx:norm Q}, to yield
  \begin{align*}
    \norm*{\lto_t^s f - \prod_{i=1}^n T_{t_{i-1}, S_i}^{t_i} f}
    &\leq \epsilon'\norm{f} +\sum_{i=1}^n\Delta_i^2\norm{\trm_{S_i}}^2\norm{f}
    \leq \epsilon'\norm{f} +\sum_{i=1}^n\Delta_i^2\norm{\ltro}^2\norm{f}
    \leq \epsilon'\norm{f} +\sum_{i=1}^n\Delta_i\delta\norm{\ltro}^2\norm{f} \\
    &= \epsilon'\norm{f} +\pr{s-t}\delta\norm{\ltro}^2\norm{f}
    \leq \frac{\epsilon}2+\frac{\epsilon}2
    = \epsilon,
  \end{align*}
  as required.
\end{proofof}

\subsection{The Reduced Poisson Generator}
The claims in Section~\ref{ssec:The reduced Poisson generator} of the main text are essentially a consequence of the following simple result.
\begin{lemma}
\label{lem:ltro^chi is a lower transition rate operator}
  Consider some \(\lowx, \upx\) in \(\stsp\) with \(\lowx \leq \upx\).
  If we let \(\chi\coloneqq\set{x\in\stsp\colon\lowx\leq x\leq\upx}\), then \(\ltro^\chi\)---as defined in Section~\ref{ssec:The reduced Poisson generator}---is a lower transition rate transformation.
\end{lemma}
\begin{proof}
  The four conditions of Definition~\ref{def:Lower transition rate transformation} are trivially satisfied.
\end{proof}
That the limit in Equation~\eqref{eqn:PoisGenChi:Lto} exists and is independent of the chosen sequence~\(\set{u_i}_{i\in\nats}\) now immediately follows from Lemma~\ref{lem:ltro^chi is a lower transition rate operator} and Propositions~\ref{prop:Gen ltro^chi:Sequence converges to ltt} and \ref{prop:Gen ltro^chi:Induced lto^chi}.
Furthermore, these three results also imply that the transformations of the form~\(\lto^\chi_{t,s}\) are lower transition transformations, and hence satisfy \ref{def:LTT:Non-negative homgeneity}--\ref{def:LTT:Bound}---or equivalently, \ref{LTO in text:Non-negateive homogeneity}--\ref{LTO in text:Bound}.
Furthermore, we additionally use Proposition~\ref{prop:Gen ltro^chi:semi-group properties} to yield that the family also satisfies---properties similar to---\ref{LTO in text:Identity}--\ref{LTO in text:Time-invariance}.
Finally, we end this section on the reduced Poisson generator~\(\ltro^\chi\) with some technical results.
\begin{lemma}
\label{lem:ltro^chi:Set of dominated transition rate matrices}
  Consider some \(\lowx, \upx\) in \(\stsp\) with \(\lowx \leq \upx\), and let \(\chi\coloneqq\set{x\in\stsp\colon\lowx\leq x \leq\upx}\).
  Then the set of dominating transition rate matrices
  \begin{equation*}
    \setoftrm^\chi
    \coloneqq \set{\trm^\chi\in\mathscr{R}\pr{\chi}\colon\pr{\forall g\in\setoffn\pr{\chi}}~\ltro^\chi g\leq\trm^\chi g},
  \end{equation*}
  where \(\mathscr{R}\pr{\chi}\) denotes the set of all transition rate matrices---that is, linear lower transition rate transformations---on \(\setoffn\pr{\chi}\), is non-empty, bounded, closed and convex.
  Furthermore, \(\trm^\chi\) is an element of \(\setoftrm^\chi\) if and only if there is a sequence \(\set{\lambda_x}_{x\in\chi\setminus\set{\upx}}\) in \(\br{\llambda, \ulambda}\) such that
  \begin{equation*}
    \br{\trm^\chi g}\pr{x}
    = \begin{dcases}
      \lambda_x\pr{g\pr{x+1}-g\pr{x}} &\text{if } \lowx\leq x<\upx \\
      0 &\text{if } x=\upx
    \end{dcases}
    \quad\text{for all \(g\in\setoffn\pr{\chi}\) and \(x\in\chi\).}
  \end{equation*}
\end{lemma}
\begin{proof}
  The first part of the stated follows immediately from \cite[Proposition~7.8]{2017Krak}.
  The second part is a matter of straightforward verification.
\end{proof}
\begin{corollary}
\label{cor:norm trm^chi <= norm ltro^chi}
  Consider some \(\lowx, \upx\) in \(\stsp\) with \(\lowx \leq \upx\), and let \(\chi\coloneqq\set{x\in\stsp\colon\lowx\leq x \leq\upx}\).
  Then for any \(\trm^\chi\) in \(\setoftrm^\chi\), with \(\setoftrm^\chi\) as defined in Lemma~\ref{lem:ltro^chi:Set of dominated transition rate matrices},
  \begin{equation*}
    \norm{\trm^\chi}
    \leq \norm{\ltro^\chi}.
  \end{equation*}
\end{corollary}
\begin{proof}
  This is an immediate corollary of Lemmas~\ref{lem:Gen ltro^chi:Norm} and \ref{lem:ltro^chi:Set of dominated transition rate matrices}:
  \begin{equation*}
    \norm{\trm^\chi}
    = 2\max\set{\lambda_x\colon x\in \chi, x<\upx}\cup\set{0}
    \leq 2\ulambda
    = \norm{\ltro^\chi}.
  \end{equation*}
\end{proof}
\begin{corollary}
\label{cor:I+delta ltro^chi f^chi <= I+delta trm^chi f^chi}
  Consider some \(\lowx, \upx\) in \(\stsp\) with \(\lowx \leq \upx\), and let \(\chi\coloneqq\set{x\in\stsp\colon\lowx\leq x \leq\upx}\).
  Fix some \(n\) in \(\nats\) and, for every \(i\) in \(\set{1, \dots, n}\), some \(\Delta_i\) in \(\nnegreals\) with \(\Delta_i\norm{\ltro^\chi}\leq 2\) and some \(\trm^\chi_i\) in \(\setoftrm^\chi\).
  Then for all \(f^\chi\) in \(\setoffn\pr{\chi}\),
  \begin{equation*}
    \prod_{i=1}^n \pr{I+\Delta_i\ltro^\chi} f^\chi
    \leq \prod_{i=1}^n \pr{I+\Delta_i\trm^\chi_i} f^\chi.
  \end{equation*}
\end{corollary}
\begin{proof}
  This is an immediate corollary of Lemma~\ref{lem:ltro^chi:Set of dominated transition rate matrices} and \cite[Lemma~F.4]{2017Krak}.

\end{proof}

\subsection{The Essential Case of Eventually Constant Functions}
Before we can prove our two results concerning eventually constant functions, we first need to establish two technical lemmas.
\begin{lemma}
\label{lem:eventually contant function:I+Delta ltro f}
  For any \(\Delta\) in \(\nnegreals\) and any \(f\) in \(\setofcafn\pr{\stsp}\) that is constant after \(\upx\), \(\pr{I+\Delta\ltro} f\) is constant after \(\upx\).
  Furthermore, if we fix any \(\lowx\) in \(\stsp\) with \(\lowx\leq\upx\) and let \(\chi\coloneqq\set{y\in\stsp\colon \lowx\leq x\leq\upx}\), then
  \begin{equation*}
    \br{\pr{I+\Delta\ltro} f}\pr{x}
    = \begin{cases}
      \br{\pr{I+\Delta\ltro^\chi} f^\chi}\pr{x} &\text{if } x \leq\upx \\
      f\pr{\upx} &\text{if } x\geq\upx
    \end{cases}
    \qquad\text{for all } x \in \stsp \text{ with } x\geq\lowx,
  \end{equation*}
  where \(f^\chi\) is the restriction of \(f\) to \(\chi\).
\end{lemma}
\begin{proof}
  That \(\pr{I+\Delta\ltro}f\) is constant after \(\upx\) is obvious: for any \(y\) in \(\stsp\) with \(y\geq \upx\),
  \begin{equation*}
    \br{\pr{I+\Delta\ltro} f}\pr{y}
    = f\pr{y} + \Delta \br{\ltro f}\pr{y}
    = f\pr{y} +\Delta \min\set{\lambda f\pr{y+1}-\lambda f\pr{y}\colon\lambda\in\br{\llambda, \ulambda}}
    = f\pr{\upx} + \Delta 0
    = f\pr{\upx},
  \end{equation*}
  where for the penultimate equality we use that \(f\) is constant after \(\upx\).

  To verify the second part of the statement, we fix any \(x\) in \(\stsp\) with \(x\geq\lowx\).
  Observe first that if \(x\geq \upx\), then
  \begin{equation*}
    \br{\pr{I+\Delta\ltro}f}\pr{x}
    = f\pr{\upx}
    = f^\chi\pr{\upx}
    = f^\chi\pr{\upx} + \Delta\br{\ltro^\chi f^\chi}\pr{\upx}
    = \br{\pr{I+\Delta\ltro^\chi}f^\chi}\pr{\upx},
  \end{equation*}
  where the first equality follows from the first part of our proof.
  If \(x<\upx\), then
  \begin{align*}
    \br{\pr{I+\Delta\ltro} f}\pr{x}
    &= f\pr{x} + \Delta \br{\ltro f}\pr{x}
    = f\pr{x} +\Delta \min\set{\lambda f\pr{x+1}-\lambda f\pr{x}\colon\lambda\in\br{\llambda, \ulambda}} \\
    &= f^\chi\pr{x} +\Delta \min\set{\lambda f^\chi\pr{x+1}-\lambda f^\chi\pr{x}\colon\lambda\in\br{\llambda, \ulambda}}
    =\br{\pr{I+\Delta\ltro^\chi}f^\chi}\pr{x}.
  \end{align*}
\end{proof}
\begin{lemma}
\label{lem:eventually contant function:I+Delta_i ltro f}
  Fix some \(n\) in \(\nats\) and, for all \(i\) in \(\set{1, \dots, n}\), a \(\Delta_i\) in \(\nnegreals\).
  Then for any \(f\) in \(\setofcafn{}\) that is constant after \(\upx\), \(\prod_{i=1}^n\pr{I+\Delta_i\ltro} f\) is constant after \(\upx\).
  Furthermore, if we fix any \(\lowx\) in \(\stsp\) with \(\lowx\leq\upx\) and let \(\chi\coloneqq\set{y\in\stsp\colon \lowx\leq x\leq\upx}\), then
  \begin{equation*}
    \br*{\prod_{i=1}^n\pr{I+\Delta_i\ltro} f}\pr{x}
    = \begin{dcases}
      \br*{\prod_{i=1}^n\pr{I+\Delta_i\ltro^\chi} f^\chi}\pr{x} &\text{if } x\leq\upx \\
      f\pr{\upx} &\text{if } x\geq\upx
    \end{dcases}
    \qquad\text{for all } x \in \stsp \text{ with } x\geq\lowx,
  \end{equation*}
  where \(f^\chi\) is the restriction of \(f\) to \(\chi\).
\end{lemma}
\begin{proof}
  We provide a proof by induction.
  First, observe that for \(n=1\), the stated follows immediately from Lemma~\ref{lem:eventually contant function:I+Delta ltro f}.

  Second, fix some \(n\) in \(\nats\) with \(n\geq2\) and assume that the stated holds for \(1\leq n'< n\).
  We show that in this case, the stated then follows for \(n\) as well.
  Let \(g\coloneqq\prod_{i=2}^n\pr{I+\Delta_i\ltro}f\).
  Then by the induction hypothesis for \(n'=n-1\), \(g\) is constant after \(\upx\), and
  \begin{equation*}
    g\pr{x}
    = \br*{\prod_{i=2}^n\pr{I+\Delta_i\ltro}f}\pr{x}
    = \begin{dcases}
      \br*{\prod_{i=2}^n\pr{I+\Delta_i\ltro^\chi}f^\chi}\pr{x} &\text{if } x\leq\upx, \\
      f\pr{\upx} &\text{if } x\geq\upx
    \end{dcases}
    \qquad\text{for all } x \in \stsp \text{ with } x\geq\lowx.
  \end{equation*}
  So clearly, \(g^\chi\), the restriction of \(g\) to \(\chi\), is equal to \(\prod_{i=2}^n\pr{I+\Delta_i\ltro^\chi}f^\chi\).
  Observe now that
  \begin{equation*}
    g'
    \coloneqq \prod_{i=1}^n \pr{I+\Delta_i\ltro} f
    = \pr{I+\Delta_1\ltro} g.
  \end{equation*}
  From the induction hypothesis for \(n'=1\), we know that \(g'\) is constant after \(\upx\) and that, for all \(x\) in \(\stsp\) with \(x\geq\lowx\),
  \begin{equation*}
    g'\pr{x}
    = \begin{dcases}
      \br{\pr{I+\Delta_1\ltro^\chi} g^\chi}\pr{x} &\text{if } x\leq\upx, \\
      g\pr{\upx} &\text{if } x\geq\upx
    \end{dcases}
    = \begin{dcases}
      \br*{\prod_{i=1}^n\pr{I+\Delta_i\ltro^\chi}f^\chi}\pr{x} &\text{if } x\leq\upx, \\
      f\pr{\upx} &\text{if } x\geq\upx,
    \end{dcases}
  \end{equation*}
  where the second equality holds because \(g^\chi=\prod_{i=2}^n\pr{I+\Delta_i\ltro^\chi}f^\chi\) and \(g\pr{\upx}=f\pr{\upx}\).
\end{proof}

\begin{proofof}{Proposition~\ref{prop:eventually bounded functions:lto f equals lto^chi f}}
  Fix any \(x\) in \(\stsp\) with \(x\geq\lowx\).
  First, observe that if \(x\geq\upx\), then
  \begin{equation*}
    \br{\lto_t^s f}\pr{x}
    = \br{\lto_t^s f'_x}\pr{0}
    = \br{\lto_t^s \pr{f\pr{\upx}}}\pr{0}
    = f\pr{\upx},
  \end{equation*}
  where we let \(f'_x\colon\stsp\to\reals\colon z\mapsto f\pr{x+z}\), and where the first equality follows from Lemma~\ref{lem:State homogeneity of lto induced by genpoisgen}, the second equality holds because clearly \(f'_x=f\pr{\upx}\) and the third equality follows from \ref{LTT:constant}.
  The obtained equality clearly agrees with the stated.

  Next, we consider the case that \(x\leq\upx\).
  Fix any \(\epsilon\) in \(\posreals\), and choose any \(\epsilon'\) in \(\posreals\) such that \(2\norm{f}\epsilon'\leq\epsilon\).
  Then by Theorem~\ref{the:LCT induced by GenPoisGen}---in combination with Equation~\eqref{eqn:PoisGen as special case of GenPoisGen}---and Proposition~\ref{prop:Gen ltro^chi:Induced lto^chi}---in combination with Lemma~\ref{lem:ltro^chi is a lower transition rate operator}---there is a sequence \(u\) in \(\setoftseq_{\br{t,s}}\) such that
  \begin{equation}
  \label{eqn:proof of lto f=lto^chi f^chi:epsilon bound}
    \norm{\lto_t^s - \Phi_u}
    \leq \epsilon'
    \qquad\text{and}\qquad
    \norm{\lto^\chi_{t,s} - \Phi^\chi_u}
    \leq \epsilon'.
  \end{equation}
  Observe now that
  \begin{align*}
    \abs*{\br{\lto_t^s f}\pr{x} - \br{\lto^\chi_{t,s} f^\chi}\pr{x}}
    &= \abs*{\br{\lto_t^s f}\pr{x} - \br{\Phi_u f}\pr{x} + \br{\Phi_u f}\pr{x} - \br{\lto^\chi_{t,s} f^\chi}\pr{x}} \\
    &\leq \abs*{\br{\lto_t^s f}\pr{x} - \br{\Phi_u f}\pr{x}} + \abs*{\br{\Phi_u f}\pr{x}-\br{\lto^\chi_{t,s} f^\chi}\pr{x}} \\
    &= \abs*{\br{\lto_t^s f}\pr{x} - \br{\Phi_u f}\pr{x}} + \abs*{\br{\Phi^\chi_u f^\chi}\pr{x}-\br{\lto^\chi_{t,s} f^\chi}\pr{x}} \\
    &\leq \norm{\lto_t^s f - \Phi_u f} + \norm{\Phi^\chi_u f^\chi-\lto^\chi_{t,s} f^\chi}
    \leq \norm{\lto_t^s - \Phi_u}\norm{f} + \norm{\Phi^\chi_u-\lto^\chi_{t,s}}\norm{f^\chi} \\
    &\leq \epsilon'\norm{f}+\epsilon'\norm{f^\chi}
    \leq 2\epsilon'\norm{f}
    \leq \epsilon,
  \end{align*}
  where the second equality follows from Lemma~\ref{lem:eventually contant function:I+Delta_i ltro f}, the third inequality follows from \ref{BNH: norm Af <= norm A norm f}, the fourth inequality follows from Equation~\eqref{eqn:proof of lto f=lto^chi f^chi:epsilon bound}, the penultimate inequality holds because clearly \(\norm{f^\chi}\leq\norm{f}\) and the final inequality is precisely our condition on \(\epsilon'\).
  Since \(\epsilon\) was any arbitrary positive real number, we conclude from this inequality that \(\br{\lto_t^s f}\pr{x} = \br{\lto^\chi_{t,s} f^\chi}\pr{x}\), as required.
\end{proofof}

\begin{proofof}{Proposition~\ref{prop:Approximate lto f using lto f indic leq upx}}
  Fix any \(\epsilon\) in \(\posreals\).
  To prove the stated, we need to verify that
  \begin{equation*}
    \pr{\exists x^\star\in\stsp}\pr{\forall \upx\in\stsp, \upx\geq x^\star}~
    \abs*{\lpoisprev_t^s\pr{f\cond x}-\lpoisprev_t^s\pr{\indica{\leq \upx} f+\indica{> \upx} f\pr{\upx}\cond x}}
    = \abs*{\br{\lto_t^s f}\pr{x}-\br{\lto_t^s\pr{\indica{\leq \upx} f+f\pr{\upx}\indica{> \upx}}}\pr{x}}
    \leq \epsilon.
  \end{equation*}
  To that end, we fix any \(\epsilon'\) in \(\posreals\) with \(2\epsilon'\norm{f}\leq\epsilon\) and we recall from Theorem~\ref{the:LCT induced by GenPoisGen} that there is a sequence \(u=t_0, \dots, t_n\) in \(\setoftseq_{\br{t,s}}\) such that \(\sigma\pr{u}\norm{\ltro}\leq2\) and
  \begin{equation*}
    \norm{\lto_t^s-\Phi_u}
    \leq \epsilon'.
  \end{equation*}
  Let \(x^\star\coloneqq x+n\), fix any \(\upx\) in \(\stsp\) such that \(\upx\geq x^\star\) and let \(f_{\upx}\coloneqq \indica{\leq \upx} f+f\pr{\upx}\indica{> \upx}\).
  Then
  \begin{align*}
    \abs*{\br{\lto_t^s f}\pr{x}-\br{\lto_t^s\pr{\indica{\leq \upx} f+f\pr{\upx}\indica{> \upx}}}\pr{x}}
    &= \abs*{\br{\lto_t^s f}\pr{x}-\br{\lto_t^s f_{\upx}}\pr{x}}
    = \abs*{\br{\lto_t^s f}\pr{x}-\br{\Phi_u f}\pr{x}+\br{\Phi_u f}\pr{x}-\br{\lto_t^s f_{\upx}}\pr{x}} \\
    &\leq \abs*{\br{\lto_t^s f}\pr{x}-\br{\Phi_u f}\pr{x}}+\abs*{\br{\lto_t^s f_{\upx}}\pr{x}-\br{\Phi_u f}\pr{x}} \\
    &= \abs*{\br{\lto_t^s f}\pr{x}-\br{\Phi_u f}\pr{x}}+\abs*{\br{\lto_t^s f_{\upx}}\pr{x}-\br{\Phi_u f_{\upx}}\pr{x}} \\
    &\leq \norm*{\lto_t^s f-\Phi_u f}+\norm*{\lto_t^s f_{\upx}-\Phi_u f_{\upx}} \\
    &\leq \norm*{\lto_t^s-\Phi_u}\norm{f}+\norm*{\lto_t^s-\Phi_u}\norm{\indica{\leq \upx} f+f\pr{\upx}\indica{> \upx}} \\
    &\leq \epsilon'\norm{f}+\epsilon'\norm{\indica{\leq \upx} f+f\pr{\upx}\indica{> \upx}}
    \leq 2\epsilon'\norm{f}
    \leq \epsilon,
  \end{align*}
  as required.
  In this expression, the third equality follows from Lemma~\ref{lem:Phi_u:Value of f does not matter above treshold}, the third inequality follows from \ref{BNH: norm Af <= norm A norm f} and the fifth inequality follows from the obvious inequality \(\norm{\indica{\leq \upx} f+f\pr{\upx}\indica{> \upx}}\leq\norm{f}\).
\end{proofof}

\section{Supplementary Material for Section~\ref{sec:Computing lower expectations}}
In this section of the Appendix, we will focus on expectations of the form \(\prev_{\prob}\pr{f\pr{X_s}\cond X_u=x_u, X_t=x}\).
Therefore, we first establish that \(\pr{f\pr{X_s}, X_u=x_u, X_t=x}\) belongs to the domain \(\prevdomain{}\) of \(\prev_\prob\).
By definition of \(\prevdomain\), this is true if \(f\pr{X_s}\) is bounded below and \(\cpfield_{u\cup t}\) measurable---that is, belongs to \(\prevdomain_{u\cup t}\).
\begin{proofof}{Lemma~\ref{lem:f of X_t is a F_u-measurable function}}
  Fix any \(\alpha\) in \([\inf f, +\infty)\).
  Then
  \begin{equation*}
    \set{f\pr{X_s}>\alpha}
    = \set{\pth\in\setofpths{}\colon f\pr{\pth\pr{s}}>\alpha}
    = \pr{X_s\in B_\alpha},
  \end{equation*}
  where \(B_\alpha\coloneqq\set{y\in\stsp\colon f\pr{y}>\alpha}\).
  As \(B_\alpha\subseteq\stsp\), it follows from Equation~\eqref{eqn:cpfield_u}, the fundamental event~\(\pr{X_s\in B_\alpha}\) is an element of the field~\(\cpfield_u\).
  Because \(\alpha\) was an arbitrary real number in \([\inf f, +\infty)\), we infer from this that \(f\pr{X_s}\) is \(\cpfield_{u\cup t}\) measurable.
\end{proofof}

\subsection{With Respect to the Consistent Poisson Processes}
Section~\ref{ssec:Lower and upper with respect to consistent Poisson} of the main text contains only one (implicit) result: Equation~\eqref{eqn:prev_lambda f(X_s) in text}.
We will here formally establish this---not exactly immediate---consequence of Theorem~\ref{the:Pois transition probabilities are poisson distributed and the converse} in Proposition~\ref{prop:expectation of poisson process as sum}.
First, however, we state a helpful---and essentially well-known---technical lemma.
\begin{lemma}
\label{lem:Mu from trans prob is sigma additive}
  Consider the Poisson process~\(\prob\) with rate~\(\lambda\) in \(\nnegreals\).
  For any \(t, s\) in \(\nnegreals\) with \(t\leq s\), \(u\) in \(\setoftseq_{<t}\), \(\pr{x_u, x}\) in \(\stsp_{u\cup t}\), the measure
  \begin{equation}
    \mu\colon\ccpevents\pr{\stsp}\to\reals\colon
    A \mapsto\mu\pr{A}
    \coloneqq\prob\pr{X_s\in A\cond X_u=x_u, X_t=x}.
  \end{equation}
  is the \(\sigma\)-additive probability measure corresponding to the probability mass function
  \begin{equation*}
    \pi\colon\stsp\to\reals\colon
    y \mapsto \pi\pr{y}
    \coloneqq \begin{cases}
      \pois_{\lambda\pr{s-t}}\pr{y-x} &\text{if } y\geq x, \\
      0 &\text{otherwise.}
    \end{cases}
  \end{equation*}
\end{lemma}
\begin{proof}
  That \(\mu\) is a probability measure on \(\ccpevents\pr{\stsp}\) follows immediately from the fact that \(\prob\) is a coherent conditional probability.
  Therefore, we only need to verify that \(\mu\) is \(\sigma\)-additive and corresponds to the probability mass function~\(\pi\).
  To that end, we observe that, for any \(y\) in \(\stsp\),
  \begin{equation}
  \label{eqn:Proof of Mu from trans prob is sigma additive:Mu on atoms}
    \mu\pr{\set{y}}
    = \prob\pr{X_s=y\cond X_u=x_u, X_t=x}
    = \begin{cases}
      \pois_{\lambda\pr{s-t}}\pr{y-x} &\text{if } y\geq x \\
      0 &\text{otherwise}
    \end{cases}
    = \pi\pr{y},
  \end{equation}
  where the second equality follows from Theorem~\ref{the:Pois transition probabilities are poisson distributed and the converse}.
  Therefore, the stated holds if we can prove that, for any subset \(A\) of \(\stsp\),
  \begin{equation}
    \mu\pr{A}
    = \sum_{y\in A} \pi\pr{y}.
  \end{equation}

  To verify this, we fix any subset \(A\) of \(\stsp\), and distinguish two cases: \(A\) is finite and \(A\) is infinite.
  In the first case, it follows from the finite additivity of \(\prob\) and Equation~\eqref{eqn:Proof of Mu from trans prob is sigma additive:Mu on atoms} that
  \begin{equation}
  \label{eqn:Proof of Mu from trans prob is sigma additive:Finite A}
    \mu\pr{A}
    = \prob\pr{X_s\in A\cond X_u=x_u, X_t=x}
    = \sum_{y\in A} \prob\pr{X_s= A\cond X_u=x_u, X_t=x}
    = \sum_{y\in A}\mu\pr{\set{y}}
    = \sum_{y\in A}\pi\pr{y},
  \end{equation}
  as required.

  The case that \(A\) is infinite is slightly more involved.
  Observe first that, for any \(\overline{z}\) in \(\stsp\) with \(\overline{z}\geq x\),
  \begin{multline*}
    \prob\pr{X_s\geq\overline{z}+1\cond X_u=x_u, X_t=x}
    = 1-\prob\pr{X_s\leq\overline{z}\cond X_u=x_u, X_t=x} \\
    = 1-\sum_{y=x}^{\overline{z}}\prob\pr{X_s=y \cond X_u=x_u, X_t=x}
    = 1-\sum_{y=x}^{\overline{z}} \pois_{\lambda\pr{s-t}}\pr{y-x},
  \end{multline*}
  where for the penultimate equality we have used Lemma~\ref{lem:Coherent prob on cpdomain:prob of <= x} and for the last equality we have used Theorem~\ref{the:Pois transition probabilities are poisson distributed and the converse}.
  Fix any \(\epsilon\) in \(\posreals\).
  As the Poisson distribution~\(\pois_{\lambda\pr{s-t}}\) is normed, there is some \(\overline{z}\) in \(\stsp\) such that \(\overline{z}\geq x\) and
  \begin{equation*}
    1-\epsilon
    \leq \sum_{y=x}^{\overline{z}} \pois_{\lambda\pr{s-t}}\pr{y-x}
    \leq 1.
  \end{equation*}
  Fix any such \(\overline{z}\) such that there is at least one \(y\) in \(A\) with \(y\leq \overline{z}\).
  Then from the left inequality and the previous equality, it now follows that
  \begin{equation}
  \label{eqn:Proof of Mu from trans prob is sigma additive:eps bound}
    \prob\pr{X_s\geq\overline{z}+1\cond X_u=x_u, X_t=x}
    \leq \epsilon.
  \end{equation}
  Let \(A'\coloneqq\set{y\in A\colon y\leq \overline{z}}\).
  As \(A'\subseteq A\subseteq A' \cup \set{y\in \stsp\colon y\geq \overline{z}+1}\) and \(A' \cap \set{y\in \stsp\colon y\geq \overline{z}+1}=\emptyset\), it follows from the monotonicity and finite additivity of \(\prob\) that
  \begin{equation*}
    \prob\pr{X_s\in A'\cond X_u=x_u, X_t=x}
    \leq \prob\pr{X_s\in A\cond X_u=x_u, X_t=x}
    \leq \prob\pr{X_s\in A'\cond X_u=x_u, X_t=x}+\epsilon,
  \end{equation*}
  where we have also used Equation~\eqref{eqn:Proof of Mu from trans prob is sigma additive:eps bound} for the right inequality.
  As \(A'\) is finite, it follows from these two inequalities and Equation~\eqref{eqn:Proof of Mu from trans prob is sigma additive:Finite A} that
  \begin{equation*}
    \sum_{y\in A'}\pi\pr{y}
    \leq \mu\pr{A}
    = \prob\pr{X_s\in A\cond X_u=x_u, X_t=x}
    \leq \sum_{y\in A'}\pi\pr{y}+\epsilon.
  \end{equation*}
  Because \(\epsilon\) was an arbitrary positive real number, we conclude from these inequalities that \(\mu\pr{A}=\sum_{y\in A}\pi\pr{y}\), as required.
\end{proof}
\begin{proposition}
\label{prop:expectation of poisson process as sum}
  Consider the Poisson process~\(\prob\) with rate~\(\lambda\) in \(\nnegreals\).
  Then for any \(t, s\) in \(\nnegreals\) with \(t\leq s\), \(u\) in \(\setoftseq_{<t}\), \(\pr{x_u, x}\) in \(\stsp_{u\cup t}\) and \(f\) in \(\setofbbfn\pr{\stsp}\),
  \begin{equation*}
    \prev_\prob\pr{f\pr{X_s}\cond X_u=x_u, X_t=x}
    = \sum_{y=x}^{+\infty} f\pr{y}\pois_{\lambda\pr{s-t}}\pr{y-x}.
  \end{equation*}
\end{proposition}
\begin{proof}
  Recall from Section~\ref{ssec:Conditional expectation with respect to a counting process} that
  \begin{equation*}
    \prev_\prob\pr{f\pr{X_s}\cond X_u=x_u, X_t=x}
    = \inf f + \int_{\inf f}^{\sup f} \prob\pr{\set{f\pr{X_s}>\alpha}\cond X_u=x_u, X_t=x} \,\mathrm{d}\alpha.
  \end{equation*}
  Observe now that \(\set{f\pr{X_s}>\alpha}=\pr{X_s\in A_\alpha}\) with \(A_\alpha\coloneqq\set{y\in\stsp\colon f\pr{y>\alpha}}\subseteq\stsp\).
  Consequently,
  \begin{equation*}
    \prev_\prob\pr{f\pr{X_s}\cond X_u=x_u, X_t=x}
    = \inf f + \int_{\inf f}^{\sup f} \mu\pr{\set{f>\alpha}} \,\mathrm{d}\alpha,
  \end{equation*}
  with \(\set{f>\alpha}\coloneqq\set{y\in\stsp\colon f\pr{y}>\alpha}\) and where \(\mu\) is the \(\sigma\)-additive measure defined as in Lemma~\ref{lem:Mu from trans prob is sigma additive}:
  \begin{equation}
    \mu\colon\ccpevents\pr{\stsp}\to\reals\colon
    A \mapsto\mu\pr{A}\coloneqq\prob\pr{X_s\in A\cond X_u=x_u, X_t=x}.
  \end{equation}
  Let \(f'\coloneqq f-\inf f\).
  It then follows from Levi's Monotone Convergence Theorem---with the sequence \(\set{f' \indica{\leq i}}_{i\in\nats}\)---that
  \begin{equation*}
    \inf f + \int_{\inf f}^{\sup f} \mu\pr{\set{f>\alpha}} \,\mathrm{d}\alpha
    = \inf f+\int_{0}^{\sup f'} \mu\pr{\set{f'>\alpha}} \,\mathrm{d}\alpha
    = \inf f + \sum_{y=0}^{+\infty} f'\pr{y}\mu\pr{\set{y}}
    = \sum_{y=0}^{+\infty} f\pr{y}\mu\pr{\set{y}},
  \end{equation*}
  where for the final equality we have also used that \(1=\sum_{y=0}^{+\infty}\mu\pr{\set{y}}\).
  The stated now follows from this equality and Lemma~\ref{lem:Mu from trans prob is sigma additive}.
\end{proof}

\subsection{With Respect to the Consistent Counting Processes}
\subsubsection{Some Notation And Intermediary Technical Results}
Before we can prove our main results, we introduce some additional notation and establish some useful technical results.
Consider a counting process~\(\prob\).
Fix some \(t,s\) in \(\nnegreals\) with \(t\leq s\), some \(u\) in \(\setoftseq_{<t}\), some \(x_u\) in \(\stsp_u\) and some \(\lowx, \upx\) in \(\stsp\) such that \(x_{\max u}\leq\lowx\leq\upx\), and let \(\chi\coloneqq\set{x\in\stsp\colon \lowx\leq x\leq\upx}\).
We now consider the operator \(T^\chi_{x_u, t, s}\colon\setoffn\pr{\chi}\to\setoffn\pr{\chi}\), defined for all \(f^\chi\) in \(\setoffn\pr{\chi}\) by
\begin{equation}
\label{eqn:T^chi}
  \br{T^\chi_{x_u, t, s} f^\chi}\pr{x}
  \coloneqq \sum_{y=x}^{\upx-1} f^\chi\pr{y}\prob\pr{X_s=y\cond X_u=x_u,X_t=x}
  + f^\chi\pr{\upx} \prob\pr{X_s\geq\upx\cond X_u=x_u,X_t=x}
  \quad\text{for all } x\in\chi,
\end{equation}
where for notational simplicity we let the empty sum equal zero.
\begin{lemma}
\label{lem:T^chi is a linear transition transformation}
  Consider a counting process~\(\prob\).
  Fix some \(t,s\) in \(\nnegreals\) with \(t\leq s\), some \(u\) in \(\setoftseq_{<t}\), some \(x_u\) in \(\stsp_u\) and some \(\lowx, \upx\) in \(\stsp\) with \(x_{\max u}\leq\lowx\leq\upx\).
  If we let \(\chi\coloneqq\set{x\in\stsp\colon \lowx\leq x\leq\upx}\), then \(T^\chi_{x_u, t, s}\colon\setoffn\pr{\chi}\to\setoffn\pr{\chi}\) is a linear counting transformation and
  \begin{equation*}
    \br{T^\chi_{x_u, t, s} f^\chi}\pr{\upx}
    = f^\chi\pr{\upx}
    \quad\text{for all } f^\chi \in \setoffn\pr{\chi}.
  \end{equation*}
\end{lemma}
\begin{proof}
  We first verify that \(T^\chi_{x_u, t, s}\) is a linear counting transformation.
  To that end, we just check the four conditions of Definition~\ref{def:Linear transition transformation}.
  That \(T^\chi_{x_u, t, s}\) is linear transformation---that is, \ref{def:LinTT:Homogeneity} and \ref{def:LinTT:Additivity}---follows immediately from Equation~\eqref{eqn:T^chi}.
  To verify \ref{def:LinTT:Bounded by inf}, we fix some \(f^\chi\) in \(\setoffn\pr{\chi}\) and some \(x\) in \(\chi\), and observe that
  \begin{align*}
    \br{T^\chi_{x_u, t, s} f^\chi}\pr{x}
    &= \sum_{y=x}^{\upx-1} f^\chi\pr{y}\prob\pr{X_s=y\cond X_u=x_u,X_t=x} + f^\chi\pr{\upx} \prob\pr{X_s\geq\upx\cond X_u=x_u,X_t=x} \\
    &\geq \sum_{y=x}^{\upx-1} \pr{\inf f^\chi} \prob\pr{X_s=y\cond X_u=x_u,X_t=x} + \pr{\inf f^\chi} \prob\pr{X_s\geq\upx\cond X_u=x_u,X_t=x},
  \end{align*}
  where the inequality holds because we replace each term by a lower term.
  From this, it now follows that
  \begin{align*}
    \br{T^\chi_{x_u, t, s} f^\chi}\pr{x}
    &= \pr{\inf f^\chi} \prob\pr{X_s\geq x\cond X_u=x_u,X_t=x}
    = \pr{\inf f^\chi} \prob\pr{X_u=x_u, X_t=x, X_s\geq x\cond X_u=x_u,X_t=x} \\
    &= \inf f^\chi \prob\pr{X_u=x_u, X_t=x \cond X_u=x_u,X_t=x}
    = \inf f^\chi,
  \end{align*}
  where for the first equality we have used the additivity of \(\prob\) and for the second equality we have used \ref{Paths:Non-decreasing}.
  The final condition \ref{def:LinTT:Counting} again follows immediately from Equation~\eqref{eqn:T^chi}.

  To verify the second part of the statement, we fix some \(f^\chi\) in \(\setoffn\pr{\chi}\) and observe that
  \begin{align*}
    \br{T^\chi_{x_u, t, s} f^\chi}\pr{\upx}
    &= f^\chi\pr{\upx}\prob\pr{X_s\geq \upx\cond X_u=x_u, X_t=\upx}
    = f^\chi\pr{\upx}\prob\pr{X_u=x_u, X_t=\upx, X_s\geq \upx\cond X_u=x_u, X_t=\upx} \\
    &= f^\chi\pr{\upx}\prob\pr{X_u=x_u, X_t=\upx\cond X_u=x_u, X_t=\upx}
    = f^\chi\pr{\upx},
  \end{align*}
  where for the third equality we have again used \ref{Paths:Non-decreasing}.
\end{proof}

The following result is heavily inspired by \cite[Proposition~4.7]{2017Krak}, and is essential to the proof of Lemma~\ref{lem:Heine-Borel generated sequence}.
\begin{lemma}
\label{lem:Bound on derivative of the transition matrix of a consistent process}
  Consider a counting process~\(\prob\) that is consistent with the rate interval~\(\Lambda\).
  Fix some \(t\) in \(\nnegreals\), \(u\) in \(\setoftseq_{<t}\), \(x_u\) in \(\stsp_u\) and \(\lowx,\upx\) in \(\stsp\) with \(x_{\max u}\leq\lowx\leq\upx\), and let \(\chi\coloneqq\set{x\in\stsp\colon \lowx\leq x\leq \upx}\).
  Then
  \begin{equation*}
    \pr{\forall\epsilon\in\posreals}
    \pr{\exists\delta\in\posreals}
    \pr{\forall \Delta\in\posreals, \Delta<\delta}
    \pr{\exists \trm^\chi \in \setoftrm^\chi}
    \norm*{\frac{T^\chi_{x_u, t, t+\Delta}-I}{\Delta}-\trm^\chi}
    \leq \epsilon
  \end{equation*}
  and, if \(t>0\),
  \begin{equation*}
    \pr{\forall\epsilon\in\posreals}
    \pr{\exists\delta\in\posreals}
    \pr{\forall \Delta\in\posreals, \Delta<\delta}
    \pr{\exists \trm^\chi \in \setoftrm^\chi}
    \norm*{\frac{T^\chi_{x_u, t-\Delta, t}-I}{\Delta}-\trm^\chi}
    \leq \epsilon.
  \end{equation*}
\end{lemma}
\begin{proof}
  Observe that if \(\lowx=\upx\), then \(\chi\) is the singleton containing \(\lowx=\upx\).
  In this case, \(\setoftrm^\chi\pr{\set{\lowx}}\) because \(0\) is the only linear transformation on \(\setoffn\pr{\chi}=\setoffn\pr{\set{\lowx}}\) that can satisfy all four conditions of Definition~\ref{def:Lower transition rate transformation}---and specifically \ref{ltrt:zero row-sums}.
  Observe now that for any \(\Delta\) in \(\posreals\),
  \begin{equation*}
    \norm*{\frac{T^\chi_{x_u,t,t+\Delta}-I}{\Delta}-0}
    = \frac{T^\chi_{x_u,t,t+\Delta}\pr{\lowx,\lowx}-I\pr{\lowx,\lowx}}{\Delta}-0\pr{\lowx,\lowx}
    = \frac{T^\chi_{x_u,t,t+\Delta}\pr{\lowx,\lowx}-1}{\Delta}
    = \frac{\br{T^\chi_{x_u,t,t+\Delta}\indica{\lowx}}\pr{\lowx}-1}{\Delta}
    = \frac{\indica{\lowx}\pr{\lowx}-1}{\Delta}
    = 0,
  \end{equation*}
  where for the second equality we have used Equation~\eqref{eqn:Norm of linear transformation of finite state space} and for the penultimate equality we have used Lemma~\ref{lem:T^chi is a linear transition transformation}.
  Hence, the first part of the stated is trivially verified.
  Similar reasoning yields that the first part is trivially satisfies for all \(\Delta\) in \(\posreals\) such that \(\Delta<t-\max u\).

  Next, we consider the alternative case that \(\lowx<\upx\).
  We here only prove the first inequality of the stated, the proof of the second inequality is entirely analoguous.
  Fix any arbitrary \(\epsilon\) in \(\posreals\), and choose some \(\epsilon_1\) and \(\epsilon_2\) in \(\posreals\) such that \(2\epsilon_1+\card{\chi}\epsilon_2\leq\epsilon\).
  Because \(\prob\) is consistent with \(\Lambda\), it follows from Equation~\eqref{eqn:Consistent rate from right} that for any \(x\) in
  \(\chi'\coloneqq\chi\setminus\set{\upx}\), there is a \(\delta_{1,x}\) in \(\posreals\) such that for all \(\Delta\) in \(\posreals\), there is a \(\lambda_{x,\Delta}\) in \(\Lambda\) such that
  \begin{equation}
  \label{eqn:Proof of Bound on derivative of the transition matrix of a consistent process:prob of one}
    \abs*{\frac{\prob\pr{X_{t+\Delta}=x+1\cond X_u=x_u, X_t=x}}{\Delta}-\lambda_{x,\Delta}}
    \leq \epsilon_1.
  \end{equation}
  Additionally, as \(\prob\) is a counting process it follows from \ref{CP:Orederliness} that for all \(x\) in \(\chi'\) there is a \(\delta_{2,x}\) in \(\posreals\) such that
  \begin{equation}
  \label{eqn:Proof of Bound on derivative of the transition matrix of a consistent process:prob of two or more}
    \pr{\forall\Delta\in\posreals, \Delta<\delta_{2,x}}~
    0 \leq
    \frac{\prob\pr{X_{t+_\Delta}\geq x+2\cond X_u=x_u, X_t=x}}{\Delta}
    \leq \epsilon_2.
  \end{equation}
  Observe that clearly \(\pr{X_{t+_\Delta}=y}\subseteq\pr{X_{t+_\Delta}\geq x+2}\) if \(x+2\leq y\), whence it follows from  Equation~\eqref{eqn:Proof of Bound on derivative of the transition matrix of a consistent process:prob of two or more} and the monotonicity of \(\prob\) that, for all \(x\) and \(y\) in \(\chi\) with \(x+2\leq y\),
  \begin{equation}
  \label{eqn:Proof of Bound on derivative of the transition matrix of a consistent process:prob of y}
    \pr{\forall\Delta\in\posreals, \Delta<\delta_{2,x}}~
    0 \leq
    \frac{\prob\pr{X_{t+_\Delta}=y\cond X_u=x_u, X_t=x}}{\Delta}
    \leq \epsilon_2.
  \end{equation}
  Similarly, it is clear that \(\pr{X_{t+_\Delta}\geq \upx}\subseteq\pr{X_{t+_\Delta}\geq x+2}\) if \(x+2\leq\upx\), whence it follows that for all \(x\) in \(\chi\) with \(x+2\leq\upx\),
  \begin{equation}
  \label{eqn:Proof of Bound on derivative of the transition matrix of a consistent process:prob of upx or more}
    \pr{\forall\Delta\in\posreals, \Delta<\delta_{2,x}}~
    0 \leq
    \frac{\prob\pr{X_{t+_\Delta}\geq \upx\cond X_u=x_u, X_t=x}}{\Delta}
    \leq \epsilon_2.
  \end{equation}
  Additionally, we use the finite additivity of \(\prob\) and Equations~\eqref{eqn:Proof of Bound on derivative of the transition matrix of a consistent process:prob of one} and \eqref{eqn:Proof of Bound on derivative of the transition matrix of a consistent process:prob of two or more} with \(x=\upx-1\), to yield
  \begin{equation}
  \label{eqn:Proof of Bound on derivative of the transition matrix of a consistent process:prob of one or more}
    \pr{\forall\Delta\in\posreals, \Delta<\delta_{1,\upx-1}, \Delta<\delta_{2,\upx-1}}~
    \abs*{\frac{\prob\pr{X_{t+\Delta}\geq\upx\cond X_u=x_u, X_t=\upx-1}}{\Delta}+\lambda_{\upx-1,\Delta}}
    \leq \epsilon_1+\epsilon_2.
  \end{equation}
  To bound \(\prob\pr{X_{t+\Delta}=x\cond X_u=x_u, X_t=x}\), we recall from Lemma~\ref{lem:Coherent prob on cpdomain:prob of <= x} that
  \begin{equation*}
    \prob\pr{X_{t+\Delta}=x\cond X_u=x_u, X_t=x}
    = 1-\prob\pr{X_{t+\Delta}=x+1\cond X_u=x_u, X_t=x}-\prob\pr{X_{t+\Delta}\geq x+2\cond X_u=x_u, X_t=x}.
  \end{equation*}
  We now combine this equality and Equations~\eqref{eqn:Proof of Bound on derivative of the transition matrix of a consistent process:prob of one} and \eqref{eqn:Proof of Bound on derivative of the transition matrix of a consistent process:prob of two or more}, to yield that for all \(x\) in \(\chi'\),
  \begin{equation}
  \label{eqn:Proof of Bound on derivative of the transition matrix of a consistent process:prob of none}
    \pr{\forall\Delta\in\posreals, \Delta<\delta_{1,x}, \Delta<\delta_{2,x}}~
    \abs*{\frac{\prob\pr{X_{t+\Delta}=x\cond X_u=x_u, X_t=x}-1}{\Delta}+\lambda_{x,\Delta}}
    \leq \epsilon_1+\epsilon_2.
  \end{equation}

  Let \(\delta\coloneqq\min\cup_{x\in\chi'}\set{\delta_{1,x}, \delta_{2,x}}\), and fix any \(\Delta\) in \(\posreals\) with \(\Delta<\delta\).
  Let \(\trm^\chi_\Delta\) be the element of \(\setoftrm^\chi\) that is characterised by the sequence~\(\set{\lambda_{x, \Delta}}_{x\in\chi'}\) in \(\Lambda\), as explained in Lemma~\ref{lem:ltro^chi:Set of dominated transition rate matrices}.
  To verify the stated, we now set out to bound
  \begin{equation*}
    \norm*{\frac{T^\chi_{x_u, t, t+\Delta}-I}{\Delta}-\trm^\chi_\Delta}
    = \max_{x\in\chi} \sum_{y\in\chi} \abs*{\frac{T^\chi_{x_u, t, t+\Delta}\pr{x, y}-I\pr{x, y}}{\Delta}-\trm^\chi_\Delta\pr{x, y}},
  \end{equation*}
  where the equality follows from Equation~\eqref{eqn:Norm of linear transformation of finite state space}.
  To that end, we take a closer look at the expression on the right for all \(x\) in \(\chi\).
  Observe first that
  \begin{equation*}
    \sum_{y\in\chi} \abs*{\frac{T^\chi_{x_u, t, t+\Delta}\pr{x, y}-I\pr{x, y}}{\Delta}-\trm^\chi_\Delta\pr{x, y}}
    = \sum_{y=x}^{\upx} \abs*{\frac{T^\chi_{x_u, t, t+\Delta}\pr{x, y}-I\pr{x, y}}{\Delta}-\trm^\chi_\Delta\pr{x, y}},
  \end{equation*}
  because, by construction, \(T^\chi_{x_u,t,t+\Delta}\pr{x,y}=0=\trm^\chi\pr{x,y}\) and, by definition, \(I\pr{x,y}=0\) for all \(y\) in \(\chi\) with \(y<x\).
  We now use the definitions of \(T^\chi_{x_u,t,t+\Delta}\) and \(\trm^\chi_\Delta\) to rewrite the expression on the right hand side.
  If \(x+2\leq\upx\), then this yields
  \begin{align*}
    \sum_{y\in\chi} \abs*{\frac{T^\chi_{x_u, t, t+\Delta}\pr{x, y}-I\pr{x, y}}{\Delta}-\trm^\chi_\Delta\pr{x, y}}
    &= \abs*{\frac{\prob\pr{X_{t+\Delta}=x\cond X_u=x_u, X_t=x}-1}{\Delta}+\lambda_{x,\Delta}}\\&\qquad+\abs*{\frac{\prob\pr{X_{t+\Delta}=x+1\cond X_u=x_u, X_t=x}}{\Delta}-\lambda_{x,\Delta}}\\&\qquad\qquad+\sum_{y=x+2}^{\upx}\abs*{\frac{\prob\pr{X_{t+\Delta}=y\cond X_u=x_u, X_t=x}}{\Delta}}\\&\qquad\qquad\qquad+\abs*{\frac{\prob\pr{X_{t+\Delta}\geq \upx\cond X_u=x_u, X_t=x}}{\Delta}}\\
    &\leq \epsilon_1+\epsilon_2+\epsilon_1+\pr{\upx-x-2}\epsilon_2+\epsilon_2
    \leq 2\epsilon_1+\card{\chi}\epsilon_2
    \leq \epsilon,
  \end{align*}
  where for the first inequality we have used Equations~\eqref{eqn:Proof of Bound on derivative of the transition matrix of a consistent process:prob of one}, \eqref{eqn:Proof of Bound on derivative of the transition matrix of a consistent process:prob of y}, \eqref{eqn:Proof of Bound on derivative of the transition matrix of a consistent process:prob of upx or more} and \eqref{eqn:Proof of Bound on derivative of the transition matrix of a consistent process:prob of none}.
  If \(x+1=\upx\), then this yields
  \begin{align*}
    \sum_{y\in\chi} \abs*{\frac{T^\chi_{x_u, t, t+\Delta}\pr{x, y}-I\pr{x, y}}{\Delta}-\trm^\chi_\Delta\pr{x, y}}
    &= \abs*{\frac{\prob\pr{X_{t+\Delta}=x\cond X_u=x_u, X_t=x}-1}{\Delta}+\lambda_{x,\Delta}}\\&\qquad+\abs*{\frac{\prob\pr{X_{t+\Delta}\geq \upx\cond X_u=x_u, X_t=x}}{\Delta}-\lambda_{x,\Delta}} \\
    &\leq \epsilon_1+\epsilon_2+\epsilon_1+\epsilon_2
    \leq 2\epsilon_1+\card{\chi}\epsilon_2
    \leq \epsilon,
  \end{align*}
  where for the first inequality we have used Equations~\eqref{eqn:Proof of Bound on derivative of the transition matrix of a consistent process:prob of one or more} and \eqref{eqn:Proof of Bound on derivative of the transition matrix of a consistent process:prob of none}.
  Finally, if \(x=\upx\), then it follows from Lemma~\ref{lem:T^chi is a linear transition transformation} that
  \begin{equation*}
    \sum_{y\in\chi} \abs*{\frac{T^\chi_{x_u, t, t+\Delta}\pr{x, y}-I\pr{x, y}}{\Delta}-\trm^\chi_\Delta\pr{x, y}}
    = \abs*{\frac{\br{T^\chi_{x_u,t,t+\Delta}\indica{\upx}}\pr{\upx}-1}{\Delta}} \\
    = \abs*{\frac{\indica{\upx}\pr{\upx}-1}{\Delta}}
    = 0
    \leq \epsilon.
  \end{equation*}
  From these three cases, we infer that
  \begin{equation*}
    \norm*{\frac{T^\chi_{x_u, t, t+\Delta}-I}{\Delta}-\trm^\chi_\Delta}
    = \max_{x\in\chi} \sum_{y\in\chi} \abs*{\frac{T^\chi_{x_u, t, t+\Delta}\pr{x, y}-I\pr{x, y}}{\Delta}-\trm^\chi_\Delta\pr{x, y}}
    \leq \epsilon,
  \end{equation*}
  as required.
\end{proof}
We now use Lemma~\ref{lem:Bound on derivative of the transition matrix of a consistent process} to establish a result---heavily inspired by \cite[Lemma~F.1]{2017Krak}---that will be essential in the proof of Proposition~\ref{prop:Lower bound for conditional expectation of consistent process}.
\begin{lemma}
\label{lem:Heine-Borel generated sequence}
  Consider a counting process~\(\prob\) that is consistent with \(\Lambda=\br{\llambda, \ulambda}\), some \(t, s\) in \(\nnegreals\) with \(t<s\), some \(u\) in \(\setoftseq_{<t}\) and some \(x_{u}\) in \(\stsp_{u}\).
  Fix some \(\lowx, \upx\) in \(\stsp\) with \(x_{\max u}\leq\lowx\leq \upx\), and let \(\chi\coloneqq\set{x\in\stsp\colon \lowx\leq x\leq \upx}\).
  Then for all \(\epsilon, \delta\) in \(\posreals\), there is a \(v=t_0, \dots, t_n\) in \(\setoftseq_{\br{t, s}}\) such that \(\sigma\pr{v} < \delta\) and
  \begin{equation*}
    \pr{\forall i \in \set{1, \dots, n}}
    \pr{\exists \trm^\chi_i \in \setoftrm^\chi}
    ~\norm{T^\chi_{x_{u}, t_{i-1}, t_{i}} - \pr{I+\Delta_i\trm^\chi_i}}
    \leq \Delta_i\epsilon.
  \end{equation*}
\end{lemma}
\begin{proof}
  Our proof is almost entirely equivalent to the proof of \cite[Lemma~F.1]{2017Krak}; the only difference is that we invoke Lemma~\ref{lem:Bound on derivative of the transition matrix of a consistent process} instead of \cite[Proposition~4.7]{2017Krak}.
  Therefore, and also because it is rather lengthy, we have chosen to omit the proof.
\end{proof}

\subsubsection{Eventually Constant Functions}
Before we consider general bounded functions, we first limit ourselves to eventually constant functions.
We first establish the following useful intermediary result.
\begin{lemma}
\label{lem:f(X_s) is a simple function}
  Consider an \(f\) in \(\setofcafn\pr{\stsp}\) that is constant from \(\upx\), some \(s\) in \(\nnegreals\) and some \(u\) in \(\setoftseq{}\) such that \(\max u\leq s\).
  Then
  \begin{equation*}
    f\pr{X_s}
    = \sum_{x=0}^{\upx-1} f\pr{x}\indica{\pr{X_s=x}} + f\pr{\upx}\indica{\pr{X_s\geq \upx}}
  \end{equation*}
  such that \(f\pr{X_s}\) is an \(\cpfield_u\)-simple function.
\end{lemma}
\begin{proof}
  It is easy to see that
  \begin{equation*}
    f\pr{X_s}
    = \sum_{y=0}^{\upx-1} f\pr{y}\indica{y}\pr{X_s} + f\pr{\upx} \indica{\geq \upx}\pr{X_s}
    = \sum_{y=0}^{\upx-1} f\pr{y}\indica{\pr{X_s=y}} + f\pr{\upx} \indica{\pr{X_s\geq \upx}}.
  \end{equation*}
  As all events in the indicators are clearly contained in \(\cpfield_u\), it furthermore follows that \(f\pr{X_s}\) is an \(\cpfield_u\)-simple function.
\end{proof}
The following result is inspired by \cite[Lemma~F.2]{2017Krak}, and one of our main reasons for introducing the notation \(T^{\chi}_{x_u,t,s}\).
\begin{lemma}
\label{lem:Product expression of the conditional expectation of any counting process}
  Consider a counting process~\(\prob\).
  Fix some \(f\) in \(\setofcafn\pr{\stsp}\) that is constant from \(\upx\), some \(t, s\) in \(\nnegreals\) with \(t\leq s\), some \(u\) in \(\setoftseq_{<t}\) and some \(\pr{x_u, x_t}\) in \(\stsp_{u\cup t}\).
  If \(x_t<\upx\), then for any \(v=t_0, t_1, \dots, t_n\) in \(\setoftseq_{\br{t,s}}\),
  \begin{equation*}
    \prev_\prob\pr{f\pr{X_s}\cond X_u=x_u, X_t=x_t}
    = \br*{T^\chi_{x_u,t,t_1} \prod_{i=2}^n T^\chi_{x_{u\cup t}, t_{i-1}, t_i} f^\chi}\pr{x_t},
  \end{equation*}
  where \(\chi\coloneqq\set{x\in\stsp\colon x_t\leq x\leq \upx}\), and \(f^\chi\) is the restriction of \(f\) to \(\chi\).
\end{lemma}
\begin{proof}
  Fix some \(\upx\) in \(\stsp\) such that \(f\) is constant starting from \(\upx\).
  Our proof is one by induction.
  First, it is an immediate consequence of Lemma~\ref{lem:f(X_s) is a simple function} and Equation~\eqref{eqn:Expectation of a simple function} that
  \begin{equation*}
    \prev_\prob\pr{f\pr{X_s}\cond X_u=x_u, X_t=x_t}
    = \sum_{y=0}^{\upx-1} f\pr{y} \prob\pr{X_s=y\cond X_u=x_u, X_t=x_t} + f\pr{\upx} \prob\pr{X_s\geq \upx\cond X_u=x_u, X_t=x_t}.
  \end{equation*}
  We use Lemma~\ref{lem:Coherent prob on cpdomain:prob of <= x}, that \(f^\chi\) is the restriction of \(f\) to \(\chi\) and Equation~\eqref{eqn:T^chi}, to yield
  \begin{align*}
    \prev_\prob\pr{f\pr{X_s}\cond X_u=x_u, X_t=x_t}
    &= \sum_{y=x_t}^{\upx-1} f\pr{y} \prob\pr{X_s=y\cond X_u=x_u, X_t=x_t} + f\pr{\upx} \prob\pr{X_s\geq \upx\cond X_u=x_u, X_t=x_t} \\
    &= \sum_{y=x_t}^{\upx-1} f^{\chi}\pr{y} \prob\pr{X_s=y\cond X_u=x_u, X_t=x_t} + f^{\chi}\pr{\upx} \prob\pr{X_s\geq \upx\cond X_u=x_u, X_t=x_t} \\
    &= \br{T^\chi_{x_u, t, s} f^\chi}\pr{x_t},
  \end{align*}
  as required.

  For the induction step, we fix some \(n\) in \(\nats\) with \(n\geq2\) and assume that the stated holds for any sequence~\(v\) of length~\(n'+1\), with \(n'\) in \(\nats\) such that \(1\leq n'<n\).
  The stated then follows for any sequence~\(v\) of length~\(n+1\), as we will now prove.
  We start with applying the induction hypothesis to the sequence \(t_0, t_2, \dots, t_n\), to yield
  \begin{align}
    \prev_\prob\pr{f\pr{X_s}\cond X_u=x_u, X_t=x_t}
    &= \br*{T^\chi_{x_u,t,t_2} \prod_{i=3}^n T^\chi_{x_{u\cup t}, t_{i-1}, t_i} f^\chi}\pr{x_t} \notag \\
    &= \sum_{y_2=x_t}^{\upx-1} \prob\pr{X_{t_2}=y_2\cond X_u=x_u, X_t=x_t} \br*{\prod_{i=3}^n T^\chi_{x_{u\cup t}, t_{i-1}, t_i} f^\chi}\pr{y_2} \notag \\
    &\qquad\qquad + \prob\pr{X_{t_2}\geq \upx\cond X_u=x_u, X_t=x_t} \br*{\prod_{i=3}^n T^\chi_{x_{u\cup t}, t_{i-1}, t_i} f^\chi}\pr{\upx}.
    \label{eqn:Proof of Product expression of the conditional expectation of any counting process:First equality}
  \end{align}

  We now substitute the probabilities in this sum with an expanded expression
  From Lemma~\ref{lem:Coherent prob on cpdomain:from t to s via r}, it follows that, for any \(y_2\) in \(\chi\setminus\set{\upx}\),
  \begin{equation}
  \label{eqn:Proof of Product expression of the conditional expectation of any counting process:Second equality}
    \prob\pr{X_{t_2}=y_2\cond X_u=x_u, X_t=x_t}
    = \sum_{y_1=x_t}^{y_2} \prob\pr{X_{t_1}=y_1\cond X_u=x_u, X_t=x_t} \prob\pr{X_{t_2}=y_2\cond X_u=x_u, X_t=x_t, X_{t_1}=y_1},
  \end{equation}
  and
  \begin{multline}
  \label{eqn:Proof of Product expression of the conditional expectation of any counting process:Third equality}
    \prob\pr{X_{t_2}\geq \upx\cond X_u=x_u, X_t=x_t}
    = \sum_{y_1=x_t}^{\upx-1} \prob\pr{X_{t_1}=y_1\cond X_u=x_u, X_t=x_t} \prob\pr{X_{t_2}\geq \upx\cond X_u=x_u, X_t=x_t, X_{t_1}=y_1} \\ + \prob\pr{X_{t_1}\geq \upx\cond X_u=x_u, X_t=x_t}.
  \end{multline}
  We substitute Equations~\eqref{eqn:Proof of Product expression of the conditional expectation of any counting process:Second equality} and \eqref{eqn:Proof of Product expression of the conditional expectation of any counting process:Third equality} in Equation~\eqref{eqn:Proof of Product expression of the conditional expectation of any counting process:First equality}, to yield
  \begin{align*}
    \MoveEqLeft\prev_\prob\pr{f\pr{X_s}\cond X_u=x_u, X_t=x_t} \\
    &=\sum_{y_2=x_t}^{\upx-1}\sum_{y_1=x_t}^{y_2} \prob\pr{X_{t_1}=y_1\cond X_u=x_u, X_t=x_t} \prob\pr{X_{t_2}=y_2\cond X_u=x_u, X_t=x_t, X_{t_1}=y_1} \br*{\prod_{i=3}^n T^\chi_{x_{u\cup t}, t_{i-1}, t_i} f^\chi}\pr{y_2} \\
    &\quad+ \sum_{y_1=x_t}^{\upx-1} \prob\pr{X_{t_1}=y_1\cond X_u=x_u, X_t=x_t} \prob\pr{X_{t_2}\geq \upx\cond X_u=x_u, X_t=x_t, X_{t_1}=y_1} \br*{\prod_{i=3}^n T^\chi_{x_{u\cup t}, t_{i-1}, t_i} f^\chi}\pr{\upx} \\
    &\qquad + \prob\pr{X_{t_1}\geq \upx\cond X_u=x_u, X_t=x_t} \br*{\prod_{i=3}^n T^\chi_{x_{u\cup t}, t_{i-1}, t_i} f^\chi}\pr{\upx}.
  \end{align*}
  Recall from Lemma~\ref{lem:T^chi is a linear transition transformation} that \(\br*{\prod_{i=2}^n T^\chi_{x_{u\cup t}, t_{i-1}, t_i} f^\chi}\pr{\upx}=\br*{\prod_{i=3}^n T^\chi_{x_{u\cup t}, t_{i-1}, t_i} f^\chi}\pr{\upx}\).
  We substitute this in the previous equality and change the summation order, to yield
  \begin{align*}
    \MoveEqLeft\prev_\prob\pr{f\pr{X_s}\cond X_u=x_u, X_t=x_t} \\
    &=\sum_{y_1=x_t}^{\upx-1}\sum_{y_2=y_1}^{\upx-1} \prob\pr{X_{t_1}=y_1\cond X_u=x_u, X_t=x_t} \prob\pr{X_{t_2}=y_2\cond X_u=x_u, X_t=x_t, X_{t_1}=y_1} \br*{\prod_{i=3}^n T^\chi_{x_{u\cup t}, t_{i-1}, t_i} f^\chi}\pr{y_2} \notag \\
    &\quad+ \sum_{y_1=x_t}^{\upx-1} \prob\pr{X_{t_1}=y_1\cond X_u=x_u, X_t=x_t} \prob\pr{X_{t_2}\geq \upx\cond X_u=x_u, X_t=x_t, X_{t_1}=y_1} \br*{\prod_{i=3}^n T^\chi_{x_{u\cup t}, t_{i-1}, t_i} f^\chi}\pr{\upx} \\
    &\qquad + \prob\pr{X_{t_1}\geq \upx\cond X_u=x_u, X_t=x_t} \br*{\prod_{i=2}^n T^\chi_{x_{u\cup t}, t_{i-1}, t_i} f^\chi}\pr{\upx}.
  \end{align*}
  Finally, we use the definition of \(T^\chi_{x_{u\cup t},t_1,t_2}\) and \(T^\chi_{x_u,t,t_1}\) to yield the stated:
  \begin{align*}
    \prev_\prob\pr{f\pr{X_s}\cond X_u=x_u, X_t=x_t}
    &=\sum_{y_1=x_t}^{\upx-1} \prob\pr{X_{t_1}=y_1\cond X_u=x_u, X_t=x_t}\br*{\prod_{i=2}^n T^\chi_{x_{u\cup t}, t_{i-1}, t_i} f^\chi}\pr{y_2} \\&\qquad\qquad + \prob\pr{X_{t_1}\geq \upx\cond X_u=x_u, X_t=x_t} \br*{\prod_{i=2}^n T^\chi_{x_{u\cup t}, t_{i-1}, t_i} f^\chi}\pr{\upx} \\
    &= \br*{T^\chi_{x_u, t, t_1} \prod_{i=2}^n T^\chi_{x_{u\cup t}, t_{i-1}, t_i} f^\chi}\pr{x_t}.
  \end{align*}
\end{proof}

The following two propositions are essential to the proof of Theorem~\ref{the:Lower conditional expectation with respect to all consistent processes}.
\begin{proposition}
\label{prop:Lower bound for conditional expectation of consistent process}
  Consider any counting process~\(\prob\) that is consistent with the rate interval~\(\Lambda\).
  Fix any \(f\) in \(\setofcafn\pr{\stsp}\), any \(t, s\) in \(\nnegreals\) with \(t\leq s\) and any \(u\) in \(\setoftseq_{<t}\).
  Then for any \(\pr{x_u, x}\) in \(\stsp_{u\cup t}\),
  \begin{equation*}
    \lpoisprev_t^s\pr{f\cond x}
    \leq\prev_\prob\pr{f\pr{X_s}\cond X_u=x_u, X_t=x}.
  \end{equation*}
\end{proposition}
\begin{proof}
  Our proof is for a large part similar to that of \cite[Proposition~8.1]{2017Krak}.
  Let \(\upx\) be in \(\stsp\) such that \(f\) is constant after \(\upx\).
  From Lemma~\ref{lem:f(X_s) is a simple function} and Equation~\eqref{eqn:Expectation of a simple function}, it then follows that
  \begin{equation*}
    \prev_\prob\pr{f\pr{X_s}\cond X_u=x_u, X_t=x}
    = \sum_{y=0}^{\upx-1} f\pr{y} \prob\pr{X_s=y\cond X_u=x_u, X_t=x} + f\pr{\upx}\prob\pr{X_s\geq\upx\cond X_u=x_u, X_t=x},
  \end{equation*}
  where we let the empty sum equal zero.
  Recall from Lemma~\ref{lem:Coherent prob on cpdomain:prob of <= x} that \(\prob\pr{X_s=y\cond X_u=x_u, X_t=x}=0\) for all \(y\) in \(\stsp\) with \(y<x\).
  Therefore,
  \begin{equation}
  \label{eqn:Proof of Lower bound for conditional expectation of consistent process:Expectation as sum}
    \prev_\prob\pr{f\pr{X_s}\cond X_u=x_u, X_t=x}
    = \sum_{y=x}^{\upx-1} f\pr{y} \prob\pr{X_s=y\cond X_u=x_u, X_t=x} + f\pr{\upx}\prob\pr{X_s\geq\upx\cond X_u=x_u, X_t=x}.
  \end{equation}

  We distinguish two cases.
  First, we consider the case \(x>\upx\).
  In this case, \(\pr{X_u=x_u, X_t=x}\subseteq\pr{X_s\geq \upx}\) due to \ref{Paths:Non-decreasing}, such that \(\prob\pr{X_s\geq\upx\cond X_u=x_u, X_t=x}=1\) due to \ref{LOP:1 if C in A}.
  We substitute this equality in Equation~\eqref{eqn:Proof of Lower bound for conditional expectation of consistent process:Expectation as sum}, to yield
  \begin{equation*}
      \prev_\prob\pr{f\pr{X_s}\cond X_u=x_u, X_t=x}
      = f\pr{\upx}.
  \end{equation*}
  As furthermore \(\lpoisprev_t^s\pr{f\cond x}=f\pr{\upx}\) due to Proposition~\ref{prop:eventually bounded functions:lto f equals lto^chi f}, this verifies the inequality of the statement in this case.

  Next, we consider the case \(x\leq\upx\), and distinguish two additional cases.
  If \(s=t\), then the equality of the statement follows immediately from Equation~\eqref{eqn:Proof of Lower bound for conditional expectation of consistent process:Expectation as sum}, some obvious properties of counting processes and Proposition~\ref{prop:Properties of lto_S induced by GenPoisGen}~\ref{prop:Properties of lto_S by GenPoisGen:IDentity}:
  \begin{equation*}
    \prev_\prob\pr{f\pr{X_s}\cond X_u=x_u, X_t=x}
    = f\pr{x}
    = \br{If}\pr{x}
    = \br{\lto_t^t f}\pr{x}
    = \lpoisprev_t^s\pr{f\cond x}.
  \end{equation*}

  Hence, from now on we furthermore assume that \(t<s\).
  Let \(\chi\coloneqq\set{y\in\stsp\colon x\leq y\leq\upx}\), and recall from Proposition~\ref{prop:eventually bounded functions:lto f equals lto^chi f} that \(\lpoisprev_t^s\pr{f\cond x}=\br{\lto^\chi_{t,s} f^\chi}\pr{x}\).
  Thus, the equality of the statement is verified if we can show that
  \begin{equation}
  \label{eqn:Proof of Lower bound for conditional expectation of consistent process:lower bound with lto^chi}
    \br{\lto^\chi_{t,s} f^\chi}\pr{x}
    \leq \prev_\prob\pr{f\pr{X_s}\cond X_u=x_u, X_t=x}.
  \end{equation}
  To that end, we fix any \(\epsilon\) in \(\posreals\), and choose any \(\epsilon_1, \epsilon_2\) in \(\posreals\) such that \(\epsilon_1\norm{f^\chi}\leq\epsilon/2\) and \(\epsilon_2\pr{s-t}\norm{f^\chi}\leq\epsilon/2\).
  It follows from Lemma~\ref{lem:ltro^chi is a lower transition rate operator} and Proposition~\ref{prop:Gen ltro^chi:Induced lto^chi} that there is some \(\delta\) in \(\posreals\) such that \(\delta\norm{\ltro^\chi}\leq 2\) and
  \begin{equation}
  \label{eqn:Proof of Lower bound for conditional expectation of consistent process:intermed eps bound 1}
    \pr{\forall v\in\setoftseq_{\br{t,s}}, \sigma\pr{v}\leq\delta}~
    \norm{\lto^\chi_{t,s}-\Phi^\chi_v}
    \leq \epsilon_1.
  \end{equation}
  As \(\prob\) is consistent with~\(\br{\llambda, \ulambda}\), it follows from Lemma~\ref{lem:Bound on derivative of the transition matrix of a consistent process} that there is some \(\Delta_1\) in \(\posreals\) with \(\Delta_1 < \max\set{\delta, s-t}\) and some \(\trm^\chi_1\) in \(\setoftrm^\chi\) such that
  \begin{equation}
  \label{eqn:Proof of Lower bound for conditional expectation of consistent process:intermed eps bound 2}
    \norm{T^\chi_{x_u,t,t+\Delta}-\pr{I+\Delta_1\trm^\chi_1}}
    \leq \Delta_1\epsilon_2.
  \end{equation}
  Furthermore, since \(t_1\coloneqq t+\Delta_1<s\) by construction, it follows from Lemma~\ref{lem:Heine-Borel generated sequence} that there is a sequence \(v'=t_1, t_2, \dots, t_n\) in \(\setoftseq_{\br{t_1, s}}\) with \(\sigma\pr{v'}<\delta\) and some \(\trm^\chi_2\), \dots, \(\trm^\chi_n\) in \(\setoftrm^\chi\) such that
  \begin{equation}
  \label{eqn:Proof of Lower bound for conditional expectation of consistent process:intermed eps bound 3}
    \pr{\forall i\in\set{2, \dots, n}}~
    \norm{T^{\chi}_{x_{u\cup t},t_{i-1}, t_i}-\pr{I+\Delta_i\trm^\chi_i}}
    \leq \Delta_i\epsilon_2.
  \end{equation}
  Let \(t_0\coloneqq t\) and \(v^\star\coloneqq t_0, t_1, \dots, t_n\).
  Recall from Lemma~\ref{lem:Product expression of the conditional expectation of any counting process} that
  \begin{equation*}
    \prev_\prob\pr{f\pr{X_s}\cond X_u=x_u, X_t=x}
    = \br*{T^\chi_{x_u,t,t_1} \prod_{i=2}^{n}T^{\chi}_{x_{u\cup t},t_{i-1},t_i}f^\chi}\pr{x},
  \end{equation*}
  such that
  \begin{equation*}
    \abs*{\prev_\prob\pr{f\pr{X_s}\cond X_u=x_u, X_t=x} - \br*{\prod_{i=1}^{n} \pr{I+\Delta_i\trm^\chi_i} f^\chi}\pr{x}}
    = \abs*{\br*{T^\chi_{x_u,t,t_1} \prod_{i=1}^{n}T^{\chi}_{x_{u\cup t},t_{i},t_{i+1}} f^\chi}\pr{x} - \br*{\prod_{i=1}^{n} \pr{I+\Delta_i\trm^\chi_i} f^\chi}\pr{x}}.
  \end{equation*}
  We now use \ref{BNH: norm Af <= norm A norm f}, Lemma~\ref{lem: bound on norm of A_1 ... A_k - B_1 ... B_k} and Equations~\eqref{eqn:Proof of Lower bound for conditional expectation of consistent process:intermed eps bound 2} and \eqref{eqn:Proof of Lower bound for conditional expectation of consistent process:intermed eps bound 3}, to yield
  \begin{align*}
    \MoveEqLeft\abs*{\prev_\prob\pr{f\pr{X_s}\cond X_u=x_u, X_t=x} - \br*{\prod_{i=1}^{n} \pr{I+\Delta_i\trm^\chi_i} f^\chi}\pr{x}} \\
    &\leq \norm*{T^\chi_{x_u,t,t_1} \prod_{i=2}^{n}T^{\chi}_{x_{u\cup t},t_{i-i},t_i}f^\chi - \prod_{i=1}^{n} \pr{I+\Delta_i\trm^\chi_i}f^\chi}
    \leq \norm*{T^\chi_{x_u,t,t_1} \prod_{i=2}^{n}T^{\chi}_{x_{u\cup t},t_{i-1},t_i} - \prod_{i=1}^{n} \pr{I+\Delta_i\trm^\chi_i}} \norm{f^\chi} \\
    &\leq \norm{T^\chi_{x_u,t,t_1} - \pr{I+\Delta_1\trm^\chi_1}} \norm{f^\chi} + \sum_{i=2}^n\norm{T^{\chi}_{x_{u\cup t},t_{i-1},t_i} - \pr{I+\Delta_i\trm^\chi_i}} \norm{f^\chi} \\
    &\leq \sum_{i=1}^n \Delta_i \epsilon_2 \norm{f^\chi}
    = \pr{s-t}\epsilon_2\norm{f^\chi}
    \leq \frac{\epsilon}2.
  \end{align*}
  Furthermore, it follows from \ref{BNH: norm Af <= norm A norm f} and Equation~\eqref{eqn:Proof of Lower bound for conditional expectation of consistent process:intermed eps bound 1}, which holds because \(\sigma\pr{v^\star}<\delta\) by construction, that
  \begin{equation*}
    \abs*{\br{\lto^\chi_{t,s} f^\chi}\pr{x}-\br*{\Phi^\chi_{v^\star} f^\chi}\pr{x}}
    \leq \norm*{\lto^\chi_{t,s}f^\chi-\Phi^\chi_{v^\star}f^\chi}
    \leq \norm*{\lto^\chi_{t,s}-\Phi^\chi_{v^\star}}\norm{f^\chi}
    \leq \epsilon_1\norm{f^\chi}
    \leq \frac{\epsilon}2.
  \end{equation*}
  Combining the two previous inequalities, we find that
  \begin{equation*}
    \br{\lto^\chi_{t,s} f^\chi}\pr{x}
    \leq \br*{\prod_{i=1}^n\pr{I+\Delta_i\ltro^\chi} f^\chi}\pr{x} + \frac{\epsilon}2
    \leq\br*{\prod_{i=1}^n\pr{I+\Delta_i\trm^\chi_i} f^\chi}\pr{x} + \frac{\epsilon}2
    \leq \prev_\prob\pr{f\pr{X_s}\cond X_u=x_u, X_t=x} + \epsilon,
  \end{equation*}
  where the second inequality holds due to Corollary~\ref{cor:I+delta ltro^chi f^chi <= I+delta trm^chi f^chi}.
  Since \(\epsilon\) was an arbitrary positive real number, this implies Equation~\eqref{eqn:Proof of Lower bound for conditional expectation of consistent process:lower bound with lto^chi}, as required.
\end{proof}

\begin{lemma}
\label{lem:System results in consistent counting process}
  Consider a counting transformation system~\(\mathcal{T}\) of the form of Corollary~\ref{cor:Constructed counting transformation system}.
  Let \(\tilde{\prob}\) be the real-valued map with domain
  \begin{equation*}
    \tilde{\ccpdomain}
    \coloneqq \set{\pr{X_{t+\Delta}=y, \pr{X_u=x_u, X_t=x}}\in\cpdomain\colon t,\Delta\in\nnegreals, u\in\setoftseq_{<t}, \pr{x_u,x}\in\stsp_{u\cup t}, y\in\stsp} \cup \set{\pr{X_0=x, \setofpths{}}\in\cpdomain{}\colon x\in\stsp},
  \end{equation*}
  that is defined for all \(t, \Delta\) in \(\nnegreals\), \(u\) in \(\setoftseq_{<t}\), \(\pr{x_u, x}\) in \(\stsp_{u\cup t}\) and \(y\) in \(\stsp\) as
  \begin{equation*}
    \tilde{\prob}\pr{X_{t+\Delta}=y \cond X_u=x_u, X_t=x}
    \coloneqq \br{T_t^{t+\Delta}\indica{y}}\pr{x}
  \end{equation*}
  and for all \(x\) in \(\stsp\) as
  \begin{equation*}
    \tilde{\prob}\pr{X_0=x \cond \setofpths}
    \coloneqq \begin{cases}
      1 &\text{if } x = 0, \\
      0 &\text{otherwise}.
    \end{cases}
  \end{equation*}
  Then \(\tilde{\prob}\) is coherent, and any extension of \(\tilde{\prob}\) to \(\cpdomain{}\) is a counting process that is consistent with \(\Lambda\).
\end{lemma}
\begin{proof}
  Fix some \(u=t_0, \dots, t_n\) in \(\setofnetseq\) with \(t_0=0\) and, for all \(i\) in \(\set{0,\dots,n}\), some sequence \(S_i\coloneqq\set{\lambda_{i,x}}_{x\in\stsp}\) in \(\br{\llambda, \ulambda}\).
  Due to Corollary~\ref{cor:Constructed counting transformation system},
  \begin{equation*}
    \mathcal{T}
    \coloneqq \mathcal{T}^{\br{0, t_1}}_{S_0}\otimes\mathcal{T}^{\br{t_1, t_2}}_{S_1}\otimes\cdots\otimes\mathcal{T}^{\br{t_{n-1}, t_n}}_{S_{n-1}}\otimes\mathcal{T}^{[t_n, +\infty)}_{S_n}
  \end{equation*}
  is a counting transformation system.
  Therefore, it follows from Lemma~\ref{lem:System results in counting process} that \(\tilde{\prob}\) is coherent, and that any coherent extension~\(\tilde{\prob}^\star\)  of \(\tilde{\prob}\) to \(\cpdomain{}\) is a counting process.
  Hence, all that remains for us is to prove that any such coherent extension \(\tilde{\prob}^\star\) is consistent with \(\br{\llambda, \ulambda}\).
  We will only verify Equation~\eqref{eqn:Consistent rate from right}, Equation~\eqref{eqn:Consistent rate from left} can be verified in a similar fashion.
  To that end, we fix any \(t\) in \(\nnegreals\), \(u\) in \(\setoftseq_{<t}\), \(\Delta\) in \(\posreals\) and \(\pr{x_u, x}\) in \(\stsp_{u\cup t}\).
  Observe that
  \begin{equation*}
    \frac{\tilde{\prob}\pr{X_{t+\Delta}=x+1\cond X_u=x_u,X_t=x}}{\Delta}
    = \frac{\br{T_t^{t+\Delta}\indica{x+1}}\pr{x}}{\Delta}.
  \end{equation*}
  We let \(i\) be the greatest element of \(\set{0, \dots, n}\) such that \(t_i\leq t\).
  We now claim that
  \begin{equation*}
    \lim_{\Delta\to0^+} \frac{\br{T_t^{t+\Delta}\indica{x+1}}\pr{x}}\Delta
    = \lambda_{i,x}.
  \end{equation*}
  If this were true, then
  \begin{equation*}
    \lim_{\Delta\to0^+}\frac{\tilde{\prob}^\star\pr{X_{t+\Delta}=x+1\cond X_u=x_u,X_t=x}}{\Delta}
    = \lambda_{i,x}
  \end{equation*}
  which implies Equation~\eqref{eqn:Consistent rate from right} because \(\llambda\leq\lambda_{i,x}\leq\ulambda\).

  We now set out to verify our claim.
  By construction, \(T_t^{t+\Delta}=T_{t,S_i}^{t+\Delta}\) for any \(\Delta\) in \(\nnegreals\) with \(t_i+\Delta<t_{i+1}\) if \(i<n\), such that
  \begin{equation*}
    \abs*{\frac{\br{T_t^{t+\Delta}\indica{x+1}}\pr{x}}{\Delta}-\lambda_{i,x}}
    = \abs*{\frac{\br{T_{t,S_i}^{t+\Delta}\indica{x+1}}\pr{x}}{\Delta}-\lambda_{i,x}}.
  \end{equation*}
  We use the two obvious equalities \(\br{I\indica{x+1}}\pr{x}\) and \(\br{\trm_{S_i} \indica{x+1}}\pr{x}=\lambda_{i,x}\) and \ref{BNH: norm Af <= norm A norm f}, to yield
  \begin{align*}
    \abs*{\frac{\br{T_t^{t+\Delta}\indica{x+1}}\pr{x}}{\Delta}-\lambda_{i,x}}
    &= \abs*{\frac{\br{T_{t,S_i}^{t+\Delta}\indica{x+1}}\pr{x}-\br{I\indica{x+1}}\pr{x}}{\Delta}-\br{\trm_{S_i} \indica{x+1}}\pr{x}} \notag \\
    &\leq \norm*{\frac{T_{t,S_i}^{t+\Delta}\indica{x+1}-I\indica{x+1}}{\Delta}-\br{\trm_{S_i} \indica{x+1}}}
    \leq \norm*{\frac{T_{t,S_i}^{t+\Delta}-I}{\Delta}-\trm_{S_i}}\norm{\indica{x+1}}
    = \norm*{\frac{T_{t,S_i}^{t+\Delta}-I}{\Delta}-\trm_{S_i}}.
  \end{align*}
  Fix any \(\epsilon\) in \(\posreals\).
  Because, by definition, \(\mathcal{T}_{S_i}\) is the family of (linear) lower counting transformations induced by the generalised Poisson generator \(\trm_{S_i}=\ltro_{S'_i}\) characterised by the sequence \(S'_i=\set{\pr{\lambda_{i,x},\lambda_{i,x}}}_{x\in\stsp}\), it follows from Lemma~\ref{lem:ltro satisfies differential equation} that there is a \(\delta\) in \(\posreals\) such that for any \(\Delta\) in \(\posreals\) with \(\Delta<\delta\),
  \begin{equation}
    \norm*{\frac{T_{t,S_i}^{t+\Delta}-I}\Delta-\trm_{S_i}}
    = \norm*{\frac{T_{t, S_i}^{t+\Delta}-T_{t, S_i}^t}\Delta-\trm_{S_i}T_{t, S_i}^t}
    \leq \epsilon,
  \end{equation}
  where the equality follows from Proposition~\ref{prop:Properties of lto_S induced by GenPoisGen}~\ref{prop:Properties of lto_S by GenPoisGen:IDentity}.
  For any \(\Delta\) in \(\posreals\) such that \(\Delta<\delta\) and \(t_i+\Delta<t_{i+i}\) if \(i<n\), it follows from the previous two inequalities that
  \begin{equation*}
    \abs*{\frac{\br{T_t^{t+\Delta}\indica{x+1}}\pr{x}}{\Delta}-\lambda_{i,x}}
    \leq \epsilon.
  \end{equation*}
  Since \(\epsilon\) was any positive real number, we infer from this inequality that
  \begin{equation*}
    \lim_{\Delta\to0^+} \frac{\br{T_t^{t+\Delta}\indica{x+1}}\pr{x}}\Delta
    = \lambda_{i,x},
  \end{equation*}
  as required.
\end{proof}
\begin{proposition}
\label{prop:Lower bound for conditional expectation is reached by a consistent process}
  Consider any \(t,s\) in \(\nnegreals\) with \(t\leq s\), any \(u\) in \(\setoftseq_{<t}\), any \(\pr{x_u,x}\) in \(\stsp_{u\cup t}\) and any \(f\) in \(\setofcafn\pr{\stsp}\).
  Then for any \(\epsilon\) in \(\posreals\), there is a counting process~\(\prob\) that is consistent with the rate interval~\(\Lambda\) such that
  \begin{equation*}
    \abs*{\lpoisprev_t^s\pr{f\cond x} - \prev_\prob\pr{f\pr{X_s}\cond X_u=x_u, X_t=x}}
    \leq \epsilon.
  \end{equation*}
\end{proposition}
\begin{proof}
  By Lemma~\ref{lem:lto:Approximate with prod of T_S}, there is a sequence \(v=t_0, \dots, t_n\) in \(\setoftseq_{\br{t,s}}\) and, for all \(i\) in \(\set{1, \dots, n}\), a sequence \(S_i=\set{\lambda^i_x}_{x\in\stsp}\) in \(\br{\llambda, \ulambda}\) such that
  \begin{equation}
  \label{eqn:Proof of Lower bound for conditional expectation is reached by a consistent process}
    \abs*{\br{\lto_t^s f}\pr{x} - \br*{\prod_{i=1}^n T_{t_{i-1}, S_i}^{t_i} f}\pr{x}}
    \leq \norm*{\lto_t^s f - \prod_{i=1}^n T_{t_{i-1}, S_i}^{t_i} f}
    \leq \epsilon.
  \end{equation}
  Furthermore, we fix any arbitrary sequence \(S=\set{\lambda_x}_{x\in\stsp}\) in \(\br{\llambda,\ulambda}\).
  It now follows from Corollary~\ref{cor:Constructed counting transformation system} that
  \begin{equation*}
    \mathcal{T}
    = \set{T_r^q \colon r,q\in\nnegreals, r\leq q}
    \coloneqq
    \mathcal{T}_{S}^{\br{0, t_1}}\otimes\mathcal{T}_{S_1}^{\br{t_1, t_2}}\otimes\mathcal{T}_{S_3}^{\br{t_2, t_3}}\otimes\cdots\otimes\mathcal{T}_{S_n}^{\br{t_{n-1},t_n}}\otimes\mathcal{T}_{S}^{[t_n, +\infty)}
  \end{equation*}
  is a counting transformation system.
  Furthermore, it follows from Lemma~\ref{lem:System results in consistent counting process} that there is a counting process~\(\prob\) that is consistent with \(\br{\llambda, \ulambda}\) and that satisfies
  \begin{equation}
  \label{eqn:Proof of Lower bound for conditional expectation is reached by a consistent process:Trans prob}
    \prob\pr{X_s=y\cond X_u=x_u, X_t=x}
    = \br{T_t^s\indica{y}}\pr{x}
    = \br*{\prod_{i=1}^nT_{t_{i-1}, S_i}^{t_i}\indica{y}}\pr{x}
    \qquad\text{for all } y\in\stsp.
  \end{equation}
  Observe furthermore that for any \(y\) in \(\stsp\) with \(y\geq 1\),
  \begin{align}
    \prob\pr{X_s\geq y\cond X_u=x_u, X_t=x}
    &= 1 - \prob\pr{X_s\leq y-1\cond X_u=x_u, X_t=x}
    = 1 - \sum_{z=0}^{y-1}\prob\pr{X_s=z\cond X_u=x_u, X_t=x} \notag \\
    &= 1-\sum_{z=0}^{y-1}\br{T_t^s \indica{z}}\pr{x}
    = \br*{T_t^s \pr*{1-\sum_{z=0}^{y-1}\indica{z}}}\pr{x}
    = \br{T_t^s \indica{\geq y}}\pr{x},
  \label{eqn:Proof of Lower bound for conditional expectation is reached by a consistent process:Trans prob geq}
  \end{align}
  where for the third equality we have used Equation~\eqref{eqn:Proof of Lower bound for conditional expectation is reached by a consistent process:Trans prob} and for the fourth equality we have used the linearity of the linear counting transformation~\(T_t^s\).

  Fix some \(\upx\) in \(\stsp\) such that \(f\) is constant starting from \(\upx\).
  By Lemma~\ref{lem:f(X_s) is a simple function} and Equation~\eqref{eqn:Expectation of a simple function},
  \begin{equation*}
    \prev_\prob\pr{f\pr{X_s}\cond X_u=x_u, X_t=x}
    = \sum_{y=0}^{\upx-1} f\pr{y} \prob\pr{X_s=y\cond X_u=x_u, X_t=x} + f\pr{\upx}\prob\pr{X_s\geq\upx\cond X_u=x_u, X_t=x}.
  \end{equation*}
  We now substitute Equations~\eqref{eqn:Proof of Lower bound for conditional expectation is reached by a consistent process:Trans prob} and \eqref{eqn:Proof of Lower bound for conditional expectation is reached by a consistent process:Trans prob geq} and again use the linearity of the linear counting transformation~\(T_t^s\), to yield
  \begin{multline*}
    \prev_\prob\pr{f\pr{X_s}\cond X_u=x_u, X_t=x}
    = \sum_{y=0}^{\upx-1} f\pr{y} \br{T_t^s\indica{y}}\pr{x} + f\pr{\upx}\br{T_t^s\indica{\geq\upx}}\pr{x} \\
    = \br*{T_t^s \pr*{\sum_{y=0}^{\upx-1} f\pr{y}\indic{y}+f\pr{\upx}\indic{\geq \upx}}}\pr{x}
    = \br{T_t^s f}\pr{x}
    = \br*{\prod_{i=1}^nT_{t_{i-1}, S_i}^{t_i} f}\pr{x},
  \end{multline*}
  where the final equality holds due to the construction of \(\mathscr{T}\).
  The stated now follows if we substitute this equality in Equation~\eqref{eqn:Proof of Lower bound for conditional expectation is reached by a consistent process}.
\end{proof}
Everything is now set up to prove our main result regarding the expectation of eventually constant functions.
\begin{theorem}
\label{the:Lower conditional expectation with respect to all consistent processes}
  For any \(t,s\) in \(\nnegreals\) with \(t\leq s\), \(u\) in \(\setoftseq_{<t}\), \(f\) in \(\setofcafn\pr{\stsp}\) and \(\pr{x_u, x}\) in \(\stsp_{u\cup t}\),
  \begin{equation*}
    \lprev_\Lambda\pr{f\pr{X_s}\cond X_u=x_u, X_t=x}
    = \lpoisprev_t^s\pr{f\cond x}.
  \end{equation*}
\end{theorem}
\begin{proof}
  On the one hand, it follows from Proposition~\ref{prop:Lower bound for conditional expectation of consistent process} and Equation~\eqref{eqn:Lower expectation with respect to set of counting processes} that
  \begin{equation*}
    \lpoisprev_t^s\pr{f\pr{X_s}\cond x}
    \leq \lprev_\Lambda\pr{f\pr{X_s}\cond X_u=x_u, X_t=x}
    = \inf\set{\prev_\prob\pr{f\pr{X_s}\cond X_u=x_u, X_t=x} \colon \prob\in \setofconscproc{\Lambda}}.
  \end{equation*}
  On the other hand, it follows from Proposition~\ref{prop:Lower bound for conditional expectation is reached by a consistent process} that
  \begin{equation*}
    \pr{\forall \epsilon\in\posreals}\pr{\exists\prob^\star\in\setofconscproc{\Lambda}}~
    \abs*{\lpoisprev_t^s\pr{f\cond x}-\prev_{\prob^\star}\pr{f\pr{X_s}\cond X_u=x_u, X_t=x}}
    \leq \epsilon.
  \end{equation*}
  From these two inequalities, it follows that for any \(\epsilon\) in \(\posreals\),
  \begin{equation*}
    \lprev_\Lambda\pr{f\pr{X_s}\cond X_u=x_u, X_t=x}
    \geq \lpoisprev_t^s\pr{f\pr{X_s}\cond x}
    \geq \prev_{\prob^\star}\pr{f\pr{X_s}\cond X_u=x_u, X_t=x} - \epsilon
    \geq \lprev_\Lambda\pr{f\pr{X_s}\cond X_u=x_u, X_t=x} - \epsilon.
  \end{equation*}
  The equality of the statement follows from these inequalities because \(\epsilon\) is an arbitrary positive real number.
\end{proof}

\subsubsection{Bounded Functions}
Next, we move from eventually constant functions to general bounded functions.
Essential to our proof of Theorem~\ref{the:Lower conditional expectation of bounded function with respect to all consistent processes} are the following two observations.
\begin{lemma}
\label{lem:upoisprev of I geq upx goes to zero}
  For any \(s, t\) in \(\nnegreals\) with \(t\leq s\),
  \begin{equation*}
    \lim_{\upx\to+\infty} \upoisprev_t^s\pr{\indica{\geq\upx}\cond x}
    = 0,
  \end{equation*}
  where \(\indica{\geq \upx}\) is the indicator of the set \(\set{z\in\stsp\colon z\geq \upx}\).
\end{lemma}
\begin{proof}
  Fix any \(\epsilon\) in \(\posreals\).
  To prove the stated, we need to verify that
  \begin{equation*}
    \pr{\exists x^\star\in\stsp}\pr{\forall \upx\in\stsp, \upx\geq x^\star}~
    0
    \leq \upoisprev_t^s\pr{\indica{\geq \upx}\cond x}
    = -\lpoisprev_t^s\pr{-\indica{\geq \upx}\cond x}
    = -\br{\lto_t^s \pr{-\indica{\geq\upx}}}\pr{x}
    \leq \epsilon.
  \end{equation*}
  To that end, we recall from Theorem~\ref{the:LCT induced by GenPoisGen} that there is a sequence \(u\) in \(\setoftseq_{\br{t,s}}\) such that \(\sigma\pr{u}\norm{\ltro}\leq2\) and
  \begin{equation*}
    \norm{\lto_t^s-\Phi_u}
    \leq \epsilon.
  \end{equation*}
  Let \(n\) be the number of time points in \(u\).
  Furthermore, we let \(x^\star\coloneqq x+n+1\) and fix any \(\upx\) in \(\stsp\) such that \(\upx\geq x^\star\).
  It then follows from Lemma~\ref{lem:Phi_u:Value of f does not matter above treshold} that
  \begin{equation*}
    \br{\Phi_u \pr{-\indica{\geq \upx}}}\pr{x}
    = \br{\Phi_u \pr{-\indica{\leq x+n}\indica{\geq \upx} -\indica{\geq \upx}\pr{x+n}\indica{>x+n}}}\pr{x}
    = \br{\Phi_u 0}\pr{x}
    = 0,
  \end{equation*}
  where for the final equality we have used \ref{LTT:constant}, which holds because \(\Phi_u\) is a lower counting transformation due to Corollary~\ref{cor:Phi_u is LTT}.
  We combine these two observations, to yield
  \begin{align*}
    0
    \leq \upoisprev_t^s\pr{\indica{\geq \upx}\cond x}
    &= -\br{\lto_t^s \pr{-\indica{\geq\upx}}}\pr{x} + \br{\Phi_u \pr{-\indica{\geq \upx}}}\pr{x}
    \leq \abs*{\br{\lto_t^s \pr{-\indica{\geq\upx}}}\pr{x} - \br{\Phi_u \pr{-\indica{\geq \upx}}}\pr{x}} \\
    &\leq \norm*{\lto_t^s \pr{-\indica{\geq \upx}}-\Phi_u \pr{-\indica{\geq \upx}}}
    \leq \norm*{\lto_t^s-\Phi_u}\norm{-\indica{\geq \upx}}
    = \norm*{\lto_t^s-\Phi_u}
    \leq \epsilon,
  \end{align*}
  where for the first equality we have used that \(\upoisprev_t^s\pr{\cdot\cond x}\) is a coherent upper prevision, for the third inequality we have used \ref{BNH: norm Af <= norm A norm f} and for the final equality we have used that \(\norm{-\indica{>\upx}}=1\).
\end{proof}

\begin{lemma}
\label{lem:Consistent counting process:Expectation via limit of eventually constant function}
  Let \(\prob\) be any counting process that is consistent with the rate interval~\(\Lambda\).
  Then for any \(t, s\) in \(\nnegreals\) with \(t\leq s\), \(u\) in \(\setoftseq_{<t}\),  \(\pr{x_u, x}\) in \(\stsp_{u\cup t}\), \(f\) in \(\setoffn\pr{\stsp}\) and \(\epsilon\) in \(\posreals\),
  \begin{equation*}
    \pr{\exists x^\star\in\stsp}
    \pr{\forall\upx\in\stsp, \upx\geq x^\star}
    \pr{\forall \prob\in\setofconscproc{\Lambda}}
    ~\abs{\prev_\prob\pr{f\pr{X_s} \cond X_u=x_u, X_t=x}-\prev_\prob\pr{\br{\indica{\leq \upx}f+f\pr{\upx}\indica{>\upx}}\pr{X_s} \cond X_u=x_u, X_t=x}}
    \leq \epsilon.
  \end{equation*}
\end{lemma}
\begin{proof}
  Fix any \(\upx\) in \(\stsp\), and let \(f_{\upx}\coloneqq f\indica{\leq \upx}+f\pr{\upx}\indica{>\upx}\).
  Observe that
  \begin{equation*}
    f_{\upx}-2\norm{f}\indica{>\upx}
    \leq f
    \leq f_{\upx}+2\norm{f}\indica{>\upx}.
  \end{equation*}
  Let \(\prob\) be any counting process that is consistent with \(\Lambda\).
  Due to the previous inequalities and the monotonicity of \(\prev_\prob\),
  \begin{multline}
  \label{eqn:Proof of Consistent counting process:Expectation via limit of eventually constant function:monotonicity with exp}
    \prev_\prob\pr{\br{f_{\upx}-2\norm{f}\indica{>\upx}}\pr{X_s} \cond X_u=x_u, X_t=x}
    \leq \prev_\prob\pr{f\pr{X_s} \cond X_u=x_u, X_t=x} \\
    \leq \prev_\prob\pr{\br{f_{\upx} +\norm{f}\indica{>\upx}}\pr{X_s} \cond X_u=x_u, X_t=x}.
  \end{multline}
  Because \(f_{\upx}-2\norm{f}\indica{>\upx}\) and \(f_{\upx}\) are both constant starting from \(\upx+1\), it follows from Lemma~\ref{lem:f(X_s) is a simple function} and Equation~\eqref{eqn:Expectation of a simple function} that
  \begin{equation}
  \label{eqn:Proof of Consistent counting process:Expectation via limit of eventually constant function:exp with -indic}
    \prev_\prob\pr{\br{f_{\upx} -2\norm{f}\indica{>\upx}}\pr{X_s}\cond X_u=x_u, X_t=x}
    = \prev_\prob\pr{f_{\upx}\pr{X_s}\cond X_u=x_u, X_t=x} - 2\norm{f}\prev_\prob\pr{\indica{>\upx}\pr{X_s}\cond X_u=x_u, X_t=x}.
  \end{equation}
  Similarly,
  \begin{equation}
  \label{eqn:Proof of Consistent counting process:Expectation via limit of eventually constant function:exp with +indic}
    \prev_\prob\pr{\br{f_{\upx} +2\norm{f}\indica{>\upx}}\pr{X_s}\cond X_u=x_u, X_t=x}
    = \prev_\prob\pr{f_{\upx}\pr{X_s}\cond X_u=x_u, X_t=x} + 2\norm{f}\prev_\prob\pr{\indica{>\upx}\pr{X_s}\cond X_u=x_u, X_t=x}.
  \end{equation}
  We now combine Equations~\eqref{eqn:Proof of Consistent counting process:Expectation via limit of eventually constant function:monotonicity with exp}--\eqref{eqn:Proof of Consistent counting process:Expectation via limit of eventually constant function:exp with +indic}, to yield
  \begin{equation}
  \label{eqn:Proof of Consistent counting process:Expectation via limit of eventually constant function:bound on expectation}
    \abs*{\prev_{\prob}\pr{f\pr{X_s}\cond X_u=x_u, X_t=x} - \prev_{\prob}\pr{f_{\upx}\pr{X_s}\cond X_u=x_u, X_t=x}}
    \leq 2 \norm{f} \prev_\prob\pr{\indica{>\upx}\pr{X_s}\cond X_u=x_u, X_t=x}.
  \end{equation}
  It now follows from Proposition~\ref{prop:Lower bound for conditional expectation of consistent process}, the conjugacy of \(\lpoisprev_t^s\pr{\cdot\cond x}\) and \(\upoisprev_t^s\pr{\cdot\cond x}\) and the obvious equality \(\prev_\prob\pr{\indica{>\upx}\pr{X_s}\cond X_u=x_u, X_t=x}=-\prev_\prob\pr{-\indica{>\upx}\pr{X_s}\cond X_u=x_u, X_t=x}\) that
  \begin{equation}
  \label{eqn:Proof of Consistent counting process:Expectation via limit of eventually constant function:bound on prob}
    \prev_\prob\pr{\indica{>\upx}\pr{X_s}\cond X_u=x_u, X_t=x}
    \leq \upoisprev_t^s\pr{\indica{>\upx}\cond x}.
  \end{equation}
  Fix any \(\epsilon\) in \(\posreals\), and choose any \(\epsilon'\) in \(\posreals\) such that \(2\norm{f}\epsilon'\leq\epsilon\).
  It now follows from  and Lemma~\ref{lem:upoisprev of I geq upx goes to zero} that there is an \(x^\star\) in \(\stsp\) such that if \(\upx\geq x^\star\), then
  \begin{equation}
  \label{eqn:Proof of Consistent counting process:Expectation via limit of eventually constant function:bound on upoisprev}
    \upoisprev\pr{\indica{>\upx}\cond x}
    \leq \epsilon'.
  \end{equation}
  Finally, we now combine Equations~\eqref{eqn:Proof of Consistent counting process:Expectation via limit of eventually constant function:bound on expectation}--\eqref{eqn:Proof of Consistent counting process:Expectation via limit of eventually constant function:bound on upoisprev} and recall that \(f_{\upx}=f\indica{\leq\upx}+f\pr{\upx}\indica{>\upx}\), to yield
  \begin{equation*}
    \pr{\forall \upx\in\stsp, \upx\geq x^\star}
    \pr{\forall \prob\in\setofconscproc{\Lambda}}
    ~\abs*{\prev_\prob\pr{f\pr{X_s}\cond X_u=x_, X_t=x} - \prev_\prob\pr{\br{f\indica{\leq\upx}+f\pr{\upx}\indica{>\upx}}\pr{X_s}\cond X_u=x_, X_t=x}}
    \leq 2\norm{f} \epsilon'
    \leq \epsilon.
  \end{equation*}
\end{proof}

\begin{proofof}{Theorem~\ref{the:Lower conditional expectation of bounded function with respect to all consistent processes}}
  Our proof is similar to the proof of Theorem~\ref{the:Lower conditional expectation with respect to all consistent processes}
  In the first part, we will show that
  \begin{equation}
  \label{eqn: proof of Lower conditional expectation of bounded function with respect to all consistent processes:first equality}
    \lpoisprev_t^s\pr{f\cond x}
    \leq \lprev_\Lambda\pr{f\pr{X_s}\cond X_u=x_u, X_t=x}
    = \inf\set{\prev_\prob\pr{f\pr{X_s}\cond X_u=x_u, X_t=x}\colon \prob\in\setofconscproc{\Lambda}}.
  \end{equation}
  In the second part, we will subsequently show that
  \begin{equation}
  \label{eqn: proof of Lower conditional expectation of bounded function with respect to all consistent processes:second equality}
    \pr{\forall\epsilon\in\posreals}
    \pr{\exists\prob^\star\in\setofconscproc{\Lambda}}~
    \abs*{\lpoisprev_t^s\pr{f\cond x} - \prev_{\prob^\star}\pr{f\pr{X_s}\cond X_u=x_u, X_t=x}}
    \leq \epsilon.
  \end{equation}
  The stated follows from Equations~\eqref{eqn: proof of Lower conditional expectation of bounded function with respect to all consistent processes:first equality} and \eqref{eqn: proof of Lower conditional expectation of bounded function with respect to all consistent processes:second equality}.
  Indeed, from these equations it follows that, for all \(\epsilon\) in \(\posreals\),
  \begin{equation*}
    \lprev_\Lambda\pr{f\pr{X_s}\cond X_u=x_u, X_t=x}
    \geq \lpoisprev_t^s\pr{f\cond x}
    \geq \prev_{\prob^\star}\pr{f\pr{X_s}\cond X_u=x_u, X_t=x}-\epsilon
    \geq \lprev_\Lambda\pr{f\pr{X_s}\cond X_u=x_u, X_t=x}-\epsilon.
  \end{equation*}
  The equality of the statement now follows from these inequalities because \(\epsilon\) is an arbitrary positive real number.

  We now set out to prove Equations~\eqref{eqn: proof of Lower conditional expectation of bounded function with respect to all consistent processes:first equality} and \eqref{eqn: proof of Lower conditional expectation of bounded function with respect to all consistent processes:second equality}.
  To that end, we fix any \(\epsilon\) in \(\posreals\), and choose some \(\epsilon_1, \epsilon_2, \epsilon_3\) in \(\posreals\) such that \(\epsilon_1+\epsilon_2+\epsilon_3\leq\epsilon\).
  Recall from Proposition~\ref{prop:Approximate lto f using lto f indic leq upx} that there is an \(x^\star_1\) such that
  \begin{equation}
  \label{eqn:Proof of Lower conditional expectation of bounded function with respect to all consistent processes:bound on lpoisprev}
    \pr{\forall \upx\in\stsp, \upx\geq x^\star_1}~
    \abs*{\lpoisprev_t^s\pr{f\cond x} - \lpoisprev_t^s\pr{\indica{\leq \upx}f+f\pr{\upx}\indica{>\upx}\cond x}}
    \leq \epsilon_1.
  \end{equation}
  Due to Lemma~\ref{lem:Consistent counting process:Expectation via limit of eventually constant function}, there is an \(x^\star_2\) such that
  \begin{equation}
  \label{eqn:Proof of Lower conditional expectation of bounded function with respect to all consistent processes:bound on prev}
    \pr{\forall \upx\in\stsp, \upx\geq x^\star_2}
    \pr{\forall \prob\in\setofconscproc{\Lambda}}~
    \abs*{\prev_\prob\pr{f\pr{X_s}\cond X_u=x_u, X_t=x} - \prev_\prob\pr{\br{\indica{\leq \upx}f+f\pr{\upx}\indica{>\upx}}\pr{X_s}\cond X_u=x_u, X_t=x}}
    \leq \epsilon_2.
  \end{equation}
  Let \(x^\star\coloneqq\max\set{x^\star_1, x^\star_2}\), and fix any \(\upx\) in \(\stsp\) such that \(\upx\geq x^\star\).
  It now follows from Equations~\eqref{eqn:Proof of Lower conditional expectation of bounded function with respect to all consistent processes:bound on lpoisprev} and \eqref{eqn:Proof of Lower conditional expectation of bounded function with respect to all consistent processes:bound on prev} that, for any \(\prob\) in \(\setofconscproc{\Lambda}\),
  \begin{multline*}
    \lpoisprev_t^s\pr{f\cond x} - \epsilon
    \leq \lpoisprev_t^s\pr{f\cond x} - \epsilon_1-\epsilon_2
    \leq \lpoisprev_t^s\pr{\indica{\leq \upx}f+f\pr{\upx}\indica{>\upx}\cond x} - \epsilon_1 \\
    \leq \prev_\prob\pr{\br{\indica{\leq \upx}f+f\pr{\upx}\indica{>\upx}}\pr{X_s}\cond X_u=x_u, X_t=x} - \epsilon_1
    \leq \prev_\prob\pr{f\pr{X_s}\cond X_u=x_u, X_t=x},
  \end{multline*}
  where the third inequality follows from Proposition~\ref{prop:Lower bound for conditional expectation of consistent process}.
  Since \(\epsilon\) was an arbitrary positive real number, we infer from this inequality that
  \begin{equation*}
    \pr{\forall \prob\in\setofconscproc{\Lambda}}~
    \lpoisprev_t^s\pr{f\cond x}\leq\prev_\prob\pr{f\pr{X_s}\cond X_u=x_u, X_t=x}.
  \end{equation*}
  We combine this with Equation~\eqref{eqn:Lower expectation with respect to set of counting processes} and use that non-strict inequalities are preserved when taking infima, to yield Equation~\eqref{eqn: proof of Lower conditional expectation of bounded function with respect to all consistent processes:first equality}.

  Next, we prove Equation~\eqref{eqn: proof of Lower conditional expectation of bounded function with respect to all consistent processes:second equality}.
  Due to Proposition~\ref{prop:Lower bound for conditional expectation is reached by a consistent process}, there is a \(\prob^\star\) in \(\setofconscproc{\Lambda}\) such that
  \begin{equation*}
    \abs*{\lpoisprev_t^s\pr{\indica{\leq \upx}f+f\pr{\upx}\indica{>\upx} f\cond x} - \prev_{\prob^\star}\pr{\br{\indica{\leq \upx}f+f\pr{\upx}\indica{>\upx}}\pr{X_s}\cond X_u=x_u, X_t=x}}
    \leq \epsilon_3.
  \end{equation*}
  We now use this inequality and Equations~\eqref{eqn:Proof of Lower conditional expectation of bounded function with respect to all consistent processes:bound on lpoisprev} and \eqref{eqn:Proof of Lower conditional expectation of bounded function with respect to all consistent processes:bound on prev}, to yield
  \begin{align*}
    \MoveEqLeft\abs*{\lpoisprev_t^s\pr{f\cond x}-\prev_{\prob^\star}\pr{f\pr{X_s}\cond X_u=x_u, X_t=x}} \\
    &\leq \abs*{\lpoisprev_t^s\pr{f\cond x}-\lpoisprev_t^s\pr{\indica{\leq x^\star} f+f\pr{x^\star}\indica{>x^\star}\cond x}} \\&\qquad + \abs*{\lpoisprev_t^s\pr{\indica{\leq x^\star} f+f\pr{x^\star}\indica{>x^\star}\cond x}-\prev_{\prob^\star}\pr{\br{\indica{\leq x^\star} f+f\pr{x^\star}\indica{>x^\star}}\pr{X_s}\cond X_u=x_u, X_t=x}} \\ &\qquad\qquad + \abs*{\prev_{\prob^\star}\pr{\br{\indica{\leq x^\star} f+f\pr{x^\star}\indica{>x^\star}}\pr{X_s}\cond X_u=x_u, X_t=x}-\prev_{\prob^\star}\pr{f\pr{X_s}\cond X_u=x_u, X_t=x}} \\
    &\leq \epsilon_1+\epsilon_2+\epsilon_3
    \leq \epsilon,
  \end{align*}
  as required by Equation~\eqref{eqn: proof of Lower conditional expectation of bounded function with respect to all consistent processes:second equality}.
\end{proofof}

Next, we consider the special case of monotone bounded functions.
\begin{lemma}
\label{lem:I+Delta GenPoisGen for monotonous functions}
  Fix any \(f\) in \(\setoffn\pr{\stsp}\), \(n\) in \(\nats\) and, for all \(i\) in \(\set{1, \dots, n}\), some \(\Delta_i\) in \(\nnegreals\) with \(\Delta_i\norm{\ltro}\leq2\).
  If \(f\) is non-decreasing, then
  \begin{equation*}
    \prod_{i=1}^n\pr{I+\Delta_i\ltro}f
    =\prod_{i=1}^n\pr{I+\Delta_i\trm_{\llambda}}f
  \end{equation*}
  is non-decreasing.
  Similarly, if \(f\) is non-increasing, then
  \begin{equation*}
    \prod_{i=1}^n\pr{I+\Delta_i\ltro}f
    =\prod_{i=1}^n\pr{I+\Delta_i\trm_{\ulambda}}f
  \end{equation*}
  is non-increasing.
\end{lemma}
\begin{proof}
  We only prove the stated for a non-decreasing \(f\), the proof for non-increasing \(f\) is entirely similar.
  Our proof is one by induction.
  In case \(n=1\), the equality of the stated follows trivially from the definition of \(\ltro\): for any \(x\) in \(\stsp\),
  \begin{align*}
    \br{\pr{I+\Delta_1\ltro} f}\pr{x}
    &= f\pr{x} + \Delta_1\br{\ltro f}\pr{x}
    = f\pr{x} + \Delta_1 \min\set{\lambda f\pr{x+1}-\lambda f\pr{x}\colon\lambda\in\br{\llambda, \ulambda}} \\
    &= f\pr{x} + \Delta_1 \pr{\llambda f\pr{x+1} - \llambda f\pr{x}}
    = f\pr{x} + \Delta_1 \br{\trm_{\llambda} f}\pr{x} \\
    &= \br{\pr{I+\Delta_1\trm_{\llambda}} f}\pr{x}.
  \end{align*}
  That \(\pr{I+\Delta_1\ltro} f\) is non-decreasing as well is easily verified in a similar fashion:
  \begin{align*}
    \br{\pr{I+\Delta_1\ltro} f}\pr{x+1} - \br{\pr{I+\Delta_1\ltro} f}\pr{x}
    &= \br{\pr{I+\Delta_1\trm_{\llambda}} f}\pr{x+1} - \br{\pr{I+\Delta_1\trm_{\llambda}} f}\pr{x} \\
    &= f\pr{x+1} + \Delta_1 \llambda\pr{f\pr{x+2} - f\pr{x+1}} - f\pr{x} - \Delta_1\llambda \pr{f\pr{x+1} - f\pr{x}} \\
    &= \pr{1-\Delta_1\llambda}\pr{f\pr{x+1} - f\pr{x}} + \Delta_1 \llambda \pr{f\pr{x+2} - f\pr{x+1}} \\
    &\geq 0,
  \end{align*}
  where the inequality follows from the inequality \(1-\Delta_1\llambda\geq0\), which holds because \(\llambda\leq\ulambda\) and \(\Delta_12\ulambda=\Delta_1\norm{\ltro}\leq2\), where the first inequality follows from Equation~\eqref{eqn:PoisGen as special case of GenPoisGen} and Lemma~\ref{lem:GenPoisGen:Norm}.

  For the induction step, we fix any \(m\) in \(\nats\) with \(m\geq2\) and assume that the stated then holds for all \(n\) in \(\nats\) with \(n<m\).
  We now show that the stated then follows for \(m\) as well.
  Let
  \begin{equation*}
    g'
    \coloneqq\prod_{i=2}^m\pr{I+\Delta_i\ltro} f
    \qquad\text{and\qquad}
    g
    \coloneqq\prod_{i=1}^m\pr{I+\Delta_i\ltro} f
    = \pr{I+\Delta_1\ltro} g'.
  \end{equation*}
  By the induction hypothesis, \(g'\) is non-decreasing and equal to
  \begin{equation*}
    \prod_{i=2}^m\pr{I+\Delta_i\trm_{\llambda}} f
  \end{equation*}
  HIt follows from this and the induction hypothesis for \(n=1\) that
  \begin{equation*}
    g
    = \pr{I+\Delta_1\ltro} g'
    = \pr{I+\Delta_1\trm_{\llambda}}g'
    = \prod_{i=1}^m\pr{I+\Delta_i\trm_{\llambda}} f
  \end{equation*}
  is non-decreasing, as required.
\end{proof}

\begin{proofof}{Proposition~\ref{prop:Monotone bounded functions:equality between conditional expectations}}
  We only prove the stated for a non-decreasing function \(f\), the proof for a non-increasing \(f\) is entirely similar.
  We first set out to prove that
  \begin{equation}
  \label{eqn:Proof of Monotone bounded functions:equality between conditional expectations:claim of equality}
    \br{\lto_t^s f}\pr{x}
    = \br{T_{t,\llambda}^s f}\pr{x}.
  \end{equation}
  To that end, we fix any \(\epsilon\) in \(\posreals\), and choose any \(\epsilon'\) in \(\posreals\) such that \(2\epsilon'\norm{f}\leq\epsilon\).
  By Theorem~\ref{the:LCT induced by GenPoisGen} and Corollary~\ref{cor:Gen trm:Existence of T_S}, there is a \(u\) in \(\setoftseq_{\br{t,s}}\) with \(\sigma\pr{u}\norm{\ltro}\leq2\)---and, due to Lemma~\ref{lem:GenPoisGen:Norm} and Corollary~\ref{cor:Gen trm:norm}, therefore also \(\sigma\pr{u}\norm{\trm_{\llambda}}\leq2\)---such that
  \begin{equation}
  \label{eqn:Proof of Monotone bounded functions:equality between conditional expectations:Approximation inequalities}
    \norm*{\lto_t^s - \prod_{i=1}^n\pr{I+\Delta_i\ltro}}
    \leq \epsilon'
    \qquad\text{and}\qquad
    \norm*{T_{t,\llambda}^s - \prod_{i=1}^n\pr{I+\Delta_i\trm_{\llambda}}}
    \leq \epsilon'.
  \end{equation}
  Observe now that
  \begin{align*}
    \abs*{\br{\lto_t^s f}\pr{x}-\br{T_{t,\llambda}^s f}\pr{x}}
    &\leq \norm{\lto_t^s f-T_{t,\llambda}^s f}
    = \norm*{\lto_t^s f-\prod_{i=1}^n\pr{I+\Delta_i\ltro}f+\prod_{i=1}^n\pr{I+\Delta_i\ltro}f-T_{t,\llambda}^s f} \\
    &= \norm*{\lto_t^s f-\prod_{i=1}^n\pr{I+\Delta_i\ltro}f+\prod_{i=1}^n\pr{I+\Delta_i\trm_{\llambda}}f-T_{t,\llambda}^s f} \\
    &\leq \norm*{\lto_t^s f-\prod_{i=1}^n\pr{I+\Delta_i\ltro}f}+\norm*{T_{t,\llambda}^s f-\prod_{i=1}^n\pr{I+\Delta_i\trm_{\llambda}}f} \\
    &\leq \norm*{\lto_t^s-\prod_{i=1}^n\pr{I+\Delta_i\ltro}}\norm{f}+\norm*{T_{t,\llambda}^s-\prod_{i=1}^n\pr{I+\Delta_i\trm_{\llambda}}}\norm{f} \\
    &\leq 2\epsilon'\norm{f}
    \leq \epsilon,
  \end{align*}
  where the second equality follows from Lemma~\ref{lem:I+Delta GenPoisGen for monotonous functions}, the second inequality follows from \ref{BNH: norm Af <= norm A norm f} and the penultimate inequality follows from Equation~\eqref{eqn:Proof of Monotone bounded functions:equality between conditional expectations:Approximation inequalities}.
  Since \(\epsilon\) was an arbitrary positive real number, these inequalities imply Equation~\eqref{eqn:Proof of Monotone bounded functions:equality between conditional expectations:claim of equality}.

  The stated basically follows from Equation~\eqref{eqn:Proof of Monotone bounded functions:equality between conditional expectations:claim of equality}.
  To see this, we let \(\prob_\lambda\) be a Poisson process with rate \(\lambda\) in the rate interval~\(\Lambda\).
  Then clearly
  \begin{equation}
  \label{eqn:Proof of Monotone bounded functions:equality between conditional expectations:Obvious inequalities}
    \lprev_\Lambda\pr{f\pr{X_s}\cond X_u=x_u, X_t=x}
    \leq \lprev_\Lambda^\star\pr{f\pr{X_s}\cond X_u=x_u, X_t=x}
    \leq \prev_{\prob_\lambda}\pr{f\pr{X_s}\cond X_u=x_u, X_t=x}.
  \end{equation}
  Recall from Theorem~\ref{the:Lower conditional expectation of bounded function with respect to all consistent processes} that
  \begin{equation}
  \label{eqn:Proof of Monotone bounded functions:equality between conditional expectations:lprev_Lambda}
    \lprev_\Lambda\pr{f\pr{X_s}\cond X_u=x_u, X_t=x}
    = \br{\lto_t^s f}\pr{x}.
  \end{equation}
  Similarly, it follows from Equation~\eqref{eqn:Singleton of Poisson with rate lambda} and Theorem~\ref{the:Lower conditional expectation of bounded function with respect to all consistent processes} that
  \begin{equation}
  \label{eqn:Proof of Monotone bounded functions:equality between conditional expectations:prev_llambda}
    \prev_{\prob_{\llambda}}\pr{f\pr{X_s}\cond X_u=x_u, X_t=x}
    = \lprev_{\br{\llambda, \llambda}}\pr{f\pr{X_s}\cond X_u=x_u, X_t=x}
    = \br{T_{t,\llambda}^s f}\pr{x}.
  \end{equation}
  It now follows from Equation~\eqref{eqn:Proof of Monotone bounded functions:equality between conditional expectations:claim of equality}, Equation~\eqref{eqn:Proof of Monotone bounded functions:equality between conditional expectations:Obvious inequalities} with \(\lambda=\llambda\) and Equations~\eqref{eqn:Proof of Monotone bounded functions:equality between conditional expectations:lprev_Lambda} and \eqref{eqn:Proof of Monotone bounded functions:equality between conditional expectations:prev_llambda} that
  \begin{equation*}
    \prev_{\prob_{\llambda}}\pr{f\pr{X_s}\cond X_u=x_u, X_t=x}
    = \br{T_{t,\llambda}^s f}\pr{x}
    = \br{T_{t,\llambda}^s f}\pr{x}
    = \lprev_\Lambda\pr{f\pr{X_s}\cond X_u=x_u, X_t=x}
    \leq \prev_{\prob_{\llambda}}\pr{f\pr{X_s}\cond X_u=x_u, X_t=x},
  \end{equation*}
  which clearly implies the equality of the statement for non-decreasing functions~\(f\).
\end{proofof}

\subsubsection{Non-Bounded Functions}
Finally, we are ready to make the transition towards bounded-below functions.
\begin{proofof}{Proposition~\ref{prop:Lower and upper expectation of non-decreasing functions}}
  Observe that if \(f\) is non-decreasing and bounded, then the stated follows immediately from Proposition~\ref{prop:Monotone bounded functions:equality between conditional expectations}.
  We therefore only have to prove the stated for a non-decreasing~\(f\) that is not bounded but bounded below.
  Observe that in this case, \(\inf f = f\pr{0}\).

  We now first set out to prove that
  \begin{equation}
  \label{eqn:Proof of Lower and upper expectation of non-decreasing functions:Claim with inequalities}
    \pr{\forall\prob\in\setofconscproc{\Lambda}}
    \prev_{\prob_{\llambda}}\pr{f\pr{X_S}\cond X_u=x_, X_t=x}
    \leq \prev_\prob\pr{f\pr{X_s}\cond X_u=x_u, X_t=x}.
  \end{equation}
  To that end, we fix any \(\prob\) in \(\setofconscproc{\Lambda}\).
  Recall that
  \begin{equation*}
    \prev_\prob\pr{f\pr{X_s}\cond X_u=X_u, X_t=x}
    = \int_{\inf f}^{\sup f} \prob\pr{\set{f\pr{X_s}>\alpha}\cond X_u=x_u, X_t=x}\, \mathrm{d}\alpha
    = \int_{f\pr{0}}^{+\infty} \prob\pr{\set{f\pr{X_s}>\alpha}\cond X_u=x_u, X_t=x}\, \mathrm{d}\alpha.
  \end{equation*}

  We now fix any \(\beta\) in \(\reals\) with \(\beta\geq f\pr{x}\), and let \(y_\beta\) be an element of \(\stsp\) such that \(f\pr{y_\beta}\leq\beta\leq f\pr{y_\beta+1}\)---this is possible because \(f\) is non-decreasing and unbounded.
  Observe that
  \begin{multline*}
    \int_{f\pr{0}}^\beta \prob\pr{\set{f\pr{X_s}>\alpha}\cond X_u=x_u, X_t=x}\, \mathrm{d}\alpha \\
    = \int_{f\pr{0}}^{f\pr{1}}\prob\pr{\set{f\pr{X_s}>\alpha}\cond X_u=x_u, X_t=x}\, \mathrm{d}\alpha
    + \int_{f\pr{1}}^{f\pr{2}}\prob\pr{\set{f\pr{X_s}>\alpha}\cond X_u=x_u, X_t=x}\, \mathrm{d}\alpha \\
    + \cdots + \int_{f\pr{y_\beta}}^{\beta}\prob\pr{\set{f\pr{X_s}>\alpha}\cond X_u=x_u, X_t=x}\, \mathrm{d}\alpha.
  \end{multline*}
  As a consequence of our assumptions on \(f\), it follows that
  \begin{equation}
  \label{eqn:Proof of Lower and upper expectation of non-decreasing functions:From level sets to events}
    \pr{\forall y\in\stsp}
    \pr{\forall\alpha\in[f\pr{y}, f\pr{y+1})}~
    \set{f\pr{X_s}>\alpha}
    = \pr{X_s>y}.
  \end{equation}
  Hence,
  \begin{multline*}
    \int_{f\pr{0}}^\beta \prob\pr{\set{f\pr{X_s}>\alpha}\cond X_u=x_u, X_t=x}\, \mathrm{d}\alpha
    = \sum_{y=0}^{y_\beta-1} \pr{f\pr{y+1}-f\pr{y}}\prob\pr{X_s>y\cond X_u=x_u, X_t=x} \\ + \pr{\beta-f\pr{y_\beta}}\prob\pr{X_s>y_\beta\cond X_u=x_u, X_t=x}
  \end{multline*}
  For any \(y\) in \(\stsp\), it follows from Equation~\eqref{eqn:Expectation of a simple function}---with the \(\cpfield_{u\cup t}\)-simple function \(\indica{>y}\pr{X_s}\)---and Proposition~\ref{prop:Monotone bounded functions:equality between conditional expectations}---with the non-decreasing function \(\indica{>y}\)---that
  \begin{multline*}
    \prob\pr{X_s>y\cond X_u=x_u, X_t=x}
    = \prev_\prob\pr{\indica{>y}\pr{X_s}\cond X_u=x_u, X_t=x} \\
    \geq \prev_{\prob_{\llambda}}\pr{\indica{>y}\pr{X_s}\cond X_u=x_u, X_t=x}
    = \prob_{\llambda}\pr{X_s>y\cond X_u=x_u, X_t=x},
  \end{multline*}
  where \(\prob_{\llambda}\) is the Poisson process with rate~\(\llambda\).
  We combine this inequality with the previous equality, to yield
  \begin{align*}
    \int_{f\pr{0}}^\beta \prob\pr{\set{f\pr{X_s}>\alpha}\cond X_u=x_u, X_t=x}\, \mathrm{d}\alpha
    &\geq \sum_{y=0}^{y_\beta-1} \pr{f\pr{y+1}-f\pr{y}}\prob_{\llambda}\pr{X_s>y\cond X_u=x_u, X_t=x} \\ &\qquad+ \pr{\beta-f\pr{y_\beta}}\prob_{\llambda}\pr{X_s>y_\beta\cond X_u=x_u, X_t=x} \\
    &= \int_{f\pr{0}}^\beta \prob_{\llambda}\pr{\set{f\pr{X_s}>\alpha}\cond X_u=x_u, X_t=x}\, \mathrm{d}\alpha.
  \end{align*}
  We take the limit for \(\beta\) going to \(+\infty\) on both sides of the inequality, to yield Equation~\eqref{eqn:Proof of Lower and upper expectation of non-decreasing functions:Claim with inequalities}:
  \begin{multline*}
    \prev_{\prob}\pr{f\pr{X_s}\cond X_u=x_u, X_t=x}
    = \lim_{\beta\to+\infty} \int_{f\pr{0}}^\beta \prob\pr{\set{f\pr{X_s}>\alpha}\cond X_u=x_u, X_t=x}\,\mathrm{d}\alpha \\
    \geq \lim_{\beta\to+\infty} \int_{f\pr{0}}^\beta \prob_{\llambda}\pr{\set{f\pr{X_s}>\alpha}\cond X_u=x_u, X_t=x}\,\mathrm{d}\alpha
    = \prev_{\prob_{\llambda}}\pr{f\pr{X_s}\cond X_u=x_u, X_t=x}.
  \end{multline*}

  Finally, we verify that the stated.
  On the one hand, we recall from Equation~\eqref{eqn:Relation between lower expectations for consistent processes} that
  \begin{equation*}
    \lprev_\Lambda\pr{f\pr{X_s}\cond X_u=x_u, X_t=x}
    \leq \lprev_\Lambda^\star\pr{f\pr{X_s}\cond X_u=x_u, X_t=x}
    \leq \prev_{\prob_{\llambda}}\pr{f\pr{X_s}\cond X_u=x_u, X_t=x},
  \end{equation*}
  where the final equality holds because \(\prob_{\ulambda}\) belongs to \(\setofconscproc{\Lambda}^\star\).
  On the other hand, because non-strict inequalities are preserved when taking the infimum, it follows from Equations~\eqref{eqn:Lower expectation with respect to set of counting processes} and \eqref{eqn:Proof of Lower and upper expectation of non-decreasing functions:Claim with inequalities} that
  \begin{equation*}
    \lprev_\Lambda\pr{f\pr{X_s}\cond X_u=x_u, X_t=x}
    \geq \prev_{\prob_{\llambda}}\pr{f\pr{X_s}\cond X_u=x_u, X_t=x}.
  \end{equation*}
  The stated equality now follows immediately from these two observations
\end{proofof}

\begin{proofof}{Corollary~\ref{cor:Lower and upper expected number of events}}
  As \(f\pr{X_s}=X_s\) is non-decreasing, it follows from Proposition~\ref{prop:Lower and upper expectation of non-decreasing functions} that
  \begin{equation*}
    \lprev_\Lambda\pr{X_s\cond X_u=x_u, X_t=x}
    = \lprev_\Lambda^\star\pr{X_s\cond X_u=x_u, X_t=x}
    = \prev_{\prob_{\llambda}}\pr{X_s\cond X_u=x_u, X_t=x}
  \end{equation*}
  and
  \begin{equation*}
    \uprev_\Lambda\pr{X_s\cond X_u=x_u, X_t=x}
    = \uprev_\Lambda^\star\pr{X_s\cond X_u=x_u, X_t=x}
    = \prev_{\prob_{\ulambda}}\pr{X_s\cond X_u=x_u, X_t=x}
  \end{equation*}
  The stated now immediately follows if we recall that
  \begin{equation*}
    \prev_{\prob_{\llambda}}\pr{X_s\cond X_u=x_u, X_t=x}
    = \sum_{y=x}^{+\infty} y\pois_{\llambda\pr{s-t}}\pr{y-x}
    = x+\llambda\pr{s-t},
  \end{equation*}
  and
  \begin{equation*}
    \prev_{\prob_{\ulambda}}\pr{X_s\cond X_u=x_u, X_t=x}
    = \sum_{y=x}^{+\infty} y\pois_{\ulambda\pr{s-t}}\pr{y-x}
    = x+\ulambda\pr{s-t},
  \end{equation*}
  where both times the first equality follows from Proposition~\ref{prop:expectation of poisson process as sum}.
\end{proofof}

\begin{lemma}
\label{lem:Expectation with respect to counting process:Via non-negative translation}
  Consider any counting process~\(\prob\).
  Fix any \(t,s\) in \(\nnegreals\) with \(t\leq s\), \(u\) in \(\setoftseq_{<t}\) and \(\pr{x_u, x}\) in \(\stsp_{u\cup t}\).
  Then for any \(f\) in \(\setofbbfn\pr{\stsp}\),
  \begin{equation*}
    \prev_\prob\pr{f\pr{X_s}\cond X_u=x_u, X_t=x}
    = \inf f + \prev_\prob\pr{f'\pr{X_s}\cond X_u=x_u, X_t=x},
  \end{equation*}
  with \(f'\coloneqq f-\inf f\).
\end{lemma}
\begin{proof}
  Follows immediately from the definition of \(\prev_\prob\).
\end{proof}

\begin{proofof}{Theorem~\ref{the:Lower and upper expectation of functions that do not increase to fast}}
  First, we observe that
  \begin{equation*}
    \inf f \leq f \leq f_{\max}.
  \end{equation*}
  Therefore, for any \(\prob\) in \(\setofconscproc{\Lambda}\),
  \begin{equation}
    \inf f
    \leq \prev_{\prob}\pr{f\pr{X_s}\cond X_u=x_u, X_t=x}
    \leq \prev_{\prob}\pr{f_{\max}\pr{X_s}\cond X_u=x_u, X_t=x}
  \end{equation}
  due to the monotonicity of \(\prev_\prob\).
  Since \(f_{\max}\) is clearly a non-decreasing bounded-below function, it follows from Proposition~\ref{prop:Lower and upper expectation of non-decreasing functions} and  Proposition~\ref{prop:expectation of poisson process as sum} that, for any \(\prob\) in \(\setofconscproc{\Lambda}\),
  \begin{equation*}
    \prev_{\prob}\pr{f_{\max}\pr{X_s}\cond X_u=x_u, X_t=x}
    \leq \prev_{\prob_{\ulambda}}\pr{f_{\max}\pr{X_s}\cond X_u=x_u, X_t=x}
    = \sum_{y=x}^{+\infty} f_{\max}\pr{y}\pois_{\ulambda\pr{s-t}}\pr{y-x}
    < +\infty,
  \end{equation*}
  where the final inequality is precisely the condition on \(f\) of the statement.
  Because non-strict inequalities are preserved when taking infima and suprema, we infer from this that
  \begin{equation*}
    \inf f
    \leq \lprev_\Lambda\pr{f\pr{X_s}\cond X_u=x_u, X_t=x}
    \leq \uprev_\Lambda\pr{f\pr{X_s}\cond X_u=x_u, X_t=x}
    \leq \sum_{y=x}^{+\infty} f_{\max}\pr{y}\pois_{\ulambda\pr{s-t}}\pr{y-x}
    < +\infty.
  \end{equation*}
  This already settles the second part of the stated, namely that the lower and upper expectations are finite.

  Next, we set out to prove the equalities of the statement.
  For any \(\upx\) in \(\stsp\), we let \(f_{\upx}\coloneqq f\indica{\leq\upx}+f\pr{\upx}\indica{>\upx}\).
  By definition of the limit, we need to prove that
  \begin{equation}
  \label{eqn:Proof of Lower and upper expectation of functions that do not increase to fast:Eps delta for lower}
    \pr{\forall\epsilon\in\posreals}
    \pr{\exists x^\star\in\stsp}
    \pr{\forall \upx\in\stsp, \upx\geq x^\star}~
    \abs{\lprev_\Lambda\pr{f\pr{X_s}\cond X_u=x_u, X_t=x}-\lpoisprev_t^s\pr{f_{\upx}\cond x}}
    \leq \epsilon
  \end{equation}
  and
  \begin{equation}
  \label{eqn:Proof of Lower and upper expectation of functions that do not increase to fast:Eps delta for upper}
    \pr{\forall\epsilon\in\posreals}
    \pr{\exists x^\star\in\stsp}
    \pr{\forall \upx\in\stsp, \upx\geq x^\star}~
    \abs{\uprev_\Lambda\pr{f\pr{X_s}\cond X_u=x_u, X_t=x}-\upoisprev_t^s\pr{f_{\upx}\cond x}}
    \leq \epsilon.
  \end{equation}
  To that end, we fix any \(\epsilon\) in \(\posreals\), and choose any \(\epsilon_1\) and \(\epsilon_2\) in \(\posreals\) such that \(\epsilon_1+\epsilon_2\leq\epsilon\).

  Our first step is to obtain a bound on
  \begin{equation*}
    \abs{\prev_\prob\pr{f\pr{X_s}\cond X_u=x_u, X_t=x}-\prev_\prob\pr{f_{\upx}\pr{X_s}\cond X_u=x_u, X_t=x}},
  \end{equation*}
  with \(\prob\) in \(\setofconscproc{\Lambda}\) and \(\upx\) in \(\stsp\).
  To that end, we let \(f'\coloneqq f-\inf f\), \(f'_{\max}\coloneqq f_{\max}-\inf f\) and \(f'_{\upx}\coloneqq f'\indica{\leq \upx}+f'\pr{\upx}\indica{>\upx}\) for any \(\upx\) in \(\stsp\).
  Due to the condition on \(f\) of the statement and the properties of the Poisson distribution,
  \begin{equation*}
    \sum_{y=x}^{+\infty} f'_{\max}\pr{y}\pois_{\ulambda\pr{s-t}}\pr{y-x}
    = \sum_{y=x}^{+\infty} \pr{f_{\max}\pr{y}-\inf f}\pois_{\ulambda\pr{s-t}}\pr{y-x}
    = \sum_{y=x}^{+\infty} f_{\max}\pr{y}\pois_{\ulambda\pr{s-t}}\pr{y-x} - \inf f
    < +\infty.
  \end{equation*}
  Hence, there is an \(x^\star\) in \(\stsp\) with \(x^\star\geq x\) such that
  \begin{equation}
  \label{eqn:Proof of the:Lower and upper expectation of functions that do not increase to fast:Bound on sum with f'_max}
    \pr{\forall\upx\in\stsp, \upx\geq x^\star}
    \abs*{\sum_{y=\upx+1}^{+\infty} f'_{\max}\pr{y}\pois_{\ulambda\pr{s-t}}\pr{y-x}}
    = \abs*{\sum_{y=x}^{+\infty} f'_{\max}\pr{y}\pois_{\ulambda\pr{s-t}}\pr{y-x}-\sum_{y=x}^{\upx} f'_{\max}\pr{y}\pois_{\ulambda\pr{s-t}}\pr{y-x}}
    \leq \epsilon_1.
  \end{equation}
  Fix any \(\upx\) in \(\stsp\) with \(\upx\geq x^\star\), and observe that
  \begin{equation}
  \label{eqn:Proof of Lower and upper expectation of functions that do not increase to fast:Bounds on f'}
    \indica{\leq \upx} f'
    \leq f'
    \leq \indica{\leq x} f' + f'_{\max} \indica{>\upx}.
  \end{equation}
  and
  \begin{equation}
  \label{eqn:Proof of Lower and upper expectation of functions that do not increase to fast:Bounds on f'_upx}
    \indica{\leq \upx} f'
    \leq f'_{\upx}
    \leq \indica{\leq x} f' + f'_{\max} \indica{>\upx}.
  \end{equation}
  Fix any \(\prob\) in \(\setofconscproc{\Lambda}\).
  Then due to Equation~\eqref{eqn:Proof of Lower and upper expectation of functions that do not increase to fast:Bounds on f'} and the monotonicity of \(\prev_\prob\),
  \begin{equation}
  \label{eqn:Proof of Lower and upper expectation of functions that do not increase to fast:Monotonicity of prev with f'}
    \prev_\prob\pr{\br{\indica{\leq \upx}f'}\pr{X_s}\cond X_u=x_u, X_t=x}
    \leq \prev_\prob\pr{f'\pr{X_s}\cond X_u=x_u, X_t=x}
    \leq \prev_\prob\pr{\br{\indica{\leq \upx}f'+\indica{>\upx}f'_{\max}}\pr{X_s}\cond X_u=x_u, X_t=x}.
  \end{equation}
  Similarly,
  \begin{equation}
  \label{eqn:Proof of Lower and upper expectation of functions that do not increase to fast:Monotonicity of prev with f'_upx}
    \prev_\prob\pr{\br{\indica{\leq \upx}f'}\pr{X_s}\cond X_u=x_u, X_t=x}
    \leq \prev_\prob\pr{f'_{\upx}\pr{X_s}\cond X_u=x_u, X_t=x}
    \leq \prev_\prob\pr{\br{\indica{\leq \upx}f'+\indica{>\upx}f'_{\max}}\pr{X_s}\cond X_u=x_u, X_t=x}.
  \end{equation}
  Furthermore, it follows from the linearity of \(\prev_\prob\) that
  \begin{equation}
  \label{eqn:Proof of Lower and upper expectation of functions that do not increase to fast:sum decomp of prev}
    \prev_\prob\pr{\br{\indica{\leq \upx}f'+\indica{>\upx}f'_{\max}}\pr{X_s}\cond X_u=x_u, X_t=x}
    = \prev_\prob\pr{\br{\indica{\leq \upx}f'}\pr{X_s}\cond X_u=x_u, X_t=x} + \prev_\prob\pr{\br{\indica{>\upx}f'_{\max}}\pr{X_s}\cond X_u=x_u, X_t=x}.
  \end{equation}
  Alternatively, we obtain Equation~\eqref{eqn:Proof of Lower and upper expectation of functions that do not increase to fast:sum decomp of prev} as follows.
  Recall that
  \begin{equation*}
    \prev_\prob\pr{\br{\indica{\leq \upx}f'+\indica{>\upx}f'_{\max}}\pr{X_s}\cond X_u=x_u, X_t=x}
    = \int_{0}^{\sup f'} \prob\pr{\set{\br{\indica{\leq \upx}f'+\indica{>\upx}f'_{\max}}\pr{X_s}>\alpha}\cond X_u=x_u, X_t=x} \,\mathrm{d}\alpha.
  \end{equation*}
  Observe now that for all \(\alpha\) in \(\nnegreals\),
  \begin{align*}
    \set{\br{\indica{\leq \upx}f'+\indica{>\upx}f'_{\max}}\pr{X_s}>\alpha}
    &= \set{\pth\in\setofpths{}\colon \indica{\leq\upx}\pr{\pth\pr{s}} f'\pr{\pth\pr{s}} + \indica{>\upx}\pr{\pth\pr{s}} f'_{\max}\pr{\pth\pr{s}} > \alpha } \\
    &= \set{\pth\in\setofpths{}\colon \pr*{\pth\pr{s} \leq \upx, f'\pr{\pth\pr{s}} > \alpha} \text{ or } \pr*{\pth\pr{s}>\upx, f'_{\max}\pr{\pth\pr{s}}> \alpha }} \\
    &= \set{\pth\in\setofpths{}\colon \pth\pr{s} \leq \upx, f'\pr{\pth\pr{s}} > \alpha} \cup \set{\pth\in\setofpths{}\colon \pth\pr{s}>\upx, f'_{\max}\pr{\pth\pr{s}} > \alpha} \\
    &= \set{\br{\indica{\leq \upx}f'}\pr{X_s} > \alpha} \cup \set{\br{\indica{>\upx}f'_{\max}}\pr{X_s} > \alpha},
  \end{align*}
  where the union is one of two disjoint sets.
  Furthermore, the first set of this union is clearly empty if \(\alpha\geq f_{\max}\pr{\upx}\).
  We use our decomposition of the level sets, to yield
  \begin{align*}
    \MoveEqLeft
    \prev_\prob\pr{\br{\indica{\leq \upx}f'+\indica{>\upx}f'_{\max}}\pr{X_s}\cond X_u=x_u, X_t=x} \\
    &= \int_{0}^{\sup f'} \prob\pr{\set{\br{\indica{\leq \upx}f'}\pr{X_s}>\alpha}\cond X_u=x_u, X_t=x}+\prob\pr{\set{\br{\indica{>\upx}f'_{\max}}\pr{X_s}>\alpha}\cond X_u=x_u, X_t=x} \,\mathrm{d}\alpha \\
    &= \int_0^{\sup \indica{\leq \upx}f'} \prob\pr{\set{\br{\indica{\leq \upx}f'}\pr{X_s}>\alpha}\cond X_u=x_u, X_t=x} + \int_{0}^{\sup f'} \prob\pr{\set{\br{\indica{>\upx}f'_{\max}}\pr{X_s}>\alpha}\cond X_u=x_u, X_t=x} \,\mathrm{d}\alpha \\
    &= \prev_\prob\pr{\br{\indica{\leq \upx}f'}\pr{X_s}\cond X_u=x_u, X_t=x} + \prev_\prob\pr{\br{\indica{>\upx}f'_{\max}}\pr{X_s}\cond X_u=x_u, X_t=x},
  \end{align*}
  where the for the second equality we have furthermore used the linearity of the (improper) Riemann integral.

  In any case, it now follows from Equations~\eqref{eqn:Proof of Lower and upper expectation of functions that do not increase to fast:Monotonicity of prev with f'}--\eqref{eqn:Proof of Lower and upper expectation of functions that do not increase to fast:sum decomp of prev} that
  \begin{equation*}
    \abs*{\prev_\prob\pr{f'\pr{X_s}\cond X_u=x_u, X_t=x}-\prev_\prob\pr{f'_{\upx}\pr{X_s}\cond X_u=x_u, X_t=x}}
    \leq \prev_\prob\pr{\br{\indica{>\upx}f'_{\max}}\pr{X_s}\cond X_u=x_u, X_t=x}.
  \end{equation*}
  Observe now that \(\indica{>\upx}f'_{\max}\) is a non-decreasing and bounded below function.
  Therefore, it follows from Proposition~\ref{prop:Lower and upper expectation of non-decreasing functions} and Proposition~\ref{prop:expectation of poisson process as sum} that
  \begin{equation*}
    \abs*{\prev_\prob\pr{f'\pr{X_s}\cond X_u=x_u, X_t=x}-\prev_\prob\pr{f'_{\upx}\pr{X_s}\cond X_u=x_u, X_t=x}}
    \leq \sum_{y=\upx+1}^{+\infty}f'_{\max}\pois_{\ulambda\pr{s-t}}\pr{y-x}.
  \end{equation*}
  From this, Lemma~\ref{lem:Expectation with respect to counting process:Via non-negative translation} and Equation~\eqref{eqn:Proof of the:Lower and upper expectation of functions that do not increase to fast:Bound on sum with f'_max}, we now infer that
  \begin{equation}
  \label{eqn:Proof of the:Lower and upper expectation of functions that do not increase to fast:Final bound with eps_1}
    \pr{\forall\upx\in\stsp, \upx\geq x^\star}
    \pr{\forall\prob\in\setofconscproc{\Lambda}}~
    \abs*{\prev_\prob\pr{f\pr{X_s}\cond X_u=x_u, X_t=x}-\prev_\prob\pr{f_{\upx}\pr{X_s}\cond X_u=x_u, X_t=x}}
    \leq \epsilon_1
  \end{equation}

  Next, we recall from Proposition~\ref{prop:Lower bound for conditional expectation is reached by a consistent process} that
  \begin{equation}
  \label{eqn:Proof of Lower and upper expectation of functions that do not increase to fast:Lower reached for eventually constant}
    \pr{\forall \upx\in\stsp}
    \pr{\exists \prob_{\text{l}}\in\setofconscproc{\Lambda}}~
    \abs*{\lpoisprev_t^s\pr{f_{\upx}\cond x} - \prev_{\prob_{\text{l}}}\pr{f_{\upx}\pr{X_s}\cond X_u=x_u, X_t=x}}
    \leq \epsilon_2
  \end{equation}
  and, due to conjugacy, that
  \begin{equation}
  \label{eqn:Proof of Lower and upper expectation of functions that do not increase to fast:Upper reached for eventually constant}
    \pr{\forall \upx\in\stsp}
    \pr{\exists \prob_{\text{u}}\in\setofconscproc{\Lambda}}~
    \abs*{\upoisprev_t^s\pr{f_{\upx}\cond x} - \prev_{\prob_{\text{u}}}\pr{f_{\upx}\pr{X_s}\cond X_u=x_u, X_t=x}}
    \leq \epsilon_2.
  \end{equation}

  Everything is now set up for us to verify Equations~\eqref{eqn:Proof of Lower and upper expectation of functions that do not increase to fast:Eps delta for lower} and \eqref{eqn:Proof of Lower and upper expectation of functions that do not increase to fast:Eps delta for upper}.
  We here only verify the former, the latter follows from entirely similar reasoning.
  Fix any \(\upx\) in \(\stsp\) such that \(\upx\geq x^\star\).
  On the one hand, we observe that
  \begin{align*}
    \lprev_\Lambda\pr{f\pr{X_s}\cond X_u=x_u, X_t=x}
    &\leq \prev_{\prob_{\text{l}}}\pr{f\pr{X_s}\cond X_u=x_u, X_t=x}
    \leq \prev_{\prob_{\text{l}}}\pr{f_{\upx}\pr{X_s}\cond X_u=x_u, X_t=x}+\epsilon_1 \\
    &\leq \lpoisprev_t^s\pr{f_{\upx}\cond x}+\epsilon_1+\epsilon_2
    \leq \lpoisprev_t^s\pr{f_{\upx}\cond x}+\epsilon,
  \end{align*}
  where the first equality holds because \(\prob_{\text{l}}\) belongs to \(\setofconscproc{\Lambda}\), and where for the subsequent inequalities we have used Equation~\eqref{eqn:Proof of the:Lower and upper expectation of functions that do not increase to fast:Final bound with eps_1}, Equation~\eqref{eqn:Proof of Lower and upper expectation of functions that do not increase to fast:Lower reached for eventually constant} and our condition on \(\epsilon_1\) and \(\epsilon_2\).
  On the other hand, we observe that
  \begin{align*}
    \lprev_\Lambda\pr{f\pr{X_s}\cond X_u=x_u, X_t=x}
    &= \inf\set{\prev_{\prob}\pr{f\pr{X_s}\cond X_u=x_u, X_t=x}\colon \prob\in\setofconscproc{\Lambda}} \\
    &\geq \inf\set{\prev_{\prob}\pr{f_{\upx}\pr{X_s}\cond X_u=x_u, X_t=x}-\epsilon_1\colon \prob\in\setofconscproc{\Lambda}} \\
    &= \inf\set{\prev_{\prob}\pr{f_{\upx}\pr{X_s}\cond X_u=x_u, X_t=x}\colon \prob\in\setofconscproc{\Lambda}}-\epsilon_1 \\
    &= \lprev_\Lambda\pr{f_{\upx}\pr{X_s}\cond X_u=x_u, X_t=x}-\epsilon_1 \\
    &= \lpoisprev_t^s\pr{f_{\upx}\cond x}-\epsilon_1
    \geq \lpoisprev_t^s\pr{f_{\upx}\cond x}-\epsilon,
  \end{align*}
  where the first two equalities follow from Equation~\eqref{eqn:Lower expectation with respect to set of counting processes}, the first inequality follows from Equation~\eqref{eqn:Proof of the:Lower and upper expectation of functions that do not increase to fast:Final bound with eps_1} and the final equality follows from Theorem~\ref{the:Lower conditional expectation of bounded function with respect to all consistent processes} because \(f_{\upx}\) is clearly bounded.
  It is now clear that these two observations imply Equation~\eqref{eqn:Proof of Lower and upper expectation of functions that do not increase to fast:Eps delta for lower}, as required.
\end{proofof}

\section{Supplementary Material for Section~\ref{sec:Justification for the term imprecise poisson process}}

\begin{proofof}{Proposition~\ref{prop:Poisson like}}
  We first consider the five properties for \(\lprev_\Lambda^\star\).
  Properties (i) and (iii)--(v) follow almost immediately from Equation~\eqref{eqn:Lower expectation with respect to set of counting processes} and Proposition~\ref{prop:expectation of poisson process as sum}.
  To verify (ii), we observe that, for all \(\Delta\) in \(\posreals\),
  \begin{align*}
    \lprev_\Lambda^\star\pr{\indica{\pr{X_{t+\Delta}\geq x+2}}\cond X_u=x_u, X_t=x}
    &= \inf\set{\prev_{\prob_\lambda}\pr{\indica{\pr{X_{t+\Delta}\geq x+2}}\cond X_u=x_u, X_t=x}\colon \lambda\in\Lambda} \\
    &= \inf\set{\prob_\lambda\pr{X_{x+2}\geq x+2\cond X_u=x_u, X_t=x}\colon \lambda\in\Lambda} \\
    &= \inf\set{1-\pois_{\lambda\Delta}\pr{0}-\pois_{\lambda\Delta}\pr{1}\colon \lambda\in\Lambda} \\
    &= 1-\pois_{\ulambda\Delta}\pr{0}-\pois_{\ulambda\Delta}\pr{1}.
  \end{align*}
  where we have used Equation~\eqref{eqn:Lower expectation with respect to set of counting processes} for the first equality, Equation~\eqref{eqn:Expectation of a simple function} for the second equality and Proposition~\ref{Prop:Pois transition probabilities are poisson distributed} for the third equality.
  Similarly, if \(t>0\), then
  \begin{equation*}
    \lprev_\Lambda^\star\pr{\indica{\pr{X_{t}\geq x+2}}\cond X_u=x_u, X_{t-\Delta}=x}
    1-\pois_{\ulambda\Delta}\pr{0}-\pois_{\ulambda\Delta}\pr{1}.
  \end{equation*}
  Therefore,
  \begin{equation*}
    \lim_{\Delta\to0^+} \frac{\lprev_\Lambda^\star\pr{\indica{\pr{X_{t+\Delta}\geq x+2}}\cond X_u=x_u, X_t=x}}\Delta
    = \lim_{\Delta\to0^+} \frac{1-\pois_{\ulambda\Delta}\pr{0}-\pois_{\ulambda\Delta}\pr{1}}\Delta
    = \lim_{\Delta\to0^+} \frac{1-e^{-\ulambda\Delta}-\ulambda\Delta e^{-\ulambda\Delta}}\Delta
    = 0,
  \end{equation*}
  and similarly for the limit from the left if \(t>0\).

  Next, we consider the five properties for \(\lprev_{\Lambda}\).
  Properties (i), (iii) and (v) follow almost immediately from Theorem~\ref{the:Lower conditional expectation of bounded function with respect to all consistent processes}.
  Property (ii) follows immediately from Theorem~\ref{the:Lower conditional expectation of bounded function with respect to all consistent processes}, Equation~\eqref{eqn:PoisGen as special case of GenPoisGen} and Lemma~\ref{lem:lto_S satisfies orederliness property}.
  Finally, property (iv) follows immediately from Theorem~\ref{the:Lower conditional expectation of bounded function with respect to all consistent processes}, Equation~\eqref{eqn:PoisGen as special case of GenPoisGen} and Lemma~\ref{lem:State homogeneity of lto induced by genpoisgen}.
\end{proofof}

\end{document}